\documentclass[10pt]{amsart}
\usepackage[nohug,heads=littlevee]{diagrams}
\usepackage{bm}
\usepackage{mathrsfs}
\usepackage{amssymb}
\usepackage[latin1]{inputenc}
\usepackage{yfonts}
\usepackage[plainpages=false,pdfpagelabels]{hyperref}
\usepackage[left=3.5cm,top=2.5cm,right=3.6cm,bottom=1.4in,asymmetric]{geometry}
\usepackage{mathtools}
\usepackage{enumitem}
\usepackage[T1]{fontenc}
\usepackage{biblatex}
\usepackage[french,english]{babel}

\makeatletter
\newenvironment{abstracts}{%
  \ifx\maketitle\relax
    \ClassWarning{\@classname}{Abstract should precede
      \protect\maketitle\space in AMS document classes; reported}%
  \fi
  \global\setbox\abstractbox=\vtop \bgroup
    \normalfont\Small
    \list{}{\labelwidth\z@
      \leftmargin3pc \rightmargin\leftmargin
      \listparindent\normalparindent \itemindent\z@
      \parsep\z@ \@plus\p@
      
      \itemsep\medskipamount
    }%
}{%
  \endlist\egroup
  \ifx\@setabstract\relax \@setabstracta \fi
}

\newcommand{\abstractin}[1]{%
  \otherlanguage{#1}%
  \item[\hskip\labelsep\scshape\abstractname.]%
}
\makeatother

\newcommand{\defnword}[1]{\textbf{#1}}
\newcommand{\comment}[1]{}
\newcommand{\on}[1]{\operatorname{#1}}
\newcommand{\bb}[1]{\mathbb{#1}}
\newcommand{\mb}[1]{\mathbf{#1}}
\newcommand{\mf}[1]{\mathfrak{#1}}
\newcommand{\ord}{\on{ord}}

\numberwithin{equation}{subsubsection}
\newtheorem{introthm}{Theorem}

\newtheorem{prp}[subsubsection]{Proposition}
\newtheorem{thm}[subsubsection]{Theorem}
\newtheorem*{thm*}{Theorem}
\newtheorem{lem}[subsubsection]{Lemma}
\newtheorem*{lem*}{Lemma}
\newtheorem{corollary}[subsubsection]{Corollary}
\theoremstyle{definition}
\newtheorem{defn}[subsubsection]{Definition}

\theoremstyle{remark}

\newtheorem{rem}[subsubsection]{Remark}
\newtheorem*{introeg}{Example}
\newtheorem*{assump*}{Assumption}

\newarrow{Equals}{=}{=}{}{=}{}

\newcommand{\Real}{\mathbb{R}}
\newcommand{\Int}{\mathbb{Z}}
\newcommand{\Nat}{\mathbb{N}}
\newcommand{\Comp}{\mathbb{C}}
\newcommand{\Adele}{\mathbb{A}}

\newcommand{\Field}{\mathbb{F}}
\newcommand{\Rat}{\mathbb{Q}}
\newcommand{\Lie}{\on{Lie}}
\newcommand{\Dieu}{\mathbb{D}}

\newcommand{\im}{\on{im}}
\newcommand{\gr}{\on{gr}}
\newcommand{\into}{\hookrightarrow}

\newcommand{\Obj}{\on{Obj}}

\newcommand{\pr}{\bm{\pi}}

\newcommand{\Rg}{\mathscr{O}}
\newcommand{\Reg}[1]{\Rg_{#1}}

\newcommand{\Hom}{\on{Hom}}
\newcommand{\End}{\on{End}}

\newcommand{\Aut}{\on{Aut}}
\newcommand{\Isom}{\on{Isom}}

\newcommand{\SHom}{\underline{\on{Hom}}}

\newcommand{\Gal}{\on{Gal}}

\newcommand{\Proj}{\on{Proj}}

\newcommand{\Gm}{\mathbb{G}_m}

\newcommand{\Gmh}[1]{\mathbb{G}_{m,#1}}

\newcommand{\st}{\on{st}}
\newcommand{\cris}{\on{cris}}
\newcommand{\dR}{\on{dR}}
\newcommand{\Bst}{B_{\st}}
\newcommand{\Bcris}{B_{\cris}}
\newcommand{\Bdr}{B_{\dR}}

\newcommand{\pow}[1]{[\vert#1\vert]}

\newcommand{\Spec}{\on{Spec}}
\newcommand{\Spf}{\on{Spf}}

\newcommand{\et}{\on{\acute{e}t}}
\newcommand{\ab}{\on{ab}}
\newcommand{\sab}{\on{sab}}
\newcommand{\mult}{\on{mult}}

\newcommand{\cl}{\on{cl}}
\newcommand{\wt}{\on{wt}}
\newcommand{\dlog}{\on{dlog}}

\newcommand{\Mon}{\on{M}}
\newcommand{\gp}{\on{gp}}

\newcommand{\Fil}{\on{Fil}}

\newcommand{\mx}{\mathfrak{m}}
\newcommand{\pf}{\mathfrak{p}}

\newcommand{\pol}{\on{pol}}

\newcommand{\rank}{\on{rank}}
\newcommand{\flogdiff}[1]{\hat{\Omega}^{1,\log}_{#1}}
\newcommand{\fdiff}[1]{\hat{\Omega}^1_{#1}}
\newcommand{\dual}[1]{{#1}^{\vee}}

\newcommand{\Sh}{\on{Sh}}
\newcommand{\Ss}{\mathcal{S}}
\newcommand{\MT}{\on{MT}}
\newcommand{\Res}{\on{Res}}
\newcommand{\an}{\on{an}}

\newcommand{\Sym}{\on{Sym}}

\newcommand{\ad}{\on{ad}}
\newcommand{\GSp}{\on{GSp}}

\newcommand{\GL}{\on{GL}}

\newcommand{\PGL}{\on{PGL}}

\newcommand{\GSpin}{\on{GSpin}}

\newcommand{\beef}{\ddagger}

\begin{document}
\title[Toroidal Compactifications]{Toroidal Compactifications of Integral Models of Shimura Varieties of Hodge Type}
\author{Keerthi Madapusi Pera}
\address{Keerthi Madapusi Pera\\%
Department of Mathematics\\%
5734 S University Ave\\%
University of Chicago\\%
Chicago, IL 60637\\%
USA}
\email{keerthi@math.uchicago.edu}

\keywords{Vari\'et\'es de Shimura; compactifications; vari\'et\'es ab\'eliennes; theorie de Dieudonn\'e logarithmique; Shimura varieties; compactifications; abelian varieties; logarithmic Dieudonn\'e theory}

\subjclass{11G18,14G35}

\begin{abstracts}

\abstractin{french}
Nous construisons compactifications toro\"idales pour les mod\`eles entiers de vari\'et\'es de Shimura de type de Hodge. Nous construisons \'egalement la compactification minimale (ou de Satake-Baily-Borel) pour ces mod\'eles entiers. Nos r\'esultats r\'eduisent le probl\`eme \`a la compr\'ehension des mod\`eles entiers eux-m\^emes. Donc ils recouvrent tous les cas d\'ej\'a connus de type PE. Quand le niveau est hyperspecial, nous montrons que nos compactifications sont canonique dans un sens pr\'ecis. Nous fournissons une nouvelle preuve de la conjecture de Y. Morita sur la bonne r\'eduction de vari\'et\'es ab\'eliennes dont le groupe de Mumford-Tate est anisotrope modulo son centre. Sur le chemin, nous d\'emontrons une propri\'et\'e de rationalit\'e int\'eressante de cycles de Hodge sur les vari\'et\'es ab\'eliennes par rapport \`a l'uniformisation analytiques \emph{p}-adiques.

\abstractin{english}
  We construct toroidal compactifications for integral models of Shimura varieties of Hodge type. We also construct integral models of the minimal (Satake-Baily-Borel) compactification. Our results essentially reduce the problem to understanding the integral models themselves. As such, they cover all previously known cases of PEL type. At primes where the level is hyperspecial, we show that our compactifications are canonical in a precise sense. We also provide a new proof of Y. Morita's conjecture on the everywhere good reduction of abelian varieties whose Mumford-Tate group is anisotropic modulo center. Along the way, we demonstrate an interesting rationality property of Hodge cycles on abelian varieties with respect to \emph{p}-adic analytic uniformizations.
\end{abstracts}

\maketitle

\setcounter{tocdepth}{2}

\section*{Introduction}\label{sec:intro}
\subsubsection*{Shimura varieties of Hodge type}
This paper is concerned with constructing compactifications for integral models of Shimura varieties of Hodge type. Essentially, these are the Shimura varieties that can be viewed as parameter spaces for polarized abelian varieties equipped with level structures and additional Hodge tensors.

More precisely, we will work with Shimura data $(G,X)$ that admit embeddings into a Siegel Shimura datum $(\GSp(V),\on{S}^{\pm}(V))$ attached to a symplectic space $V$ over $\Rat$. Given such an embedding and a small enough compact open $K\subset G(\Adele_f)$, we have the associated Shimura variety $\Sh_K(G,X)$, which has the above moduli interpretation over $\Comp$.

If we are in the more familiar PEL setting, the additional Hodge tensors parameterized by $\Sh_K(G,X)(\Comp)$ can be chosen to consist of endomorphisms and polarizations. One can then define representable PEL type moduli problems over the reflex field $E=E(G,X)$, and even over a suitable localization of its ring of integers, which recover the moduli interpretation for $\Sh_K(G,X)$ over $\Comp$, and are thus \emph{canonical} models for $\Sh_K(G,X)$ over $E$ or even its ring of integers; cf.~\cite{deligne:travaux} for the theory over $E$, and \cite{kottwitz:points} for the integral theory (away from primes where the level is not hyperspecial).

The theory of \cite{deligne:travaux} applies more generally to show that Shimura varieties of Hodge type admit canonical models over their reflex fields\footnote{We now know that every Shimura variety admits such a canonical model; cf.~\cite{milne:canonical}.}, and Milne has used Deligne's results on absolute Hodge cycles to give these canonical models a modular interpretation; cf.~\cite{milne:motives}.

\begin{introeg}
  An important class of Shimura data of Hodge type arises from quadratic forms over $\Rat$ of signature $(n+,2-)$. Suppose that we have a vector space $U$ over $\Rat$ equipped with such a quadratic form. Then the group $G=\GSpin(U)$ acts naturally on the Clifford algebra $C$ attached to $U$. We can equip $C$ with an appropriate symplectic form such that we have an embedding $\GSpin(U)\into\GSp(C)$. Moreover, if we take $X$ to be the space of negative definite oriented $2$-planes in $U_{\Real}$, then $(G,X)$ is a Shimura datum, and we in fact get an embedding $(G,X)\into (\GSp(C),\on{S}^{\pm})$ of Shimura data. This is the \defnword{Kuga-Satake construction}; cf.~\cite{deligne:k3}. So $(G,X)$ is of Hodge type, but is not of PEL type when $n\geq 6$.

The Shimura varieties attached to such data are important, for example, in the study of the moduli of K3 surfaces (when $n=19$). Moreover, the Shimura varieties attached to the Spin group Shimura data play a significant role in S. Kudla's program (cf.~\cite{kudla:special_cycles}) for relating intersection numbers on Shimura varieties with Fourier coefficients of Eisenstein series. They have also been used by the author to prove the Tate conjecture for K3 surfaces; cf.~\cite{mp:tatek3}.
\end{introeg}

\subsubsection*{Integral models}

Unfortunately, since Hodge cycles are transcendentally defined, there is no natural way to use them to obtain a modular interpretation over the ring of integers of $E$. But an \emph{ad hoc} construction of integral models can be given as follows: Fix a prime $p$ and a place $v\vert p$ of $E$. Suppose that we have an embedding $(G,X)\into (\GSp,\on{S}^{\pm})$ into a Siegel Shimura datum. For any level ${K}^{\beef}\subset\GSp(\Adele_f)$, the Siegel Shimura variety $\Sh_{{K}^{\beef}}(\GSp,\on{S}^{\pm})$ has a natural integral model $\mathcal{S}_{{K}^{\beef}}$ over $\Int_{(p)}$: this is Mumford's construction.

We can now take the normalization of $\mathcal{S}_{{K}^{\beef}}$ in $\Sh_K(G,X)$\footnote{Here, and in the rest of the paper, given a normal, excellent $\Int_{(p)}$-algebraic stack $S$, an open dense substack $j:U\hookrightarrow S$, and a finite map $f:S'\to U$, with $S'$ normal, the \defnword{normalization of $S$ in $S'$} will be the finite $S$-algebraic stack, whose associated coherent sheaf of $\Reg{S}$-algebras is the normalization of $\Reg{S}$ in $j_*f_*\Reg{S'}$.}: This gives us a normal integral model $\Ss_K$ over $\Reg{E,(v)}$, which is finite over $\Ss_{K^\beef}$.

When $p>2$ and the level at $p$ is hyperspecial, Kisin showed in~\cite{kisin:abelian} that $\Ss_K$ is a smooth scheme $\Reg{E,(v)}$, and is canonical in a precise sense.~\footnote{A related result due to Vasiu can be found in \cite{vasiu:preab}. The result was also extended to the case $p=2$ in~\cite{Kim2016-fb}.}  In particular, it is independent of the choice of symplectic embedding.

In general, however, one does not know if $\Ss_K$ has any good properties. Moreover, it need not be independent of the choice of symplectic embedding.

\subsubsection*{Compactifications}
In any case, since our interest lies mainly in the computation of the zeta function of the Shimura variety, and hence its cohomology, we are led to consider the question of compactifying $\Ss_K$.\footnote{We will see below that this question is largely independent of the properties of the integral model itself.}

Another motivation to study compactifications of integral models is the role they play in constructing regular proper models over $\Int$ for the orthogonal Shimura varieties mentioned above. Such models are a crucial ingredient in carrying out Kudla's program on the arithmetic intersection theory of these spaces; cf.~\cite{mp:reg} for a description of these models over $\Int[1/2]$.

Over $\Comp$, Mumford and his collaborators (cf.~\cite{amrt}) constructed good, toroidal compactifications in the general setting of arithmetic quotients of hermitian symmetric domains. In \cite{harris:functorial} and \cite{pink:thesis} these compactifications are constructed for Shimura varieties in their natural ad\'elic setting. All these constructions depend on a choice of a certain cone decomposition $\Sigma$, called a complete admissible rpcd (cf. \S\ref{sec:strata} for the terminology). Given such a choice they produce a compactification $\Sh^{\Sigma}_K$ of the Shimura variety $\Sh_K\coloneqq\Sh_K(G,X)$ with good properties.

In the Siegel case, when the level ${K}^{\beef}$ is hyperspecial at $p$, Chai and Faltings~\cite{faltings_chai} studied degenerations of abelian varieties, and used this to construct smooth compactifications $\mathcal{S}^{\Sigma'}_{{K}^{\beef}}$ of $\mathcal{S}_{{K}^{\beef}}$ over $\Int_{(p)}$ attached to smooth cone decompositions $\Sigma'$ for the symplectic group. It was shown by K.-W. Lan~\cite{lan:analytic} that this construction is compatible in characteristic $0$ with the analytic construction of Mumford, \emph{et. al.} mentioned above. Lan's proof uses a careful comparison of the algebraic and analytic definitions of theta functions. We give an independent proof of this fact here, using the compatibility of Mumford's construction with cohomological realizations; cf.~\eqref{background:prp:analytic_mumford_comp}.

For the general Hodge type case, there is a natural cone decomposition $\Sigma$ attached to $(G,X)$ such that the normalization of $\Sh^{\Sigma'}_{{K}^{\beef}}$ in $\Sh_K$ is canonically isomorphic to the toroidal compactification $\Sh_K^{\Sigma}$; cf.~\cite{harris:functorial}. 

Now, assume that ${K}^{\beef}$ is chosen to be hyperspecial at $p$ (this can always be arranged using Zarhin's trick, and by replacing the Sigel Shimura data with one associated with a larger symplectic space). If we take the normalization of $\mathcal{S}^{\Sigma^\beef}_{{K}^{\beef}}$ in $\Sh_K$, we obtain a proper normal algebraic space $\Ss^{\Sigma}_K$ over $\Reg{E,(v)}$, whose generic fiber is $\Sh^{\Sigma}_K$, and which contains $\Ss_K$ as an open sub-scheme.

The main result of this paper is:
\begin{introthm}\label{main}
$\Ss^{\Sigma}_K$ is a compactification of $\Ss_K$. More precisely, the complement of $\Ss_K$ in $\Ss^{\Sigma}_K$ is a relative Cartier divisor over $\Reg{E,(v)}$. Moreover, $\Ss^{\Sigma}_K$ admits a stratification of the expected shape, extending that of its generic fiber. After replacing $\Sigma$ by an appropriate refinement if necessary, the singularities of $\Ss^{\Sigma}_K$ are no worse than those of $\Ss_K$: Every complete local ring of $\Ss^{\Sigma}_K$ at a geometric point is isomorphic to a complete local ring of $\Ss_K$.\footnote{The original version of this theorem imposed the further condition that the level subgroup at $p$ satisfy $K_p = K^\beef_p \cap G(\Rat_p)$. We thank the referee for encouraging us to consider the situation at arbitrary level.}
\end{introthm}

For the reader familiar with the general structure of toroidal compactifications, we can unpack Theorem~\ref{main} a bit (cf.~\ref{strata:thm:strata}): $\Ss^{\Sigma}_K$ is stratified by locally closed sub-schemes, and each stratum in this stratification can be described as follows: There is a normal integral model $\Ss_{K_{\Phi,h}}(G_{\Phi,h},D_{\Phi,h})$ over $\Reg{E,(v)}$ of a Shimura variety, a projective scheme $\Ss_{\overline{K}_\Phi}(\overline{Q}_\Phi,\overline{D}_\Phi)\to\Ss_{K_{\Phi,h}}(G_{\Phi,h},D_{\Phi,h})$ that is in many cases of interest a torsor under an abelian group scheme over $\Ss_{K_{\Phi,h}}(G_{\Phi,h},D_{\Phi,h})$; a torus $\mb{E}_K({\Phi})$ over $\Int$; an $\mb{E}_K({\Phi})$-torsor $\Ss_{K_\Phi}(Q_\Phi,D_\Phi)\to\Ss_{\overline{K}_\Phi}(\overline{Q}_\Phi,\overline{D}_\Phi)$; a rational polyhedral cone in $\Real\otimes\mb{B}_K({\Phi})$ (where $\mb{B}_K({\Phi})$ is the co-character group of $\mb{E}_K(\Phi)$) determined by the cone decomposition $\Sigma$ with corresponding twisted torus embedding
\[
 \Ss_{K_\Phi}(Q_\Phi,D_\Phi)\into\Ss_{K_\Phi}(Q_{\Phi},D_{\Phi},\sigma);
\]
such that the stratum is isomorphic to the closed stratum $\mathcal{Z}_{K_\Phi}(Q_\Phi,D_\Phi,\sigma)$ of the twisted toric scheme $\Ss_{K_\Phi}(Q_{\Phi},D_{\Phi},\sigma)$. 

Moreover the formal completion of $\Ss_{K_\Phi}(Q_{\Phi},D_{\Phi},\sigma)$ along $\mathcal{Z}_{K_\Phi}(Q_\Phi,D_\Phi,\sigma)$ is canonically isomorphic to that of $\Ss^{\Sigma}_K$ along the corresponding stratum.

In particular, over a fine enough \'etale neighborhood of any point in this stratum, the open immersion $\Ss_K\into\Ss^{\Sigma}_K$ is isomorphic to an \'etale neighborhood of $\Ss_{K_\Phi}(Q_\Phi,D_\Phi)\into\Ss_{K_\Phi}(Q_{\Phi},D_{\Phi},\sigma)$. 

This allows us to deduce essentially all of Theorem~\ref{main}. For the assertion about the singularities of the integral models, observe that $\Ss_{K_\Phi}(Q_\Phi,D_\Phi)$ is smooth over $\Ss_{\overline{K}_{\Phi}}(\overline{Q}_{\Phi},\overline{D}_{\Phi})$, and, after suitably refining $\Sigma$, so is $\Ss_{K_\Phi}(Q_{\Phi},D_{\Phi},\sigma)$. 

Therefore, the complete local ring at a geometric point of $\Ss^{\Sigma}_K$ lying in the stratum corresponding to $(Q_{\Phi},D_{\Phi},\sigma)$ is isomorphic to a power series ring over a complete local ring of $\Ss_{\overline{K}_{\Phi}}(\overline{Q}_{\Phi},\overline{D}_{\Phi})$. By the same token, one can also find a geometric point of $\Ss_K$ with the same complete local ring.

From this, it is immediate that many \'etale local properties of $\Ss_K$ (and hence of $\Ss_{K_\Phi}(Q_\Phi,D_\Phi)$)---such as smoothness, reducedness, or being Cohen-Macaulay or a local complete intersection---transfer over to $\Ss_{K_\Phi}(Q_{\Phi},D_{\Phi},\sigma)$, and thence to $\Ss^\Sigma_K$.

\subsubsection*{The unramified case}

To say more about the singularities requires stronger hypotheses. In the case where $\Ss_K$ is an integral canonical model constructed by Kisin, $\Ss_{K_{\Phi,h}}(G_{\Phi,h},D_{\Phi,h})$ is also a smooth integral canonical model of its generic fiber. Moreover, one can show via a simple argument that
\[
\Ss_{\overline{K}_\Phi}(\overline{Q}_{\Phi},\overline{D}_\Phi)\to \Ss_{K_{\Phi,h}}(G_{\Phi,h},D_{\Phi,h})
\]
is a torsor under an abelian group scheme, and is hence also smooth.

By choosing a smooth cone decomposition $\Sigma$, we then obtain (cf.~\ref{strata:subsec:smooth}):
\begin{introthm}\label{smooth}
If $K$ is hyperspecial at $p$, the integral canonical model $\Ss_K$ over $\Reg{E,(v)}$ admits smooth projective compactifications $\Ss^{\Sigma}_K$ such that the boundary $\Ss^{\Sigma}_K\backslash\Ss_K$ is a normal crossings divisor. Moreover, these compactifications, as well as their stratifications, depend only on the choice of cone decomposition $\Sigma$, and not on the choice of symplectic embedding used to construct them.
\end{introthm}

% Note that the condition on the prime $p$ is entirely because of a similar restriction in the paper~\cite{kisin:abelian}. 
In the special case where $(G,X)$ is of PEL type, the first assertion of the theorem is originally due to K.-W. Lan~\cite{lan:thesis}.

\subsubsection*{The ramified case}
The most striking and somewhat unexpected (at least, to the author) point about Theorem~\ref{main} is that it is entirely agnostic to the nature of the special fiber of $\Ss_K$. Therefore, we can now construct compactifications at places where the group $G$ is \emph{ramified}.

As observed in the introduction to~\cite{faltings_chai}, in attempting to deal with such situations, one is `\emph{led to very hard new problems which require new methods...}' Our method is to work locally, using $p$-adic Hodge theory, and some notions from basic rigid analytic geometry. This is explained in a bit more detail towards the end of this introduction.

The main benefit of the local method is that we never have to deal with problems concerning the algebraizing and gluing of formal charts---processes that constitute some of the trickier parts of~\cites{faltings_chai,lan:thesis}. Indeed, we have an already constructed global space, and we only have to show that it has the right properties.

In fact, this observation has been made independently by K.-W. Lan, who has proven a version of Theorem~\ref{main} in the special case of PEL type Shimura varieties through more direct means, also without any conditions on the level; cf.~\cite{lan:ramified}. Even in this setting, the group theoretic formulation of our results offers a different perspective, and the flexibility it provides might be helpful in some applications.

In general, for any class of Shimura varieties of Hodge type, once one has a reasonable theory of normal integral models, our results will immediately supply good compactifications. In particular, this paper subsumes the results of~\cite{lan:thesis},~\cite{stroh:thesis} and~\cite{lan:ramified}.

We also note that work of Kisin and Pappas shows that many Shimura varieties of Hodge type with parahoric level have good, normal integral models\cite{kisin_pappas}, whose local properties are governed by those of the local models of~\cite{pappas_zhu}. Our results will immediately apply to give toroidal compactifications of these models, with stratifications that can be described explicitly. With a bit more work, one can show that the Shimura varieties involved in these stratifications are once again Kisin-Pappas integral models. We do not do this work in this article, but hope to return to it in the future.

\subsubsection*{The minimal compactification}
The toroidal compactifications of Mumford, \emph{et. al.} are resolutions of the minimal or Baily-Borel-Satake compactification, which is important from the automorphic perspective, since its $L^2$ or intersection cohomology is intimately related with the discrete automorphic spectrum of $G$; cf.~\cite{morel:icm}. Adopting the methods of~\cite[\S V.2]{faltings_chai},\cite[\S 7.2]{lan:thesis},\cite{chai:hilbmin}, we can construct the integral model for the minimal compactification via the $\Proj$ construction applied to a certain graded ring of automorphic forms on $\Ss^{\Sigma}_K$. This gives us (cf.~\ref{min:subsec:const}):

\begin{introthm}\label{min}
 The minimal compactification of $\Sh_K$ admits a projective, normal model $\Ss_K^{\min}$ over $\Reg{E,(v)}$ that is stratified by quotients by finite groups of integral models of Shimura varieties of Hodge type. Moreover, the Hecke action of $G(\Adele_f^p)$ on $\Ss_K$ extends naturally to an action on $\Ss_K^{\min}$. Given a complete admissible rpcd $\Sigma$ as in Theorem~\ref{main}, there exists a unique map $p_{\Sigma}:\Ss_K^{\Sigma}\rightarrow\Ss_K^{\min}$ that extends the identity on $\Ss_K$ and is compatible with the stratifications on domain and target.

 When the level at $p$ is hyperspecial, $\Ss_K^{\min}$ is canonically determined and is independent of the choice of symplectic embedding.
\end{introthm}

\subsubsection*{A rationality property for Hodge cycles}
Implicit in the main theorem is a new rationality property for Hodge cycles on abelian varieties with respect to $p$-adic uniformizations, closely related to notions considered by Andr\'e in \cite{andre:padic_betti}. To explain this in its simplest form, we will put ourselves in the following situation: Let $A$ be an abelian variety over a number field $F$, whose reduction at a place $v$ of $F$ is totally degenerate: that is, a torus with character group $X_g=\Int^g$, where $g=\dim A$. By a theorem of Raynaud~\cite{raynaud:icm}, we can find a rigid analytic uniformization over $F_v$:
\[
1\to X_g\xrightarrow{\iota}\SHom(X_g,\Gmh{F_v}^{\an})\to A^{\an}_{F_v}\to 1.
\]
In particular, for any prime $\ell$, the $\ell$-adic Tate module $V_{\ell}(A)$ acquires an extension structure as a $\Gal(\overline{F}_v/F_v)$-module:
\[
 0\to X_g\otimes\Rat_{\ell}(1)\to V_{\ell}(A)\to\Hom(X_g,\Rat_{\ell})\to 0.
\]

Let $U_{\ell}\subset\GL(V_{\ell}(A))$ be the unipotent radical of the parabolic sub-group stabilizing this extension structure. Its Lie algebra is now equipped with a natural integral structure:
\[
 \Rat_{\ell}(1)\otimes\Hom(\dual{X}_g,X_g)=\Lie U_{\ell}.
\]

Let $G=\on{MT}(A)$ be the Mumford-Tate group of $A$; then $G_{\Rat_{\ell}}$ has a natural action on $V_{\ell}(A)$ determined up to inner automorphisms. In particular, the image of the induced representation $G_{\Rat_{\ell}}\to\GL(V_{\ell}(A))$ is canonically determined. We can therefore consider the intersection $\Lie U_{\ell,G}\coloneqq\Lie U_{\ell}\cap\Lie G_{\Rat_{\ell}}$ within $\End(V_{{\ell}}(A))$.

\begin{introthm}\label{rationality}
The $\Int$-module
\[
\Hom(\dual{X}_g,X_g)\cap(\Lie G_{\Rat_{{\ell}}})(-1)\subset\Hom(\dual{X}_g,X_g)
\]
is independent of the choice of prime $\ell$. Denote it by $\mb{B}_G$. Then, for any prime $\ell$, the natural map:
\[
 \Rat_{\ell}(1)\otimes\mb{B}_G\to\Lie U_{\ell,G}
\]
is an isomorphism.
\end{introthm}

Although we will extract it as an immediate consequence of the flatness of the boundary divisor of the associated Shimura variety (cf.~\ref{strata:rem:etale_analog}), we believe that the result warrants emphasizing: The Mumford-Tate group $G$ is a transcendental object defined using a complex analytic uniformization for $A$. The theorem says that it also enjoys strong rationality properties with respect to the $p$-adic analytic uniformization. Of course, such behavior is predicted by the Hodge conjecture, and it was our confidence in its validity that led to Theorem~\ref{main}.

In fact, a significant input into the main Theorem~\ref{main} is the proof of a crystalline version of Theorem~\ref{rationality}; cf.~\ref{boundary:thm:main}.

Also, \emph{a priori} knowledge of the result---for instance, in the PEL case, where the Hodge cycles are generated by endomorphisms and polarizations, Theorem~\ref{rationality} follows directly from the functoriality of $p$-adic uniformizations---would make the proof of Theorem~\ref{main} quite straightforward. For more of an explanation of this, cf.~\eqref{boundary:rem:pel_case}.

\subsubsection*{Morita's conjecture}
Theorem \ref{main} also has the following pleasant consequence:
\begin{introthm}\label{morita:good_reduction}
  Suppose that $A$ is an abelian variety defined over a number field $F$, and suppose that its Mumford-Tate group is anisotropic modulo its center. Then $A$ has potentially good reduction.
\end{introthm}
The hypothesis on the Mumford-Tate group ensures that $A$ does not `degenerate in characteristic $0$'. The theorem says that this is enough to keep it from degenerating in finite characteristic as well. This result gives a different proof of Y. Morita's conjecture (see \cite{morita:good_reduction}). Related results can be found in \cite{paugam:morita},\cite{vasiu:projective} and \cite{lan:elevators}, with a proof of the full conjecture appearing in \cite{lee:morita}, building on the previous results. Our proof is independent of all these efforts, and applies uniformly, without any consideration of special cases.

\subsubsection*{The main difficulty}

We will now give a rough idea of the main difficulty that has to be overcome in this paper. The most important situation is the one where $K = K^\beef\cap G(\Adele_f)$. In this case, we can assume that $K^\beef$ has been chosen so that $\Sh_K\hookrightarrow \Sh_{K^\beef}$ is a closed immersion.

Consider the case where we are working around a point of $\Ss_{K^\beef}^{\Sigma^\beef}$ where the universal abelian variet{}y has totally degenerate multiplicative reduction: This corresponds to the situation where the point lies over a zero-dimensional stratum of the minimal compactification. Here, the complete local ring of $\Ss_{K^\beef}^{\Sigma^\beef}$ is the completion of a toric scheme associated with the torus over $\Int_{(p)}$ with cocharacter group $B(X_g)$, the space of symmetric bilinear forms on $X_g = \Int^g$. One therefore expects the complete local ring of $\Ss_K^\Sigma$ to then be identified with that of a subtoric scheme corresponding a to a cocharacter subgroup of $B(X_g)$. Which one should this be? It is of course the subgroup $\mb{B}_G$ guaranteed to us by Theorem~\ref{rationality}! The theorem tells us that it has the right dimension.

In the general case, the complete local ring of $\Ss_{K^\beef}^{\Sigma^\beef}$ is isomorphic to that of a smooth twisted torus embedding
\[
\Ss_{K^\beef_{\Phi^\beef}}(Q_{\Phi^\beef},D_{\Phi^\beef})\hookrightarrow \Ss_{K^\beef_{\Phi^\beef}}(Q_{\Phi^\beef},D_{\Phi^\beef},\sigma^\beef),
\]
where 
\[
\Ss_{K^\beef_{\Phi^\beef}}(Q_{\Phi^\beef},D_{\Phi^\beef})\to \Ss_{\overline{K}^\beef_{\Phi^\beef}}(\overline{Q}_{\Phi^\beef},\overline{D}_{\Phi^\beef})
\]
is a torsor under a certain torus whose cocharacter group is a subgroup of $B(X_g)$. 

Therefore, the complete local ring of $\Ss_{K^\beef}^{\Sigma^\beef}$ at such a point is formally smooth over a subring $R^{\sab}$, which is a complete local ring of $\Ss_{\overline{K}^\beef_{\Phi^\beef}}(\overline{Q}_{\Phi^\beef},\overline{D}_{\Phi^\beef})$.

Now, suppose that $R_G$ is the complete local ring of $\Ss_K^\Sigma$ at such a point, and let $R_G^{\sab}\subset R_G$ be the integral closure of the image of $R^{\sab}$.

Theorem~\ref{rationality} gives us a torus $\mathbf{E}_G$ with cocharacter group $\mb{B}_G$ of the correct rank. The main point now is to show that the $\mathbf{E}_{K^\beef}(\Phi^\beef)$-torus torsor
\[
\Ss_{K^\beef_{\Phi^\beef}}(Q_{\Phi^\beef},D_{\Phi^\beef})\to \Ss_{\overline{K}^\beef_{\Phi^\beef}}(\overline{Q}_{\Phi^\beef},\overline{D}_{\Phi^\beef})
\]
admits a canonical reduction of structure group to an $\mb{E}_G$-torsor over $\Spec R^{\sab}_G$.

In the PEL case such a reduction can be constructed directly, using the moduli description of the spaces involved. This is essentially what is done in~\cite{lan:ramified}. In our situation, we have to show the existence of the reduction using more abstract and indirect arguments. 

In more detail: Consider the induced $\mb{E}_{K^\beef}(\Phi^\beef)/\mb{E}_G$-torsor $\mb{\Xi}^G\to \Spec R_G^{\sab}$. Over the fraction field $Q(R_G)$, one has a canonical trivialization of $\mb{\Xi}^G$. The reduction of structure will follow once we know that this canonical trivialization is already defined over $R^{\sab}_G$.

The proof of this takes up much of the technical material in \S~\ref{sec:boundary}; cf.~\eqref{boundary:subsec:rationality} and~\eqref{boundary:subsec:reduction}. 

Once the reduction of structure is known in this particular way, it is quite easy to identify $R_G$ with a complete local ring of a twisted torus embedding of the $\mb{E}_G$-torsor we have obtained. One consequence of this is that the intersection of $\Ss_K^\Sigma$ with the boundary divisor is flat over $\Int_{(p)}$. This is an important result, because it shows that points at the boundary of $\Ss_K^\Sigma$ in characteristic $p$ admit lifts to characteristic $0$ that are once again at the boundary of the \emph{same} Shimura variety. In the PEL case, this was shown by K.-W. Lan in~\cite{lan:elevators} via a direct argument using degeneration data. In the general Hodge type case, we do not see how to show this without essentially understanding the full structure of the boundary.

\subsubsection*{Ingredients for the proof of Theorem~\ref{main}}
As should be clear from the description above, our proof is local. It makes essential use of the rigid analytic space attached to the formal neighborhood of a closed point of $\Ss^{\Sigma}_K$: Specifically, we work with the rigid space $\widehat{\mathcal{U}}^{\beef,\an}$ attached to the complete local ring of the corresponding point of $\Ss^{\Sigma^\beef}_{{K}^{\beef}}$. 

Now, $\Sh_K$ has an interpretation as the sub-space of $\Sh_{{K}^{\beef}}$ where certain Hodge cycles propagate as parallel tensors on the de Rham cohomology of the universal family of abelian varieties. 

The essential point is that even though we do not know that these de Rham tensors arise from algebraic cycles, we do have access to a shadow of their conjectural motivic origins: Namely, their good behavior with respect to the $p$-adic comparison isomorphisms. It is precisely this that allows us to prove the requisite descent assertion for the trivialization of $\mb{\Xi}^G$.

To execute this plan, we need to develop the logarithmic Dieudonn\'e theory of degenerating abelian varieties, and we also need some information about the $p$-adic comparison isomorphisms for semi-stable abelian varieties, and their behavior in families. Section~\ref{sec:semistable} is devoted to the exposition of these results. A useful ingredient here is the interpretation of the log crystalline realization of a semi-stable abelian variety as a space of `nearby unipotent cycles' in the terminology of \cite{vologodsky:hodge}; cf.~\eqref{semistable:subsec:unip}. This allows us to give a construction of the semi-stable comparison isomorphism for abelian varieties using log Dieudonn\'e theory: this is a direct generalization of the construction for abelian varieties with good reduction given in \cite{faltings:very_ramified}; cf.~\eqref{semistable:prp:comparison}. As a consequence, we also recover the following result due to Coleman and Iovita~\cite{coleman_iovita} (cf.~\eqref{semistable:subsubsec:coleman_iovita}):
\begin{introthm}\label{padicmonodromy}
Let $K$ be a non-archimedean local field of characteristic $0$, and let $A$ be an abelian variety over $K$ with semi-abelian reduction. Then $A$ has good reduction if and only if its $p$-adic Tate module $T_p(A)$ is a crystalline representation of the absolute Galois group of $K$.
\end{introthm}

In \S\ref{sec:background}, we summarize what we need of the existing theory of toroidal compactifications, both the analytic theory in characteristic $0$ and the arithmetic theory of Chai-Faltings over $\Int$. A good part of this amounts to setting up notation for what follows, and we expend some effort to harmonize between the two theories in a resolutely ad\'elic language.

With this background in hand, the technical heart of the paper can be found in sub-sections~\eqref{boundary:subsec:logfcrys}--\eqref{boundary:subsec:final}. The main result is Theorem~\eqref{boundary:thm:main}, which gives the desired description of the local structure at the boundary. This description is quite precise, and it seems likely that it could be used to provide a complete theory of integral compactifications of Hodge type without any reference to the characteristic $0$ theory.

We do not pursue this line of reasoning in this article. Instead, in \S\ref{sec:strata}, we use this local result to reduce the proof of our main theorems to already known facts about compactifications in characteristic $0$. The reader interested only in the statements of results is encouraged to jump directly to this section (though she should refer to the preceding sections for the notation used), and then to its companion \S\ref{sec:min}, which deals with the minimal compactification,

\subsection*{Acknowledgements}

This article has been through several iterations at this point, but it has its genesis in my University of Chicago Ph. D. thesis~\cite{mp:thesis}, completed under the guidance of Mark Kisin. I am very grateful to him, for introducing me to this subject, and for his encouragement and insight. I also thank B. Conrad, B. Howard, K. Kato, R. Kottwitz, K.-W. Lan, S. Morel, G. Pappas, G. Prasad, J. Rabinoff, J. Suh, B. Stroh and Y. Zhu for enlightening conversations and correspondence. I also thank the anonymous referee for useful remarks. 

This work was partially supported by NSF Postdoctoral Research Fellowship DMS-1204165 and by NSF grant DMS-1502142.

\section*{Conventions}\label{sec:notation}

\begin{enumerate}[itemsep=0.1in]
  \item All rings and monoids will be commutative, unless otherwise noted.
  \item For any prime $p$, $\vert\cdot\vert_p$ will denote the standard $p$-adic norm with $\vert p\vert_p=p^{-1}$, and $\nu_p$ will be the $p$-adic valuation $-\log_p\vert\cdot\vert_p$.
  \item For any prime $p$, we will write
      \[
      \widehat{\Int}^p\coloneqq\prod_{\ell\neq p}\Int_p\subset\widehat{\Int}
      \]
      for the pro-finite prime-to-$p$ integers. We will also write
      \[
       \Adele_f^p = \Rat\otimes\widehat{\Int}^p\subset\Adele_f
      \]
      for the prime-to-$p$ finite ad\'eles.
  \item If $L$ is a discrete valuation field, then $\Reg{L}$ will denote its ring of integers and $\mx_L\subset\Reg{L}$ its maximal ideal.
  \item We will use the geometric notation for change of scalars. If $f:R\to S$ is a map of rings and $M$ is an $R$-module, then we will denote the induced $S$-module $S\otimes_{f,R}M$ by $f^*M$. If the map $f$ is clear from context, then we will also write $M_S$ for the same $S$-module.
  \item If $\varphi:R\to R$ is an endomorphism of $R$, then a \defnword{$\varphi$-module} over $R$ is an $R$-module $M$ equipped with a map $\varphi^*M\to M$ of $R$-modules.
  \item
  Suppose that $R$ is a ring and suppose that $\mathbf{C}$ is an $R$-linear tensor category that is a faithful tensor sub-category of $\operatorname{Mod}_R$, the category of $R$-modules. Suppose in addition that $\mathbf{C}$ is closed under taking duals, symmetric and exterior powers in $\operatorname{Mod}_R$. Then, for any object $D\in\Obj(\mathbf{C})$, we will denote by $D^\otimes$ the direct sum of the tensor, symmetric and exterior powers of $D$ and its dual.

  \item Many Shimura data and varieties, both pure and mixed, appear in this paper. We will always use $\Sh$ to denote the canonical models of these varieteis and $\Ss$ for their integral models. In general, Roman characters will be used for spaces over $\Rat$, and calligraphic characters will be used for their integral models.
\end{enumerate}

\tableofcontents

\section{Semi-stable abelian varieties}\label{sec:semistable}

This section is meant to be used as a reference for Section~\ref{sec:boundary}. We review the theory of degenerating abelian varieties due to Mumford and Chai-Faltings. Using ideas of Kato, we employ this theory to study the corresponding degenerations of their cohomological realizations, especially the de Rham and crystalline ones.

Another goal is to understand the behavior of the $p$-adic comparison isomorphisms in families. Towards this, we give a construction for the semi-stable comparison isomorphism for abelian varieties using logarithmic Dieudonn\'e theory.

To avoid distracting from the main focus of the article, we only give the statements of the relevant results here, and postpone the proofs to the appendix.

\subsection{$1$-motifs and their realizations}\label{semistable:subsec:1motifs}

In this sub-section, we will assume that the reader is familiar with the notion of a bi-extension of a pair of group schemes; cf.~\cite[\S 10.2]{deligne:hodge-iii},~\cite[\S VII (2.1)]{sga_7_i} for details. For the theory of $1$-motifs, cf.~\cite[\S 10]{deligne:hodge-iii} and \cite{andreatta_bviale}.

\subsubsection{}
For any pair $(H,G)$ of sheaves of groups over a scheme $S$, we will denote by $\mb{1}_{H\times G}$ the trivial $\Gm$-bi-extension of $H\times G$.

A \defnword{$1$-motif} $Q$ over a scheme $S$ is a tuple $(B,\underline{Y},\underline{X},c,\dual{c},\tau)$, where:
\begin{itemize}
  \item $B$ is an abelian scheme over $S$, which we will denote $Q^{\ab}$.
  \item $\underline{Y}$ and $\underline{X}$ are \'etale sheaves of free abelian groups of finite rank over $S$, trivialized over a finite \'etale cover of $S$. We will denote them as $Q^{\et}$ and $Q^{\mult,C}$, respectively.
  \item $c:\underline{Y}\rightarrow B$ and $\dual{c}:\underline{X}\rightarrow\dual{B}$ are maps of sheaves of groups over $S$. We will denote them by $c_Q$ and $\dual{c}_Q$, respectively.
  \item $\tau:\mb{1}_{\underline{Y}\times\underline{X}}\xrightarrow{\simeq}(c\times\dual{c})^*\mathcal{P}_B$ is a trivialization of $\Gm$-bi-extensions of $\underline{Y}\times\underline{X}$. We will denote it by $\tau_Q$.
\end{itemize}
Here, $\mathcal{P}_B$ is the Poincar\'e $\Gm$-bi-extension of $B\times\dual{B}$.

A map $\varphi:Q_1\rightarrow Q_2$ of $1$-motifs is a tuple $(\varphi^{\ab},\varphi^{\et},\varphi^{\mult,C})$, for $?=\ab,\et$, $\varphi^{?}:Q^{?}_1\rightarrow Q^{?}_2$ is a map of sheaves of groups over $S$ and $\varphi^{\mult,C}:Q_2^{\mult,C}\rightarrow Q_1^{\mult,C}$. The tuple satisfies: $c_{Q_2}\varphi^{\et}=\varphi^{\ab}c_{Q_1}$, $\dual{c}_{Q_1}\varphi^{\mult,C}=\varphi^{\ab,\vee}\dual{c}_{Q_2}$, and a certain compatibility between $\tau_{Q_1}$ and $\tau_{Q_2}$~\cite[10.2.12]{deligne:hodge-iii}. To explain this, observe that we have natural isomorphisms of $\Gm$-bi-extensions of $Q_1^{\et}\times Q_2^{\mult,C}$:
\[
 (c_{Q_1}\times \dual{c}_{Q_1}\varphi^{\mult,C})^*\mathcal{P}_{Q_1^{\ab}}\xrightarrow{\simeq}(c_1\times\varphi^{\ab,\vee}\dual{c}_{Q_2})^*\mathcal{P}_{Q_1^{\ab}}\xrightarrow{\simeq}(\varphi^{\ab}c_{Q_1}\times \dual{c}_{Q_2})^*\mathcal{P}_{Q_2^{\ab}}\xrightarrow{\simeq}(c_{Q_2}\varphi^{\et}\times\dual{c}_{Q_2})^*\mathcal{P}_{Q_2^{\ab}}.
\]
The compatibility condition is that these isomorphisms carry the trivialization $(1\times\varphi^{\mult,C})^*\tau_{Q_1}$ to $(\varphi^{\et}\times 1)^*\tau_{Q_2}$.

The \defnword{dual} $\dual{Q}$ of a $1$-motif $Q$ is the tuple $(\dual{\bigl(Q^{\ab}\bigr)},Q^{\mult,C},Q^{\et},\dual{c},c,\dual{\tau})$, where $\dual{\tau}$ is the trivialization of the $\Gm$-bi-extension $(\dual{c}_Q\times c_Q)^*\mathcal{P}_{\dual{\bigl(Q^{\ab}\bigr)}}$ induced from $\tau$ via the symmetricity of the Poincar\'e bi-extension.

A \defnword{polarization} of a $1$-motif $Q$ is a map $\lambda:Q\rightarrow\dual{Q}$ such that $\lambda^{\ab}:Q^{\ab}\rightarrow\dual{\bigl(Q^{\ab}\bigr)}$ is a polarization of abelian schemes, such that $\lambda^{\et}:Q^{\et}\rightarrow Q^{\mult,C}$ is injective, and such that $\lambda^{\mult,C}=\dual{(\lambda^{\et})}$.

There is a canonical \defnword{weight filtration} $W_\bullet Q$ on any $1$-motif $Q$ with:
\begin{align*}
  W_iQ&=\begin{cases}
    0,&\text{ if $i<-2$};\\
    (0,0,Q^{\mult,C},0,0,1),&\text{ if $i=-2$};\\
    (Q^{\ab},0,Q^{\mult,C},0,\dual{c}_Q,1)&,\text{ if $i=-1$};\\
    Q,&\text{ if $i=0$}.
  \end{cases}
\end{align*}

\subsubsection{}\label{semistable:subsubsec:etale}
Write $Q^{\mult}$ for the torus over $S$ with character group $Q^{\mult,C}$ and let $Q^{\sab}$ be the extension of $Q^{\ab}$ by $Q^{\mult}$ classified by the map $-\dual{c}_Q:Q^{\mult,C}\to\dual{(Q^{\ab})}$: This satisfies the property that, for any section $y\in Q^{\mult,C}$, the $\Gm$-extension induced by pushing $Q^{\sab}$ along the map $y:Q^{\mult}\to\Gm$ corresponds to the point $-\dual{c}_Q(y)\in\dual{(Q^{\ab})}$.

Giving the trivialization $\tau$ is now equivalent to giving a map $Q^{\et}\rightarrow Q^{\sab}$ that lifts $c_Q$. This establishes an equivalence of categories between $1$-motifs over $S$ and two-term complexes of the form $[\underline{Y}\to H]$, where $\underline{Y}$ is an \'etale locally constant sheaf of finite free $\Int$-modules over $S$ and $H$ is an extension of an abelian scheme by a torus over $S$.\footnote{For us, a torus is always split locally in the finite \'etale topology.} For all this, cf.~\cite[\S 10]{deligne:hodge-iii}.

In particular, given a $1$-motif $Q$ and $n\in\Int_{\geq 1}$ we can define its \defnword{$n$-torsion $Q[n]$}: this is the flat group scheme $H^{-1}\bigl(\operatorname{cone}(Q\xrightarrow{n}Q)\bigr)$.

For any prime $p$, the system $(Q[p^n])_{n\in\Int_{\geq 0}}$ is a $p$-divisible group over $S$, which we will denote by $Q[p^{\infty}]$; cf.~\cite[\S 1.3]{andreatta_bviale}.

The filtration $W_\bullet Q$ induces a 3-step ascending weight filtration on $Q[p^{\infty}]$:
\[
 0=W_{-3}Q[p^{\infty}]\subset W_{-2}Q[p^{\infty}]=Q^{\mult}[p^{\infty}]\subset W_{-1}Q[p^{\infty}]=Q^{\sab}[p^{\infty}]\subset W_0Q[p^{\infty}]=Q[p^{\infty}].
\]
We have further identifications:
\[
 \gr^W_{-1}Q[p^{\infty}]=Q^{\ab}[p^\infty]\;;\;\gr^W_0Q[p^\infty]=Q^{\et}\otimes\underline{(\Rat_p/\Int_p)}.
\]

There is a perfect pairing $Q[p^{\infty}]\times\dual{Q}[p^{\infty}]\rightarrow\mu_{p^{\infty}}$ identifying $\dual{Q}[p^{\infty}]$ with the Cartier dual of $Q[p^{\infty}]$ defined precisely as in \cite[10.2.5]{deligne:hodge-iii}. The pairing is compatible with weight filtrations: Here, we equip $\mu_{p^\infty}$ with the ascending weight filtration $W_\bullet\mu_{p^{\infty}}$ with $W_{-2}\mu_{p^\infty}=\mu_{p^\infty}$ and $W_{-3}\mu_{p^{\infty}}=0$.

\subsubsection{}\label{semistable:subsubsec:derham}
We can also attach to $Q$ its de Rham realization $\underline{H}^1_{\dR}(Q)$: This will be a vector bundle over $S$ of rank $2\dim Q^{\ab}+\rank Q^{\et}+\rank Q^{\mult,C}$, equipped with two filtrations, the descending \defnword{Hodge filtration} $F^\bullet\underline{H}^1_{\dR}(Q)$ and the ascending \defnword{weight filtration} $W_{\bullet}\underline{H}^1_{\dR}(Q)$.

It is defined using universal vector extensions; cf.~\cite[\S 10.2]{deligne:hodge-iii} or \cite[\S 2.4]{andreatta_bviale}. Here is what we will need for now: There is a canonical $2$-term complex $[Q^{\et}\rightarrow E_Q]$ that is an extension of $[Q^{\et}\to Q^{\sab}]$ by the vector group attached to the locally free sheaf $\underline{\Lie}(Q^{\sab})$. This is the universal such vector extension, in the sense that given any extension of $[Q^{\et}\to E]$ by a vector group attached to a locally free coherent sheaf $H$ over $S$, there exists a unique map of coherent sheaves $f:\underline{\Lie}(Q^{\sab})\to H$ such that $[Q^{\et}\to E]$ is the push-forward of $[Q^{\et}\to E_Q]$ along $f$.

We set $\underline{H}_1^{\dR}(Q)=\underline{\Lie}(Q)$, and $\underline{H}^1_{\dR}(Q)=\dual{\underline{H}_1^{\dR}(Q)}$. We write $F^1\underline{H}^1_{\dR}(Q)$ for the annihilator of $\underline{\Lie}(Q^{\sab})\subset\underline{H}_1^{\dR}(Q)$: This is a direct summand, and it determines a two-step descending filtration $F^\bullet\underline{H}^1_{\dR}(Q)$ with $F^0\underline{H}^1_{\dR}(Q)=\underline{H}^1_{\dR}(Q)$ and $F^2\underline{H}^1_{\dR}(Q)=0$.

Again, the weight filtration on $Q$ induces an ascending $3$-step weight filtration on $\underline{H}^1_{\dR}(Q)$:
\[
 0\subset W_0\underline{H}^1_{\dR}(Q)=\underline{H}^1_{\dR}(Q/W_0Q)\subset W_1\underline{H}^1_{\dR}(Q)=\underline{H}^1_{\dR}(Q/W_{-1}Q)\subset\underline{H}^1_{\dR}(Q).
\]

Write $\Reg{S}(-1)$ for the trivial vector bundle $\Reg{S}$ equipped with the descending Hodge filtration $F^\bullet\Reg{S}(-1)$ satisfying $F^1\Reg{S}(-1)=\Reg{S}(-1)$ and $F^2\Reg{S}(-1)=0$, and the ascending weight filtration $W_{\bullet}\Reg{S}(-1)$ satisfying $W_2\Reg{S}(-1)=\Reg{S}(-1)$ and $W_1\Reg{S}(-1)=0$. Then there is a canonical perfect pairing of vector bundles respecting both Hodge and weight filtrations: $\underline{H}^1_{\dR}(Q)\times\underline{H}^1_{\dR}(\dual{Q})\to\Reg{S}(-1)$.

$\underline{H}^1_{\dR}(Q)$ has additional structure. Suppose that $S$ is a separated scheme over a base $S_0$. Let $S$ be the first-order infinitesimal neighborhood of the diagonal embedding of $S$ in $S\times_{S_0}S$: It is equipped with two projections $p_1,p_2:S\to S$. There is a canonical isomorphism of complexes of $S$-group schemes~\cite[3.3.2]{andreatta_bviale}:
\[
p_1^*[Q^{\et}\to E_Q]\xrightarrow{\simeq}p_2^*[Q^{\et}\to E_Q].
\]
Applying the Lie algebra functor now gives us a canonical isomorphism $p_1^*\underline{H}^1_{\dR}(Q)\xrightarrow[\eta]{\simeq}p_2^*\underline{H}^1_{\dR}(Q)$. As usual, the map $s\mapsto p_2^*s-\eta(p_1^*s)$ now defines an integrable connection:
\[
 \nabla_{S/S_0}:\underline{H}^1_{\dR}(Q)\to\underline{H}^1_{\dR}(Q)\otimes_{\Reg{S}}\Omega^1_{S/S_0}.
\]

\subsubsection{}\label{semistable:subsubsec:crystalline}
Finally, in the case where a prime $p$ is locally nilpotent in $\Reg{S}$, we can functorially attach a (contra-variant) Dieudonn\'e crystal $\Dieu(Q)$ to $Q$. This is simply the Dieudonn\'e crystal attached to the $p$-divisible group $Q[p^{\infty}]$ by the theory of \cite[\S~4]{bbm:cris_ii}; cf. also~\cite{andreatta_bviale}. As such, it is actually a tuple $(\Dieu(Q),\varphi_{\Dieu(Q)},\on{V}_{\Dieu(Q)},F^\bullet\Dieu(Q)(S))$, where $\Dieu(Q)$ is a crystal of locally free coherent sheaves over the crystalline site $(S/\Int_p)_{\cris}$. To describe the remaining data, we first need to note that the absolute Frobenius $\on{Fr}_{S\otimes\Field_p}$ on $S\otimes\Field_p$ canonically induces a pull-back functor $\on{Fr}^*$ on crystals over $(S/\Int_p)_{\cris}$. Now, $\varphi_{\Dieu(Q)}$ and $\on{V}_{\Dieu(Q)}$ are maps:
\[
 \varphi_{\Dieu(Q)}:\on{Fr}^*\Dieu(Q)\to\Dieu(Q)\;;\;\on{V}_{\Dieu(Q)}:\Dieu(Q)\to\on{Fr}^*\Dieu(Q)
\]
such that  $\varphi_{\Dieu(Q)}\circ \on{V}_{\Dieu(Q)}=p\cdot\on{1}_{\Dieu(Q)}$ and $\on{V}_{\Dieu(Q)}\circ\varphi_{\Dieu(Q)}=p\cdot\on{1}_{\on{Fr}^*\Dieu(Q)}$.

Write $\Dieu(Q)(S)$ for the vector bundle over $S$ obtained by restricting $\Dieu(Q)$ to the Zariski site of $S$: The crystalline nature of $\Dieu(Q)$ equips this with an integrable connection:
\[
 \Dieu(Q)(S)\to\Dieu(Q)(S)\otimes\Omega^1_{S/\Int_p}.
\]
There is now a canonical parallel isomorphism of vector bundles with integrable connections $\Dieu(Q)(S)\xrightarrow{\simeq}\underline{H}^1_{\dR}(Q)$~\cite[4.3.1]{andreatta_bviale}. Via this isomorphism, the Hodge filtration on $\underline{H}^1_{\dR}(Q)$ induces a filtration $F^\bullet\Dieu(Q)(S)$ on $\Dieu(Q)(S)$. It satisfies:
\begin{align}\label{semistable:eqn:hodgefil}
 \on{Fr}^*\bigl(F^1{\Dieu(Q)}(S)\otimes\Field_p\bigr)=\ker\bigl(\varphi_{\Dieu(Q)}\otimes 1:\on{Fr}^*{\Dieu(Q)}(S)\otimes\Field_p\rightarrow {\Dieu(Q)}(S)\otimes\Field_p\bigr).
\end{align}

As above, $\Dieu(Q)$ is equipped with an ascending $3$-step weight filtration $W_{\bullet}\Dieu(Q)$ by sub-Dieudonn\'e crystals:
\[
 0=W_{-1}\Dieu(Q)\subset W_0\Dieu(Q)=\Dieu(Q/W_0Q)\subset W_1\Dieu(Q)=\Dieu(Q/W_{-1}Q)\subset W_2\Dieu(Q)=\Dieu(Q).
\]

We will write $\mb{1}$ for the trivial crystal over $S$; this is naturally a Dieudonn\'e $F$-crystal when equipped with $\varphi_{\mb{1}}=1$, and $\on{V}_{\mathbb{M}}=p$, with Hodge filtration $F^0\mb{1}(S)=\mb{1}(S)$, $F^1\mb{1}(S)=0$, and weight filtration $W_0\mb{1}=\mb{1}$, $W_{-1}\mb{1}=0$. The \defnword{Tate twist} $\mb{1}(-1)$ is the Dieudonn\'e $F$-crystal whose underlying crystal is still the trivial crystal, but $\varphi_{\mb{1}(-1)}=p$, $\on{V}_{\mb{1}(1)}=1$; the Hodge filtration is given by $F^1\mb{1}(-1)(S)=\mb{1}(-1)(S)$ and $F^2\mb{1}(-1)(S)=0$, and the weight filtration by $W_{2}\mb{1}(1)=\mb{1}(-1)$, $W_{1}\mb{1}(-1)=0$.

Cartier duality induces a canonical perfect pairing $\Dieu(Q)\times\Dieu(\dual{Q})\to\mb{1}(-1)$ that is compatible with both Hodge and weight filtrations; cf.~\cite[\S~3.4]{andreatta_bviale} and~\cite[\S~5.3]{bbm:cris_ii}.

\subsection{Degenerating abelian varieties}\label{semistable:subsec:degen}

In this sub-section, we will fix a complete local normal Noetherian ring $R$, and an effective Cartier divisor $D\subset S\coloneqq\Spec R$ with complement $j:U\into S$.

\subsubsection{}\label{semistable:subsubsec:degpol}
Given an irreducible effective divisor $D'\subset S$ and a line bundle $\mathcal{L}$ over $S$, the order of vanishing $\nu_{D'}(\alpha)$ along $D'$ of any section $\alpha\in H^0(U,\mathcal{L})$ is well-defined: we choose a trivialization $\iota:\mathcal{L}\vert_V\xrightarrow{\simeq}\Reg{V}$ in an open sub-scheme $V\subset S$ containing the generic point of $D'$, and set $\nu_{D'}(\alpha)=\nu_{D'}(\iota(\alpha))$, where $\nu_{D'}$ is the discrete valuation attached to $D'$.

With this prelude, let $\mb{DD}_{\on{pol}}(S,U)$ be the category of polarized $1$-motifs $(Q_U,\lambda_U)$ over $U$ such that:
\begin{itemize}
  \item The polarized abelian scheme $(Q_U^{\ab},\lambda^{\ab}_U)$ extends (uniquely) to a polarized abelian scheme $(Q^{\ab},\lambda^{\ab})$ over $S$.
  \item The maps $c_{Q_U}$ and $\dual{c}_{Q_U}$ extend to maps $c_Q:Q^{\et}\to Q^{\ab}$ and $\dual{c}_Q:Q^{\mult,C}\to Q^{\ab,\vee}$. Here, we need the fact that the locally constant sheaves $Q_U^{\et}$ and $Q_U^{\mult,C}$ extend canonically to sheaves $Q^{\et}$ and $Q^{\mult,C}$ over $S$, as do the maps between them.
  \item For any section $y$ of $Q^{\et}$, consider the section
  \[
  \tau(y,\lambda^{\et}(y))\in H^0\bigl(U,(c(y)\times\dual{c}(\lambda^{\et}(y)))^*\mathcal{P}_{Q^{\ab}}\bigr).
  \]
  We require that, for any irreducible divisor $D'$ with support in $D$, we have:
  \[
   \nu_{D'}(\tau(y,\lambda^{\et}(y)))\geq 0.
  \]
\end{itemize}

Let $\mb{DD}(S,U)$ be the full sub-category of the category of $1$-motifs over $U$ consisting of those $1$-motifs $Q_U$ that admit a polarization $\lambda_U$ such that $(Q_U,\lambda_U)$ belongs to $\mb{DD}_{\on{pol}}(S,U)$.

\subsubsection{}\label{semistable:subsubsec:mumford}
Let $\mb{DEG}(S,U)$ be the category of abelian schemes $A$ over $U$ such that $A$ extends (uniquely) to a smooth group scheme $G$ over $S$ with semi-abelian fibers. Then a construction of Mumford and Chai-Faltings (cf.~\cite[Ch. III]{faltings_chai},\cite[4.4.16,4.5.5]{lan:thesis}) gives a canonical equivalence of categories:
\[
 \on{M}_{(S,U)}:\mb{DEG}(S,U)\xrightarrow{\simeq}\mb{DD}(S,U).
\]
When $R=\Reg{K}$ is the ring of integers of a $p$-adic local field, and $D=\Spec k$ is its special point, the result is due to Raynaud~\cite{raynaud:icm}.

Fix $A$ in $\mb{DEG}(S,U)$, and let $Q_U=\on{M}_{(S,U)}(A)$. We will need the following facts about the relationship between $A$ and $Q_U$:

\begin{itemize}
\item There is a canonical duality isomorphism:
\[
 \on{M}_{(S,U)}(\dual{A})\xrightarrow{\simeq}\dual{Q}_U.
\]
This is shown in~\cite[Ch. III,\S~6, pp. 73]{faltings_chai}.

\item  For any $n\in\Int_{\geq 1}$, there is a canonical isomorphism~\cite[III.7.3]{faltings_chai} of finite flat group schemes over $U$, compatible with polarization pairings:
\begin{align}\label{semistable:eqn:torsion}
 A[n]\xrightarrow{\simeq}Q_U[n].
\end{align}

\item There is a canonical isomorphism of filtered vector bundles over $U$ with integrable connection:
\begin{align}\label{semistable:eqn:derham}
\underline{H}^1_{\dR}(A)\xrightarrow{\simeq}\underline{H}^1_{\dR}(Q_U).
\end{align}
It is compatible with polarization pairings. This amounts to seeing that the universal vector extension of $A$ can be constructed using that of $Q_U$ via Mumford's construction. This can be deduced from \cite[Ch. III,\S~9]{faltings_chai}, especially the proof of Theorem 9.4.
\end{itemize}

\subsection{Log $F$-crystals}\label{semistable:subsec:logfcrys}
Let $(S,U)$ be as in \eqref{semistable:subsec:degen}. Given $A$ in $\mb{DEG}(S,U)$, we will now use the equivalence of categories from (\ref{semistable:subsubsec:mumford}) to define a logarithmic crystalline realization $\bb{D}(A)$ over $S$.

\subsubsection{}
We assume that the reader is conversant with the language of log schemes; that is, of pairs $(Z,\Mon_Z)$, where $Z$ is a scheme and $\Mon_Z$ is an \'etale sheaf of commutative monoids over $Z$ equipped with a map of monoids $\alpha:\Mon_Z\to\Reg{Z}$ such that $\alpha^{-1}(\Reg{Z}^\times)$ maps isomorphically onto $\Reg{Z}^\times$. We will assume that the stalks of $\Mon_Z/\Reg{Z}^\times$ at all geometric points of $Z$ are cancellative, sharp and saturated. Here, a monoid $M$ is \defnword{cancellative} if the map $M\to M^{\gp}$ to its group envelope is injective; it is \defnword{sharp} if it contains no non-trivial invertible elements; and it is \defnword{saturated} if there is no $x\in M^{\gp}\backslash M$ such that $x^n\in M$, for some $n\in\Int_{>1}$.

Our primary example will be the log scheme attached to the pair $(S,U)$ as above with $Z=S$, $\Mon_Z=j_*\Reg{U}^\times\cap\Reg{Z}$ and $\alpha:\Mon_Z\to\Reg{Z}$ the natural inclusion. It is easy to see that the stalks of $\Mon_Z/\Reg{Z}^\times$ are cancellative and sharp; one needs the normality of $S$ to see that the stalks are also saturated.

We will usually drop the monoid $\Mon_Z$ from our notation and denote the log scheme by the single letter $Z$.

\subsubsection{}
Suppose that $Z$ is a scheme over $\Int_p$. Recall from \cite[\S~6]{kato:fontaine_illusie_i} the notion of a crystal of vector bundles (or log crystal) over the logarithmic crystalline site of $Z$ (with respect to the canonical divided powers on $p\Int_p$). For our purposes it suffices to note that giving such a crystal $\bb{M}$ is equivalent to the following data:
\begin{itemize}
  \item Given a log scheme $V\to Z$ in which $p$ is nilpotent, and an exact nilpotent thickening $V\to \tilde{V}$ of log schemes\footnote{This means that the map $\Mon_{\tilde{V}}/\Reg{\tilde{V}}^\times\to\Mon_V/\Reg{V}^\times$ is an isomorphism.} equipped with divided powers, a finite, locally free $\Reg{\tilde{V}}$-module $\bb{M}\bigl((V\into\tilde{V})\bigr)$.
  \item Given another log scheme $V'\to Z$ in which $p$ is nilpotent, an exact nilpotent thickening $V'\into\widetilde{V}'$ equipped with divided powers, and a map $f:\widetilde{V}'\to\widetilde{V}$ carrying $V'$ into $V$, an isomorphism of $\Reg{\widetilde{V}'}$-modules:
      \[
       \bb{M}(f):f^*\bb{M}\bigl((V\into\widetilde{V})\bigr)\xrightarrow{\simeq}\bb{M}\bigl((V'\into\widetilde{V}')\bigr).
      \]
      Furthermore, if $g:(\widetilde{V}'',V'')\to(\widetilde{V}',V')$ is another such map, we have: $g^*\bb{M}(f)\circ\bb{M}(g)=\bb{M}(f\circ g)$.
\end{itemize}

\subsubsection{}
Just as in the classical crystalline case, the absolute Frobenius $\on{Fr}_{Z\otimes\Field_p}$\footnote{For log schemes, the absolute Frobenius induces the multiplication-by-$p$ map on $\Mon_Z$.} induces an endofunctor $\on{Fr}^*$ on the category of log crystals. A \defnword{log Dieudonn\'e crystal} over $Z$ is a tuple $(\bb{M},\varphi_{\bb{M}},\on{V}_{\bb{M}},F^\bullet\bb{M}(\hat{Z}))$ where $\bb{M}$ is a log crystal over $Z$, and $\varphi_{\bb{M}}:\on{Fr}^*\bb{M}\to\bb{M}$, $\on{V}_{\bb{M}}:\bb{M}\to\on{Fr}^*\bb{M}$ are morphisms of log crystals satisfying:
\[
 \varphi_{\bb{M}}\on{V}_{\bb{M}}=p\mb{1}_{\bb{M}}\;;\;\on{V}_{\bb{M}}\varphi_{\bb{M}}=p\mb{1}_{\on{Fr}^*\bb{M}}.
\]
To describe the final piece of data, let $\hat{Z}$ be the completion of $Z$ along $Z\otimes\Field_p$. Then we obtain a finite, locally free module $\bb{M}(\hat{Z})$ over the formal scheme $\hat{Z}$ attached to the system $(\bb{M}(Z\otimes\Int/p^n\Int))_n$. Now, $F^\bullet\bb{M}(\hat{Z})$ is a $2$-step descending \emph{Hodge filtration} by direct summands:
\[
 0=F^2\bb{M}(\widehat{Z})\subset F^1\bb{M}(\widehat{Z})\subset F^0\bb{M}(\widehat{Z})=\bb{M}(\widehat{Z}).
\]
It satisfies:
\begin{align}\label{semistable:eqn:hodgefillog}
 \on{Fr}^*\bigl(F^1{\bb{M}}(\widehat{Z})\otimes\Field_p\bigr)=\ker\bigl(\varphi_{\bb{M}}\otimes 1:\on{Fr}^*{\bb{M}}(Z\otimes\Field_p)\rightarrow {\bb{M}}(Z\otimes\Field_p)\bigr).
\end{align}

Completely analogously to the classical case, we can define the Cartier dual for any log Dieudonn\'e crystal $\bb{M}$: Its underlying log crystal is just the  dual $\dual{\bb{M}}$ and the additional structure is given by $\varphi_{\dual{\bb{M}}}=\dual{\on{V}}_{\bb{M}}$ ,$\on{V}_{\dual{\bb{M}}}=\dual{\varphi}_{\bb{M}}$, and $F^1\dual{\bb{M}}(\hat{Z}))=\on{ann}(F^1\bb{M}(\hat{Z}))$.

Any $F$-crystal over the scheme underlying $Z$ can be naturally viewed as a log $F$-crystal over $Z$. If the $F$-crystal has the structure of a Dieudonn\'e crystal, then the associated log $F$-crystal will have the structure of a log Dieudonn\'e crystal. In particular, we always have the log Dieudonn\'e crystals $\mb{1}$ and $\mb{1}(-1)$.

There is then a natural perfect pairing on $\bb{M}\times\dual{\bb{M}}$ with values in $\mb{1}(-1)$.

\subsubsection{}

Let $\mb{LDieu}(Z)$ be the category of log Dieudonn\'e crystals over $Z$, and let $\mb{LDieu}_{\on{wt}}(Z)$ be category of pairs $(\bb{M},W_{\bullet}\bb{M})$ where $\bb{M}$ is a log Dieudonn\'e crystal, and $W_{\bullet}\bb{M}$ is an ascending $3$-step filtration of $\bb{M}$ by log Dieudonn\'e sub-crystals:
\[
0=W_{-1}\bb{M}\subset W_0\bb{M}\subset W_1\bb{M}\subset W_2\bb{M}=\bb{M}.
\]

The \defnword{dual} $\dual{(\bb{M},W_{\bullet}\bb{M})}$ will be the pair $(\dual{\bb{M}},W_{\bullet}\dual{\bb{M}})$, where $\dual{\bb{M}}$ is the Cartier dual of $\bb{M}$ and $W_i\dual{\bb{M}}\subset\dual{\bb{M}}$ is the annihilator of $W_{1-i}\bb{M}$.

If $S^{\log}$ is the log scheme attached to a pair $(S,U)$ as in (\ref{semistable:subsec:degen}), we will write $\mb{LDieu}(S,U)$ and $\mb{LDieu}_{\wt}(S,U)$ for the associated categories. For any log crystal $\bb{M}$ over $S^{\log}$, $\bb{M}(\hat{S}^{\log})$ is a finite free $R$-module, and we will denote it simply by $\bb{M}(S)$.

\begin{prp}\label{semistable:prp:logdieu}
There is a natural functor:
\[
 \Dieu:\mb{DEG}(S,U)\to\mb{LDieu}_{\wt}(S,U)
\]
such that, for any $A$ in $\mb{DEG}(S,U)$:
\begin{enumerate}[itemsep=0.11in]
\item We have canonical isomorphisms in $\mb{LDieu}(S,U)$:
\[
W_0\Dieu(A)\xrightarrow{\simeq}\SHom(Q^{\et},\mb{1});\; W_1\Dieu(A)\xrightarrow{\simeq}\Dieu(Q^{\sab});\; \gr^W_2\Dieu(A)\xrightarrow{\simeq}\mb{1}(-1)\otimes Q^{\mult,C}
\]
\item There is a canonical duality isomorphism in $\mb{LDieu}_{\on{wt}}(S,U)$: $\Dieu(\dual{A})\xrightarrow{\simeq}\dual{\Dieu(A)}$.
\item There is a natural horizontal isomorphism of $\Reg{U}$-modules $\underline{H}^1_{\dR}(A)\xrightarrow{\simeq}\Dieu(A)(S)\vert_U$ respecting both Hodge and weight filtrations.
\end{enumerate}
\end{prp}

Using (\ref{semistable:subsubsec:mumford}), we see that it is enough to construct a functor
\[
 \Dieu:\mb{DD}(S,U)\to\mb{LDieu}_{\wt}(S,U)
\]
such that, for any $Q_U$ on the left-hand side, the assertions about duality and the weight filtration hold, and such that there is a natural horizontal isomorphism
\[
\underline{H}^1_{\dR}(Q_U)\xrightarrow{\simeq}\Dieu(Q_U)(S)\vert_U
\]
respecting Hodge and weight filtrations. We will do this in \eqref{appendix:subsec:logfcrys}.

\subsection{Unipotent nearby cycles}\label{semistable:subsec:unip}
This sub-section is mostly a review of some material from \cite[\S~3]{vologodsky:hodge}. We will develop the language to construct the key comparison isomorphism (\ref{semistable:prp:comparison}).

\subsubsection{}
Let $k=\overline{\Field}_p$ be an algebraic closure of $\Field_p$, and set $W=W(k)$ and $K_0=W\bigl[p^{-1}\bigr]$. Write $\sigma:W\to W$ for the canonical lift of the $p$-power Frobenius on $k$. Consider a pair $(S,U)$ as in (\ref{semistable:subsec:degen}) with $S=\Spec R$, so that the divisor $D=S\backslash U\subset S$ is an effective Cartier divisor. We will impose some conditions on this pair:
\begin{itemize}
\item $R$ is a formally log smooth, topologically finitely generated $W$-algebra with residue field $k$. Explicitly, there exists $j\in\Int_{\geq 0}$, and a finitely generated, cancellative, saturated monoid $P$ with $P^\times=\{1\}$ such that, as $W$-algebras:
\[
 R\xrightarrow{\simeq}W\pow{t_1,\ldots,t_j}\widehat{\otimes}_WW\pow{P}.
\]
Here, $W\pow{t_1\ldots,t_k}$ is the power series ring over $W$ in $r$ variables, and $W\pow{P}$ is the completion of the monoid ring $W[P]$ along the ideal generated by $P\backslash\{1\}$.
\item The divisor $D\subset S$ is the vanishing locus of the elements of $P\backslash\{1\}$.
\end{itemize}

Let $S^{\log}$ be the log scheme attached to $(S,U)$ and fix a lift $\varphi:S^{\log}\to S^{\log}$ of the absolute Frobenius on $S^{\log}\otimes\Field_p$. Concretely, this means $\varphi$ is a Frobenius lift on $S$ satisfying an additional condition: There exists an isomorphism $R\xrightarrow{\simeq}W\pow{t_1,\ldots,t_j}\widehat{\otimes}W\pow{P}$ as above such that, for all $m\in P$, we have $\varphi(m)=m^p$.

Let $\flogdiff{R/W}$ be the module of continuous logarithmic differentials on $R$ with poles along the divisor $D$. Let $a\in R$ be an equation defining $D$. We have an exact sequence:
\begin{align}\label{semistable:eqn:residue}
 \fdiff{R/W}\to\flogdiff{R/W}\to W\otimes_{\Int}(R[a^{-1}]^\times/R^\times)\to 0,
\end{align}
where, under the map on the right, for any $r\in R[a^{-1}]^\times$, $\dlog(r)$ is carried to the image $[r]\in R[a^{-1}]^\times/R^\times$ of $r$.

\subsubsection{}\label{semistable:subsubsec:osanlog}
Let $S^{\an}$ be the rigid analytic space attached to the formal scheme $\hat{S}=\Spf S$.\footnote{We will be using the analytification functor of Berthelot, exposed in detail in \cite[\S 7]{dejong:formal_rigid}.} Let $\Reg{S}^{\an}$ be the sheaf of analytic functions on $S^{\an}$. A \defnword{log isocrystal} over $S^{\an}$ is a locally free coherent $\Reg{S}^{\an}$-module $M$ equipped with an integrable logarithmic connection:
\[
 \nabla_M:M\to M\otimes_R\flogdiff{R/W}.
\]

A \defnword{log $F$-isocrystal} over $S^{\an}$ is a log isocrystal $(M,\nabla_M)$ equipped with a horizontal isomorphism $\varphi_M:\varphi^*M\xrightarrow{\simeq}M$. Write $\mb{LFI}(S,U)$ for the category of such triples: it is naturally a $\Rat_p$-linear tensor category. The unit object $\mb{1}$ is simply the structure sheaf $\Reg{S^{\an}}$ equipped with the obvious identification $\varphi^*\Reg{S^{\an}}=\Reg{S^{\an}}$ and the trivial connection.

Set
\[
\Reg{S}^{\an,\log}=\frac{\Reg{S}^{\an}[\ell_r:r\in R[a^{-1}]^\times]}{\begin{pmatrix}\ell_{rs}-\ell_r-\ell_s\text{, for all $r,s\in R[a^{-1}]^\times$};\\\ell_u-\log(u):\text{ for all $u\in R^\times$}\end{pmatrix}}.
\]

If we choose an isomorphism $R\xrightarrow{\simeq}W\pow{t_1,\ldots,t_j}\widehat{\otimes}W\pow{P}$ as above, then we obtain an isomorphism of $\Reg{S}^{\an}$-algebras:
\begin{align}\label{semistable:eqn:polynomial}
 \Reg{S}^{\an}\otimes\Sym(P^{\gp})&\xrightarrow{\simeq}\Reg{S}^{\an,\log};\\
 1\otimes m&\mapsto \ell_{m}.\nonumber
\end{align}
Here, $\Sym(P^{\gp})$ is the symmetric algebra for the group envelope $P^{\gp}$ of $P$.

$\Reg{S}^{\an,\log}$ can naturally be equipped with the structure of an ind-object in $\mb{LFI}(S,U)$: The connection is the unique $\Reg{S}^{\an,\log}$-derivation that satisfies $\nabla(\ell_m)=1\otimes\dlog(m)$, and the $\varphi$-module structure is the algebra homomorphism induced by the map $\ell_m\mapsto p\ell_m$. Let $D_i\Reg{S}^{\an,\log}\subset\Reg{S}^{\an,\log}$ be the image of $\Reg{S}^{\an}\otimes\Sym^{\leq i}P^{\gp}$ under the isomorphism (\ref{semistable:eqn:polynomial}); then $D_i\Reg{S}^{\an,\log}$ is finite free over $\Reg{S}^{\an}$ and is stable under both $\nabla$ and $\varphi$. We clearly have: $\Reg{S}^{\an,\log}=\bigcup_iD_i\Reg{S}^{\an,\log}$.

\subsubsection{}\label{semistable:subsubsec:unip}
Given a tuple $(M,\nabla_M,\varphi_M)$ in $\mb{LFI}(S,U)$, we can equip $M^{\an,\log}\coloneqq\Reg{S}^{\an,\log}\otimes_{\Reg{S}^{\an}}M$ with the tensor product structure of an ind-object in $\mb{LFI}(S,U)$. In particular, it is equipped with the connection $\nabla=\nabla\otimes 1+1\otimes\nabla_M$. Set:
\[
 \Psi_{\on{un}}(M)=\bigl(M^{\an,\log}\bigr)^{\nabla=0}.
\]

We have the following facts about this object~\cite[3.7]{vologodsky:hodge}:
\begin{itemize}
  \item $\Psi_{\on{un}}(M)$ is a finite dimensional $K_0$-vector space, and the $\varphi$-module structure on $M^{\an,\log}$ equips it with a $\sigma$-module structure $\varphi_{\on{un}}:\sigma^*\Psi_{\on{un}}(M)\to\Psi_{\on{un}}(M)$.
  \item The map
  \[
   \Reg{S}^{\an,\log}\otimes_{\Reg{S}^{\an}}M\xrightarrow{\nabla\otimes 1}M^{\an,\log}\otimes_R\flogdiff{R/W}\rightarrow M^{\an,\log}\otimes_{\Int}(R[a^{-1}]^\times/R^\times)
  \]
  restricts to a map $N:\Psi_{\on{un}}(M)\to\Psi_{\on{un}}(M)\otimes(R[a^{-1}]^\times/R^\times)$ satisfying
  \[
  N\varphi_{\on{un}}=p(\varphi_{\on{un}}\otimes 1)N.
  \]
  \item The natural map
  \[
   \Reg{S}^{\an,\log}\otimes_{K_0}\Psi_{\on{un}}(M)\to M^{\an,\log}
  \]
  is an isomorphism of ind-objects in $\mb{LFI}(S,U)$. Here, the $\varphi$-module structure on the left is the diagonal one, and the restriction of the connection to $\Psi_{\on{un}}(M)$ is trivial.
\end{itemize}

\subsubsection{}\label{semistable:subsubsec:unipequiv}
Here is one way to formulate the above results. Set $\Lambda\coloneqq\Lambda(S,U)=R[a^{-1}]^\times/R^\times$: This is a free abelian group of finite rank. Let $\mb{LFI}(k,\Lambda)$ be the category of tuples $(M_0,\varphi_0,N_0)$, where $M_0$ is a finite dimensional $K_0$-vector space, $\varphi_0:\sigma^*M_0\to M_0$ is a $\sigma$-module structure, and $N_0:M_0\to M_0\otimes\Lambda$ is a map satisfying $N_0\varphi_0=p(\varphi_0\otimes 1)N_0$.

This category has a natural rigid tensor structure: The associated dual object to a tuple as above is given by $(\dual{M}_0,(\dual{\varphi}_0)^{-1},-\dual{N}_0)$. Given another tuple $(M'_0,\varphi'_0,N'_0)$, we have:
\[
 (M_0,\varphi_0,N_0)\otimes (M'_0,\varphi'_0,N'_0)=(M_0\otimes_{K_0}M'_0,\varphi_0\otimes\varphi'_0,N_0\otimes 1+1\otimes N'_0).
\]
The unit object is the tuple $\mb{1}=(K_0,\sigma^*K_0=K_0,0)$.

Now we can summarize the discussion above as follows: The functor $M\to\Psi_{\on{un}}(M)$ is an equivalence of $\Rat_p$-linear tensor categories between $\mb{LFI}(S,U)$ and $\mb{LFI}(k,\Lambda)$.

The functor has a natural inverse: Let
\[
 N_{\on{triv}}:\Reg{S}^{\an,\log}\to\Reg{S}^{\an,\log}\otimes\Lambda
\]
be the residue of the connection on $\Reg{S}^{\an,\log}$~\eqref{semistable:eqn:residue}. Then, given $(M_0,\varphi_0,N_0)$ in $\mb{LFI}(k,\Lambda)$, the ind-object $\Reg{S}^{\an,\log}\otimes_{K_0}M_0$ is equipped with the operator
\[
 N=N_{\on{triv}}\otimes 1+1\otimes N_0:\Reg{S}^{\an,\log}\otimes_{K_0}M_0\to(\Reg{S}^{\an,\log}\otimes_{K_0}M_0)\otimes_{\Int}\Lambda.
\]
It can now be checked that $\bigl(\Reg{S}^{\an,\log}\otimes_{K_0}M_0\bigr)^{N=0}$ is stable under both $\varphi$ and $\nabla$ and is naturally an object in $\mb{LFI}(S,U)$.

The inverse to $\Psi_{\on{un}}$ is now given by $M_0\mapsto\bigl(\Reg{S}^{\an,\log}\otimes_{K_0}M_0\bigr)^{N=0}$.

\subsubsection{}\label{semistable:subsubsec:uniplogsmooth}
Suppose now that we have $A$ in $\mb{DEG}(S,U)$. By (\ref{semistable:prp:logdieu}), we can attach to it the log $F$-crystal $\Dieu(A)$ over $S^{\log}$. The $\Reg{S}^{\an}$-module $M^{\an}(A)=\Reg{S}^{\an}\otimes_{\Reg{S}}\Dieu(A)(S)$ is naturally equipped with the structure of an object in $\mb{LFI}(S,U)$. It is also equipped with Hodge and weight filtrations: $F^\bullet M^{\an}(A)$ and $W_{\bullet}M^{\an}(A)$. Its restriction to the complement $U^{\an}\subset S^{\an}$ of $D^{\an}$ is, as a filtered vector bundle with integrable connection, canonically isomorphic to the analytification of $\underline{H}^1_{\dR}(A)$.

We set:
\[
 M_0(A)=\Psi_{\on{un}}(M^{\an}(A))\in\mb{LFI}(k,\Lambda).
\]
This will be the module of \defnword{unipotent nearby cycles} attached to $A$.

There is a canonical isomorphism of ind-objects in $\mb{LFI}(S,U)$:
\begin{align}\label{semistable:eqn:paralleltrans}
  \Reg{S}^{\an,\log}\otimes_{K_0}M_0(A)\xrightarrow{\simeq}M^{\an,\log}(A)\coloneqq \Reg{S}^{\an,\log}\otimes_{\Reg{S}^{\an}}M^{\an}(A).
\end{align}

\subsubsection{}\label{semistable:subsubsec:spibreuil}
Fix an algebraic closure $\overline{K}_0$ of $K_0$. Let $K/K_0$ be a finite extension within $\overline{K}_0$. Fix a uniformizer $\pi\in K$ and let $\log_{\pi}:K^\times\to K$ be the branch of the $p$-adic logarithm such that $\log_{\pi}(\pi)=0$. Let $E(u)\in W[u]$ be the monic Eisenstein polynomial satisfying $E(\pi)=0\in\Reg{K}$. Then we can view $\Reg{K}$ as the quotient $W[u]/(E(u))$. Let $\mathcal{S}_{\pi}$ be the $p$-adic completion of the divided power envelope of the surjection $W[u]\to\Reg{K}$ carrying $u$ to $\pi$, and set $\Fil^1\mathcal{S}_{\pi}=\ker(\mathcal{S}_{\pi}\to\Reg{K})$: by construction $\Fil^1\mathcal{S}_{\pi}$ is equipped with divided powers compatible with those on $pS$. Concretely, we have (cf.~\cite[2.1.1]{breuil:pdiv}):
\begin{align}\label{semistable:eqn:spiexpl}
\mathcal{S}_{\pi}=\bigl\{\sum_ia_i\frac{u^i}{q(i)!}\in K_0\pow{u}:a_i\in W,\text{}\lim_{i\to\infty}a_i=0\bigr\}.
\end{align}
Here, $q(i)=\lfloor\frac{i}{p}\rfloor$.

$\Spec\mathcal{S}_{\pi}$ is equipped with the log structure induced by the divisor $u=0$, and the natural log structure on $\Spec\Reg{K}$ is induced from this one via the surjection $\mathcal{S}_{\pi}\twoheadrightarrow\Reg{K}$.

The Frobenius lift $\varphi:W[u]\to W[u]$ with $\varphi(u)=u^p$ induces one on $\mathcal{S}_{\pi}$: $\varphi:\mathcal{S}_{\pi}\to\mathcal{S}_{\pi}$. The induced endomorphism of $\Spec\mathcal{S}_{\pi}$, again denoted $\varphi$, is an endomorphism of log schemes.

Given any semi-stable abelian variety $A$ over $K$, we can now evaluate the associated log Dieudonn\'e crystal $\Dieu(A)$ over $\Reg{K}$ along the formal divided power thickening $\Spf\Reg{K}\into\Spf\mathcal{S}_{\pi}$. This gives us a $\varphi$-module $\mathcal{M}(A)$ over $\mathcal{S}_{\pi}$ equipped with an integrable logarithmic connection
\[
\nabla_{\mathcal{M}(A)}:\mathcal{M}(A)\to\mathcal{M}(A)\dlog(u).
\]

Let $Y$ be the rigid analytic space over $K_0$ attached to $\Spf W\pow{u}$: This is the rigid analytic open disk of radius $1$ around the $K_0$-valued point $y_0$ attached to $u\mapsto 0$. Let $\varphi:Y\to Y$ be the endomorphism induced by the endomorphism $\varphi$ of $W[u]$. Let $Y(e)\subset Y$ be the rigid analytic open disk of radius $p^{-1/e(p-1)}$: it is preserved by the endomorphism $\varphi$.

We have a natural $\varphi$-equivariant map $\mathcal{S}_\pi\to\Reg{Y(e)}^{\an}$; this follows for instance from \eqref{semistable:eqn:spiexpl}. Let $\Reg{Y(e)}^{\an,\log}$ be the restriction of $\Reg{Y}^{\an,\log}$ to $Y(e)$. Set
\[
 M_0(A)=(\Reg{Y(e)}^{\an,\log}\otimes_{\mathcal{S}_\pi}\mathcal{M}(A))^{\nabla=0}.
\]
Then, just as in (\ref{semistable:subsubsec:unip}), the logarithmic connection on $\mathcal{M}(A)$ induces a residue map:
\begin{align}\label{semistable:eqn:monodromy_residue}
N_0(A): M_0(A)\to M_0(A)\otimes(W\pow{u}[u^{-1}])^\times/W\pow{u}^\times\xrightarrow{u\mapsto \pi}M_0(A)\otimes\bigl(\overline{K}_0^\times/\Reg{\overline{K}_0}^\times\bigr).
\end{align}
$M_0(A)$ also has a natural $\sigma$-module structure, making it an object in $\mb{LFI}(k,\overline{K}_0^\times/\Reg{\overline{K}_0}^\times)$. We will refer to an object in this category as a \defnword{$(\varphi,N)$-module over $K_0$}.

There is a canonical isomorphism of ind-log $F$-isocrystals over $Y(e)$:
\begin{align}\label{semistable:eqn:separallel}
\Reg{Y(e)}^{\an,\log}\otimes_{K_0} M_0(A)\xrightarrow{\simeq}\Reg{Y(e)}^{\an,\log}\otimes_{\mathcal{S}_e}\mathcal{M}(A).
\end{align}

The evaluation-at-$\pi$ map $f\mapsto f(\pi)$ from $\Reg{Y(e)}^{\an}$ can be extended to a map on $\Reg{Y(e)}^{\an,\log}$ by sending $\ell_r$ to $\log_{\pi}(r(\pi))$. Specializing \eqref{semistable:eqn:separallel} along this extension gives us an isomorphism:
\begin{align}\label{semistable:eqn:absthyodokato}
 \beta_{\on{H-K},A,\pi}:K\otimes_{K_0}M_0(A)\xrightarrow{\simeq}K\otimes_{\mathcal{S}_e}\mathcal{M}(A)\xrightarrow{\simeq}K\otimes_{\Reg{K}}\Dieu(A)(\Reg{K})\xrightarrow{\simeq}H^1_{\dR}(A/K).
\end{align}

\subsubsection{}\label{semistable:subsubsec:hyodokatocomp}
Now, we return to the situation of \eqref{semistable:subsubsec:uniplogsmooth}. Suppose that we have a $K$-valued point $x\in S^{\an}(K)$ not lying on the divisor $D^{\an}$. Then we can take the fiber $A_x$ at $x$ of the abelian variety $A$, and the fiber of $M^{\an}(A,\psi)$ at $x$ is canonically identified with $H^1_{\dR}(A_x/K)$ as a filtered $K$-vector space.

As above, the evaluation map $f\mapsto f(x)$ on $\Reg{S}^{\an}$ can now be extended to a map $x_{\pi}$ on $\Reg{S}^{\an,\log}$ with $x_{\pi}(\ell_r)=\log_{\pi}(r(x))$. Specializing (\ref{semistable:eqn:paralleltrans}) along $x_{\pi}$, we obtain an isomorphism of $K$-vector spaces:
\begin{align}\label{semistable:eqn:absthyodokatox}
 \beta_{\on{H-K},x,\pi}:K\otimes_{K_0}M_0(A)\xrightarrow{\simeq}H^1_{\dR}(A_x/K).
\end{align}

On the other hand, we also have the module $M_0(A_x)$ of unipotent nearby cycles for $A_x$ and the isomorphism (\ref{semistable:eqn:absthyodokato}):
\[
 \beta_{\on{H-K},A_x,\pi}:K\otimes_{K_0}M_0(A_x)\xrightarrow{\simeq}H^1_{\dR}(A_x/K).
\]

Evaluation at $x$ induces a map $x^\sharp:\Lambda=R[a^{-1}]^\times/R^\times\to \overline{K}_0^\times/\Reg{\overline{K}_0}^\times$. From this we obtain an obvious functor $x^*:\mb{LFI}(k,\Lambda)\to\mb{LFI}(k,\overline{K}_0^\times/\Reg{\overline{K}_0}^\times)$.

\begin{prp}\label{semistable:prp:phinmodule}
There is a canonical isomorphism of $(\varphi,N)$-modules $x^*M_0(A)\xrightarrow{\simeq}M_0(A_x)$ compatible with the isomorphisms $\beta_{\on{H-K},x,\pi}$ and $\beta_{\on{H-K},A_x,\pi}$ above.
\end{prp}
\begin{proof}
  We will view $\Reg{K}$ as a quotient of $W\pow{u}$ by the ideal $(E(u))$. Since $R$ is log smooth, the map $x:R\to\Reg{K}$ admits a lift $\tilde{x}:R\to W\pow{u}$ that respects log structures. Let $A_{\tilde{x}}$ be the corresponding abelian variety over $W\pow{u}[u^{-1}]$; then we have the associated log Dieudonn\'e crystal $\Dieu(A_{\tilde{x}})$ over the log scheme $(\Spec W\pow{u})^{\log}$. The log crystalline nature of $\Dieu(A)$ now gives us natural $\varphi$-equivariant parallel isomorphisms:
  \[
   \mathcal{S}_{\pi}\otimes_{\tilde{x},R}\Dieu(A)(R)\xrightarrow{\simeq}\mathcal{S}_{\pi}\otimes_{W\pow{u}}\Dieu(A_{\tilde{x}})(W\pow{u})\xrightarrow{\simeq}\mathcal{M}(A).
  \]
  Tensoring everything with $\Reg{Y(e)}^{\an,\log}$, and using (\ref{semistable:eqn:paralleltrans}) and \eqref{semistable:eqn:separallel}, we get a parallel $\varphi$-equivariant isomorphism:
  \[
   \Reg{Y(e)}^{\an,\log}\otimes_{K_0}M_0(A)\xrightarrow{\simeq}\Reg{Y(e)}^{\an,\log}\otimes_{K_0}M_0(A_x).
  \]
  From this the proposition is immediate.
\end{proof}

\subsubsection{}\label{semistable:subsubsec:compstatement}
Fix a finite extension $K/K_0$ as above. Let $\Gamma_K$ be the absolute Galois group $\Gal(\overline{K}_0/K)$. Fix $A$ in $\mb{DEG}(\Spec\Reg{K},\Spec K)$, and let $M_0(A)$ be the associated $\varphi$-module over $K_0$.

We refer to \cite{fontaine:periodes} for the period rings $\Bdr,\Bcris,\Bst$ and their properties. We will only note that both $\Bst$ is naturally a $\Bcris$-algebra, and that there is a canonical embedding $K\otimes_{K_0}\Bcris\into\Bdr$. The choice of uniformizer $\pi$ now permits us to extend this to an embedding $K\otimes_{K_0}\Bst\into\Bdr$. We fix such a choice.

\begin{prp}\label{semistable:prp:comparison}
There is a canonical isomorphism
\[
\beta_{\st,A}:\Bst\otimes_{\Rat_p}H^1_{\et}\bigl(A_{\overline{K}},\Rat_p\bigr)\xrightarrow{\simeq}\Bst\otimes_{K_0}M_0(A)
\]
compatible with all additional structures, and is such that, after base-change along the embedding $K\otimes_{K_0}\Bst\into\Bdr$ attached to $\pi$, the following diagram commutes:
\begin{equation}\label{semistable:diagram}
\begin{diagram}
  \Bdr\otimes_{\Rat_p}H^1_{\et}\bigl(A_{\overline{K}},\Rat_p\bigr)&\rTo_{\simeq}^{1\otimes\beta_{\st,A}}&\Bdr\otimes_K(K\otimes_{K_0}M_0(A))\\
  &\rdTo_{\simeq}&\dTo^{\simeq}_{1\otimes\beta_{\on{H-K},A,\pi}}\\
  &&\Bdr\otimes_KH^1_{\dR}(A/K).
\end{diagram}
\end{equation}
Here, the diagonal arrow is the canonical $p$-adic de Rham comparison isomorphism.
\end{prp}

The proof will be given in \eqref{appendix:subsec:comparison}.

\subsubsection{}\label{semistable:subsubsec:coleman_iovita}
As a corollary, we obtain a proof of Theorem~\ref{padicmonodromy}, a result that is originally due to Coleman-Iovita~\cite{coleman_iovita}:
\begin{proof}[Proof of Theorem~\ref{padicmonodromy}]
From~\eqref{semistable:prp:comparison}, we find that the representation $T_p(A)$ is crystalline if and only if the monodromy operator $N_0(A)$ on $M_0(A)$ is trivial.

Note that $N_0(A)$ is, by construction~\eqref{semistable:eqn:monodromy_residue}, the residue at $u=0$ of the logarithmic connection on $\mathcal{M}(A)$. This connection in turn was determined by the fact that $\mathcal{M}(A)$ was the evaluation of $\Dieu(A)$ along the formal divided power thickening $\Spf\Reg{K}\into\Spf\mathcal{S}_{\pi}$. The residue vanishes exactly when $\Dieu(A)$ is in the image of the natural (fully faithful) functor from the category of Dieudonn\'e crystals over $\Reg{K}$ to that of log Dieudonn\'e crystals over $\Reg{K}$.

So we have shown: $T_p(A)$ is crystalline if and only if $\Dieu(A)$ is a Dieudonn\'e crystal (and not just a log Dieudonn\'e crystal).

Set $S=\Spec\Reg{K}$, $U=\Spec K$, and let $Q_U$ be the object in $\mb{DD}(S,U)$ associated with $A$. Without loss of generality, we can replace $K$ by an unramified extension, and assume that $Y\coloneqq Q^{\et}$ and $X\coloneqq Q^{\mult}$ are constant. Associated with $Q_U$ is a trivialization $\tau$ over $U$ of the bi-extension $(c_Q\times\dual{c}_Q)^*\mathcal{P}_{Q^{\ab}}$ of $Y\times X$. The construction of $\Dieu(A)$ in~\eqref{semistable:subsubsec:logffull} now shows that it is an $F$-crystal precisely when $\tau$ extends to a trivialization over all of $S$. Equivalently, since $(y,x)\mapsto \nu_{\pi}(\tau(y,x))$ is non-degenerate, $\Dieu(A)$ is an $F$-crystal exactly when $Q^{\et}=Q^{\mult}=0$. But this can happen if and only if $A$ has good reduction.
\end{proof}

\section{Toroidal compactifications: background}\label{sec:background}

In this section, we review essential background material: The first topic is the general theory of toroidal compactifications of Shimura varieties in characteristic $0$; this is essentially due to Ash-Mumford-Rapoport-Tai~\cite{amrt}, who constructed the compactifications over $\Comp$, and Pink~\cite{pink:thesis}, who constructed canonical models for these compactifications over their reflex fields.

The second topic is the arithmetic construction of compactifications of Siegel modular varieties, due to Chai and Faltings~\cite{faltings_chai}.

In both cases, we have preferred to use resolutely ad\'elic constructions. This has the disadvantage that we have to re-do some of the presentation of the construction of boundary charts from~\cite{faltings_chai} and~\cite{lan:thesis}, but it also presents at least two advantages: The first of course is the conceptual clarity it brings to the discussion of Hecke correspondences; but the second, and more essential one, is that it allows for the construction of compactifications for Shimura varieties of abelian type, following a natural extension of the strategy in \cite{deligne:corvallis}. We will present this application in a future article.

\subsection{Compactifications in characteristic $0$}\label{background:subsec:char0}

In this sub-section, we will give a quick summary of the theory of toroidal compactifications of Shimura varieties in characteristic $0$. Our treatment will be rather formal, but details can be found in our main references,~\cites{amrt,pink:thesis}.

\subsubsection{}
We will use Pink's slightly more general definition of (pure) Shimura data from \cite[\S2.1]{pink:thesis}. Therefore, a Shimura datum will be a triple $(G,X,h)$, where $X$ is a $G(\Real)$-homogeneous space, and $h:X\to\Hom\bigl(\bb{S},G_{\Real}\bigr)$ is a $G(\Real)$-equivariant map with finite fibers such that $(G,h(X))$ is a Shimura datum in the sense of Deligne~\cite[1.5]{deligne:travaux}. As usual, $\bb{S}=\Res_{\Comp/\Real}\Gm$ is the Deligne torus. In what follows, $h$ will be clear from context, and we will consistently omit it from our notation. We will write $E(G,X)\subset\bb{C}$ for the reflex field attached to $(G,X)$.

By the definition of Shimura data, it follows that the composition $\Gmh{\Real}\into\bb{S}\xrightarrow{h_x}G_{\Real}$, for $x\in X$, is independent of $x$ and maps into the center of $G_{\Real}$: It is the \defnword{weight co-character} attached to $(G,X)$, and we will denote it by $w_0$.

\emph{We will always assume that $w_0$ is defined over $\Rat$.}

Let $K\subset G(\Adele_f)$ be a compact open sub-group of the ad\'{e}lic points of $G$. We will always assume that $K$ is of the form $K^pK_p$, where $K_p\subset G(\Rat_p)$ and $K^p\subset G(\Adele_f^p)$.

Attached to the triple $(G,X,K)$, we have the Shimura variety $\Sh_K(G,X)$ over the reflex field $E(G,X)$: In general, it is an algebraic stack over $E(G,X)$ equipped with a natural isomorphism of complex orbifolds:
\[
\Sh_K(G,X)(\Comp)=G(\Rat)\backslash X \times G(\Adele_f)/K.
\]

\subsubsection{}\label{background:subsubsec:siegel}
Our chief examples of Shimura data are Siegel Shimura data. These are associated with symplectic spaces $(H,\psi)$ over $\Rat$. The relevant reductive group is the group of symplectic similitudes $\GSp(H,\psi)$, and the symmetric domain is the space $\on{S}^{\pm}(H,\psi))$ of maps $h:\mathbb{S}\to G_{\Real}$ such that:
  \begin{enumerate}
    \item $h$ induces a Hodge structure of type $(-1,0),(0,-1)$ on $H(\Rat)$, so that we have a corresponding decomposition
    \[
    H(\Comp)=H_h^{-1,0}\oplus H_h^{0,-1};
    \]
    \item The symmetric form $(x,y)\mapsto\psi(x,h(i)y)$ is (positive or negative) definite on $V_{\Real}$.
  \end{enumerate}
The reflex field of a Siegel Shimura datum is $\Rat$.

Following \cite[2.6]{pink:thesis}, we can also make sense of a Siegel Shimura datum when $H=0$. We set $\GSp(0)\coloneqq\Gm$, and we take $\on{S}^{\pm}(0)$ to be the two-element set of isomorphisms of $\Int$-Hodge structures $\Isom(\Int,\Int(1))$ (equivalently, that of square roots of $-1$), with the action of $\Real^\times$ given by the sign character. We equip $\on{S}^{\pm}(0)$ with the constant map $h:\on{S}^{\pm}(0)\to\Hom\bigl(\bb{S},\Gmh{\Real}\bigr)$ carrying either isomorphism to the norm map $z\mapsto z\overline{z}$.

\subsubsection{}
Let $(G,X)$ be a Shimura datum. Let $G^{\ad}=\prod_{i=1}^m G_i$ be the decomposition into simple factors of the adjoint group $G^{\ad}$. We will say that a parabolic sub-group $P\subset G$ is \defnword{admissible} if, for $1\leq i\leq r$, the image $P_i$ of $P$ in $G_i$ is either a proper maximal parabolic sub-group, or all of $G_i$. In the language of \cite[Ch. III]{amrt}, such a sub-group corresponds to a \emph{rational boundary component} for $(G,X)$. Observe that this definition allows for the possibility that $P=G$.

Given a parabolic sub-group $P\subset G$, there is a canonical increasing filtration $(\Lie G)_{\bullet}$ on $\Lie G$ whose stabilizer in $G$ is $P$, and is such that $(\Lie G)_0=\Lie P$. Let $U_P\subset P$ be the unipotent radical; then $(\Lie G)_{-1}=\Lie U_P$. If $P$ is in addition admissible, then the filtration satisfies $(\Lie G)_j=0$, for $j\leq -3$, and the center $W_P=Z(U_P)$ of $U_P$ satisfies: $(\Lie G)_{-2}=\Lie W_P$.

\subsubsection{}
Fix an admissible parabolic sub-group $P\subset G$. Choose a co-character $w:\Gm\to P$ that splits the filtration $(\Lie G)_{\bullet}$ and such that $ww_0^{-1}$ factors through the derived group $G^{\mathrm{der}}$ of $G$.

Given $x\in X$ and an algebraic representation $M$ of $G$ over $\Rat$, let $F^\bullet_xM(\Comp)$ be the Hodge filtration on $M(\Comp)$ associated with the Hodge structure on $M(\Rat)$ induced from $h_x$. Also, let $W_{\bullet}M$ be the unique increasing filtration on $M$ that is split by $w$. Then, for an appropriate choice of the cocharacter $w$, it follows from \cite[\S 4]{pink:thesis} that, for all such representations $M$, $F^\bullet_xM(\Comp)$ determines a canonical mixed Hodge structure on $M(\Rat)$, for which $W_{\bullet}M(\Rat)$ is the weight filtration; cf. also~\cite[(4.1.2)]{brylinski:1motifs}. Moreover, from~\cite[(4.6)]{pink:thesis}, we find that there exists a canonical homomorphism:
\[
 \varpi_x:\bb{S}_{\Comp}=\Gmh{\Comp}\times\Gmh{\Comp}\to P_{\Comp}
\]
splitting the mixed Hodge structure on $M(\Rat)$, for every representation $M$, and whose restriction to the diagonal embedding of $\Gmh{\Comp}$ in $\bb{S}_{\Comp}$ is conjugate under $P(\Comp)$ to $w$.\footnote{In the notation of \emph{loc. cit.}, this is actually the map $\varpi_x\circ h_{\infty}$.}

\subsubsection{}\label{background:subsubsec:fp2}
Let $Q_P\subset P$ be the smallest normal sub-group such that the maps $\varpi_x$, for $x\in X$, factor through $Q_{P,\Comp}$. Then the group $Q_P(\Real)W_\Phi(\Comp)$ acts on the set $\pi_0(X)$ of connected components of $X$ via the maps
\[
 \pi_0\bigl(Q_P(\Real)W_\Phi(\Comp)\bigr)\to \pi_0\bigl(Q_P(\Real)\bigr)\to\pi_0(G(\Real)).
\]

In \cite[4.11]{pink:thesis}, Pink shows that, given $x,x'$ in the same connected component of $X$, $\varpi_x$ and $\varpi_{x'}$ are conjugate under an element of $Q_P(\Real)W_P(\Comp)$. In particular, given any connected component $X^+\subset X$ and $x\in X^+$, the $Q_P(\Real)W_\Phi(\Comp)$-orbit $F_{P,X^+}$ of $(X^+,\varpi_x)$ in $\pi_0(X)\times\Hom(\bb{S}_{\Comp},Q_{\Phi,\Comp})$ does not depend on the choice of $x\in X^+$.

There is a unique complex structure on $F_{P,X^+}$ such that the natural map to $\Hom(\bb{S}_{\Comp},Q_{\Phi,\Comp})$ is holomorphic.

\subsubsection{}\label{background:subsubsec:tube}
The map $X^+\to F_{P,X^+}$ carrying $x$ to $(X^+,\varpi_x)$ is an open immersion of complex analytic spaces. It exhibits $X^+$ as a tube domain.

To see this, we first note that there is a continuous map $u:F_{P,X^+}\to W_P(\Real)(-1)$ that assigns to a pair $(X_1^+,\varpi)\in F_{P,X^+}$ the unique element $u(\varpi)\in W_\Phi(\Real)(-1)$ such that $u(\varpi)\varpi u(\varpi)^{-1}$ is defined over $\Real$. Here, we are viewing $W_P(\Real)$ as an $\Real$-vector space, and as usual $W_P(\Real)(-1)=(2\pi\sqrt{-1})^{-1}W_P(\Real)\subset W_P(\Comp)$.

By \cite[4.15]{pink:thesis}, we can now find a canonical open non-degenerate self-adjoint convex cone (cf.~\cite[Ch. II,\S 1.1]{amrt} for the terminology) $\mb{H}_{P,X^+}\subset W_P(\Real)(-1)$, homogeneous under $P(\Real)$, such that
\[
X^+ = u^{-1}\bigl(\mb{H}_{P,X^+}\bigr)\subset F_{P,X^+}.
\]

\subsubsection{}\label{background:subsubsec:mixeddata}
A \defnword{cusp label representative} or \defnword{clr} for short is a triple 
\[
\Phi = (P_\Phi,X^+_\Phi,g_\Phi),
\]
where $P_\Phi\subset G$ is an admissible parabolic, $X^+_\Phi\subset X$ is a connected component, and $g_\Phi\in G(\Adele_f)$. 

Set $Q_\Phi = Q_\Phi$, $U_\Phi = U_P$, $W_\Phi = W_\Phi$. Also, set $\overline{Q}_\Phi=Q_\Phi/W_\Phi$. Its unipotent radical is $V_\Phi=U_\Phi/W_\Phi$, which is a commutative group over $\Rat$. Let $G_{\Phi,h}=Q_\Phi/U_\Phi$ be the Levi quotient of $Q_\Phi$: The conjugation action of $Q_\Phi$ on $W_\Phi$ and $V_P$ factors through $G_{\Phi,h}$.

Also, set
\[
D_\Phi = F_{P,X^+}\;;\;\overline{D}_\Phi =W_\Phi(\Comp)\backslash D_\Phi\;;\; D_{\Phi,h} = V_\Phi(\Real)\backslash \overline{D}_\Phi
\]

The map $D_\Phi = F_{P,X^+}\to\Hom(\bb{S}_{\Comp},Q_{\Phi,\Comp})$ induces a map
\[
 D_{\Phi,h}\to\Hom(\bb{S}_{\Comp},G_{\Phi,h,\Comp}),
\]
which actually factors through $\Hom(\bb{S},G_{\Phi,h,\Real})$. This makes the pair $(G_{\Phi,h},D_{\Phi,h})$ a (pure) Shimura datum, whose reflex field is again $E$.

Moreover, the pairs $(Q_\Phi,D_\Phi)$ and $(\overline{Q}_\Phi,\overline{D}_\Phi)$ are  mixed Shimura data in the language of \cite[Ch. 2]{pink:thesis}.

Given any compact open sub-group $K\subset G(\Adele_f)$, set $K_{\Phi,P}=P_{\Phi}(\Adele_f)\cap gKg^{-1}$, $K_{\Phi}=Q_{\Phi}(\Adele_f)\cap K_{\Phi,P}$. Let $\overline{K}_{\Phi}\subset \overline{Q}_\Phi(\Adele_f)$ be the image of $K_{\Phi}$ and let $K_{\Phi,h}\subset G_{\Phi,h}(\Adele_f)$ be the image of $K_{\Phi}$. 

We can now form the mixed Shimura varieties
\begin{align*}
\Sh_{K_\Phi}(Q_\Phi,D_\Phi)(\Comp)&=Q_\Phi(\Rat)\backslash\bigl(D_\Phi\times Q_\Phi(\Adele_f)\bigr)/K_{\Phi};\\
\Sh_{\overline{K}_\Phi}(\overline{Q}_\Phi,\overline{D}_\Phi)(\Comp)&=\overline{Q}_\Phi(\Rat)\backslash\bigl(\overline{D}_\Phi\times\overline{Q}_\Phi(\Adele_f)\bigr)/\overline{K}_{\Phi};\\
\Sh_{K_{\Phi,h}}(G_{\Phi,h},D_{\Phi,h})(\Comp)&=G_{\Phi,h}(\Rat)\backslash\bigl(D_{\Phi,h}\times G_{\Phi,h}(\Adele_f)\bigr)/K_{\Phi,h}.
\end{align*}

We obtain a tower: $\Sh_{K_\Phi}(Q_\Phi,D_\Phi)(\Comp)\to \Sh_{\overline{K}_\Phi}(\overline{Q}_\Phi,\overline{D}_\Phi)(\Comp)\to \Sh_{K_{\Phi,h}}(G_{\Phi,h},D_{\Phi,h})(\Comp)$.

Here, by construction, $\Sh_{K_{\Phi,h}}(G_{\Phi,h},D_{\Phi,h})(\Comp)$ is the space of $\Comp$-valued points of the Shimura variety $\Sh_{K_{\Phi,h}}(G_{\Phi,h},D_{\Phi,h})=\Sh_{K_{\Phi,h}}(G_{\Phi,h},D_{\Phi,h})$.

\subsubsection{}\label{background:subsubsec:varmixedhodge}
We will need some generalities about sheaves on $\Sh_{K_\Phi}(Q_\Phi,D_\Phi)(\Comp)$. Let $M$ be an algebraic representation of $Q_\Phi$ defined over $\Rat$. The co-character $w:\Gm\to G$ factors through $Q_\Phi$ and so induces a filtration $W_\bullet M$ on $M$.

Now, $D_{\Phi}\times M(\Rat)$ (with $M(\Rat)$ given the discrete topology) is a $Q_\Phi(\Rat)$-equivariant local system of $\Rat$-vector spaces over $D_\Phi$. It is equipped with a $Q_\Phi(\Rat)$-equivariant filtration $D_\Phi\times W_{\bullet}M(\Rat)$. For every point $y\in D_\Phi$, the corresponding element $\varpi_y:\bb{S}_{\Comp}\to Q_{\Phi,\Comp}$ splits a canonical mixed Hodge structure on $M(\Rat)$ with underlying weight filtration $W_{\bullet}M(\Rat)$, which is again $Q_\Phi(\Rat)$-equivariant.

Therefore, the quotient:
\[
 \bm{M}_B(\Phi)=Q_\Phi(\Rat)\backslash\bigl((D_\Phi\times M(\Rat))\times Q_\Phi(\Adele_f)/K_{\Phi}\bigr)
\]
is a local system of $\Rat$-vector spaces over $\Sh_{K_\Phi}(Q_\Phi,D_\Phi)(\Comp)$, underlying a variation of mixed Hodge structures:
\[
 \bm{M}_{\on{MH}}(\Phi)=\bigl(\bm{M}_B(\Phi),W_{\bullet}\bm{M}_B(\Phi),F^\bullet(\Reg{\Sh_{K_\Phi}(Q_\Phi,D_\Phi)(\Comp)}^{\an}\otimes\bm{M}_B(\Phi))\bigr).
\]

This construction is functorial in $M$.

If the representation on $M$ factors through $\overline{Q}_\Phi$ (resp. $G_{\Phi,h}$) then $\bm{M}_{\on{MH}}(\Phi)$ is canonically identified with the pull-back of a variation of mixed Hodge structures on $\Sh_{\overline{K}_\Phi}(\overline{Q}_\Phi,\overline{D}_\Phi)(\Comp)$ (resp. $\Sh_{K_{\Phi,h}}(G_{\Phi,h},D_{\Phi,h})(\Comp)$), which we will again denote by $\bm{M}_{\on{MH}}(\Phi)$.

\subsubsection{}\label{background:subsubsec:varmixedintegral}
Set:
\begin{equation}\label{background:eqn:inverse_limit}
 \Sh(Q_\Phi,D_\Phi)(\Comp)=\varprojlim_{K}\Sh_{K_\Phi}(Q_\Phi,D_\Phi)(\Comp),
\end{equation}
where the projective limit is taken over the directed system of neat compact open sub-groups $K\subset G(\Adele_f)$. 

We have a natural action of $K_{\Phi}$ on $\Sh(Q_\Phi,D_\Phi)(\Comp)$ and an identification $\Sh(Q_\Phi,D_\Phi)(\Comp)/K_{\Phi}=\Sh_{K_\Phi}(Q_\Phi,D_\Phi)$.

Let $Z_\Phi\subset Q_\Phi$ be the center, and let $Z_\Phi(\Rat)^0\subset Z_\Phi(\Rat)$ be the sub-group of elements acting trivially on $F_{P,X^+}$. Let $\Gamma_Z\subset Z_\Phi(\Rat)^0$ be any arithmetic sub-group, and let $\overline{\Gamma}_Z\subset Q_\Phi(\Adele_f)$ be its closure. By \cite[3.7]{pink:thesis}, we have:
\begin{align}\label{background:eqn:xiphiinvlimit}
\Sh(Q_\Phi,D_\Phi)(\Comp)=Q_\Phi(\Rat)\backslash\bigl(D_\Phi\times Q_\Phi(\Adele_f)/\overline{\Gamma}_Z\bigr).
\end{align}

For any $Q_\Phi$-representation $M$, we have a $Q_\Phi(\Rat)\times\overline{\Gamma}_Z$-equivariant isomorphism:
\begin{align}\label{background:eqn:madeletriv1}
 D_\Phi\times  Q_\Phi(\Adele_f)\times M(\Adele_f)&\xrightarrow{\simeq}D_\Phi\times M(\Adele_f)\times Q_\Phi(\Adele_f)\\
(y,q,m)&\mapsto (y,qm,q).\nonumber
\end{align}
Here, on the left, $Q_\Phi(\Rat)$ acts trivially on $M(\Adele_f)$, and on the right via the representation of $Q_\Phi$ on $M$

Set $\bm{M}_{\Adele_f}(\Phi)=\Adele_f\otimes\bm{M}_B(\Phi)$. Then, quotienting~\eqref{background:eqn:madeletriv1} by $Q_\Phi(\Rat)\times\overline{\Gamma}_Z$ and using~\eqref{background:eqn:xiphiinvlimit}, we obtain a canonical isomorphism of $\Adele_f$-sheaves:
\begin{align}\label{background:eqn:madeletriv}
\underline{M}(\Adele_f)&\xrightarrow{\simeq}\bm{M}_{\Adele_f}(\Phi)\vert_{\Sh(Q_\Phi,D_\Phi)(\Comp)}.
\end{align}

In particular, if $M(\widehat{\Int})\subset M(\Adele_f)$ is a $\widehat{\Int}$-lattice stabilized by $K_{\Phi}$, then its image under~\eqref{background:eqn:madeletriv} will be a $K_{\Phi}$-equivariant $\widehat{\Int}$-lattice in $\bm{M}_{\Adele_f}(\Phi)\vert_{\Sh(Q_\Phi,D_\Phi)(\Comp)}$ and will thus descend to a $\widehat{\Int}$-lattice $\bm{M}_{\widehat{\Int}}(\Phi)\subset\bm{M}_{\Adele_f}(\Phi)$. In turn this gives us a $\Int$-lattice:
\[
 \bm{M}_{B}(\Phi)_{\Int}=\bm{M}_B(\Phi)\cap\bm{M}_{\widehat{\Int}}(\Phi)\subset\bm{M}_B(\Phi).
\]
This refines $\bm{M}_{\on{MH}}(\Phi)$ to a variation of mixed $\Int$-Hodge structures $\bm{M}_{\on{MH}}(\Phi)_{\Int}$ over $\Sh_{K_\Phi}(Q_\Phi,D_\Phi)(\Comp)$. Of course, this refinement depends on the choice of the $K_{\Phi}$-stable lattice $M(\widehat{\Int})$.

\subsubsection{}\label{background:subsubsec:abscheme}
The variation of Hodge structures $\bm{V}_{\on{MH}}(\Phi)$ over $\Sh_{K_{\Phi,h}}(G_{\Phi,h},D_{\Phi,h})(\Comp)$ attached to the representation $V_P$ has weights $(-1,0),(0,-1)$; that is, for every point $\varpi\in D_\Phi$, the corresponding Hodge structure on $V_\Phi(\Comp)$ arises from a complex structure on $V_\Phi(\Real)$. By quite general principles, this variation is in fact polarizable; cf.~\cite[1.12]{pink:thesis}.

As above, let $Z_\Phi(\Rat)^0\subset Z_\Phi(\Rat)$ be the sub-group of elements acting trivially on $D_\Phi$, and let $K_{\Phi,U}\subset U_\Phi(\Adele_f)$ be the image of the sub-set:
\[
 \{(z,u)\in Z_\Phi(\Rat)^0\times U_\Phi(\Adele_f): z\cdot u\in K_{\Phi}\}\subset Z_\Phi(\Rat)\times U_\Phi(\Adele_f)
\]
under the projection map $Z_\Phi(\Rat)\times U_\Phi(\Adele_f)\to U_\Phi(\Adele_f)$. Let $K_{\Phi,V}\subset V_P(\Adele_f)$ be the image of $K_{\Phi,U}$: This is a $K_{\Phi}$-stable $\widehat{\Int}$-lattice, and therefore refines $\bm{V}_{\on{MH}}(\Phi)$ to a variation of $\Int$-Hodge structures $\bm{V}_{\on{MH}}(\Phi)_{\Int}$.

Thus, we have a canonical smooth family of abelian varieties $A_K(\Phi)(\Comp)\to \Sh_{K_{\Phi,h}}(G_{\Phi,h},D_{\Phi,h})(\Comp)$, whose relative integral homology is identified with $\bm{V}_{\on{MH}}(\Phi)_{\Int}$.

The underlying vector bundle with integrable connection is the quotient of the $Q_\Phi(\Rat)$-equivariant bundle:
\[
 V_\Phi(\Real)\times D_\Phi\to D_\Phi,
\]
where, for any point $\varpi\in D_\Phi$, $V_\Phi(\Real)\times\{\varpi\}$ is equipped with the complex structure attached to $\varpi$.

The natural action of $V_\Phi(\Real)$ on $\overline{D}_\Phi$ via conjugation descends to an action of the family of abelian varieties $A_K(\Phi)(\Comp)$ on $\Sh_{\overline{K}_\Phi}(\overline{Q}_\Phi,\overline{D}_\Phi)(\Comp)$ over $\Sh_{K_{\Phi,h}}(G_{\Phi,h},D_{\Phi,h})(\Comp)$, making $\Sh_{\overline{K}_\Phi}(\overline{Q}_\Phi,\overline{D}_\Phi)(\Comp)$ an $A_K(\Phi)(\Comp)$-torsor over $\Sh_{K_{\Phi,h}}(G_{\Phi,h},D_{\Phi,h})(\Comp)$.

% If we choose a splitting $\overline{Q}_\Phi = V_\Phi\rtimes G_{\Phi,h}$ such that $\overline{K}_{\Phi}$ is of the form $K_{\Phi,V}\rtimes K_{\Phi,h}$,\footnote{This is not always possible, but will be so for a cofinal collection of level subgroups $K$.} then we obtain a section $D_\Phi\to F^{(1)}_{P,X^+}$, which descends to a section of $\Sh_{\overline{K}_\Phi}(\overline{Q}_\Phi,\overline{D}_\Phi)(\Comp)$ over $\Sh_{K_{\Phi,h}}(G_{\Phi,h},D_{\Phi,h})(\Comp)$, and makes it a \emph{trivial} $A_K(\Phi)(\Comp)$-torsor. For all this, cf.~\cite[3.12]{pink:thesis}

\subsubsection{}\label{background:subsubsec:torustors}
The action of $G_{\Phi,h}$ on $W_\Phi$ is via a character $\nu_\Phi:G_{\Phi,h}\to\Gm$; cf.~\cite[2.14]{pink:thesis}. Set $K_{\Phi,W}=W_\Phi(\Adele_f)\cap K_{\Phi,U}$. Then the variation of $\Int$-Hodge structures associated with the representation $W_\Phi$, and the lattice $K_{\Phi,W}$ is the homology of a canonical family of algebraic tori over $\Sh_{K_{\Phi,h}}(G_{\Phi,h},D_{\Phi,h})(\Comp)$. Moreover, from~\cite[3.12(b)]{pink:thesis}, we see that the map $\Sh_{K_\Phi}(Q_\Phi,D_\Phi)(\Comp)\to \Sh_{\overline{K}_\Phi}(\overline{Q}_\Phi,\overline{D}_\Phi)(\Comp)$ is naturally a torsor under this family.

In fact, this family of tori is constant. To see this, first extend $\nu_\Phi$ to a surjective map of Shimura data:
\begin{equation}\label{background:eqn:nuP_shimura_data}
 \nu_\Phi:(G_{\Phi,h},D_{\Phi,h}) \to (\Gm,\on{S}^{\pm}(0)).
\end{equation}
Such an extension is determined entirely by where it sends a point in $X^+$ under the induced composition
\[
X^+ \to D_{\Phi}\to D_{\Phi,h}\to \on{S}^{\pm}(0).
\]
This shows that there are exactly two possibilities for it.

Set $P_\Phi(0)=W_\Phi\rtimes\Gm$, where $\Gm$ acts on $W_\Phi =\Lie W_\Phi$ via scalar multiplication. Write $\pi(0):P_\Phi(0)\to\Gm$ for the natural projection. Set
\[
 D_\Phi(0) = \{(\varpi,\lambda)\in\Hom(\bb{S}_{\Comp},P_\Phi(0)_{\Comp})\times \on{S}^{\pm}(0):\;(\pi\circ h)(x,y)=xy\}.
\]
Here, we are identifying $\bb{S}_{\Comp}$ with $\Gmh{\Comp}\times\Gmh{\Comp}$ in the usual way.

Set
\[
 K_{\Phi}(0) = K_{\Phi,W}\rtimes\nu_\Phi(K_{\Phi,h})\subset P_\Phi(0)(\Adele_f).
\]

Then it follows from~\cite[3.12(a)]{pink:thesis} that the projection of complex manifolds
\begin{equation}\label{background:eqn:torus_family}
 P_\Phi(0)(\Rat)\backslash D_\Phi(0)\times P_\Phi(0)(\Adele_f)/K_{\Phi}(0)\to \Rat^\times\backslash \on{S}^{\pm}(0)\times\Adele_f^\times/\nu_\Phi(K_{\Phi,h})
\end{equation}
is naturally a family of smooth commutative groups over the base, which is identified with $\Sh_{\nu_\Phi(K_{\Phi,h})}(\Gm,\on{S}^{\pm}(0))(\Comp)$.

In fact, we can be more precise. Set
\begin{equation}\label{background:eqn:bkphi_defn}
 \mb{B}_K(\Phi) \coloneqq (W_{\Phi}(\Rat)\cap K_{\Phi,W})(-1)\subset W_{\Phi}(\Rat)(-1),
\end{equation}
and let $\mb{E}_K(\Phi)$ be the torus over $\Int$ with cocharacter group $\mb{B}_K(\Phi)$. Equivalently, it is the torus with character group $\mb{S}_K(\Phi) \coloneqq \dual{\mb{B}_K(\Phi)}$.

Then by~\cite[3.16]{pink:thesis}, we find that there is a canonical isomorphism of families of complex groups over $\Sh_{\nu_\Phi(K_{\Phi,h})}(\Gm,\on{S}^{\pm}(0))(\Comp)$:
\[
 P_\Phi(0)(\Rat)\backslash D_\Phi(0)\times P_\Phi(0)(\Adele_f)/K_{\Phi}(0)\xrightarrow{\simeq}\mb{E}_K(\Phi)(\Comp)\times \Sh_{\nu_\Phi(K_{\Phi,h})}(\Gm,\on{S}^{\pm}(0))(\Comp).
\]
Moreover, the family of tori in question is simply the pullback over $\Sh_{K_{\Phi,h}}(G_{\Phi,h},D_{\Phi,h})(\Comp)$ under the map induced by~\eqref{background:eqn:nuP_shimura_data} of the constant torus on the right hand side of the above isomorphism.

Of course, this identification with the constant torus $\mb{E}_K(\Phi)(\Comp)\times\Sh_{K_{\Phi,h}}(G_{\Phi,h},D_{\Phi,h})(\Comp)$ depends on the choice of the map~\eqref{background:eqn:nuP_shimura_data}. A different choice will change the identification by a sign. We will assume such an identification in the sequel, and will ignore its lack of canonicity from now on, since it does not play any essential role in this article.

\subsubsection{}\label{background:subsubsec:canmodels}
The spaces $\Sh_{K_\Phi}(Q_\Phi,D_\Phi)(\Comp)$, $\Sh_{\overline{K}_\Phi}(\overline{Q}_\Phi,\overline{D}_\Phi)(\Comp)$ and $\Sh_{K_{\Phi,h}}(G_{\Phi,h},D_{\Phi,h})(\Comp)$ admit canonical models $\Sh_{K_\Phi}(Q_\Phi,D_\Phi)$, $\Sh_{\overline{K}_\Phi}(\overline{Q}_\Phi,\overline{D}_\Phi)$ and $\Sh_{K_{\Phi,h}}(G_{\Phi,h},D_{\Phi,h})$ over the reflex field $E\coloneqq E(G,X)$; cf. Ch. 11 of~\cite{pink:thesis}.

These models are characterized by the following properties:
\begin{itemize}
\item $\Sh_{K_{\Phi,h}}(G_{\Phi,h},D_{\Phi,h})$ is the canonical model over $E$;
\item As the level $K$ varies, the transition maps in the inverse limit~\eqref{background:eqn:inverse_limit} are defined over $E$, giving us a descent $\Sh(Q_\Phi,D_\Phi)$ over $E$ for $\Sh(Q_\Phi,D_\Phi)(\Comp)$.
\item The Hecke action of $Q_{\Phi}(\Adele_f)$ on $\Sh(Q_\Phi,D_\Phi)(\Comp)$ by right translation via the uniformization~\eqref{background:eqn:xiphiinvlimit} descends to an action on $\Sh(Q_\Phi,D_\Phi)$.
\end{itemize}

The structures defined above also descend over the reflex field: $\Sh_{K_\Phi}(Q_\Phi,D_\Phi)\to \Sh_{\overline{K}_\Phi}(\overline{Q}_\Phi,\overline{D}_\Phi)$ is an $\mb{E}_K(\Phi)$-torsor; the family of abelian varieties $A_K(\Phi)(\Comp)\to \Sh_{K_{\Phi,h}}(G_{\Phi,h},D_{\Phi,h})(\Comp)$ descends canonically to an abelian scheme $A_K(\Phi)\to \Sh_{K_{\Phi,h}}(G_{\Phi,h},D_{\Phi,h})$, and the $A_K(\Phi)$-torsor structure on $\Sh_{\overline{K}_\Phi}(\overline{Q}_\Phi,\overline{D}_\Phi)(\Comp)$ descends to one on $\Sh_{\overline{K}_\Phi}(\overline{Q}_\Phi,\overline{D}_\Phi)$ over $\Sh_{K_{\Phi,h}}(G_{\Phi,h},D_{\Phi,h})$.

 % For a cofinal collection of level subgroups $K$, $\Sh_{\overline{K}_\Phi}(\overline{Q}_\Phi,\overline{D}_\Phi)$ will be a trivializable $A_K(\Phi)$-torsor.

\subsubsection{}\label{background:subsubsec:opencovering}
Fix a neat compact open subgroup $K\subset G(\Adele_f)$. For any subgroup $H\subset G$ defined over $\Rat$, let $H(\Rat)_+\subset H(\Rat)$ be the pre-image in $H(\Rat)$ of the connected component of the identity in the real Lie group $G^{\ad}(\Real)$. 

Consider the open immersion (cf.~\cite[6.10]{pink:thesis}):
\begin{align}\label{background:eqn:adelicopenimm}
 \on{U}_{K_\Phi}(Q_\Phi,D_\Phi)\coloneqq Q_\Phi(\Rat)_+\backslash X^+_\Phi\times Q_\Phi(\Adele_f)/K_{\Phi}&\to\Sh_{K_\Phi}(Q_\Phi,D_\Phi)(\Comp)\\
 [(x,q)]&\mapsto [(X^+_\Phi,\varpi_x,q)]\nonumber.
\end{align}

Let $\Sh_K=\Sh_K(G,X)$; then we have a natural map:
\begin{align}\label{background:eqn:adeliccover}
\on{U}_{K_\Phi}(Q_\Phi,D_\Phi)&\to\Sh_K(\Comp)\\
[(x,q)]&\mapsto [(x,qg)]\nonumber.
\end{align}
On connected components, this map is isomorphic to $\Gamma_Q\backslash X^+_\Phi\to \Gamma_G\backslash X^+_\Phi$, where $\Gamma_G\subset G(\Rat)_+$ and $\Gamma_Q\subset Q_\Phi(\Rat)_+$ are arithmetic sub-groups with $\Gamma_Q\subset\Gamma_G$. Moreover, since $K$ is neat, the action of $\Gamma_G$ on $X^+_\Phi$ is free and properly discontinuous, so that~\eqref{background:eqn:adeliccover} is a local isomorphism of analytic spaces.

\subsubsection{}\label{background:subsubsec:equivrelation}
Let $\on{U}_K$ be the disjoint union of the spaces $\on{U}_{K_\Phi}(Q_\Phi,D_\Phi)$, as $\Phi$ varies over the clrs for $(G,X)$. Then we obtain a surjective locally \'etale covering $\on{U}_K\to\Sh_K(\Comp)$. There is now a Hausdorff equivalence relation $\sim$ on $\on{U}_K$ so that $\Sh_K(\Comp)$ is identified with $\on{U}_K/\sim$. Following~\cite[6.11]{pink:thesis}, we can describe this relation explicitly.

For this, it will be convenient to introduce some new notation: Given two admissible parabolics $P_1,P_2\subset G$ and $\gamma\in G(\Rat)$, we will write $P_1\xrightarrow{\gamma}P_2$ if the following equivalent conditions hold (cf.~\cite[III.4.8]{amrt} for a proof of their equivalence):
\begin{itemize}
\item $\gamma W_{P_1}\gamma^{-1}\supset W_{P_2}$;
\item $\gamma Q_{P_1}\gamma^{-1}\subset Q_{P_2}$.
\end{itemize}

Given two clrs $\Phi_1$ and $\Phi_2$, $\gamma\in G(\Rat)$ and $q_2\in Q_{\Phi_2}(\Adele_f)$, we will write $\Phi_1\xrightarrow{(\gamma,q_2)_K}\Phi_2$ if the following hold:
\begin{itemize}
  \item $P_{\Phi_1}\xrightarrow{\gamma}P_{\Phi_2}$;
  \item $\gamma\cdot X_{\Phi_1}^+\in \pi_0(X)$ is contained in the $Q_{\Phi}(\Rat)$-orbit of $X_{\Phi_2}^+$;
  \item $\gamma g_1\in q_2g_2K$.
\end{itemize}

Suppose now that $\Phi_1\xrightarrow{(\gamma,q_2)_K}\Phi_2$. Then we obtain a map
\begin{align}\label{background:eqn:mapcusplabels}
\rho(\gamma,q_2):\on{U}_{K_{\Phi_1}}(Q_{\Phi_1},D_{\Phi_1})&\to\on{U}_{K_{\Phi_2}}(Q_{\Phi_2},D_{\Phi_2})\\
[(x,q)]&\mapsto [(\gamma\cdot x,\on{int}(\gamma)(q)q_2)]\nonumber.
\end{align}
Here, we are using the identification
\[
\on{U}_{K_{\Phi_i}}(Q_{\Phi_i},D_{\Phi_i}) = Q_{\Phi_i}(\Rat)\backslash \widetilde{X}_{\Phi_i}\times Q_{\Phi_i}(\Adele_f)/K_{\Phi_i},
\]
where $\widetilde{X}_{\Phi_i}\subset X$ is the union of the connected components in the $Q_{\Phi}(\Rat)$-orbit of $X^+_{\Phi}\in \pi_0(X)$, and the fact that $\gamma$ carries $\widetilde{X}_{\Phi_1}$ onto $\widetilde{X}_{\Phi_2}$.

It follows from Lemma 6.12 of~\cite[6.11]{pink:thesis} that the equivalence relation $\sim$ on $\on{U}_K$ is generated by the graphs of all maps of this form.

\subsubsection{}\label{background:subsubsec:cuspsisom}
In the case where $\gamma\cdot P_{\Phi_1}=P_{\Phi_2}$, we can make~\eqref{background:eqn:mapcusplabels} even more explicit: In fact, it will be the restriction of an isomorphism $\Sh_{K_{\Phi_1}}(Q_{\Phi_1},D_{\Phi_1})\xrightarrow{\simeq}\Sh_{K_{\Phi_2}}(Q_{\Phi_2},D_{\Phi_2})$. We can reduce to the following three cases:
\begin{itemize}
\item $\gamma=1$, $q_2=1$ and $g_{\Phi_2}=g_{\Phi_1}k$, for $k\in K$: In this case, $K_{\Phi_2}=K_{\Phi_1}$, and the isomorphism is simply the identity.
\item $\gamma=1$, $q_2=q\in Q_{\Phi_1}(\Adele_f)$, and $q_fg_{\Phi_2}=g_{\Phi_1}$: In this case, $K_{\Phi_2}=qK_{\Phi_1}q^{-1}$, and right multiplication by $q$ on $Q_{\Phi_1}(\Adele_f)$ induces an isomorphism 
\[
[\cdot q]:\Sh_{K_{\Phi_1}}(Q_{\Phi_1},D_{\Phi_1})\xrightarrow{\simeq}\Sh_{K_{\Phi_2}}(Q_{\Phi_2},D_{\Phi_2}).
\]
\item $q_2=1$, and $\Phi_2=\on{int}(\gamma)(\Phi_1)$: In this case, conjugation by $\gamma$ induces an isomorphism $[\on{int}(\gamma)]:\Sh_{K_{\Phi_1}}(Q_{\Phi_1},D_{\Phi_1})\xrightarrow{\simeq}\Sh_{K_{\Phi_2}}(Q_{\Phi_2},D_{\Phi_2})$.
\end{itemize}

\emph{A priori}, all these maps are only defined on the level of $\Comp$-valued points. We obtain their descent to the canonical models by the very characterizing properties of such models.

In all these cases, we actually obtain isomorphisms of mixed Shimura varieties:
\begin{diagram}
\Sh_{K_{\Phi_1}}(Q_{\Phi_1},D_{\Phi_1})&\rTo& \Sh_{\overline{K}_{\Phi_1}}(\overline{Q}_{\Phi_1},\overline{D}_{\Phi_1})&\rTo&\Sh_{K_{\Phi_1,h}}(G_{\Phi_1,h},D_{\Phi_1,h})\\
\dTo_{\simeq}&&\dTo_{\simeq}&&\dTo_{\simeq}\\
\Sh_{K_{\Phi_2}}(Q_{\Phi_2},D_{\Phi_2})&\rTo& \Sh_{\overline{K}_{\Phi_2}}(\overline{Q}_{\Phi_2},\overline{D}_{\Phi_2})&\rTo&\Sh_{K_{\Phi_2,h}}(G_{\Phi_2,h},D_{\Phi_2,h})
\end{diagram}
There is a natural isomorphism $\mb{E}_K(\Phi_1)\xrightarrow{\simeq}\mb{E}_K(\Phi_2)$, giving $\Sh_{K_{\Phi_2}}(Q_{\Phi_2},D_{\Phi_2})$ the structure of an $\mb{E}_K(\Phi_1)$-torsor over $ \Sh_{\overline{K}_{\Phi_2}}(\overline{Q}_{\Phi_2},\overline{D}_{\Phi_2})$, and is such that, if we view $\Sh_{K_{\Phi_1}}(Q_{\Phi_1},D_{\Phi_1})$ as a scheme over $ \Sh_{\overline{K}_{\Phi_2}}(\overline{Q}_{\Phi_2},\overline{D}_{\Phi_2})$ via the middle isomorphism, then the isomorphism on the left is one of $\mb{E}_K(\Phi_1)$-torsors.

\subsubsection{}\label{background:subsubsec:deltakphi}
An even more particular case is where $\Phi_1= \Phi_2=\Phi$, and $\gamma$ belongs to the sub-group
\begin{equation}\label{background:eqn:deltaphi_prequotient}
P_\Phi(\Rat)_{\heartsuit}\cap(Q_\Phi(\Adele_f)g_\Phi Kg_\Phi^{-1}).
\end{equation}
Here, $P_\Phi(\Rat)_{\heartsuit}\subset P_\Phi(\Rat)$ is the stabilizer in $P_\Phi(\Rat)$ of the $Q_\Phi(\Real)$-orbit of $X^+_\Phi$ in $\pi_0(X)$.

Now, there exists $q\in Q_\Phi(\Adele_f)$ such that $\Phi\xrightarrow{(\gamma,q)_K}\Phi$. It is easy to see that the associated automorphism $\Sh_{K_\Phi}(Q_\Phi,D_\Phi)\xrightarrow{\simeq}\Sh_{K_\Phi}(Q_\Phi,D_\Phi)$ does not depend on the choice of $q_f$. Therefore, we obtain an action of the group
\begin{equation}\label{background:eqn:deltaphi}
\Delta_K(\Phi)=\frac{P_\Phi(\Rat)_{\heartsuit}\cap(Q_\Phi(\Adele_f)g_\Phi Kg_\Phi^{-1})}{Q_\Phi(\Rat)}
\end{equation}
on the tower $\Sh_{K_\Phi}(Q_\Phi,D_\Phi)\to \Sh_{\overline{K}_\Phi}(\overline{Q}_\Phi,\overline{D}_\Phi)\to \Sh_{K_{\Phi,h}}(G_{\Phi,h},D_{\Phi,h})$.

By construction, $\Delta_K(\Phi)$ is an arithmetic sub-group of $G_{\Phi,\ell}(\Rat)$, where $G_{\Phi,\ell} \coloneqq P_\Phi/Q_\Phi$. As mentioned in \cite[6.3]{pink:thesis}, a finite index subgroup of $\Delta_K(\Phi)$ admits a lift to the centralizer of $G_{\Phi,h}$ in the Levi quotient $L_{\Phi} = P_\Phi/U_{\Phi}$. Therefore, $\Delta_K(\Phi)$ acts on $\Sh_{K_{\Phi,h}}(G_{\Phi,h},D_{\Phi,h})$ via a finite quotient $\Delta^{\mathrm{fin}}_K(\Phi)$. In particular, if $L_\Phi\to G_{\Phi,\ell}$ admits a section giving an identification $L_\Phi=G_{\Phi,h}\times G_{\Phi,\ell}$, then $\Delta_K(\Phi)$ will act \emph{trivially} on $\Sh_{K_{\Phi,h}}(G_{\Phi,h},D_{\Phi,h})$.

Since the action of $Q_\Phi$ on $W_{\Phi}$ is via the cocharacter $\nu_\Phi$ (cf.~\ref{background:subsubsec:torustors}), the conjugation action of $P_\Phi$ on $W_{\Phi}$ induces a map
\[
 G_{\Phi,\ell}\to\PGL(W_{\Phi}).
\]

Let $ {G}^\beef_{\Phi,\ell}$ be the $\Gm$-extension $\GL(W_{\Phi})\times_{\PGL(W_{\Phi})}G_{\Phi,\ell}$ of $G_{\Phi,\ell}$. Note that the group~\eqref{background:eqn:deltaphi_prequotient} acts on $\mb{B}_K(\Phi)\subset W_{\Phi}(\Rat)(-1)$ via conjugation, and that this action factors through its image $\widetilde{\Delta}_K(\Phi)\subset {G}^\beef_{\Phi,\ell}(\Rat)$. Since $K$ is neat, $\widetilde{\Delta}_K(\Phi)$ maps isomorphically onto $\Delta_K(\Phi)\subset G_{\Phi,\ell}(\Rat)$. Therefore, we find that $\Delta_K(\Phi)$ has a natural action on $W_{\Phi}(\Real)(-1)$, which preserves $\mb{H}_{P,X^+}$ as well as $\mb{B}_K(\Phi)$.

\subsubsection{}\label{background:subsubsec:toric}
We will need some standard terminology about torus embeddings; cf.~\cite{kkms}. Let $\mb{V}$ be a finite dimensional $\Rat$-vector space. A \defnword{rational polyhedral cone} $\sigma\subset \Real\otimes\mb{V}$ is a sub-set, for which there exist finitely many linear functions $f_1,\ldots,f_r\in\dual{\mb{V}}$ such that:
\[
 \sigma=\{x\in \Real\otimes\mb{V}:\;f_i(x)\geq 0,\text{ for $1\leq i\leq r$}\}.
\]

A \defnword{face} of $\sigma$ is a subset of the form $\{f_i=0:i\in I\}\subset\sigma$, where $I\subset\{1,\ldots,r\}$. The \defnword{interior} $\sigma^{\circ}\subset\sigma$ is the complement of the proper faces of $\sigma$.

We say that $\sigma$ is \defnword{non-degenerate} if it does not contain any non-zero linear sub-spaces of $\mb{V}$.

Fix a $\Int$-lattice $\mb{X}\subset\mb{V}$. We will say that $\sigma$ is \defnword{smooth} (with respect to $\mb{X}$) if $f_1,\ldots,f_r\in\dual{\mb{V}}$ can be chosen to be a sub-set of a basis for $\dual{\mb{X}}$.

For any rational polyhedral cone $\sigma\subset\Real\otimes\mb{X}$, set
\[
 \dual{\mb{X}}_{\sigma}=\{f\in\dual{\mb{X}}:\;f(x)\geq 0,\text{ for all $x\in\sigma$}\}
\]
Let $\on{T}_{\mb{X}}$ be the torus over $\Int$ with co-character group $\mb{X}$. To $\sigma$, we can attach the affine torus embedding:
\[
 \on{T}_{\mb{X}}=\Spec\Int[\dual{\mb{X}}]\into\Spec\Int[\dual{\mb{X}}_{\sigma}]=\on{T}_{\mb{X}}(\sigma).
\]
$\on{T}_{\mb{X}}(\sigma)$ is smooth over $\Int$ precisely when $\sigma$ is smooth.

There is a unique closed $\on{T}_{\mb{X}}$-orbit $\on{O}_{\mb{X}}(\sigma)\subset \on{T}_{\mb{X}}(\sigma)$. This is defined by the ideal $\mb{I}_{\sigma}\subset\Int[\dual{\mb{X}}_{\sigma}]$ which is generated by the set:
\[
 \biggl\{f\in\dual{\mb{X}}:\; f(x)>0\;\text{for all $x\in\sigma^{\circ}$}\biggr\}.
\]

Given an inclusion of polyhedral cones $\tau\subset\sigma$, we get a $\on{T}_{\mb{X}}$-equivariant map $\on{T}_{\mb{X}}(\tau)\to\on{T}_{\mb{X}}(\sigma)$. This is an open immersion precisely when $\tau$ is a face of $\sigma$.

Suppose that we are given a scheme $S$ and a $\on{T}_{\mb{X}}$-torsor $\on{P}\to S$. Then the \defnword{twisted torus embedding} attached to the cone $\sigma$ is the open immersion of $S$-schemes:
\[
 \on{P}\into \on{P}(\sigma)=\bigl(\on{P}\times \on{T}_{\mb{X}}(\sigma)\bigr)/\on{T}_{\mb{X}}.
\]
Here, $\on{T}_{\mb{X}}$ acts diagonally on $\on{P}\times\on{T}_{\mb{X}}$. The stratification of $\on{T}_{\mb{X}}(\sigma)$ by the orbits under the $\on{T}_{\mb{X}}$-action induces a stratification on $\on{P}(\sigma)$. In particular, there is a unique closed stratum in $\on{P}(\sigma)$ corresponding to the closed orbit $\on{O}_{\mb{X}}(\sigma)\subset\on{T}_{\mb{X}}(\sigma)$. 

% Moreover, for every $1$-dimensional face $\tau\subset\sigma$, we have the open stratum $\on{P}(\tau)\subset\on{P}(\sigma)$.

\subsubsection{}\label{background:subsubsec:compactcovering}
Fix a clr $\Phi=(P,X^+,g)$. Given a rational polyhedral cone $\sigma\subset W_{\Phi}(\Real)(-1)$, we can form the twisted torus embedding $\Sh_{K_\Phi}(Q_\Phi,D_\Phi)\into \Sh_{K_\Phi}(Q_{\Phi},D_{\Phi},\sigma)\coloneqq\Sh_{K_\Phi}(Q_\Phi,D_\Phi)(\sigma)$ over $\Sh_{\overline{K}_\Phi}(\overline{Q}_\Phi,\overline{D}_\Phi)$. Let $Z_{K_\Phi}(Q_\Phi,D_\Phi,\sigma)\subset\Sh_{K_\Phi}(Q_{\Phi},D_{\Phi},\sigma)$ be the closed stratum. Write $\on{U}_{K_\Phi}(Q_\Phi,D_\Phi,\sigma)$ for the closure of $\on{U}_{K_\Phi}(Q_\Phi,D_\Phi)$ in $\Sh_{K_\Phi}(Q_{\Phi},D_{\Phi},\sigma)$: This is an open sub-space of $\Sh_{K_\Phi}(Q_{\Phi},D_{\Phi},\sigma)$.

Set $\mb{H}(\Phi)\coloneqq\mb{H}_{P,X^+}$. It follows from the discussion in \cite[6.13]{pink:thesis} that $\on{U}_{K_\Phi}(Q_\Phi,D_\Phi,\sigma)$ contains $Z_{K_\Phi}(Q_\Phi,D_\Phi,\sigma)$ precisely when $\sigma^{\circ}\subset\mb{H}(\Phi)$.

Suppose that we have $\Phi_1\xrightarrow{(\gamma,q_2)_K}\Phi_2$ as in~\eqref{background:subsubsec:equivrelation}. Then conjugation by $\gamma^{-1}$ induces an embedding $\on{int}(\gamma^{-1}):W_{\Phi_2}(\Real)(-1)\into W_{\Phi_1}(\Real)(-1)$. Suppose that, for $i=1,2$, $\sigma_i\subset W_{\Phi_i}(\Real)(-1)$ are rational polyhedral cones such that $\on{int}(\gamma^{-1})(\sigma_2)$ is a face of $\sigma_1$. In this situation we will write $(\Phi_1,\sigma_1)\xrightarrow{(\gamma,q_2)_K}(\Phi_2,\sigma_2)$.

Suppose that $\gamma\cdot P_1=P_2$; then the map~\eqref{background:eqn:mapcusplabels} extends to a map~\cite[6.15]{pink:thesis}:
\begin{align}\label{background:eqn:mapcusplabelssigma}
\rho(\gamma,q):\on{U}_{K_{\Phi_1}}(Q_{\Phi_1},D_{\Phi_1},\sigma_1)&\xrightarrow{\simeq}\on{U}_{K_{\Phi_2}}(Q_{\Phi_2},D_{\Phi_2},\sigma_2).
\end{align}
Indeed, we can certainly assume that $\on{int}(\gamma^{-1})(\sigma_2)=\sigma_1$. In this case, our extension is obtained from the isomorphism of twisted torus embeddings 
\[
\Sh_{K_{\Phi_1}}(Q_{\Phi_1},D_{\Phi_1},\sigma_1)\xrightarrow{\simeq}\Sh_{K_{\Phi_2}}(Q_{\Phi_2},D_{\Phi_2},\sigma_2)
\]
extending the corresponding isomorphism of torus torsors over $ \Sh_{\overline{K}_{\Phi_2}}(\overline{Q}_{\Phi_2},\overline{D}_{\Phi_2})$. Therefore,~\eqref{background:eqn:mapcusplabelssigma} is a strata preserving isomorphism. In particular, if $\sigma_i^{\circ}\subset \mb{H}(\Phi_i)$, then, for $i=1,2$, the closed stratum $Z_{K_{\Phi_i}}(Q_{\Phi_i},D_{\Phi_i},\sigma_i)$ is contained in $\on{U}_{K_{\Phi_i}}(Q_{\Phi_i},D_{\Phi_i},\sigma_i)$, and the isomorphism carries $Z_{K_{\Phi_1}}(Q_{\Phi_1},D_{\Phi_1},\sigma_1)$ onto $Z_{K_{\Phi_2}}(Q_{\Phi_2},D_{\Phi_2},\sigma_2)$.

\subsubsection{}\label{background:subsubsec:deltakphisigma}
A particular case of the above situation is where $(\Phi_i,\sigma_i)=(\Phi,\sigma)$, for $i=1,2$. Here, the construction gives us a strata preserving action on $\Sh_{K_\Phi}(Q_{\Phi},D_{\Phi},\sigma)$ of the subgroup
\[
 \Delta_K(\Phi,\sigma)\subset\Delta_K(\Phi)
\]
consisting of elements that stabilize the cone $\sigma$; see~\eqref{background:subsubsec:deltakphi} for the notation.

Now, the stabilizer of $\sigma$ in $\Aut(\mb{B}_K(\Phi))$ is a finite group; cf.~\cite[Corollary II.4.9]{amrt}. Since $K$ is neat, we see that $\Delta_K(\Phi,\sigma)$ is the kernel of the map $\Delta_K(\Phi)\to\Aut(\mb{B}_K(\Phi))$, and so is independent of $\sigma$. We will therefore denote it by $\Delta^{\circ}_K(\Phi)$.

\begin{lem}\label{background:lem:deltacirc_trivial}
Suppose that the connected center $Z_G^{\circ}\subset G$ is isogenous to a product of split and compact tori over $\Rat$. Then $\Delta^{\circ}_K(\Phi)$ is trivial.
\end{lem}
\begin{proof}
$\Delta^{\circ}_K(\Phi)$ is a torsion-free arithmetic subgroup of the kernel of the map $G_{\Phi,\ell}\to\PGL(W_{\Phi})$, which is contained in the image of $Z_\Phi=Z_G$. But our hypothesis implies that $Z_G(\Rat)$ has no non-trivial torsion-free arithmetic subgroups.
\end{proof}

\subsubsection{}\label{background:subsubsec:derham_extension}
Fix an algebraic representation $M$ of $Q_\Phi$, and consider the associated complex variation of mixed Hodge structures (cf.~\ref{background:subsubsec:varmixedhodge}):
\[
 \bm{M}_{\dR}(\Phi)\vert_{\Sh_{K_\Phi}(Q_\Phi,D_\Phi)(\Comp)}\coloneqq \Reg{\Sh_{K_\Phi}(Q_\Phi,D_\Phi)(\Comp)}^{\an}\otimes\bm{M}_B(\Phi)
\]
over $\Sh_{K_\Phi}(Q_\Phi,D_\Phi)(\Comp)$. This is equipped with an integrable connection, and weight and Hodge filtrations. We claim that this vector bundle has a canonical extension to a vector bundle $\bm{M}_{\dR}(\Phi,\sigma)\vert_{\Sh_{K_\Phi}(Q_{\Phi},D_{\Phi},\sigma)(\Comp)}$, which is equipped with an integrable connection with log poles along the boundary $\Sh_{K_\Phi}(Q_{\Phi},D_{\Phi},\sigma)(\Comp)\backslash\Sh_{K_\Phi}(Q_\Phi,D_\Phi)(\Comp)$, and to which the weight and Hodge filtrations also extend.

For this, set
\[
X_K(\Phi)\coloneqq W_{\Phi}(\Rat)\backslash D_\Phi\times Q_\Phi(\Adele_f)/K_{\Phi}\;;\;Y_K(\Phi)=\overline{D}_\Phi\times\overline{Q}_\Phi(\Adele_f)/\overline{K}_{\Phi}.
\]
Then $X_K(\Phi)\to Y_K(\Phi)$ is a $\overline{Q}_\Phi(\Rat)$-equivariant $\mb{E}_K(\Phi)(\Comp)$-torsor whose quotient by $\overline{Q}_\Phi(\Rat)$ is exactly $\Sh_{K_\Phi}(Q_\Phi,D_\Phi)(\Comp)\to\Sh_{\overline{K}_\Phi}(\overline{Q}_\Phi,\overline{D}_\Phi)(\Comp)$. In particular, $\Sh_{K_\Phi}(Q_{\Phi},D_{\Phi},\sigma)$ is the $\overline{Q}_\Phi(\Rat)$-quotient of the twisted torus embedding $X_K(\Phi,\sigma)$ over $Y_K(\Phi)$.

The isomorphism
\begin{align}\label{background:eqn:torus_torsor_triv_1}
 M(\Comp)\times W_{\Phi}(\Comp)\times D_\Phi&\xrightarrow{\simeq} W_{\Phi}(\Comp)\times D_\Phi\times M(\Comp)\\
 (m,w,\varpi)\mapsto (w,\varpi,w\cdot m)\nonumber
\end{align}
gives rise to an isomorphism of vector bundles over $\mb{E}_K(\Phi)(\Comp)\times X_K(\Phi)$:
\begin{align}\label{background:eqn:torus_torsor_triv_2}
M(\Comp)\times\mb{E}_K(\Phi)(\Comp)\times X_K(\Phi)&\xrightarrow{\simeq}\bm{M}_{\dR}(\Phi)\vert_{\mb{E}_K(\Phi)(\Comp)\times X_K(\Phi)}.
\end{align}
Here, we are viewing $\mb{E}_K(\Phi)(\Comp)\times X_K(\Phi)$ as a space over $\Sh_{K_\Phi}(Q_\Phi,D_\Phi)(\Comp)$ via the composition:
\[
 \mb{E}_K(\Phi)(\Comp)\times X_K(\Phi)\to X_K(\Phi)\to\Sh_{K_\Phi}(Q_\Phi,D_\Phi)(\Comp),
\]
where the first map is given by the $\mb{E}_K(\Phi)(\Comp)$-action on $X_K(\Phi)$, and the second is the natural projection.

The left hand side of~\eqref{background:eqn:torus_torsor_triv_2} has an obvious extension to a trivial vector bundle over $\mb{E}_K(\Phi,\sigma)(\Comp)\times X_K(\Phi)$. The Hodge filtration depends only on the $X_K(\Phi)$-factor and so also extends. The induced connection on this extension has logarithmic poles along the boundary. More precisely, the difference between this connection and the trivial one is (up to sign) the linear map:
\[
 \Theta_P:M(\Comp)\to M(\Comp)\otimes\Omega^1_{\mb{E}_K(\Phi)(\Comp)/\Comp}=\Hom(\Lie W_{\Phi},M(\Comp))\otimes\Reg{\mb{E}_K(\Phi)(\Comp)},
\]
induced by the natural map $\Lie W_{\Phi}\to\End(M(\Rat))$.

The successive quotient of this extension by $\mb{E}_K(\Phi)$ and then $\overline{Q}_\Phi(\Rat)$ gives us the desired vector bundle $\bm{M}_{\dR}(\Phi,\sigma)\vert_{\Sh_{K_\Phi}(Q_{\Phi},D_{\Phi},\sigma)(\Comp)}$.

\subsubsection{}\label{background:subsubsec:conedecomp}
Fix a clr $\Phi=(P,X_+,g)$. Let $\mb{H}^*(\Phi)\subset W_{\Phi}(\Real)(-1)$ be the union of the images of the cones $\on{int}(\gamma^{-1})(\mb{H}(\Phi'))$, for all $\Phi'\xrightarrow{(\gamma,q)_K}\Phi$.

A \defnword{rational polyhedral cone decomposition} for $\mb{H}^*(\Phi)$ is a set $\Sigma(\Phi)$ of rational polyhedral cones $\sigma\subset W_{\Phi}(\Real)(-1)$ such that:
\begin{itemize}
\item $\sigma\subset\mb{H}^*(\Phi)$, for all $\sigma\in\Sigma(\Phi)$;
\item If $\sigma\in\Sigma(\Phi)$ then every face of $\sigma$ is also in $\Sigma(\Phi)$;
\item For $\sigma_1,\sigma_2\in\Sigma(\Phi)$, then $\sigma_1\cap\sigma_2$ is a face of both $\sigma_1$ and $\sigma_2$.
\end{itemize}

We will say that $\Sigma(\Phi)$ is in addition \defnword{complete} if $\mb{H}^*(\Phi)=\bigcup_{\sigma\in\Sigma}\sigma^{\circ}$. It is \defnword{smooth} (with respect to $K$) if each $\sigma\in\Sigma(\Phi)$ is smooth with respect to the lattice $\mb{B}_K(\Phi)\subset W_{\Phi}(\Rat)(-1)$.

Let $\Sigma^{\circ}(\Phi)\subset\Sigma(\Phi)$ be the subset of cones $\sigma$ such that $\sigma^\circ\subset \mb{H}(\Phi)$,

Given two rational polyhedral decompositions $\Sigma_1(\Phi)$, $\Sigma_2(\Phi)$ for $\mb{H}^*(\Phi)$, we will say that $\Sigma_2(\Phi)$ is a \defnword{refinement} of $\Sigma_1(\Phi)$ if, given any $\sigma_2\in\Sigma_2(\Phi)$, there exists $\sigma_1\in\Sigma_1(\Phi)$ such that $\sigma_2\subset\sigma_1$.

Given a decomposition $\Sigma(\Phi)$, we can construct a global twisted torus embedding that is locally of finite type over $\Sh_{\overline{K}_\Phi}(\overline{Q}_\Phi,\overline{D}_\Phi)$:
\[
 \Sh_{K_\Phi}(Q_\Phi,D_\Phi)\into\Sh_{K_\Phi}(Q_\Phi,D_\Phi,\Sigma\bigr).
\]
Here, $\Sh_{K_\Phi}(Q_\Phi,D_\Phi,\Sigma\bigr)$ is a union of (relatively) affine open sub-schemes $\Sh_{K_\Phi}(Q_{\Phi},D_{\Phi},\sigma)$, for $\sigma\in\Sigma^\circ(\Phi)$, where, for $\sigma_1,\sigma_2\in\Sigma^\circ(\Phi)$, $\Sh_{K_\Phi}(Q_\Phi,D_\Phi,\sigma_1)$ and $\Sh_{K_\Phi}(Q_\Phi,D_\Phi,\sigma_2)$ are glued along the common open sub-scheme $\Sh_{K_\Phi}(Q_\Phi,D_\Phi,\sigma_1\cap\sigma_2)$.

Write $\on{U}_{K_\Phi}(Q_\Phi,D_\Phi,\Sigma)$ for the closure of $\on{U}_{K_\Phi}(Q_\Phi,D_\Phi)$ in $\Sh_{K_\Phi}(Q_\Phi,D_\Phi,\Sigma)(\Comp)$: This is a union of open sub-spaces of the form $\on{U}_{K_\Phi}(Q_\Phi,D_\Phi,\sigma)$, for $\sigma\in\Sigma^\circ(\Phi)$.

If $\Sigma_2(\Phi)$ is a refinement of $\Sigma_1(\Phi)$, then we obtain a map $\Sh_{K_\Phi}(Q_\Phi,D_\Phi,\Sigma_1\bigr)\to\Sh_{K_\Phi}(Q_\Phi,D_\Phi,\Sigma_2\bigr)$ of twisted torus embeddings: It carries $\on{U}_{K_\Phi}(Q_\Phi,D_\Phi,\Sigma_2)$ to $\on{U}_{K_\Phi}(Q_\Phi,D_\Phi,\Sigma_1)$.

\subsubsection{}\label{background:subsubsec:rpcd}
An \defnword{admissible rational polyhedral cone decomposition}, or admissible rpcd, for $(G,X,K)$ is an assignment $\Phi\mapsto\Sigma(\Phi)$ attaching to each clr $\Phi$ a rational polyhedral cone decomposition $\Sigma(\Phi)$ for $\mb{H}^*(\Phi)$, and satisfying the following property: Suppose that we have $\Phi_1\xrightarrow{(\gamma,q_2)_K}\Phi_2$, with the corresponding embedding $\on{int}(\gamma^{-1}):W_{\Phi_2}(\Rat)(-1)\into W_{\Phi_1}(\Rat)(-1)$; then:
\[
 \Sigma(\Phi_2)=\{\sigma\subset W_{\Phi_1}(\Rat)(-1):\;\on{int}(\gamma^{-1})(\sigma)\in\Sigma(\Phi_1)\}.
\]

We will say that $\Sigma$ is \defnword{complete} (resp. \defnword{smooth}) if, for any clr $\Phi$, $\Sigma(\Phi)$ is complete (resp. smooth). An admissible rpcd $\Sigma_2$ is a refinement of another, $\Sigma_1$, if, for each clr $\Phi$, $\Sigma_2(\Phi)$ is a refinement of $\Sigma_1(\Phi)$.

Given an admissible rpcd $\Sigma$ for $(G,X)$, given a clr $\Phi=(P,X^+,g)$, $\Sigma(\Phi)$ is stable under the conjugation action on $W_{\Phi}(\Rat)(-1)$ of the sub-group $\Delta_K(\Phi)$ from~\eqref{background:eqn:deltaphi}. We will say that $\Sigma$ is \defnword{finite} if, for any clr $\Phi$, the orbit space $\Delta_K(\Phi)\backslash\Sigma(\Phi)$ is a finite set.

Observe that, in this case, using~\eqref{background:eqn:mapcusplabelssigma}, we can extend the action of $\Delta_K(\Phi)$ on $\on{U}_{K_\Phi}(Q_\Phi,D_\Phi)$ to an action on $\on{U}_{K_\Phi}(Q_\Phi,D_\Phi,\Sigma)$.

We will only be using finite admissible rpcd's in the sequel, so, from now on, admissible rpcd will always mean `finite admissible rpcd'. In fact, we will also impose the following additional \defnword{`no self-intersections'} condition on our rpcd's~\cite[6.5.2.25]{lan:thesis},~\cite[7.12]{pink:thesis}:

Suppose that we have $(\Phi_1\xrightarrow{(\gamma,q_2)_K}\Phi_2)$, and $\sigma\in\Sigma(\Phi_2)$ and $\tau\in\Sigma(\Phi_1)$ such that $\on{int}(\gamma^{-1})(\sigma)$ is a face of $\tau$. If $\eta\in\Delta_K(\Phi_2)$ is such that $\on{int}(\gamma^{-1}\eta^{-1})(\sigma)$ is also a face of $\tau$, then we must have $\on{int}(\eta^{-1})(\sigma)=\sigma$.

\subsubsection{}\label{background:subsubsec:char0compact}
Given an admissible rpcd $\Sigma$, we take $\on{U}_K(\Sigma)$ to be the disjoint union of the spaces $\on{U}_{K_\Phi}(Q_\Phi,D_\Phi,\Sigma)$ as $\Phi$ varies over the clr's for $(G,X)$. By construction, we have an open immersion $\on{U}_K\into\on{U}_K(\Sigma)$. Let $\sim_{\Sigma}$ be the closure in $\on{U}_K(\Sigma)\times\on{U}_K(\Sigma)$ of the equivalence relation $\sim$.

\begin{thm}[Ash-Mumford-Rapoport-Tai]\label{background:thm:amrt}
Admissible (finite) rpcd's exist. Any admissible rpcd can be refined to be smooth. After replacing $K$ by a sub-group of finite index, if necessary, the admissible rpcd can be chosen to satisfy the no self-intersections condition.

Given an admissible rpcd $\Sigma$ for $(G,X)$, $\sim_{\Sigma}$ defines an equivalence relation on $\on{U}_K(\Sigma)$. Set $\Sh_K^{\Sigma}(\Comp)=\on{U}_K(\Sigma)/\sim_{\Sigma}$. Then:
\begin{enumerate}
\item\label{amrt:equiv}$\Sh_K^{\Sigma}(\Comp)$ is the set of $\Comp$-valued points of a normal complex algebraic space $\Sh_{K,\Comp}^{\Sigma}$ of finite type over $\Comp$.
\item\label{amrt:algebraic}The natural map $\Sh_K(\Comp)=\on{U}_K/\sim\to\Sh_K^{\Sigma}(\Comp)$ is induced by an open immersion of algebraic spaces $\Sh_{K,\Comp}\into\Sh_{K,\Comp}^{\Sigma}$.
\item\label{amrt:compsmooth}If $\Sigma$ is complete (resp. smooth), then $\Sh^{\Sigma}_{K,\Comp}$ is proper (resp. smooth) over $\Comp$.
\end{enumerate}
\end{thm}
\begin{proof}
This is essentially the main result of~\cite{amrt}. The existence of admissible rpcd's is shown in Ch. II,\S~5.3-4 of \emph{loc. cit.}; cf. also~\cite[Ch. 9]{pink:thesis}. That they can be chosen to satisfy the no self-intersections condition by shrinking $K$ is shown in~\cite[7.13]{pink:thesis}.

For the fact that $\sim_{\Sigma}$ is an equivalence relation on $\on{U}_K(\Sigma)$, see~\cite[6.17]{pink:thesis}. The same result implies that the induced map $\Sh_K(\Comp)\to\Sh_K^{\Sigma}(\Comp)$ is an open immersion.

The remaining assertions---except for the one about smoothness, for which, cf.~\cite[6.26]{pink:thesis}---can be found in~\cite[9.34]{pink:thesis}.
\end{proof}

\subsubsection{}\label{background:subsubsec:char0strata}
Fix a clr $\Phi$ and $\sigma\in\Sigma^\circ(\Phi)$ so that $\on{U}_{K_\Phi}(Q_\Phi,D_\Phi,\sigma)$ contains $Z_{K_\Phi}(Q_\Phi,D_\Phi,\sigma)(\Comp)$. Since $K$ is neat, the stabilizer of $\sigma$ in $\Delta_K(\Phi)$ is the torsion-free arithmetic group $\Delta^\circ_K(\Phi)$; cf.~\eqref{background:subsubsec:deltakphisigma}. It now follows from \cite[7.15]{pink:thesis} that, for a sufficiently small neighborhood $\mathcal{V}_{K_\Phi}(Q_\Phi,D_\Phi,\sigma)\subset\on{U}_{K_\Phi}(Q_\Phi,D_\Phi,\sigma)$ of $Z_{K_\Phi}(Q_\Phi,D_\Phi,\sigma)$, the composition:
\[
 \mathcal{V}_{K_\Phi}(Q_\Phi,D_\Phi,\sigma)\into\on{U}_{K_\Phi}(Q_\Phi,D_\Phi,\sigma)\into\on{U}_{K_\Phi}(Q_\Phi,D_\Phi,\Sigma)\into\on{U}_K(\Sigma)\to\Sh_K^\Sigma(\Comp)
\]
induces an open immersion
\[
 \Delta^{\circ}_K(\Phi)\backslash\mathcal{V}_{K_\Phi}(Q_\Phi,D_\Phi,\sigma)\into\Sh^{\Sigma}_K(\Comp)
\]
The restriction of this open immersion produces a locally closed immersion
\[
\Delta^{\circ}_K(\Phi)\backslash Z_{K_\Phi}(Q_\Phi,D_\Phi,\sigma)(\Comp)\into\Sh_K^{\Sigma}(\Comp).
\]

Let $\on{Cusp}_K^{\Sigma}(G,X)$ be the set of equivalence classes of pairs $(\Phi,\sigma)$, where $\Phi$ is a clr and $\sigma\in\Sigma^\circ(\Phi)$: Here, we say that $(\Phi_1,\sigma_1)$ is equivalent to $(\Phi_2,\sigma_2)$ if we have $\Phi_1\xrightarrow{(\gamma,q_2)_K}\Phi_2$ with $\gamma\cdot P_1=P_2$, and $\on{int}(\gamma)(\sigma_1)=\sigma_2$. This set has a natural poset structure $\preccurlyeq$, where $[(\Phi_1,\sigma_1)]\preccurlyeq[(\Phi_2,\sigma_2)]$ if $(\Phi_1,\sigma_1)\xrightarrow{(\gamma,q_2)_K}(\Phi_2,\sigma_2)$, for some $\gamma\in G(\Rat)$ and $q_2\in Q_{\Phi_2}(\Adele_f)$.

By the definition of the equivalence relation $\sim_{\Sigma}$, we see that the locally closed immersion
\[
\Delta^{\circ}_K(\Phi)\backslash Z_{K_\Phi}(Q_\Phi,D_\Phi,\sigma)(\Comp)\into\Sh_K^{\Sigma}(\Comp),
\]
up to canonical isomorphism, depends only on the equivalence class $\Upsilon\coloneqq[(\Phi,\sigma)]\in\on{Cusp}^{\Sigma}_K(G,X)$. Therefore, we can denote the corresponding locally closed sub-space of $\Sh_K^{\Sigma}(\Comp)$ unambiguously as $Z_K(\Upsilon)(\Comp)$. Observe that the complex analytic space $Z_K(\Upsilon)(\Comp)$ is algebraic, and in fact has a canonical model $Z_K(\Upsilon)$ over $E$: For any representative $(\Phi,\sigma)$ of $\Upsilon$, we have
\[
Z_K(\Upsilon)=\Delta^{\circ}_K(\Phi)\backslash Z_{K_\Phi}(Q_\Phi,D_\Phi,\sigma).
\]

From~\cite[12.4]{pink:thesis}, we obtain:
\begin{thm}[Pink]\label{background:thm:pink}
There is a canonical model $\Sh^{\Sigma}_K$ for $\Sh^{\Sigma}_{K,\Comp}$ over $E$ such that the open immersion $\Sh_K\into\Sh^{\Sigma}_K$ is also defined over $E$. Suppose in addition that $\Sigma$ is complete. Then:
\begin{enumerate}[itemsep=0.11in]
\item\label{pink:rationality}For every $\Upsilon\in\on{Cusp}^{\Sigma}_K(G,X)$, the locally closed immersion $Z_K(\Upsilon)(\Comp)\into\Sh^\Sigma_K(\Comp)$ arises from a map $Z_K(\Upsilon)\into\Sh^{\Sigma}_K$ of algebraic spaces over $E$.

\item\label{pink:stratification}There is a canonical stratification:
\[
 \Sh^{\Sigma}_K=\bigsqcup_{\Upsilon}Z_K(\Upsilon),
\]
where $\Upsilon$ ranges over $\on{Cusp}^\Sigma_K(G,X)$. For any fixed $\Upsilon$, the closure of $Z_K(\Upsilon)$ in $\Sh_K^\Sigma$ is precisely the closed sub-space:
\[
 \overline{Z}_K(\Upsilon)=\bigsqcup_{\Upsilon'\preccurlyeq\Upsilon}Z_K(\Upsilon').
\]

\item\label{pink:completion}Given $\Upsilon=[(\Phi,\sigma)]$, let $\widehat{\Sh}_{K_\Phi}(Q_\Phi,D_\Phi,\sigma)$ be the the formal completion of $\Sh_{K_\Phi}(Q_{\Phi},D_{\Phi},\sigma)$ along $Z_{K_\Phi}(Q_\Phi,D_\Phi,\sigma)$. Then the isomorphism $\Delta^{\circ}_K(\Phi)\backslash Z_{K_\Phi}(Q_\Phi,D_\Phi,\sigma)\xrightarrow{\simeq}Z_K(\Upsilon)$ extends to an isomorphism of formal algebraic spaces between $\Delta^{\circ}_K(\Phi)\backslash\widehat{\Sh}_{K_\Phi}(Q_\Phi,D_\Phi,\sigma)$ and the completion of $\Sh^{\Sigma}_K$ along $Z_K(\Upsilon)$. This extension is characterized by the property that, over $\Comp$, it converges to the holomorphic open immersion $\Delta^{\circ}_K(\Phi)\backslash\mathcal{V}_{K_\Phi}(Q_\Phi,D_\Phi,\sigma)\into\Sh_K^{\Sigma}(\Comp)$.
\end{enumerate}
\end{thm}
\qed

\subsubsection{}\label{background:subsubsec:funct}
We will end this summary with some results on the functoriality of the toroidal compactifications and their stratifications. For this, fix a closed immersion of Shimura data $\iota:(G,X)\into ( {G}^\beef, X^\beef)$ of Shimura data, as well as an element $g^\beef\in {G}^\beef(\Adele_f)$. Let ${E}^\beef$ be the reflex field for $( {G}^\beef, X^\beef)$. Then we get a map of Shimura varieties:
\[
 (\iota,g^\beef):\Sh_K\to \Sh_{{K}^\beef,E}\coloneqq E\otimes_{{E}^\beef}\Sh_{{K}^\beef}( {G}^\beef, X^\beef).
\]
On the level of $\Comp$-points, this map carries $[(x,g)]$ to $[(\iota(x),\iota(g)g^\beef)]$.

Given an admissible parabolic sub-group $P\subset G$, there is a unique minimal admissible parabolic sub-group $\iota_*P\subset {G}^\beef$ containing $\iota(P)$; cf.~\cite[4.16]{pink:thesis}. Moreover, we have $\iota(Q_P)\subset Q_{\iota_*P}$ and $\iota(W_P)\subset W_{\iota_*P}$.

Given a clr $\Phi$ for $(G,X)$, we can now define a clr $\Phi^\beef = (\iota,g^\beef)_*\Phi$ for $( {G}^\beef, X^\beef)$: We set $P_{\Phi^\beef} = \iota_*P_\Phi$; $X^{\beef,+}_{\Phi^\beef}$ will be the unique connected component of $ X^\beef$ containing $\iota(X^+_\Phi)$; and $g_{\Phi^\beef} = g_\Phi g^\beef$. 

Tthe convex cone $\iota\bigl(\mb{H}(\Phi)\bigr)\subset W_{\Phi^\beef}(\Real)(-1)$ is contained in $\mb{H}({{\Phi}^\beef})$, and, since $\iota(Q_\Phi(\Real)W_\Phi(\Real))$ is contained in $Q_{\Phi^\beef}(\Real)W_{\Phi^\beef}(\Comp)$, the map $\iota$ extends to a map of mixed Shimura data 
\[
(Q_\Phi,D_\Phi)\to (Q_{\Phi^\beef},D_{\Phi^\beef}). 
\]
Let $K\subset G(\Adele_f)$ and ${K}^\beef\subset {G}^\beef(\Adele_f)$ be compact opens such that $\iota(K)$ is contained in $g^{\beef,-1}{K}^\beef g^\beef$. Then $\iota(K_{\Phi})\subset K_{{\Phi}^\beef}$, and we therefore obtain a map of mixed Shimura varieties
\begin{align}\label{background:eqn:phitildephi}
\Sh_{K_\Phi}(Q_\Phi,D_\Phi)\to E\otimes_{{E}^\beef}\Sh_{K^\beef_{\Phi^\beef}}(Q_{\Phi^\beef},D_{\Phi^\beef}).
\end{align}

If $\Sigma^\beef$ is an admissible (finite) rpcd for $( {G}^\beef, X^\beef,{K}^\beef)$, then the assignment:
\[
 (\iota,g^\beef)^*\Sigma^\beef:\Phi\mapsto\{\iota^{-1}(\sigma^\beef):\;\sigma^\beef\in\Sigma^\beef(\Phi^\beef);\;\sigma^\beef\subset\iota(W_{\Phi}(\Real)(-1))\}
\]
is an admissible (finite) rpcd for $(G,X,K)$; cf.~\cite[\S 3.3]{harris:functorial}.

Fix such a $\Sigma^\beef$ and let $\Sigma$ be a refinement of $(\iota,g^\beef)^*\Sigma^\beef$. Suppose that we are given $\sigma\in\Sigma(\Phi)$; then there is a unique $\sigma^\beef\in\Sigma^\beef({\Phi}^\beef)$ such that $\iota(\sigma)\subset\sigma^\beef$, and such that no proper face of $\sigma^\beef$ contains $\iota(\sigma)$. The map~\eqref{background:eqn:phitildephi} extends to a strata respecting map
\begin{align}\label{background:eqn:phitildephisigma}
\Sh_{K_\Phi}(Q_{\Phi},D_{\Phi},\sigma)\to E\otimes_{{E}^\beef}\Sh_{K^\beef_{\Phi^\beef}}(Q_{\Phi^\beef},D_{\Phi^\beef},\sigma^\beef).
\end{align}

If $\Phi_1,\Phi_2$ are clr's for $(G,X)$ with $\Phi_1\xrightarrow{(\gamma,q_2)_K}\Phi_2$, then we have:
\[
 \Phi^\beef_1\xrightarrow{(\iota(\gamma),\iota(q_2))_{{K}^\beef}}\Phi^\beef_2.
\]
Therefore, $(\Phi,\sigma)\mapsto({\Phi}^\beef,\sigma^\beef)$ induces a map of posets:
\[
 (\iota,g^\beef)_*:\on{Cusp}_K^{\Sigma}(G,X)\to\on{Cusp}_{{K}^\beef}^{\Sigma^\beef}( {G}^\beef, X^\beef).
\]

Given $\Upsilon=[(\Phi,\sigma)]$ on the left hand side, set $\Upsilon^\beef=(\iota,g^\beef)_*\Upsilon=[({\Phi}^\beef,\sigma^\beef)]$. Let $(\Sh_K^{\Sigma})^{\wedge}_{Z_K(\Upsilon)}$ (resp. $(\Sh_{{K}^\beef}^{\Sigma^\beef})^{\wedge}_{Z_{{K}^\beef}(\Upsilon^\beef)}$) be the completion of $\Sh_K^{\Sigma}$ (resp. $\Sh_{{K}^\beef}^{\Sigma^\beef}$) along $Z_{{K}^\beef}(\Upsilon^\beef)$. Then, using~\eqref{pink:completion} of~\eqref{background:thm:pink}, the map~\eqref{background:eqn:phitildephisigma} gives us a canonical map:
\begin{align}\label{background:eqn:functstratacomp}
 (\iota,g^\beef):(\Sh_K^{\Sigma})^{\wedge}_{Z_K(\Upsilon)}\xrightarrow{\simeq}\Delta^{\circ}_K(\Phi)\backslash\widehat{\Sh}_{K_\Phi}(Q_\Phi,D_\Phi,\sigma)\to \Delta^{\circ}_{{K}^\beef}({\Phi}^\beef)\backslash\widehat{\Sh}_{{K}^\beef_{\Phi^\beef}}(Q_{\Phi^\beef},D_{\Phi^\beef},\sigma^\beef)\xrightarrow{\simeq}(\Sh_{{K}^\beef}^{\Sigma^\beef})^{\wedge}_{Z_{{K}^\beef}(\Upsilon^\beef)}.
\end{align}

From \cite[6.25,12.4]{pink:thesis}, we now obtain:
\begin{prp}[Pink]\label{background:prp:funct}
Suppose that $\Sigma$ and $\Sigma^\beef$ are complete. Then the map $(\iota,g):\Sh_K\to\Sh_{{K}^\beef,E}$ extends uniquely to a map $(\iota,g):\Sh^{\Sigma}_K\to\Sh^{\Sigma^\beef}_{{K}^\beef,E}$. For every $\Upsilon\in\on{Cusp}_K^{\Sigma}(G,X)$, $(\iota,g)$ carries $Z_K(\Upsilon)$ to $E\otimes_{{E}^\beef}Z_{{K}^\beef}(\Upsilon^\beef)$, and the corresponding map between the formal completions along these locally closed sub-schemes is identified with~\eqref{background:eqn:functstratacomp}
\end{prp}
\qed

One particular case of this functoriality is when $( {G}^\beef, X^\beef)=(G,X)$, and $\iota$ is the identity; then~\eqref{background:prp:funct} gives us the action of Hecke correspondences on toroidal compactifications.

\subsection{Siegel modular varieties}\label{background:subsec:siegel}
We will now fix a symplectic space $(H,\psi)$ over $\Rat$. In this sub-section, we will show that the mixed Shimura varieties attached to rational boundary components of the associated Siegel Shimura variety are moduli spaces for $1$-motifs with additional structures.

\subsubsection{}\label{background:subsubsec:siegelbdry}
Let $(G,X)=(\GSp(H,\psi),\on{S}^{\pm}(H,\psi))$ be the Siegel Shimura datum associated with $(H,\psi)$. The first order of business is to describe the rational boundary components for $(G,X)$. This is given in~\cite[4.25]{pink:thesis}. Here is the summary:

Every admissible parabolic $P\subset G$ is obtained as the stabilizer of an isotropic sub-space $I\subset H$. Equivalently, it is the stabilizer of the filtration $W_{\bullet}H$, where:
\[
 0=W_{-3}H\subset W_{-2}H=I\subset W_{-1}H=I^{\perp}\subset W_0H=H
\]

Suppose that we have a clr $\Phi$ with $P_\Phi$ the stabilizer of a filtration as above. Then $Q_{\Phi}\subset P_\Phi$ is the largest sub-group acting trivially on $H/I^{\perp}$; it must necessarily act on $I$ via the similitude character. Moreover, $W_{\Phi}\subset Q_{\Phi}$ is the largest sub-group acting trivially on both $H/I$ and $I^{\perp}$.

The assignment $f\mapsto\psi(\cdot,-f\cdot)$ identifies $\Lie W_\Phi$ with the space of symmetric bilinear forms on $H/I^{\perp}$. Therefore, we can view $W_\Phi(\Real)=\Real\otimes\Lie W_\Phi$ as the space of symmetric $\Real$-valued forms on $H/I^{\perp}$.

The choice of a connected component $X^+_\Phi\subset X$ corresponds to a choice $i=\sqrt{-1}$; equivalently, to a choice of isomorphism $\Rat(-1)\xrightarrow{\simeq}\Rat$ or to an element in $\on{S}^{\pm}(0)$. This gives us an isomorphism $W_\Phi(\Real)(-1)\xrightarrow{\simeq}W_\Phi(\Real)$. The cone $\mb{H}(\Phi)\subset W_\Phi(\Real)(-1)$ is the pre-image under this isomorphism of the space of positive definite bilinear forms on $(H/I^{\perp})(\Real)$.

Since $Q_\Phi(\Real)$ acts transitively on the connected components of $X$, the space $D_\Phi$ does not depend on the choice of the connected component $X^+_\Phi$. It is the space of pairs $(F^\bullet H(\Comp),\lambda)$, where $\lambda:\Rat\xrightarrow{\simeq}\Rat(1)$ is an isomorphism, and $(F^\bullet H(\Comp),W_\bullet H(\Rat))$  is a mixed Hodge structure on $H(\Rat)$ of weights $(-1,-1),(-1,0),(0,-1),(0,0)$ that is polarized by $\lambda\circ\psi$. When $I\neq I^{\perp}$, $\lambda$ is determined uniquely by the Hodge filtration and this condition.

The $Q_\Phi(\Real)W_\Phi(\Comp)$-equivariant open immersion $X\into D_\Phi$ can be described as follows: Given $x\in X$ with associated Hodge filtration $F^\bullet_x H(\Comp)$, we attach to it the mixed Hodge structure $(H(\Rat),W_{\bullet}H(\Rat),F^\bullet_x H(\Comp))$, and the unique isomorphism $\lambda:\Rat\to\Rat(1)$ such that $F^\bullet_x H(\Comp)$ is polarized by $\lambda\circ\psi$.

\subsubsection{}\label{background:subsubsec:h_integral}
A \defnword{polarized lattice} in $(H(\Rat),\psi)$ is a $\Int$-lattice $H(\Int)\subset H(\Rat)$ on which $\psi$ restricts to a $\Int$-valued. Let $\dual{H}(\Int)\subset H(\Rat)$ be the dual lattice with respect to the pairing $\psi$, and let $d\in\Int_{>0}$ be such that the order of the finite group $\dual{H}(\Int)/H(\Int)$ is $d^2$. We will call $d$ the \defnword{discriminant} of the polarized lattice $H(\Int)$.

Fix a compact open sub-group $K\subset G(\Adele_f)$ and a polarized $\Int$-lattice $H(\Int)\subset H$ such that $H(\widehat{\Int})=\widehat{\Int}\otimes H(\Int)\subset H(\Adele_f)$ is stabilized by $K$. Extend $(P,X^+)$ to a clr $\Phi$.

Set $H^g(\widehat{\Int})=g\cdot H(\widehat{\Int})\subset H(\Adele_f)$, and let $H^g(\Int)=H(\Rat)\cap H^g(\widehat{\Int})$. Then $H^g({\Int})$ inherits a filtration $W_{\bullet}H^g({\Int})$ from the filtration $W_{\bullet}H$. Let $N(g)\in\Rat_{>0}$ be the unique positive element mapping to the image of $\nu(g)^{-1}$ under the surjection:
\[
 \Rat^\times\to\Rat^\times/\Int^\times\xrightarrow{\simeq}\Adele_f^\times/\widehat{\Int}^\times.
\]
If we equip $H$ with the symplectic form $\psi^g=N(g)\cdot \psi$, then $H^g(\Int)\subset H(\Rat)$ will be a polarized lattice for $\psi^g$.

Let $\dual{H}(\nu)$ be the $G$-representation obtained by twisting the dual representation $\dual{H}$ by the similitude character $\nu:G\to\Gm$. The pairing $\psi^g$ on $H$ induces an isomorphism of $G$-representations $f^g:H\xrightarrow{\simeq}\dual{H}(\nu)$. Equip $\dual{H}(\nu)$ with the dual filtration $W_{\bullet}\dual{H}(\nu)$, so that $W_i\dual{H}(\nu)\subset\dual{H}(\nu)$ is the annihilator of $W_{-3-i}H$. Then $f^g$ preserves $W_{\bullet}$-filtrations.

Let $H^{g,\vee}(\Int)\subset\dual{H}$ be lattice dual to $H^g(\Int)$, and set $H^{g,\vee}(\Int)(\nu)=N(g)H^{g,\vee}\subset\dual{H}=\dual{H}(\nu)$. Then $f^g$ carries $H^g(\Int)$ into $H^{g,\vee}(\Int)(\nu)$.

\subsubsection{}\label{background:subsubsec:1motifsmoduli}
Given a $1$-motif $Q$ over $\Rat$-scheme $S$, let $\widehat{T}(Q)$ be the total Tate module:
\[
 \widehat{T}(Q)=\varprojlim_nQ[n].
\]
Similarly, we have the total Tate module $\widehat{T}(\dual{Q})$. Both these sheaves are equipped with weight filtrations, $W_{\bullet}\widehat{T}(Q)$ and $W_{\bullet}\widehat{T}(\dual{Q})$, induced from the weight filtrations on the $1$-motifs themselves.

Fix a polarization $\lambda:Q\to\dual{Q}$. The induced map $\widehat{T}(Q)\to\widehat{T}(\dual{Q})=\dual{\widehat{T}(Q)}(1)$ produces a Weil pairing:
\[
\widehat{e}_{\lambda}:\widehat{T}(Q)\times\widehat{T}(Q)\to\underline{\widehat{\Int}}(1).
\]
Suppose that we are given an $S$-scheme $T$ and a pair of isomorphisms of \'etale sheaves over $T$:
\[
\eta:\underline{H}^g(\widehat{\Int})\xrightarrow{\simeq}\widehat{T}(Q)\vert_T;\; u:\underline{\widehat{\Int}}\xrightarrow{\simeq}\underline{\widehat{\Int}}(1).
\]
We will say that such a pair is \defnword{compatible with $\lambda$} if we have:
\[
 \widehat{e}_{\lambda}\circ(\eta\times \eta)=u\circ\psi_g:\underline{H}^g(\widehat{\Int})\times\underline{H}^g(\widehat{\Int})\to\underline{\widehat{\Int}}(1).
\]

We will say that the pair is \defnword{compatible with $W_{\bullet}$-filtrations} if $\eta$ carries $W_\bullet\underline{H}^g(\widehat{\Int})$ onto $W_{\bullet}\widehat{T}(Q)\vert_T$.

Suppose now that we are given isomorphisms of sheaves of $\Int$-modules over $S$:
\[
 \alpha:\gr^W_0\underline{H}^g(\Int)\xrightarrow{\simeq}Q^{\et}\;;\;\dual{\alpha}:\gr^W_0\underline{H}^{\vee,g}(\Int)(\nu)\xrightarrow{\simeq}Q^{\vee,\et}=Q^{\mult,C}.
\]
Then we obtain isomorphisms of $\widehat{\Int}$-sheaves:
\begin{align*}
\beta:\gr^W_0\underline{H}^g(\widehat{\Int})=\widehat{\Int}\otimes\gr^W_0\underline{H}^g(\Int)&\xrightarrow[\simeq]{1\otimes\alpha}\widehat{\Int}\otimes Q^{\et}=\gr^W_0\widehat{T}(Q).\\
\dual{\beta}:\underline{\widehat{\Int}}\otimes\gr^W_0\underline{H}^{g,\vee}(\Int)(\nu)&\xrightarrow[\simeq]{1\otimes\dual{\alpha}}\widehat{\Int}\otimes Q^{\mult,C}=\gr^W_0\widehat{T}(\dual{Q}).
\end{align*}

Let $(\eta,u)$ be a pair as above compatible with $\lambda$ and $W_{\bullet}$-filtrations. We then obtain an isomorphism:
\[
 \dual{\eta}(1):\underline{H}^{g,\vee}(\widehat{\Int})(\nu)\xrightarrow[\simeq]{N(g)^{-1}}\underline{H}^{g,\vee}(\widehat{\Int})\xrightarrow[\simeq]{(\dual{\eta})^{-1}}\dual{\widehat{T}(Q)}\xrightarrow[\simeq]{u}\dual{\widehat{T}}(Q)(1)=\widehat{T}(\dual{Q}).
\]
We will say that the pair is \defnword{compatible with $\alpha$ and $\dual{\alpha}$} if $\gr^W_0\eta=\beta$, and if $\gr^W_0\dual{\eta}(1)=\dual{\beta}$.

Write $\underline{\Isom}\bigl(H^g(\widehat{\Int}),\widehat{T}(Q)\bigr)$ for the sheaf over $S$ that assigns to any $S$-scheme $T$ the set of pairs $(\eta,u)$ as above that are compatible with $\lambda$, $W_{\bullet}$-filtrations, $\alpha$ and $\dual{\alpha}$.

Observe that $K_{\Phi}\subset Q_\Phi(\Adele_f)$, acts naturally on this sheaf via right composition: $(\eta,u)\cdot k=(\eta\circ k,u\circ\nu(k))$. A \defnword{$K_{\Phi}$-level structure} on $(Q,\lambda,\alpha,\dual{\alpha})$ is a section:
\[
\varepsilon\in H^0\bigl(S,\underline{\Isom}(H^g(\widehat{\Int}),\widehat{T}(Q))/K_{\Phi}\bigr).
\]

Let $\widetilde{\xi}_K(\Phi)$ be the stack of groupoids over $\Rat$ parameterizing, for each $\Rat$-scheme $S$, the category of tuples $(Q,\lambda,\alpha,\dual{\alpha},\varepsilon)$, where $(Q,\lambda)$ and $\alpha^{\et},\alpha^{\mult}$ are as above, and $\varepsilon$ is a $K_{\Phi}$-level structure.

\subsubsection{}\label{background:subsubsec:comppoints}
There is a canonical tuple $(\mathcal{Q},\lambda,\alpha,\dual{\alpha},\varepsilon)$ over $\Sh_{K_\Phi}(Q_\Phi,D_\Phi)_{\Comp}$, and thus a canonical map of $\Comp$-varieties $\Sh_{K_\Phi}(Q_\Phi,D_\Phi)_{\Comp}\to\widetilde{\xi}_K(\Phi)_{\Comp}$, constructed as follows:

First, applying the construction of~\eqref{background:subsubsec:varmixedhodge} to the representation $H$ with the lattice $H^g(\Int)\subset H(\Rat)$, we obtain a variation of mixed $\Int$-Hodge structures $\bm{H}_{\on{MH}}(\Phi)_{\Int}$ over $\Sh_{K_\Phi}(Q_\Phi,D_\Phi)(\Comp)$ of weights $(-1,-1),(-1,0),(0,-1),(0,0)$.

Applying the same construction to the $1$-dimensional representation associated with the similitude character $\nu:Q_\Phi\to\Gm$, we obtain a variation of Hodge structures of rank $1$, which, arguing as in \cite[3.16]{pink:thesis}, one can show to be canonically isomorphic to $\mb{1}(1)$.

Here, given a smooth complex analytic space $S$, $\mb{1}(n)$ is the variation of $\Rat$-Hodge structures on $S$ whose underlying local system is $\underline{\Rat}$ and whose weight filtration is concentrated in degree $-2n$. We will write $\mb{1}$ for $\mb{1}(0)$. The obvious lattice $\Int\subset\Rat$ refines $\mb{1}(n)$ to a variation of $\Int$-Hodge structures $\mb{1}_{\Int}(n)$.

The $Q_\Phi$-equivariant pairing $\psi^g:H\times H\to \Rat(\nu)$ now gives rise to a polarization of variations of mixed Hodge structures:
\begin{align}\label{background:eqn:polarize1motif}
 \bm{H}_{\on{MH}}(\Phi)_{\Int}\times\bm{H}_{\on{MH}}(\Phi)_{\Int}\to\mb{1}_{\Int}(1).
\end{align}

Thus, using the equivalence of categories between polarized $1$-motifs over $\Sh_{K_\Phi}(Q_\Phi,D_\Phi)_{\Comp}$ and polarized variations of mixed $\Int$-Hodge structures of weights $(-1,-1),(-1,0),(0,-1),(0,0)$ over $\Sh_{K_\Phi}(Q_\Phi,D_\Phi)(\Comp)$~\cite[10.1.3]{deligne:hodge-iii}, we obtain a canonical polarized $1$-motif $(\mathcal{Q},\lambda)$ over $\Sh_{K_\Phi}(Q_\Phi,D_\Phi)_{\Comp}$, whose homology, as a polarized variation of Hodge structures, is identified with $\bm{H}_{\on{MH}}(\Phi)_{\Int}$. In particular, we have a canonical identification $\widehat{T}(\mathcal{Q})=\bm{H}_{\widehat{\Int}}(\Phi)$.

Since $Q_\Phi$ acts trivially on $\gr^W_0H$ and $\gr^W_0\dual{H}(\nu)$, we obtain canonical trivializations:
\[
\alpha:\gr^W_0\underline{H}^g(\Int)\xrightarrow{\simeq}\gr^W_0\bm{H}_{\on{MH}}(\Phi)_{\Int}=\mathcal{Q}^{\et}\;;\;\dual{\alpha}:\gr^W_0\underline{H}^{\vee,g}(\Int)(\nu)\xrightarrow{\simeq}\gr^W_0\dual{\bm{H}_{\on{MH}}(\Phi)}(1)=\mathcal{Q}^{\vee,\et}.
\]

Moreover, by~\eqref{background:eqn:madeletriv}, we have a pair of canonical trivializations:
\[
 \eta:\underline{H}^g(\widehat{\Int})\xrightarrow{\simeq}\widehat{T}(\mathcal{Q})\vert_{\Sh(Q_\Phi,D_\Phi)_{\Comp}}\;;\;u:\underline{\widehat{\Int}}\xrightarrow{\simeq}\underline{\widehat{\Int}}(1)\vert_{\Sh(Q_\Phi,D_\Phi)_{\Comp}}
\]
compatible with $\lambda$, $W_{\bullet}$-filtrations, $\alpha$ and $\dual{\alpha}$. The $K_{\Phi}$-orbit of $(\eta,u)$ now determines a canonical $K_{\Phi}$-level structure $\varepsilon$ on $(\mathcal{Q},\lambda,\alpha,\dual{\alpha})$ over $\Sh_{K_\Phi}(Q_\Phi,D_\Phi)_{\Comp}$.

\subsubsection{}\label{background:subsubsec:xiphirep}
Standard methods now show that $\widetilde{\xi}_K(\Phi)$ is representable by a $\Rat$-variety, and checking on $\Comp$-points shows that the map $\Sh_{K_\Phi}(Q_\Phi,D_\Phi)_{\Comp}\to\widetilde{\xi}_K(\Phi)_{\Comp}$ is an isomorphism. See~\cite[Ch. 10]{pink:thesis} for an illustration of this in a related scenario, and also further below in this subsection, where we consider integral models for $\Sh_{K_\Phi}(Q_\Phi,D_\Phi)$ in the situation where $H(\Int)$ has discriminant $1$.

In fact, this isomorphism realizes $\widetilde{\xi}_K(\Phi)$ as a canonical model for $\Sh_{K_\Phi}(Q_\Phi,D_\Phi)(\Comp)$. For this, note that, by~\cite[Th\'eor\`eme 4.21]{deligne:travaux} the canonical model $\Sh_{K_{\Phi,h}}(G_{\Phi,h},D_{\Phi,h})$ for $\Sh_{K_{\Phi,h}}(G_{\Phi,h},D_{\Phi,h})(\Comp)$ is the moduli space of triples $(B,\lambda^{\ab},\varepsilon^{\ab})$ over $\Rat$-schemes $S$, where $(B,\lambda^{\ab})$ is a polarized abelian scheme over $S$ and
\[
 \varepsilon^{\ab}\in H^0\bigl(S,\underline{\Isom}\bigl(\gr^W_{-1}H^g(\widehat{\Int}),\widehat{T}(B)\bigr)/K_{\Phi,h}\bigr)
\]
is a $K_{\Phi,h}$-level structure on $(B,\lambda^{\ab})$ (the definition of such a level structure is essentially a special case of the one given above for polarized $1$-motifs).

There is a canonical map $\widetilde{\xi}_K(\Phi)\to \Sh_{K_{\Phi,h}}(G_{\Phi,h},D_{\Phi,h})$ that takes a tuple $(\mathcal{Q},\lambda,\alpha,\dual{\alpha},\varepsilon)$ to the triple $(\mathcal{Q}^{\ab},\lambda^{\ab},\varepsilon^{\ab})$, where $\varepsilon^{\ab}=\gr^W_{-1}\varepsilon$ is the $K_{\Phi,h}$-level structure on $(\mathcal{Q}^{\ab},\lambda^{\ab})$ induced from $\varepsilon$.

Therefore, given the characterization of canonical models described in~\eqref{background:subsubsec:canmodels}, it is enough to show that the system $\{\widetilde{\xi}_K(\Phi)\}$ as $K$ varies is equivariant for the action of $Q_\Phi(\Adele_f)$. For the transition maps $\widetilde{\xi}_{K_1}(\Phi)\to\widetilde{\xi}_{K_2}(\Phi)$ for $K_1\subset K_2$, this is clear: Any $K_{1,\Phi}$-level structure on $(\mathcal{Q},\lambda,\alpha,\dual{\alpha})$ naturally also induces a $K_{2,\Phi}$-level structure. The remaining instances of equivariance will be described in the next paragraph.

\subsubsection{}\label{background:subsubsec:functoriality}
Suppose that we have two clrs $\Phi_1=(P_1,X_1^+,g_1)$ and $\Phi_2=(P_2,X_2^+,g_2)$; and $\gamma\in G(\Rat)$, $q\in Q_{\Phi_2}(\Adele_f)$ with $\Phi_1\xrightarrow{(\gamma,q)_K}\Phi_2$ and $\gamma\cdot P_1=P_2$. Then, from \eqref{background:subsubsec:cuspsisom}, we obtain a natural isomorphism of mixed Shimura varieties $\Sh_{K_{\Phi_1}}(Q_{\Phi_1},D_{\Phi_1})\xrightarrow{\simeq}\Sh_{K_{\Phi_2}}(Q_{\Phi_2},D_{\Phi_2})$. This has the following analogue for the moduli schemes considered above.

First, observe that we have $\gamma g_1=qg_2k$, for some $k\in K$. Therefore, multiplication by $\gamma^{-1}q$ produces isomorphisms:
\[
 H^{g_2}({\widehat{\Int}})\xrightarrow[\simeq]{\gamma^{-1}q}H^{g_1}({\widehat{\Int}})\;;\;H^{g_2,\vee}({\widehat{\Int}})(\nu)\xrightarrow[\simeq]{\gamma^{-1}q}H^{g_1,\vee}({\widehat{\Int}})(\nu)
\]
This isomorphism has the property that it preserves $W_{\bullet}$-filtrations. Moreover, since $q$ acts trivially on $\gr^W_0H$ and $\gr^W_0\dual{H}(\nu)$, the induced isomorphisms of the associated graded objects restrict to isomorphisms of $\Int$-modules:
\[
 \gr^W_0H^{g_2}(\Int)\xrightarrow[\simeq]{\gamma^{-1}q}\gr^W_0H^{g_1}(\Int)\;;\;\gr^W_{0}H^{g_2,\vee}(\Int)(\nu)\xrightarrow[\simeq]{\gamma^{-1}q}\gr^W_{0}H^{g_1,\vee}({\Int})(\nu).
\]

Let $(\mathcal{Q},\lambda,\alpha_1,\dual{\alpha}_1,\varepsilon_1)$ be the tautological tuple over $\Sh_{K_{\Phi_1}}(Q_{\Phi_1},D_{\Phi_1})$. Then we obtain isomorphisms:
\begin{align*}
 \alpha_2:\gr^W_0\underline{H}^{g_2}({\Int})\xrightarrow[\simeq]{\gamma^{-1}q}\gr^W_0\underline{H}^{g_1}({\Int})&\xrightarrow[\simeq]{\alpha_1}\mathcal{Q}^{\et};\\
 \dual{\alpha}_2:\gr^W_{0}\underline{H}^{g_2,\vee}({\Int})(\nu)\xrightarrow[\simeq]{\gamma^{-1}q}\gr^W_{0}\underline{H}^{g_1,\vee}({\Int})(\nu)&\xrightarrow[\simeq]{\dual{\alpha}_1}\mathcal{Q}^{\mult,C}.
\end{align*}

Given an $\widetilde{\xi}_K(\Phi_1)$-scheme $T$ and a section $(\eta,u)$ of $\underline{\Isom}\bigl(H^{g_1}({\widehat{\Int}}),\widehat{T}(\mathcal{Q})\bigr)$, the pair $(\eta\circ(\gamma^{-1}q),u)$ is a section of $\underline{\Isom}\bigl(H^{g_2}({\widehat{\Int}}),\widehat{T}(\mathcal{Q})\bigr)$. Therefore, from $\varepsilon_1$, we obtain a canonical $K_{\Phi_2}$-level structure $\varepsilon_2$ on $(\mathcal{Q},\lambda,\alpha_2^{\et},\alpha_2^{\mult})$.

The tuple $(\mathcal{Q},\lambda,\alpha_2,\dual{\alpha}_2,\varepsilon_2)$ over $\widetilde{\xi}_K(\Phi_1)$ is precisely the one corresponding to the isomorphism with $\widetilde{\xi}_K(\Phi_2)$.

This completes the proof of Hecke equivariance and shows that we have a canonical isomorphism of $\Rat$-schemes:
\begin{align}\label{background:eqn:1motifsmoduli}
\Sh_{K_\Phi}(Q_\Phi,D_\Phi)&\xrightarrow{\simeq}\widetilde{\xi}_K(\Phi).
\end{align}
From now on, we will freely use this isomorphism as an identification.

\subsubsection{}\label{background:subsubsec:ekphi_siegel}
Recall that in~\eqref{background:subsubsec:torustors}, we defined the torus $\mb{E}_K(\Phi)$ with cocharacter group
\[
 \mb{B}_K(\Phi) = (W_{\Phi}(\Rat)\cap K_{\Phi,W})(-1).
\]

We will now describe it explicitly. Let $(P_\Phi(0),D_\Phi(0))$ be as in \emph{loc. cit.}, so that $P_\Phi(0)=W_\Phi\rtimes\Gm$. In our situation, we can describe this pair as a rational boundary component for a particular Siegel Shimura datum. Namely, set $H(0)=\gr^W_0H\oplus W_{-2}H$, and equip it with the symplectic pairing $\psi(0)$ induced from $\psi$. Let $(G(0),X(0))$ be the Siegel Shimura datum associated with this symplectic space. Then $(P_\Phi(0),D_\Phi(0))$ is a rational boundary component of $(G(0),X(0))$ associated with the Lagrangian sub-space $W_{-2}H\subset H(0)$.

There is a natural lattice:
\[
 H(0)(\Int) = \gr^WH^g(\Int)\oplus W_{-2}H^g(\Int)\subset H(0).
\]

Let $\tilde{P}(0)\subset\GSp(H(0),\psi(0))$ be the parabolic subgroup stabilizing $W_{-2}H$. Set $\Phi(0)=(\tilde{P}(0),X(0)^+,1)$, where $X(0)^+\subset X(0)$ is the connected component determining the same element of $\on{S}^{\pm}(0)$ as $X^+$ (cf.~\ref{background:subsubsec:siegelbdry}). This is a clr for $(G(0),X(0))$. Let $K(0)\subset G(0)(\Adele_f)$ be any compact open whose intersection with $P_\Phi(0)(\Adele_f)$ is $K_{\Phi,W}\rtimes\nu(K)$.

The variety $\Sh_{K(0)_{\Phi(0),h}}(G_{\Phi(0),h},D_{\Phi(0),h})$ is the $0$-dimensional Shimura variety $\Sh_{\nu(K)}(\Gm,\on{S}^{\pm}(0))$.

From the discussion in~\eqref{background:subsubsec:torustors}, we find that $\mb{E}_K(\Phi)\times \Sh_{K(0)_{\Phi(0),h}}(G_{\Phi(0),h},D_{\Phi(0),h})$ can be canonically identified with the mixed Shimura variety $\Sh_{K(0)_{\Phi(0)}}(Q_{\Phi(0)},D_{\Phi(0)})$, to which we can give a moduli interpretation via~\eqref{background:eqn:1motifsmoduli}. Namely, it is the moduli space of tuples $(\mathcal{Q}(0),\lambda(0),\alpha,\dual{\alpha},\varepsilon(0))$ over $\Rat$-schemes $S$, where $(\mathcal{Q}(0),\lambda(0))$ is a polarized $1$-motif over $S$, $\alpha$ and $\dual{\alpha}$ are exactly as in~\eqref{background:subsubsec:1motifsmoduli}, and
\[
 \varepsilon(0)\in H^0\bigl(S,\underline{\Isom}(H(0)(\widehat{\Int}),\widehat{T}(\mathcal{Q}(0)))/K(0)\bigr)
\]
is a $K(0)$-level structure on $(\mathcal{Q}(0),\lambda(0),\alpha,\dual{\alpha})$.

The canonical section $\Sh_{K(0)_{\Phi(0),h}}(G_{\Phi(0),h},D_{\Phi(0),h})\to\Sh_{K(0)_{\Phi(0)}}(Q_{\Phi(0)},D_{\Phi(0)})$ inducing its trivialization as an $\mb{E}_K(\Phi)$-torsor is also easily described: It corresponds to the \emph{split} $1$-motif
\[
 \bigl[\gr^W_0H^g(\Int)\xrightarrow{0}\SHom(W_{-2}H^g(\Int),\Gm)],
\]
equipped with the obvious trivialization of its Tate module.

\subsubsection{}\label{background:subsubsec:ekphiaction}
We can use the above description to give a moduli interpretation to the action of $\mb{E}_K(\Phi)$ on $\Sh_{K_\Phi}(Q_\Phi,D_\Phi)$. This amounts to giving an action of $\Sh_{K(0)_{\Phi(0)}}(Q_{\Phi(0)},D_{\Phi(0)})$ on $\Sh_{K_\Phi}(Q_\Phi,D_\Phi)$ over $\Sh_{K(0)_{\Phi(0),h}}(G_{\Phi(0),h},D_{\Phi(0),h})$.

Suppose that we are given a $\Rat$-scheme $S$ and a pair of tuples
\[
\bigl((\mathcal{Q}(0),\lambda(0),\alpha(0),\dual{\alpha}(0),\varepsilon(0)),(\mathcal{Q},\lambda,\alpha,\dual{\alpha},\varepsilon)\bigr)\in \Sh_{K(0)_{\Phi(0)}}(Q_{\Phi(0)},D_{\Phi(0)})(S)\times\Sh_{K_\Phi}(Q_\Phi,D_\Phi)(S)
\]
The image of this tuple $(\mathcal{Q}',\lambda',\alpha',\dual{(\alpha')},\varepsilon')$ in $\Sh_{K_\Phi}(Q_\Phi,D_\Phi)(S)$ under the natural action can be described explicitly.

The isomorphisms $\alpha(0)\circ\alpha^{-1}$ and $\dual{\alpha}(0)\circ(\dual{\alpha})^{-1}$ allow us to identify:
\[
 \mathcal{Q}^{\et} = \mathcal{Q}(0)^{\et}\;;\;\mathcal{Q}^{\mult,C} = \mathcal{Q}(0)^{\mult,C},
\]
compatibly with the maps $\lambda^{\et}$ and $\lambda(0)^{\et}$.

So the $1$-motif $\mathcal{Q}$ is a tuple $(\mathcal{Q}^{\ab},\mathcal{Q}^{\et},\mathcal{Q}^{\mult},c,\dual{c},\tau)$, and the $1$-motif $\mathcal{Q}(0)$ is a tuple $(0,\mathcal{Q}^{\et},\mathcal{Q}^{\mult},0,0,\tau(0))$. Here, $\tau(0)$ is a trivialization of the trivial $\Gm$-bi-extension of $\mathcal{Q}^{\et}\times\mathcal{Q}^{\mult,C}$. The $1$-motif $\mathcal{Q}'$ is now associated with the same tuple as $\mathcal{Q}$, except that the trivialization $\tau'$ is now $\tau(0)\tau$.

As for the remaining data, by construction, we have canonical identifications $\mathcal{Q}^{',?}=\mathcal{Q}^{?}$, for $?=\ab,\et,\mult$. Therefore, the tuple $(\lambda^{\ab},\lambda^{\et},\lambda^{\mult,C})$ still defines a polarization $\lambda'$ on $\mathcal{Q}'$, and $\alpha$ and $\dual{\alpha}$ induce in an obvious way the isomorphisms $\alpha'$ and $\alpha^{',\vee}$.

First, observe that both $\varepsilon(0)$ and $\varepsilon$ determine a $\nu(K)$-orbit of isomorphisms $\underline{\widehat{\Int}}\xrightarrow{\simeq}\underline{\widehat{\Int}}(1)$; equivalently, they determine a point in $\Sh_{K(0)_{\Phi(0),h}}(G_{\Phi(0),h},D_{\Phi(0),h})$. Since our tuples, by assumption, lie above the same point in $\Sh_{K(0)_{\Phi(0),h}}(G_{\Phi(0),h},D_{\Phi(0),h})$, it follows that they determine the \emph{same} $\nu(K)$-orbit. In particular, we can assume that $\varepsilon(0)$ and $\varepsilon$ admit lifts
\[
 (\eta(0),u(0))\in\underline{\Isom}(H(0)(\widehat{\Int}),\widehat{T}(\mathcal{Q}(0)))\;;\;(\eta,u)\in\underline{\Isom}(H^g(\widehat{\Int}),\widehat{T}(\mathcal{Q}))
\]
with $u(0)=u$.

$\widehat{T}(\mathcal{Q}(0))$ is canonically an extension of $\gr^W_0\widehat{T}(\mathcal{Q})$ by $W_{-2}\widehat{T}(\mathcal{Q})$, and its push-out along the inclusion
\[
j:W_{-2}\widehat{T}(\mathcal{Q})\hookrightarrow W_{-1}\widehat{T}(\mathcal{Q})
\]
gives an extension $j_*\widehat{T}(\mathcal{Q}(0))$ of $\gr^W_0\widehat{T}(\mathcal{Q})$ by $W_{-1}\widehat{T}(\mathcal{Q})$. By construction, $\widehat{T}(\mathcal{Q}')$ is canonically isomorphic to the Baer sum of $\widehat{T}(\mathcal{Q})$ and $j_*\widehat{T}(\mathcal{Q}(0))$ in the category of extensions of $\gr^W_0\widehat{T}(\mathcal{Q})$ by $W_{-1}\widehat{T}(\mathcal{Q})$.

Similarly, $H^g(\widehat{\Int})$ is canonically isomorphic to the Baer sum with the \emph{trivial} extension of $\gr^W_0H^g(\widehat{\Int})$ by $W_{-1}H^g(\widehat{\Int})$. Therefore, the Baer sum of $\eta$ and $j_*\eta(0)$ now determines a canonical isomorphism:
\[
 \eta':H^g(\widehat{\Int})\xrightarrow{\simeq}\widehat{T}(\mathcal{Q}').
\]

The $K_{\Phi}$-orbit of $(\eta',u)$ now gives us the level structure $\varepsilon'$.

\subsubsection{}\label{background:subsubsec:sheavesxikphi}
Let $(\mathcal{Q},\lambda)$ be the tautological polarized $1$-motif over $\Sh_{K_\Phi}(Q_\Phi,D_\Phi)$. The $\widehat{\Int}$-sheaf $\bm{H}_{\widehat{\Int}}(\Phi)$ over $\Sh_{K_\Phi}(Q_\Phi,D_\Phi)_{\Comp}$, along with its weight filtration and polarization pairing, has a canonical descent over $\Sh_{K_\Phi}(Q_\Phi,D_\Phi)$ realized by the total Tate module $\widehat{T}(\mathcal{Q})$: We will sometimes denote this descent also by $\bm{H}_{\widehat{\Int}}(\Phi)$.

Now, consider the analytic vector bundle $\Reg{\Sh_{K_\Phi}(Q_\Phi,D_\Phi)(\Comp)}^{\an}\otimes\bm{H}_B(\Phi)$ over $\Sh_{K_\Phi}(Q_\Phi,D_\Phi)(\Comp)$: It is equipped with an integrable connection, a parallel weight filtration, a Hodge filtration, and a non-degenerate polarization pairing with values in the structure sheaf. The de Rham realization of $\mathcal{Q}$ (or rather, its dual) gives a canonical descent $\bm{H}_{\dR}(\Phi)$ over $\Sh_{K_\Phi}(Q_\Phi,D_\Phi)$ of this vector bundle, along with all its additional structures.

In the particular case where $\Phi=(G,X^+,1)$, $\Sh_{K_\Phi}(Q_\Phi,D_\Phi)=\Sh_K$ is the Siegel modular variety itself, and we obtain the usual moduli interpretation of $\Sh_K$ as the space parameterizing triples $(A,\lambda,\varepsilon)$ over $\Rat$-schemes $S$, where $(A,\lambda)$ is a polarized abelian scheme over $S$ and
\[
 \varepsilon\in H^0\bigl(S,\underline{\Isom}\bigl(H_{\widehat{\Int}},\widehat{T}(A)\bigr)
\]
is a $K$-level structure on $(A,\lambda)$.

In this case, we will omit the $\Phi$ from our notation for the sheaves over $\Sh_K$, and simply write $\bm{H}_B$, $\bm{H}_{\on{MH}}$, $\bm{H}_{\widehat{\Int}}$, $\bm{H}_{\dR}$, for the sheaves $\bm{H}_B(\Phi)$, etc.

\subsubsection{}\label{background:subsubsec:boundarymapmoduli}
We will identify $W_\Phi(\Real)(-1)$ with $W_\Phi(\Real)$ via the choice of connected component $X^+$ determined by $\Phi$: This also identifies $\mb{H}(\Phi)$ with the cone of positive definite symmetric bilinear forms on $\gr^W_0H(\Real)$. Fix a rational polyhedral cone $\sigma\subset W_\Phi(\Real)$, satisfying $\sigma^{\circ}\subset\mb{H}(\Phi)$. Then we have the associated twisted toric embedding $\Sh_{K_\Phi}(Q_\Phi,D_\Phi)\into\Sh_{K_\Phi}(Q_{\Phi},D_{\Phi},\sigma)$ over $\Sh_{\overline{K}_\Phi}(\overline{Q}_\Phi,\overline{D}_\Phi)$. Let $Z_{K_\Phi}(Q_\Phi,D_\Phi,\sigma)\subset\Sh_{K_\Phi}(Q_{\Phi},D_{\Phi},\sigma)$ be the closed stratum.

Given a point $s$ of $Z_{K_\Phi}(Q_\Phi,D_\Phi,\sigma)$, write $R(\Phi,\sigma,s)$ for the complete local ring of $\Sh_{K_\Phi}(Q_{\Phi},D_{\Phi},\sigma)$ at $s$. Set $S\coloneqq S(\Phi,\sigma,s)=\Spec R(\Phi,\sigma,s)$, and
\[
 V\coloneqq V(\Phi,\sigma,s)=S(\Phi,\sigma,s)\times_{\Sh_{K_\Phi}(Q_{\Phi},D_{\Phi},\sigma)}\Sh_{K_\Phi}(Q_\Phi,D_\Phi)\subset S(\Phi,\sigma,s).
\]

Consider the polarized $1$-motif $(\mathcal{Q},\lambda)\vert_V$ over $V$ obtained from the restriction of the tautological one over $\Sh_{K_\Phi}(Q_\Phi,D_\Phi)$. The condition that $\sigma\subset\mb{H}(\Phi)$ ensures that this object satisfies the positivity condition in the definition from (\ref{semistable:subsubsec:degpol}): Indeed, it implies that for any divisor $D'\subset S$ with support in $S\backslash V$, the symmetric bi-linear pairing $(y,x)\mapsto\nu_{D'}(\tau(y,x))$ lies in the closure of $\mb{H}(\Phi)$ within $W_\Phi(\Real)$, and thus is positive semi-definite.

Therefore, in the notation of (\ref{semistable:subsec:degen}), $(\mathcal{Q},\lambda)\vert_V$ is an object in $\mb{DD}_{\on{pol}}(S,V)$. In particular, by the equivalence of categories in \eqref{semistable:subsubsec:mumford}, there exists a canonical polarized abelian scheme $(\mathcal{A}(\Phi,\sigma,s),\psi(\Phi,\sigma,s))$ over $V$ with semi-abelian extension over $S$ such that:
\[
 \on{M}_{(S,V)}\bigl((\mathcal{A}(\Phi,\sigma,s),\psi(\Phi,\sigma,s))\bigr)=(\mathcal{Q},\lambda)\vert_V.
\]

For simplicity, set $\mathcal{A}'_V=\mathcal{A}(\Phi,\sigma,s)$ and $\psi'=\psi(\Phi,\sigma,s)$. Using \eqref{semistable:eqn:torsion}, we have a canonical isomorphism of total Tate modules over $V$:
\begin{align}\label{background:eqn:mumfordtotaltate}
\widehat{T}(\mathcal{Q})\vert_V\xrightarrow{\simeq}\widehat{T}(\mathcal{A}'_V)
\end{align}
This respects the Weil pairings on both sides induced by the respective polarizations.

Suppose now that we are given a section $(\eta,u)\in H^0\bigl(S,\underline{\Isom}(H^g_{\widehat{\Int}},\widehat{T}(\mathcal{Q}))\bigr)$; then we obtain an isomorphism:
\[
 \tilde{\eta}:\underline{H}({\widehat{\Int}})\xrightarrow[\simeq]{g}\underline{H}^g({\widehat{\Int}})\xrightarrow[\simeq]{\eta}\widehat{T}(\mathcal{Q})\vert_V\xrightarrow[\simeq]{\eqref{background:eqn:mumfordtotaltate}}\widehat{T}(\mathcal{A}'_V).
\]
The pair $(\tilde{\eta},u\circ(N(g)\nu(g)))$ is now a section of $\underline{\Isom}(H(\widehat{\Int}),\widehat{T}(\mathcal{A}'_V))$.

From this, we deduce that the tautological $K_{\Phi}$-level structure $\varepsilon'$ on $(\mathcal{Q},\lambda)$ produces in a canonical way a $K$-level structure $\eta'$ on $(\mathcal{A}'_V,\psi')$.

Therefore, using the moduli interpretation of $\Sh_K$, we obtain a canonical map:
\begin{align}\label{background:eqn:phisigmasmap}
i'(\Phi,\sigma,s):V\to\Sh_K
\end{align}
corresponding to the triple $(\mathcal{A}'_V,\psi',\eta')$.

\subsubsection{}\label{background:subsubsec:formalboundarychar0}
Now, suppose that we have an admissible rpcd $\Sigma$ for $(G,X,K)$ with $\sigma\in\Sigma(\Phi)$. Set $\Upsilon=[(\Phi,\sigma)]\in\on{Cusp}_K^{\Sigma}(G,X)$.

By~\eqref{background:lem:deltacirc_trivial}, $\Delta_K^{\circ}(\Phi)$ is trivial, and so $Z_K(\Upsilon)$ is canonically isomorphic to $Z_{K_\Phi}(Q_\Phi,D_\Phi,\sigma)$. Under this isomorphism, the point $s\in Z_{K_\Phi}(Q_\Phi,D_\Phi,\sigma)$ corresponds to a point $s'\in Z_K(\Upsilon)\subset\Sh^{\Sigma}_K$.

By \eqref{background:thm:pink}, $\widehat{\Sh}_{K_\Phi}(Q_\Phi,D_\Phi)$ is canonically isomorphic to the completion of $\Sh^{\Sigma}_K$ along the stratum $Z_K(\Upsilon)$. Here, $\widehat{\Sh}_{K_\Phi}(Q_\Phi,D_\Phi)$ is the completion of $\Sh_{K_\Phi}(Q_\Phi,D_\Phi)$ along $Z_{K_\Phi}(Q_\Phi,D_\Phi,\sigma)$. Therefore, there is an identification of $R(\Phi,\sigma,s)$ with the complete local ring of $\Sh^{\Sigma}_K$ at $s'$. This gives us a map $S(\Phi,\sigma,s)\to\Sh^{\Sigma}_K$, whose restriction to $V(\Phi,\sigma,s)$ factors through $\Sh_K$ thus giving us:
\begin{align}\label{background:eqn:phisigmasmap2}
i(\Phi,\sigma,s):V\to\Sh_K.
\end{align}
This map corresponds to a triple $(\mathcal{A}_V,\psi,\eta)$ over $V$.

\subsubsection{}\label{background:subsubsec:totaltateisom}
Consider the complex analytic open sub-space $\on{U}_{K_\Phi}(Q_\Phi,D_\Phi)\subset\Sh_{K_\Phi}(Q_\Phi,D_\Phi)(\Comp)$, equipped with its local isomorphism $\on{U}_{K_\Phi}(Q_\Phi,D_\Phi)\xrightarrow{\eqref{background:eqn:adeliccover}}\Sh_K(\Comp)$.

Over it, we have the polarized variation of mixed $\Int$-Hodge structures $\bm{H}_{\on{MH}}(\Phi)_{\Int}\vert_{\on{U}_{K_\Phi}(Q_\Phi,D_\Phi)}$: By the description of the map $X^+\to F_{P,X^+} = D_\Phi$ in \eqref{background:subsubsec:siegelbdry}, we find that, if we \emph{forget} the weight filtration on this variation of mixed Hodge structures, the remaining data forms a polarized variation of \emph{pure} Hodge structures of weights $(-1,0),(0,-1)$. By construction, this is simply the restriction of $\bm{H}_{\on{MH},\Int}$ via the map~\eqref{background:eqn:adeliccover}. In particular, we obtain a canonical isomorphism of $\widehat{\Int}$-sheaves:
\begin{align}\label{background:eqn:ukphicanisom}
\widehat{T}(\mathcal{Q})\vert_{\on{U}_{K_\Phi}(Q_\Phi,D_\Phi)}\xrightarrow{\simeq}\bm{H}_{\widehat{\Int}}(\Phi)\vert_{\on{U}_{K_\Phi}(Q_\Phi,D_\Phi)}\xrightarrow{\simeq}\bm{H}_{\widehat{\Int}}\vert_{\on{U}_{K_\Phi}(Q_\Phi,D_\Phi)}\xrightarrow{\simeq}\widehat{T}(\mathcal{A})\vert_{\on{U}_{K_\Phi}(Q_\Phi,D_\Phi)}.
\end{align}

Suppose now that $s$ belongs to $Z_{K_\Phi}(Q_\Phi,D_\Phi,\sigma)(\Comp)$. Then the field of rational functions $Q(V)$ on $V$ (which is also the fraction field of $R(\Phi,\sigma,s)$) contains the field of meromorphic functions on a small neighborhood of $s$ in $\on{U}_{K_\Phi}(Q_\Phi,D_\Phi,\sigma)$. Therefore,~\eqref{background:eqn:ukphicanisom} gives us a canonical isomorphism of sheaves over $V$:
\begin{align}\label{background:eqn:analytictotaltate}
\widehat{T}(\mathcal{Q})\vert_V\xrightarrow{\simeq}\widehat{T}(\mathcal{A}_V).
\end{align}

Combining~\eqref{background:eqn:analytictotaltate} with~\eqref{background:eqn:mumfordtotaltate} gives us a canonical isomorphism of total Tate modules:
\begin{align}\label{background:eqn:analytic_mumford_comp}
\widehat{T}(\mathcal{A}'_V)\xrightarrow{\simeq}\widehat{T}(\mathcal{A}_V).
\end{align}

\begin{prp}\label{background:prp:analytic_mumford_comp}
The isomorphism~\eqref{background:eqn:analytic_mumford_comp} arises from a unique isomorphism
\[
 (\mathcal{A}'_V,\psi',\eta')\xrightarrow{\simeq}(\mathcal{A}_V,\psi,\eta)
\]
of triples over $V$. In particular, the maps \eqref{background:eqn:phisigmasmap} and~\eqref{background:eqn:phisigmasmap2} are identical.
\end{prp}
\begin{proof}
The existence of the isomorphism follows from the main technical result of~\cite{lan:analytic}, and comes down to a comparison between the algebraic and analytic definitions of theta functions; cf. in particular, Prop. 4.2.2 of \emph{loc. cit.} However, we can sketch a direct proof here. Set $Y=\gr^W_0H^g(\Int)$; then, restricting $\mathcal{Q}$ to $\on{U}_{K_\Phi}(Q_\Phi,D_\Phi)$ gives us a $1$-motif $[Y\xrightarrow{j}\mathcal{Q}^{\sab}]\vert_{\on{U}_{K_\Phi}(Q_\Phi,D_\Phi)}$. By construction, $\mathcal{A}\vert_{\on{U}_{K_\Phi}(Q_\Phi,D_\Phi)}$ is the analytic quotient of $\mathcal{Q}^{\sab}\vert_{\on{U}_{K_\Phi}(Q_\Phi,D_\Phi)}$ by $j(Y)$. It is easy to see that this quotient has semi-abelian degeneration over $\on{U}_{K_\Phi}(Q_\Phi,D_\Phi,\sigma)$.

This immediately implies that $\mathcal{A}_V$ has a semi-abelian extension over $S(\Phi,\sigma,s)$, and that the formal completion of this semi-abelian extension is canonically isomorphic to that of $\mathcal{Q}^{\sab}$.

Similarly, the dual $\dual{\mathcal{A}}\vert_{\on{U}_{K_\Phi}(Q_\Phi,D_\Phi)}$ is a quotient of $\mathcal{Q}^{\vee,\sab}$, which is classified by the projection of $j$ onto $\mathcal{Q}^{\ab}$.

In particular, $\mathcal{A}_V$ is an object of $\mb{DEG}(S,V)$, and by the equivalence in~\eqref{semistable:subsubsec:mumford}, corresponds to an object in $\mb{DD}(S,V)$ of the form
\[
\mathcal{Q}'=[Y\xrightarrow{j'}\mathcal{Q}^{\sab}\vert_V].
\]
The proposition amounts to the equality $j=j'$.

Now, the projection of $j'$ onto $\mathcal{Q}^{\ab}\vert_V$ must agree with that of $j$. Indeed, this projection also classfies the semi-abelian scheme $\mathcal{Q}^{\sab,\vee}\vert_V$. Therefore, $j$ and $j'$ must differ by a map $f:Y\to\mathcal{Q}^{\mult}\vert_V$, which determines a $1$-motif $\mathcal{Q}(0)$ over $V$ with $\mathcal{Q}(0)^{\ab}=0$. Note that $f$ is determined by a pairing $\tau_0:Y\times X\to\Gmh{V}$, where $X=\gr^W_0H^{g,\vee}(\Int)(\nu)$. We have to show that this pairing is trivial.

Putting together~\eqref{background:eqn:mumfordtotaltate} and~\eqref{background:eqn:analytictotaltate} gives us a canonical isomorphism $\widehat{T}(\mathcal{Q})\vert_V\xrightarrow{\simeq}\widehat{T}(\mathcal{Q}')$ as extensions of $Y\otimes\underline{\widehat{\Int}}$ by $\widehat{T}(\mathcal{Q}^{\sab})$. In turn, this implies that $\widehat{T}(\mathcal{Q}(0))$ must be the \emph{split} extension of $Y\otimes\underline{\widehat{\Int}}$ by $\SHom(X,\widehat{\underline{\Int}}(1))$ over $V$.

Now, the only global sections of $\Reg{V}^\times$ admitting $n^{\text{th}}$-roots in $\Reg{V}^\times$ for all $n\in\Int_{>0}$ are the constants. From this, and the explicit description of $\widehat{T}(\mathcal{Q}(0))$ in terms of the pairing $\tau_0$ (cf.~\eqref{semistable:subsubsec:ntorsion}), we find that it can be a split extension if and only if $\tau_0$ takes values in $\Comp^\times$.

Therefore, to show that $\tau_0$ is trivial, we can restrict to a small unit disk around $s$ in $\on{U}_{K_\Phi}(Q_\Phi,D_\Phi,\sigma)$, and so can replace $R$ with $\Comp\pow{t}$, and $V$ with $\Spec\Comp((t))$. In this case, by~\cite[Ch. III, Prop. 8.1]{faltings_chai}, there is a natural quotient map:
\[
 \mathcal{Q}^{\sab}\bigl(\Comp((t))\bigr) \to \mathcal{A}_V\bigl(\Comp((t))\bigr),
\]
whose kernel is $j'(Y)$. On the other hand, the analytic construction gives us another such quotient map whose kernel is $j(Y)$. Arguing as in~\cite[4.5.5.3]{lan:thesis}, one checks that both quotient maps are actually equal. In particular, $j$ and $j'=j+f$ have the same image in $\mathcal{Q}^{\sab}\vert_V$. This immediately implies that we must have $f=0$.
\end{proof}

\subsubsection{}\label{background:subsubsec:1motifsmodulip}
Assume now that $H(\Int)$ has discriminant $1$. Fix a prime $p$. We will now look at moduli spaces over $\Int_{(p)}$ of polarized $1$-motifs.

We will consider level sub-groups of the form $K=K_pK^p$, where $K^p\subset G(\Adele_f^p)$ is neat and $K_p\subset G(\Rat_p)$ is the stabilizer of $H(\Int_p)\subset H(\Rat_p)$. Fix a clr $\Phi$: then we obtain an induced decomposition $K_{\Phi}=K_{\Phi,p}K^p_{\Phi}$, where $K_{\Phi,p}\subset Q_\Phi(\Adele_f)$ is the stabilizer of $H^g(\Int_p)\subset H(\Rat_p)$.

We will now, using its moduli interpretation, produce an integral model $\Ss_{K_\Phi}(Q_\Phi,D_\Phi)$ over $\Int_{(p)}$ for $\Sh_{K_\Phi}(Q_\Phi,D_\Phi)$

Fix a $\Int_{(p)}$-scheme $S$ and a polarized $1$-motif $(Q,\lambda)$, equipped with isomorphisms $\alpha$ and $\dual{\alpha}$ exactly as in \eqref{background:subsubsec:1motifsmoduli}. We will require that the map $Q^{\et}\to Q^{\vee,\et}$ is identified along $\alpha$ and $\dual{\alpha}$ with the map $\lambda^{\et}:Y\to Y'$.\footnote{In characteristic $0$, this condition was implied by the existence of level structures.}

If $\widehat{T}^p(Q)$ is the prime-to-$p$ Tate module of $Q$, in complete analogy with the definitions above, we can speak of the sheaf $\underline{\Isom}\bigl(H^g(\widehat{\Int}^p),\widehat{T}^p(Q)\bigr)$ of pairs of isomorphisms $\eta^p:\underline{H}^g(\widehat{\Int}^p)\xrightarrow{\simeq}\widehat{T}^p(Q)$ and $u^p:\underline{\widehat{\Int}}^p\xrightarrow{\simeq}\underline{\widehat{\Int}}^p(1)$ that are compatible with $\lambda$, the filtrations $W_{\bullet}$, and $\alpha$ and $\dual{\alpha}$.

A $K_{\Phi}^p$-level structure $\varepsilon^p$ on the tuple $(Q,\lambda,\alpha^{\et},\alpha^{\mult})$ will now be a global section of $\underline{\Isom}\bigl(H^g(\widehat{\Int}^p),\widehat{T}^p(Q)\bigr)/K_{\Phi}^p$.

Let $\Ss_{K_\Phi}(Q_\Phi,D_\Phi)$ be the stack of groupoids over $\Int_{(p)}$ such that $\Ss_{K_\Phi}(Q_\Phi,D_\Phi)(S)$ is the category of tuples $(\mathcal{Q},\lambda,\alpha,\dual{\alpha},\varepsilon)$ as above. The restriction of this stack over $\Rat$ is represented by $\Sh_{K_\Phi}(Q_\Phi,D_\Phi)$.

In fact, we will see below that $\Ss_{K_\Phi}(Q_\Phi,D_\Phi)$ is itself represented by a smooth quasi-projective scheme over $\Int_{(p)}$.

If $\Phi$ is a clr of the form $(G,X^+,1)$, then the associated $\Int_{(p)}$-scheme $\Ss_{K_\Phi}(Q_\Phi,D_\Phi)$ is an integral model for $\Sh_K$, and we will denote it by $\mathcal{S}_K$.

\subsubsection{}\label{background:subsubsec:mbmkphi}
Let $\Ss_{K_{\Phi,h}}(G_{\Phi,h},D_{\Phi,h})$ be the moduli space of triples $(\mathcal{B},\lambda^{\ab},\varepsilon^{p,\ab})$ over $\Int_{(p)}$-schemes $S$, where $(\mathcal{B},\lambda^{\ab})$ is a principally polarized abelian scheme over $S$, and $\varepsilon^{p,\ab}$ is a $K^p_{\Phi,h}$-level structure on $(\mathcal{B},\lambda^{\ab}))$; that is, a section
\[
 \varepsilon^{\ab,p}\in H^0\bigl(S,\underline{\Isom}(\gr^W_{-1}H(\widehat{\Int}^p),\widehat{T}^p(\mathcal{B}))/K^p_{\Phi,h}\bigr).
\]

By~\cite[Theorem 7.9]{mumford:git}, $\Ss_{K_{\Phi,h}}(G_{\Phi,h},D_{\Phi,h})$ is represented by a quasi-projective scheme over $\Int_{(p)}$, which we denote by the same symbol; its smoothness amounts to the fact that the deformation problem for principally polarized abelian varieties is formally smooth; for instance, cf.~\cite[Corollary 2.12]{chai_oort:moduli}.

The generic fiber of $\Ss_{K_{\Phi,h}}(G_{\Phi,h},D_{\Phi,h})$ is $\Sh_{K_{\Phi,h}}(G_{\Phi,h},D_{\Phi,h})$: This amounts to seeing that, over a $\Rat$-scheme $S$, giving a $K_{\Phi,h}$-level structure is equivalent to giving a $K^{p}_{\Phi.h}$-level structure. This in turn follows from the fact that the $p$-primary part $K_{\Phi,h,p}$ of $K_{\Phi,h}$ is the hyperspecial compact open sub-group:
\[
 K_{\Phi,h,p} = \GSp(\gr^W_{-1}H^g(\Int_p),\gr^W_{-1}\psi^g)\subset G(\Rat_p).
\]

\subsubsection{}\label{background:subsubsec:mbckphi}
We will now see that we have a tower $\Ss_{K_\Phi}(Q_\Phi,D_\Phi)\to\Ss_{\overline{K}_\Phi}(\overline{Q}_\Phi,\overline{D}_\Phi)\to\Ss_{K_{\Phi,h}}(G_{\Phi,h},D_{\Phi,h})$, where $\Ss_{\overline{K}_\Phi}(\overline{Q}_\Phi,\overline{D}_\Phi)$ is a $\Int_{(p)}$-model for $\Sh_{\overline{K}_\Phi}(\overline{Q}_\Phi,\overline{D}_\Phi)$ mapping smoothly onto $\Ss_{K_{\Phi,h}}(G_{\Phi,h},D_{\Phi,h})$, and $\Ss_{K_\Phi}(Q_\Phi,D_\Phi)$ is an $\mb{E}_K(\Phi)$-torsor over $\Ss_{\overline{K}_\Phi}(\overline{Q}_\Phi,\overline{D}_\Phi)$.

To construct $\Ss_{\overline{K}_\Phi}(\overline{Q}_\Phi,\overline{D}_\Phi)$, consider, for any $\Ss_{K_{\Phi,h}}(G_{\Phi,h},D_{\Phi,h})$-scheme $S$, a $1$-motif
\[
 \mathcal{Q}_1 = [\gr^W_0\underline{H}(\Int)\xrightarrow{c}\mathcal{B}\vert_S]
\]
over $S$.

Let $\mathrm{I}(\mathcal{Q}_1)$ be the \'etale sheaf over $S$ parameterizing pairs $(\eta^{\sab,p},u^p)$ of isomorphisms:
\[
 \eta^{\sab,p}:\frac{\underline{H}^g(\widehat{\Int}^p)}{W_{-2}\underline{H}^g(\widehat{\Int}^p)}\xrightarrow{\simeq}\widehat{T}(\mathcal{Q}_1)\;;\;u^p:\widehat{\underline{\Int}}^p\xrightarrow{\simeq}\widehat{\underline{\Int}}^p(1),
\]
satisfying the following conditions:
\begin{itemize}
  \item $\eta^{\sab,p}$ preserves $W_{\bullet}$-filtrations;
  \item The composition
  \[
\gr^W_0\underline{H}^g(\widehat{\Int}^p)\xrightarrow[\gr^W_0\eta^{\sab,p}]{\simeq}\widehat{T}^p(\mathcal{Q}_1^{\et})=\gr^W_0\underline{H}^g(\widehat{\Int}^p)
\]
is the identity.
\item The induced pair $(\gr^W_{-1}\eta^{\sab,p},u^p)$ is a section of $\underline{\Isom}\bigl(\gr^W_{-1}H^g(\widehat{\Int}^p),\widehat{T}^p(\mathcal{B})\bigr)$.
\end{itemize}

There is a natural action of $\overline{K}^p_{\Phi}$ on $\mathrm{I}(\mathcal{Q}_1)$ via right composition, and a \defnword{$\overline{K}^p_{\Phi}$-level structure} on $\mathcal{Q}_1$ is a section:
\[
 \varepsilon^{p,\sab}\in H^0\bigl(S,\mathrm{I}(\mathcal{Q}_1)/\overline{K}^p_{\Phi}\bigr).
\]

Note that every $\overline{K}^p_{\Phi}$-level structure on $\mathcal{Q}_1$ induces (via the $\gr^W_{-1}$ operation) a $K^p_{\Phi,h}$-level structure on $(\mathcal{B},\lambda^{\ab})\vert_S$.

$\Ss_{\overline{K}_\Phi}(\overline{Q}_\Phi,\overline{D}_\Phi)$ will now be the moduli space over $\Ss_{K_{\Phi,h}}(G_{\Phi,h},D_{\Phi,h})$ parameterizing pairs $(\mathcal{Q}_1,\varepsilon^{\sab,p})$ inducing the tautological $K^p_{\Phi,h}$-level structure on $(\mathcal{B},\lambda^{\ab})\vert_S$.
ruby
We can also define a canonical abelian scheme $\mathcal{A}_K(\Phi)\to\Ss_{K_{\Phi,h}}(G_{\Phi,h},D_{\Phi,h})$ extending $A_K(\Phi)\to \Sh_{K_{\Phi,h}}(G_{\Phi,h},D_{\Phi,h})$. This will be defined just as $\Ss_{\overline{K}_\Phi}(\overline{Q}_\Phi,\overline{D}_\Phi)$ was, except that $\mathrm{I}(\mathcal{Q}_1)$ will be replaced with the sheaf $\mathrm{I}_0(\mathcal{Q}_1)$ parameterizing pairs $(\eta^{0,p},u^p)$ with $u^p$ as usual, and $\eta^{0,p}$ an isomorphism:
\[
 \eta^{0,p}:\gr^W_{-1}\underline{H}^g(\widehat{\Int}^p)\oplus\gr^W_0\underline{H}^g(\widehat{\Int}^p)\xrightarrow{\simeq}\widehat{T}(\mathcal{Q}_1).
\]
The pair $(\eta^{0,p},u^p)$ is required to satisfy the same set of conditions as a section of $\mathrm{I}(\mathcal{Q}_1)$.

$\mathcal{A}_K(\Phi)$ will now classify pairs $(\mathcal{Q}_1,\varepsilon^{p,0})$, where $\varepsilon^{p,0}$ is a section of $\mathrm{I}_0(\mathcal{Q}_1)/(K^p_{\Phi,V}\rtimes K^p_{\Phi,h})$.

As in~\eqref{background:subsubsec:ekphiaction}, we can define natural actions of $\mathcal{A}_K(\Phi)$ on itself and on $\Ss_{\overline{K}_\Phi}(\overline{Q}_\Phi,\overline{D}_\Phi)$ using Baer sums. The split $1$-motif $[\gr^W_0\underline{H}^g(\Int)\xrightarrow{0}\mathcal{B}]$, equipped with the natural trivialization of its Tate module, provides us the identity section $\Ss_{K_{\Phi,h}}(G_{\Phi,h},D_{\Phi,h})\to\mathcal{A}_K(\Phi)$. This makes $\mathcal{A}_K(\Phi)$ an abelian scheme over $\Ss_{K_{\Phi,h}}(G_{\Phi,h},D_{\Phi,h})$, and $\Ss_{\overline{K}_\Phi}(\overline{Q}_\Phi,\overline{D}_\Phi)\to\Ss_{K_{\Phi,h}}(G_{\Phi,h},D_{\Phi,h})$ a torsor under $\mathcal{A}_K(\Phi)$.

 In particular, $\Ss_{\overline{K}_\Phi}(\overline{Q}_\Phi,\overline{D}_\Phi)$ is relatively representable and smooth over $\Ss_{K_{\Phi,h}}(G_{\Phi,h},D_{\Phi,h})$.

Using the methods of~\eqref{background:subsubsec:comppoints} and~\eqref{background:subsubsec:xiphirep}, it is not hard to see that the generic fibers of $\Ss_{\overline{K}_\Phi}(\overline{Q}_\Phi,\overline{D}_\Phi)$ and $\mathcal{A}_K(\Phi)$ are precisely $\Sh_{\overline{K}_\Phi}(\overline{Q}_\Phi,\overline{D}_\Phi)$ and $A_K(\Phi)$, and that the induced action of $A_K(\Phi)$ on $\Sh_{\overline{K}_\Phi}(\overline{Q}_\Phi,\overline{D}_\Phi)$ is the canonical one.

\subsubsection{}\label{background:subsubsec:mbekphi_action}
Finally, we will describe the $\mb{E}_K(\Phi)$-action on $\Ss_{K_\Phi}(Q_\Phi,D_\Phi)$, and will leave it to the reader to check that this makes it an $\mb{E}_K(\Phi)$-torsor over $\Ss_{\overline{K}_\Phi}(\overline{Q}_\Phi,\overline{D}_\Phi)$.

In the notation of~\eqref{background:subsubsec:ekphi_siegel}, one has the moduli schemes 
\[
\Sh_{K(0)_{\Phi(0)}}(Q_{\Phi(0)},D_{\Phi(0)})\;;\;\Sh_{K(0)_{\Phi(0),h}}(G_{\Phi(0),h},D_{\Phi(0),h}).
\]
From the discussion above, they both admit integral models 
\[
\Ss_{K(0)_{\Phi(0)}}(Q_{\Phi(0)},D_{\Phi(0)})\;;\;\Ss_{K(0)_{\Phi(0),h}}(G_{\Phi(0),h},D_{\Phi(0),h})
\]
over $\Int_{(p)}$. In fact, $\Ss_{K(0)_{\Phi(0),h}}(G_{\Phi(0),h},D_{\Phi(0),h})$ is the finite \'etale $\Int_{(p)}$-scheme parameterizing $\nu(K^p)$-orbits of isomorphisms $\underline{\widehat{\Int}}^p\xrightarrow{\simeq}\underline{\widehat{\Int}}^p(1)$.

The section $\Sh_{K(0)_{\Phi(0),h}}(G_{\Phi(0),h},D_{\Phi(0),h})\to\Ss_{K(0)_{\Phi(0)}}(Q_{\Phi(0)},D_{\Phi(0)})$, defined in~\eqref{background:subsubsec:ekphi_siegel} using the split $1$-motif, immediately extends to one over $\Ss_{K(0)_{\Phi(0),h}}(G_{\Phi(0),h},D_{\Phi(0),h})$, and so identifies
\[
\Ss_{K(0)_{\Phi(0)}}(Q_{\Phi(0)},D_{\Phi(0)}) = \mb{E}_K(\Phi)\times\Ss_{K(0)_{\Phi(0),h}}(G_{\Phi(0),h},D_{\Phi(0),h}).
\]

The level structure $\varepsilon^{p,\ab}$ on $(\mathcal{B},\lambda^{\ab})$ gives a morphism 
\[
\Ss_{K_{\Phi,h}}(G_{\Phi,h},D_{\Phi,h})\to\Ss_{K(0)_{\Phi(0),h}}(G_{\Phi(0),h},D_{\Phi(0),h}), 
\]
and, as in~\eqref{background:subsubsec:ekphiaction}, one easily checks that there is a natural action of the moduli scheme $\Ss_{K(0)_{\Phi(0)}}(Q_{\Phi(0)},D_{\Phi(0)})$ on $\Ss_{K_\Phi}(Q_\Phi,D_\Phi)$ over $\Ss_{K(0)_{\Phi(0),h}}(G_{\Phi(0),h},D_{\Phi(0),h})$, which can be described using Baer sums.

\subsubsection{}\label{background:subsubsec:intfunct}
Suppose that we have two clrs $\Phi_1=(P_1,X_1^+,g_1)$ and $\Phi_2=(P_2,X_2^+,g_2)$; and $\gamma\in G(\Rat)$, $q\in Q_{\Phi_2}(\Adele_f)$ with $\Phi_1\xrightarrow{(\gamma,q)_K}\Phi_2$ and $\gamma\cdot P_1=P_2$. Then the moduli interpretation of the induced isomorphism $\Sh_{K_{\Phi_1}}(Q_{\Phi_1},D_{\Phi_1})\xrightarrow{\simeq}\Sh_{K_{\Phi_2}}(Q_{\Phi_2},D_{\Phi_2})$ from~\eqref{background:subsubsec:functoriality} shows that it extends to an isomorphism
\begin{align}\label{background:eqn:intfunct}
\rho(\gamma,q):\Ss_{K_{\Phi_1}}(Q_{\Phi_1},D_{\Phi_1})\xrightarrow{\simeq}\Ss_{K_{\Phi_2}}(Q_{\Phi_2},D_{\Phi_2}).
\end{align}
We leave the details to the reader.

\subsubsection{}\label{background:subsubsec:chaifal}
We now review the arithmetic theory of compactifications of $\Ss_K$ developed by Chai and Faltings in~\cite{faltings_chai}. The only thing of (minor) note is that we present their results entirely in the ad\'elic language.

As in \eqref{background:subsubsec:boundarymapmoduli}, for any rational polyhedral cone $\sigma\subset W_\Phi(\Real)$ satisfying $\sigma^\circ\subset\mb{H}(\Phi)$, we can consider the twisted toric embedding $\Ss_{K_\Phi}(Q_\Phi,D_\Phi)\into\Ss_{K_\Phi}(Q_{\Phi},D_{\Phi},\sigma)$ over $\Ss_{\overline{K}_\Phi}(\overline{Q}_\Phi,\overline{D}_\Phi)$. Let $\mathcal{Z}_{K_\Phi}(Q_\Phi,D_\Phi,\sigma)$ be the closed stratum in $\Ss_{K_\Phi}(Q_{\Phi},D_{\Phi},\sigma)$. Given a point $s_0^\beef$ of $\mathcal{Z}_{K_\Phi}(Q_\Phi,D_\Phi,\sigma)$, write $R(\Phi,\sigma,s_0^\beef)$ for the complete local ring of $\Ss_{K_\Phi}(Q_{\Phi},D_{\Phi},\sigma)$ at $s_0^\beef$, and set:
\[
 S\coloneqq S(\Phi,\sigma,s_0^\beef)=\Spec R(\Phi,\sigma,s_0^\beef);\;; V\coloneqq V(\Phi,\sigma,s_0^\beef)=S(\Phi,\sigma,s)\times_{\Ss_{K_\Phi}(Q_{\Phi},D_{\Phi},\sigma)}\Ss_{K_\Phi}(Q_\Phi,D_\Phi).
\]

Exactly as in \emph{loc. cit.} the polarized $1$-motif $(\mathcal{Q},\lambda)\vert_V$ produces a canonical polarized abelian scheme $(\mathcal{A}(\Phi,\sigma,s_0^\beef),\psi(\Phi,\sigma,s_0^\beef))$ over $V$, and produces a map:
\begin{align}\label{background:eqn:iphisigs0}
 i(\Phi,\sigma,s_0^\beef):V(\Phi,\sigma,s_0^\beef)\to\mathcal{S}_K.
\end{align}

Using~\eqref{background:subsubsec:intfunct}, we find that, if we have another clr $\Phi'$ and $\gamma\in G(\Rat)$, $q\in Q_\Phi(\Adele_f)$ with $\Phi\xrightarrow{(\gamma,q)_K}\Phi'$ and $\gamma\cdot P=P'$, then there exists an isomorphism
\[
 \rho(\gamma,q):\mathcal{Z}_{K_\Phi}(Q_\Phi,D_\Phi,\sigma)\xrightarrow{\simeq}\mathcal{Z}_{K_{\Phi'}}(Q_{\Phi'},D_{\Phi'},\sigma'),
\]
where $\sigma'=\gamma(\sigma)$. In particular, the scheme $\mathcal{Z}_{K_\Phi}(Q_\Phi,D_\Phi,\sigma)$ depends only on the class $\Upsilon$ of $(\Phi,\sigma)$ in $\on{Cusp}^{\Sigma}_K(G,X)$. We will therefore denote it by $\mathcal{Z}_K(\Upsilon)$.

The next theorem follows from~\cite[Ch. IV]{faltings_chai}, ~\cite[6.4.1.1]{lan:thesis} and~\eqref{background:prp:analytic_mumford_comp}.

\begin{thm}[Chai-Faltings]\label{background:thm:chaifal}
Let $\Sigma$ be a smooth admissible rpcd for $(G,X,K)$. Then the open immersion $\Sh_K\into\Sh^\Sigma_K$ over $\Rat$ extends to an open immersion of smooth $\Int_{(p)}$-schemes $\mathcal{S}_K\into\mathcal{S}^{\Sigma}_K$ with the following properties:
\begin{enumerate}
\item\label{chaifal:stratum}For every $\Upsilon\in\on{Cusp}^\Sigma_K(G,X)$, the immersion $Z_K(\Upsilon)\into\Sh^\Sigma_K$ extends to a locally closed immersion
\[
 \mathcal{Z}_K(\Upsilon)\into\mathcal{S}^\Sigma_K.
\]
\item\label{chaifal:strata}The stratification~\eqref{background:thm:pink}\eqref{pink:stratification} extends to an integral one:
\[
 \mathcal{S}^\Sigma_K=\bigsqcup_{\Upsilon}\mathcal{Z}_K(\Upsilon).
\]
For any fixed $\Upsilon$, the closure of $\mathcal{Z}_K(\Upsilon)$ in $\mathcal{S}^\Sigma_K$ is the closed sub-space:
\[
 \overline{\mathcal{Z}}_K(\Upsilon)=\bigsqcup_{\Upsilon'\preceq\Upsilon}\mathcal{Z}_K(\Upsilon').
\]
\item\label{chaifal:completion}Given $\Upsilon=[(\Phi,\sigma)]$, let $\widehat{\Ss}_{K_{\Phi}}(Q_{\Phi},D_{\Phi},\sigma)$ be the formal completion of $\Ss_{K_\Phi}(Q_{\Phi},D_{\Phi},\sigma)$ along $\mathcal{Z}_{K_\Phi}(Q_\Phi,D_\Phi,\sigma)$. Then the isomorphism between $\mathcal{Z}_{K_\Phi}(Q_\Phi,D_\Phi,\sigma)$ and $\mathcal{Z}_K(\Upsilon)$ lifts to an isomorphism of formal algebraic spaces:
\[
 \widehat{\Ss}_{K_{\Phi}}(Q_{\Phi},D_{\Phi},\sigma)\xrightarrow{\simeq}\bigl(\mathcal{S}^{\Sigma}_K\bigr)^{\wedge}_{\mathcal{Z}_K(\Upsilon)}
\]
restricting to the one from~\eqref{background:thm:pink}\eqref{pink:completion} in characteristic $0$. This extension is characterized by the following property: Given a point $s_0^\beef$ of $\mathcal{Z}_{K_\Phi}(Q_\Phi,D_\Phi,\sigma)$, with corresponding point $s'_0$ of $\mathcal{Z}_K(\Upsilon)$, consider the induced map:
\[
 \Spec R(\Phi,\sigma,s_0^\beef)\xrightarrow{\simeq}\Spec\widehat{\Rg}_{\mathcal{S}^{\Sigma}_K,s'_0}\to\mathcal{S}^{\Sigma}_K.
\]
The restriction of this map to $V(\Phi,\sigma,s_0^\beef)$ factors through $\Ss_K$ and is identified with $i(\Phi,\sigma,s_0^\beef)$.
\end{enumerate}
\end{thm}

\qed

\section{The local structure at the boundary of a Shimura variety of Hodge type}\label{sec:boundary}

\emph{We will fix a prime number $p$ for the entirety of this section}. The goal is to prove the main technical theorem~\eqref{boundary:thm:main} on the local structure at the boundary of an integral model for a Shimura variety of Hodge type over a $p$-adic place of its reflex field.

\subsection{Shimura varieties and absolute Hodge cycles}\label{boundary:subsec:abshodge}

This sub-section is a direct generalization of the first part of \cite[\S 2]{kisin:abelian}: We are simply transposing the results from the world of pure Shimura varieties (of Hodge type) and abelian varieties to that of mixed Shimura varieties\footnote{We will only be considering mixed Shimura varieties attached to rational boundary components of pure ones, but our discussion should apply just as well to arbitrary mixed Shimura varieties of Hodge type.} and $1$-motifs.

We will fix a Shimura datum $(G,X)$ of Hodge type. This means that it is equipped with an embedding into a Siegel Shimura datum
\[
\iota:(G,X)\into (\GSp(H),\on{S}^{\pm}(H)),
\]
which we will also fix.

Set $G^{\beef}=\GSp(H)$ and $X^{\beef}=\on{S}^{\pm}(H)$. Let $E=E(G,X)$ be the reflex field of $(G,X)$.

Suppose that ${K}^{\beef}\subset G^{\beef}(\Adele_f)$ and $K\subset G(\Adele_f)$ are neat compact open sub-groups with $K\subset K^\beef$. 

\emph{For the remainder of this section, we will always assume that $H$ and ${K}^{\beef}$ have been chosen so that $H$ admits a self-dual lattice $H(\Int)$ such that $H(\widehat{\Int})$ is stabilized by $K$ and such that ${K}^{\beef}_p\subset {G}^\beef(\Rat_p)$ is exactly the stabilizer of $H(\Int_p)$.}

It is not possible to arrange for the existence of such a lattice for every choice of symplectic representation $H$; however, using Zarhin's trick, one can always ensure it after replacing $H$ with $H^{\oplus 8}$. 

The main reason for this restriction is that we need to apply the theory of~\eqref{background:subsubsec:chaifal}.

\subsubsection{}\label{boundary:subsubsec:abshodge}
Suppose that we are given a clr $\Phi$ for $(G,X)$, with ${\Phi}^\beef=\iota_*\Phi$ the corresponding clr for $(G^{\beef},X^{\beef})$. The notation here is as in~\eqref{background:subsubsec:functoriality}.

We then obtain a map of mixed Shimura varieties: $\Sh_{K_\Phi}(Q_\Phi,D_\Phi)\to E\otimes\Sh_{K^\beef_{\Phi^\beef}}(Q_{\Phi^\beef},D_{\Phi^\beef})$. In particular, restriction along this map gives us a $1$-motif $\mathcal{Q}(\Phi)$ over $\Sh_{K_\Phi}(Q_\Phi,D_\Phi)$. In the language of \cite[\S 2.2]{brylinski:1motifs}, this $1$-motif is equipped with a certain family of absolute Hodge cycles, which we will now describe.

Let $\{s_{\alpha}\}\subset H(\Rat)^\otimes$ be a family of tensors such that $G\subset GL(H)$ is their point-wise stabilizer. For each index $\alpha$, via the functor in~\eqref{background:subsubsec:varmixedhodge}, we then obtain a canonical map of mixed Hodge structures $\bm{s}_{\alpha,\Phi}:\mb{1}\to \bm{H}^\otimes_{\on{MH}}(\Phi)$. In particular, we obtain canonical sections:
\begin{align*}
 s_{\alpha,\Phi,B}&\in H^0\bigl(\Sh_{K_\Phi}(Q_\Phi,D_\Phi)(\Comp),W_0\bm{H}^\otimes_B(\Phi)\bigr);\\
 s_{\alpha,\Phi,\dR}&\in H^0\bigl(\Sh_{K_\Phi}(Q_\Phi,D_\Phi)(\Comp),F^0\bm{H}^\otimes_{\dR}(\Phi)\cap W_0\bm{H}^\otimes_{\dR}(\Phi)\bigr)^{\nabla=0};\\
 s_{\alpha,\Phi,\et}&\in H^0\bigl(\Sh_{K_\Phi}(Q_\Phi,D_\Phi)(\Comp),W_0\bm{H}^\otimes_{\Adele_f}(\Phi)\bigr).
\end{align*}
The latter two are identified with $1\otimes s_{\alpha,\Phi,B}$ via the canonical comparison isomorphisms between the Betti and de Rham (resp. \'etale) realizations.

\begin{prp}\label{boundary:prp:abshodge}
\mbox{}
\begin{enumerate}
\item\label{abshodge:derham}For each $\alpha,$ $s_{\alpha,\Phi,\dR}$ descends to a section of $\bm{H}^\otimes_{\dR}(\Phi)$ defined over $\Sh_{K_\Phi}(Q_\Phi,D_\Phi)$.
\item\label{abshodge:etale}For each $\alpha$, $s_{\alpha,\Phi,\et}$ descends to a section of $\bm{H}^\otimes_{\Adele_f}(\Phi)$ defined over $\Sh_{K_\Phi}(Q_\Phi,D_\Phi)$.
\end{enumerate}
\end{prp}
\begin{proof}
When $\Phi=(G,X^+,1)$, and $\Sh_{K_\Phi}(Q_\Phi,D_\Phi)$ is a pure Shimura variety of Hodge type, this is deduced in~\cite[(2.2)]{kisin:abelian} from Deligne's theorem that all Hodge classes on abelian varieties over $\Comp$ are absolutely Hodge~\cite[Ch. I, Theorem 2.11]{dmos}. The same proof applies in our setting, except that we have to replace the use of Deligne's theorem with the result of Brylinski that all Hodge classes on $1$-motifs over $\Comp$ are absolutely Hodge~\cite[Th\'eor\`eme (2.2.5)]{brylinski:1motifs}.
\end{proof}

\subsubsection{}\label{boundary:subsubsec:blasius}
We now specialize to the case where $\Phi=(G,X^+,1)$. The $1$-motif $\mathcal{Q}(\Phi)$ in this case is simply an abelian scheme $\mathcal{A}$. As usual, in this case, we will omit $\Phi$ from the notation for all the cohomological cycles we have defined above.

Suppose that $F$ is a finite extension of $E_v$, where $v\vert p$ is a finite place of $E$, and suppose that $x\in\Sh_K(F)$. Then we have the de Rham comparison isomorphism
\[
\Bdr\otimes_{F}H^1_{\dR}(\mathcal{A}_x/F)\xrightarrow{\simeq}\Bdr\otimes_{\Rat_p}H^1_{\et}(\mathcal{A}_{x,\bar{F}},\Rat_p).
\]
\begin{prp}\label{boundary:prp:blasius}
Under this isomorphism, $s_{\alpha,\dR,x}$ is carried to $s_{\alpha,\et,x}$.
\end{prp}
\begin{proof}
  This is the main result of \cite{blasius:padic}, which applies directly when $\mathcal{A}_x$ is in fact defined over a number field. For the generality we need, as pointed out in \cite[5.6.3]{moonen:models}, we can either appeal to a trick of Lieberman as in \cite[5.2.16]{vasiu:preab}, or we can directly use the fact that $\mathcal{A}_x$ arises from the family $\mathcal{A}$ defined over the number field $E$.
\end{proof}

\subsubsection{}\label{boundary:subsubsec:twotensors}
The clr $\Phi$ will again be arbitrary. Fix $\sigma\subset W_\Phi(\Real)(-1)$ with $\sigma^{\circ}\subset\mb{H}(\Phi)$. Fix a point $t$ in $Z_{K_\Phi}(Q_\Phi,D_\Phi,\sigma)$, and let $R(\Phi,\sigma,t)$ be the complete local ring of $\Sh_{K_\Phi}(Q_{\Phi},D_{\Phi},\sigma)$ at $t$. Set
\[
 V\coloneqq V(\Phi,\sigma,t)=\Spec R(\Phi,\sigma,t)\times_{\Sh_{K_\Phi}(Q_{\Phi},D_{\Phi},\sigma)}\Sh_{K_\Phi}(Q_\Phi,D_\Phi).
\]
There is a canonical $1$-motif $\mathcal{Q}=\mathcal{Q}(\Phi)\vert_V$ over $V$, equipped with additional tensors
\begin{align*}
\{s_{\alpha,\Phi,\et}\}\subset H^0\bigl(V,\Adele_f\otimes\widehat{T}(\mathcal{Q})^\otimes\bigr)&;\{s_{\alpha,\Phi,\dR}\}\subset H^0(V,H^1_{\dR}(\mathcal{Q})^\otimes).
\end{align*}

Via the isomorphism $Z_{K_\Phi}(Q_\Phi,D_\Phi,\sigma)\xrightarrow{\simeq}Z_K(\Upsilon)$, we obtain a point $t'\in Z_K(\Upsilon)$ (where $\Upsilon=[(\Phi,\sigma)]$) and an identification of $R(\Phi,\sigma,t)$ with the complete local ring of $\Sh^\Sigma_K$ at $t'$. Through this, we obtain a map $V(\Phi,\sigma,t)\to\Sh_K$, and thus an abelian scheme $\mathcal{A}$ over $V$, equipped with additional tensors
\begin{align*}
\{s_{\alpha,\et}\}\subset H^0\bigl(V,\Adele_f\otimes\widehat{T}(\mathcal{A})^\otimes\bigr)&;\{s_{\alpha,\Phi,\dR}\}\subset H^0(V,H^1_{\dR}(\mathcal{A})^\otimes).
\end{align*}

We also have canonical isomorphisms
\begin{align*}
\widehat{T}(\mathcal{Q})\xrightarrow{\simeq}\widehat{T}(\mathcal{A})&;H^1_{\dR}(\mathcal{Q})\xrightarrow{\simeq}H^1_{\dR}(\mathcal{A}).
\end{align*}

\begin{prp}\label{boundary:prp:twotensors}
The two isomorphisms carry, for each $\alpha$, $s_{\alpha,\Phi,\et}$ to $s_{\alpha,\et}$ (resp. $s_{\alpha,\Phi,\dR}$ to $s_{\alpha,\dR}$).
\end{prp}
\begin{proof}
We can assume that $t$ is $\Comp$-valued. By~\eqref{background:prp:analytic_mumford_comp}, we can identify the isomorphism $\widehat{T}(\mathcal{Q})\xrightarrow{\simeq}\widehat{T}(\mathcal{A})$ with the analytic one arising from~\eqref{background:eqn:ukphicanisom}. For the latter, the proposition is clear by construction. A similar argument works for the de Rham realizations as well.
\end{proof}

\subsection{The canonical log $F$-crystal and its properties}\label{boundary:subsec:logfcrys}

For the rest of the section, we will fix a clr $\Phi^\beef$ for $({G}^\beef,X^\beef)$. We will use the choice of isomorphism of Hodge structures $\Rat\xrightarrow{\simeq}\Rat(1)$ determined by $X^{\beef,+}$ to identify $W_{\Phi^{\beef}}(\Real)(-1)$ with $W_{\Phi^{\beef}}(\Real)$.

\subsubsection{}\label{boundary:subsubsec:sigma}
Let $\sigma^\beef\subset W_{\Phi^{\beef}}(\Real)$ be a rational polyhedral cone with $\sigma^{\beef,\circ}\subset\mb{H}(\Phi^\beef)$. Given a point $s_0^\beef\in\mathcal{Z}_{{K}^{\beef}}(\Phi^\beef,\sigma^\beef)(\overline{\Field}_p)$, let $R(\sigma^\beef,s_0^\beef)\coloneqq R(\Phi^\beef,\sigma^\beef,s_0^\beef)$ be the complete local ring of $\Ss_{K^\beef_{\Phi^\beef}}(Q_{\Phi^\beef},D_{\Phi^\beef},\sigma^\beef)$ at $s_0^\beef$, and let $R^{\sab}$ be the complete local ring of $\Ss_{\overline{K}^\beef_{\Phi^\beef}}(\overline{Q}_{\Phi^\beef},\overline{D}_{\Phi^\beef},\sigma^\beef)$ at the image of $s_0^\beef$. Set $S(\sigma^\beef,s_0^\beef)=\Spec R(\sigma^\beef,s_0^\beef)$ and let $V(\sigma^\beef,s_0^\beef)\subset S(\sigma^\beef,s_0^\beef)$ be the complement of the boundary divisor.

Let $\mathcal{Q}(\sigma^\beef,s_0^\beef)$ be the tautological object in $\mb{DD}(S(\sigma^\beef,s_0^\beef),V(\sigma^\beef,s_0^\beef))$ determined by the natural map $V(\sigma^\beef,s_0^\beef)\to\mb{\Xi}$, and let $\mathcal{A}(\sigma^\beef,s_0^\beef)$ be the corresponding object of $\mb{DEG}(S(\sigma^\beef,s_0^\beef),V(\sigma^\beef,s_0^\beef))$ associated with it via the equivalence in~\eqref{semistable:subsubsec:mumford}.

Set $\mathsf{X}=\gr^W_0H^g(\Int)$; then the perfect pairing
\[
 \psi^g: H^g(\Int)\times H^g(\Int)\to\Int
\]
identifes $\gr^W_0H^{\vee,g}(\nu)$ with $\gr^W_0H^g(\Int)=\mathsf{X}$.

We have canonical identifications
\[
\mathcal{Q}(\sigma^\beef,s_0^\beef)^{\et}=\mathcal{Q}(\sigma^\beef,s_0^\beef)^{\mult,C} = \mathsf{X}.
\]
Moreover, the abelian part $\mathcal{B}=\mathcal{Q}(\sigma^\beef,s_0^\beef)^{\ab}$ depends only on the image of $s_0^\beef$ in $\Ss_{\overline{K}^\beef_{\Phi^\beef}}(\overline{Q}_{\Phi^\beef},\overline{D}_{\Phi^\beef},\sigma^\beef)$: It is equipped with a principal polarization $\lambda^{\ab}$, which we can use to identify it with $\dual{\mathcal{B}}$.

Over $R^{\sab}$, we also have a canonical homomorphism $\dual{c}:\mathsf{X}\to\dual{\mathcal{B}}=\mathcal{B}$, which classifies the semi-abelian part $\mathcal{G}\coloneqq\mathcal{Q}^{\sab}$. This is an extension:
\begin{equation}\label{boundary:eqn:semi_abelian_ext}
1\to\underline{\Hom}(\mathsf{X},\Gm)\to\mathcal{G}\to\mathcal{B}\to 1.
\end{equation}

Set $\widehat{\mathcal{U}}(\sigma^\beef,s_0^\beef)=\Spf R(\sigma^\beef,s_0^\beef)$. Set $K_0=W[p^{-1}]$, and let $\widehat{\mathcal{U}}^{\an}(\sigma^\beef,s_0^\beef)$ be the analytic space over $K_0$ associated with $\widehat{\mathcal{U}}({\sigma^\beef,s_0^\beef})$ (cf.~\ref{appendix:subsec:rigid}). We will write $\widehat{\mathcal{U}}^{\an,\circ}(\sigma^\beef,s_0^\beef)$ for the complement of the boundary divisor in $\widehat{\mathcal{U}}^{\an}(\sigma^\beef,s_0^\beef)$.

The goal of this sub-section is to explicitly describe the log Dieudonn\'e $F$-crystal $\mathcal{M}(\sigma^\beef,s_0^\beef)$ over $R(\sigma^\beef,s_0^\beef)$ that is associated with $\mathcal{Q}(\sigma^\beef,s_0^\beef)$. That is, we want to write down in concrete terms a $\varphi$-module over $R(\sigma^\beef,s_0^\beef)$, along with the logarithmic topologically nilpotent connection on it. 

Also, for later technical purposes, we will need to allow ourselves to change $\sigma^\beef$ in a way that replaces $R(\sigma^\beef,s_0^\beef)$ by complete local rings in a combination of blow-ups and blow-downs of $\Ss_{K^\beef_{\Phi^\beef}}(Q_{\Phi^\beef},D_{\Phi^\beef},\sigma^\beef)$. This necessitates a bit of notational and conceptual baggage that the reader can ignore if she so chooses.

\subsubsection{}\label{boundary:subsubsec:tori_mult_1motifs}
In the notation of~\eqref{background:subsubsec:torustors}, set $\mb{S}=\mb{S}_{{K}^{\beef}}(\Phi^\beef)$, $\mb{B}=\mb{B}_{{K}^{\beef}}(\Phi^\beef)$, and $\mb{E}=\mb{E}_{{K}^{\beef}}(\Phi^\beef)$.

As observed in~\eqref{background:subsubsec:mbekphi_action}, we can identify $\Spec R^{\sab}\times\mb{E}$ with the scheme over $R^{\sab}$ that parameterizes principally polarized $1$-motifs $(\mathcal{Q}(0),\lambda(0))$ satisfying:
\[
\mathcal{Q}(0)^{\et}=\mathcal{Q}(0)^{\mult,C}=\mathsf{X}\;;\;\mathcal{Q}(0)^{\ab}=0,
\]
and equipped with additional level structures, which will not be essential here.

Equivalently, $\Spec R^{\sab}\times\mb{E}$ parameterizes symmetric bilinear pairings:
\[
 \tau_0:\mathsf{X}\times\mathsf{X}\to\Gm
\]
with some additional structures. In particular, if $\mb{E}_0$ is the torus, with character group $\mb{S}_0\coloneqq\on{Sym}^2\mathsf{X}$, then we have a natural finite \'etale map of tori:
\[
 \Spec R^{\sab}\times\mb{E}\to\Spec R^{\sab}\times\mb{E}_0
\]
corresponding to an injective map of character groups $\mb{S}_0\into\mb{S}$.

\subsubsection{}\label{boundary:subsubsec:noncansplit}
Let $\mb{\Xi}$ be the restriction of the $\mb{E}$-torsor $\Ss_{K^\beef_{\Phi^\beef}}(Q_{\Phi^\beef},D_{\Phi^\beef})$ to $\Spec R^{\sab}$. Since $R^{\sab}$ is strictly henselian, $\mb{\Xi}$ is trivializable. Fix an $\mb{E}$-equivariant isomorphism:
\[
\iota^{\sab}:\Spec R^{\sab}\times\mb{E}\xrightarrow{\simeq}\mb{\Xi}.
\]

In particular, by the discussion in~\eqref{boundary:subsubsec:tori_mult_1motifs}, such a trivialization equips any $\mb{\Xi}$-scheme $S$ with a natural pairing $\tau_0:\mathsf{X}\times\mathsf{X}\to\Reg{S}^\times$ corresponding to the identity section of the trivial $\mb{E}_0$-torsor.

Let $\mb{E}\hookrightarrow \mb{E}(\sigma^\beef)$ be the torus embedding associated with $\sigma^\beef$, and let $\mb{S}(\sigma^\beef)\subset \mb{S}$ be the sub-monoid such that
\[
\mb{E}(\sigma^\beef) = \Spec~\Int[\mb{S}(\sigma^\beef)].
\]
Let $\mb{\Xi}(\sigma^\beef)$ be the restriction of $\Ss_{K^\beef_{\Phi^\beef}}(Q_{\Phi^\beef},D_{\Phi^\beef},\sigma^\beef)$ to $\Spec R^{\sab}$.

Then $\iota^{\sab}$ allows us to identify $\Spec R^{\sab}\times\mb{E}(\sigma^\beef)$ with $\mb{\Xi}(\sigma^\beef)$. In particular, given a point $s_0^\beef$ as above, it gives us a point $s'_0$ in the closed orbit of $\mb{E}(\sigma^\beef)$, and an isomorphism of $R^{\sab}$-algebras:
\begin{align}\label{boundary:eqn:noncansplit}
R^{\sab}\widehat{\otimes}_WB{(\sigma^\beef,s'_0)}\xrightarrow{\simeq}R(\sigma^\beef,s_0^\beef).
\end{align}
Here, $B(\sigma^\beef,s'_0)$ is the complete local ring of $\mb{E}(\sigma^\beef)$ at $s'_0$.

\subsubsection{}\label{boundary:subsubsec:blowups}
Suppose that $\tilde{\sigma}^\beef\subset \sigma^\beef$ is another rational polyhedral cone that is not contained in a proper face of $\sigma^\beef$. Then we obtain a map of torus embeddings:
\begin{align}\label{boundary:eqn:sigmatildetoremb}
\mb{E}(\tilde{\sigma}^\beef)\to\mb{E}(\sigma^\beef).
\end{align}
Moreover, under this map, the closed orbit of the target is contained within the image of the closed orbit of the source.

Twisting~\eqref{boundary:eqn:sigmatildetoremb} by the torsor $\mb{\Xi}$ produces a birational map of $R^{\sab}$-schemes:
\begin{align}\label{boundary:eqn:sigmatildecorr}
\mb{\Xi}(\tilde{\sigma}^\beef)\to\mb{\Xi}(\sigma^\beef)
\end{align}
such that the closed stratum $\mathcal{Z}(\sigma^\beef)\subset\mb{\Xi}(\sigma^\beef)$ is contained within the image of $\mathcal{Z}(\tilde{\sigma}^\beef)$.

If $\tilde{s}^\beef_0\in\mathcal{Z}(\tilde{\sigma}^\beef)(\overline{\Field}_p)$ is a point lifting $s_0^\beef\in\mathcal{Z}(\sigma^\beef)(\overline{\Field}_p)$
Then (\ref{boundary:eqn:sigmatildecorr}) induces a birational map of $R^{\sab}$-schemes
\begin{equation}\label{boundary:eqn:blowupab}
f_{\tilde{s}^\beef_0,s_0^\beef}:S(\tilde{\sigma}^\beef,\tilde{s}^\beef_0)\to S(\sigma^\beef,s_0^\beef),
\end{equation}
carrying $V(\tilde{\sigma}^\beef,\tilde{s}^\beef_0)$ into $V(\sigma^\beef,s_0^\beef)$. This map is compatible with the splittings~\eqref{boundary:eqn:noncansplit} induced by the secion $\iota^{\sab}$.

\subsubsection{}\label{boundary:subsubsec:open}
Analytifying $f_{\tilde{s}^\beef_0,s_0^\beef}$, we obtain a map of analytic spaces
\[
f_{\tilde{s}^\beef_0,s_0^\beef}:\widehat{\mathcal{U}}^{\an,\circ}(\tilde{\sigma}^\beef,\tilde{s}^\beef_0)\to\widehat{\mathcal{U}}^{\an,\circ}(\sigma^\beef,s_0^\beef).
\]
This identifies $\widehat{\mathcal{U}}^{\an,\circ}(\widetilde{\sigma}^\beef,\tilde{s}^\beef_0)$ with an open subspace of $\widehat{\mathcal{U}}^{\an,\circ}(\sigma^\beef,s_0^\beef)$. 

To see this, note that the decomposition \eqref{boundary:eqn:noncansplit} reduces us to considering the map $(\Spf B(\tilde{\sigma},\tilde{s}'_0))^{\an}\to(\Spf B(\sigma^\beef,s'_0))^{\an}$. 

As above, we will use the superscript $\circ$ to denote the complement of the boundary divisor. Then every $h\in\mb{S}$ can be viewed as a non-vanishing global function on $(\Spf B({\sigma^\beef,s_0^\beef}))^{\an,\circ}$.

 The point $\tilde{s}'_0$ determines a map of monoids $\tilde{s}^\sharp_0:\mb{S}(\tilde{\sigma})\to\overline{\Field}_p$. Let $\beta(h)\in W^\times$ be the Teichm\"uller lift of $\tilde{s}^\sharp_0(h)\in\overline{\Field}_p^{\times}$. Then
 \[
  \bigl(\Spf B({\tilde{\sigma}^\beef,\tilde{s}^\beef_0})\bigr)^{\an,\circ}\subset\bigl(\Spf B({\sigma^\beef,s_0^\beef})\bigr)^{\an,\circ}
 \]
 is the subspace defined by the conditions $\vert h-\beta(h)\vert<1$, as $h$ varies over elements in $\mb{S}(\tilde{\sigma}^\beef)^\times\subset\mb{S}$.

\subsubsection{}\label{boundary:subsubsec:equiv}
More generally, consider the set
\begin{align}\label{boundary:eqn:equivset}
 \bigl\{(\sigma^\beef,s_0^\beef):\;\sigma^\beef\subset W_{\Phi^{\beef}}(\Real)(-1)\text{ rational polyhedral cone with $\sigma^{\beef,\circ}\subset\mb{H}(\Phi^\beef)$, and }s_0^\beef\in\mathcal{Z}(\sigma^\beef)(\overline{\Field}_p)\bigr\}.
\end{align}
We will equip it with the minimal equivalence relation $\sim$ such that $(\tilde{\sigma}^\beef,\tilde{s}^\beef_0)\sim(\sigma^\beef,s_0^\beef)$ whenever $\tilde{\sigma}^\beef\subset\sigma^\beef$ and $\tilde{s}^\beef_0$ maps to $s_0^\beef$ under the map $\mathcal{Z}(\tilde{\sigma}^\beef)\to\mathcal{Z}(\sigma^\beef)$.

Concretely, this means that $(\tilde{\sigma}^\beef,\tilde{s}^\beef_0)$ is equivalent to $(\sigma^\beef,s_0^\beef)$ if two conditions hold: the intersection $\tilde{\sigma}^{\beef,\circ}\cap\sigma^{\beef,\circ}$ is non-empty; and, with $\sigma'\subset W_{\Phi^{\beef}}(\Real)(-1)$ the closure of $\tilde{\sigma}^{\beef,\circ}\cap\sigma^{\beef,\circ}$, there is a point $s'_0\in\mathcal{Z}(\sigma')(\overline{\Field}_p)$ mapping to $\tilde{s}^\beef_0$ and $s_0^\beef$ under the maps $\mathcal{Z}(\sigma')\to\mathcal{Z}(\sigma^\beef)$ and $\mathcal{Z}(\sigma')\to\mathcal{Z}(\tilde{\sigma}^\beef)$, respectively.

By \eqref{boundary:subsubsec:open}, we can naturally view $\widehat{\mathcal{U}}^{\an,\circ}({\sigma',s'_0})$ as an open sub-space of both $\widehat{\mathcal{U}}^{\an,\circ}({\tilde{\sigma}^\beef,\tilde{s}^\beef_0})$ and $\widehat{\mathcal{U}}^{\an,\circ}({\sigma^\beef,s_0^\beef})$. We will therefore suggestively denote it as an `intersection':
\[
\widehat{\mathcal{U}}^{\an,\circ}({\tilde{\sigma}^\beef,\tilde{s}^\beef_0})\cap\widehat{\mathcal{U}}^{\an,\circ}({\sigma^\beef,s_0^\beef}) \coloneqq \widehat{\mathcal{U}}^{\an,\circ}({\sigma',s'_0}).
\]

\subsubsection{}\label{boundary:subsubsec:fcrystal}
Fix any Frobenius lift $\varphi:R^{\sab}\to R^{\sab}$ and equip $B(\sigma^\beef,s_0^\beef)$ with the Frobenius lift $\varphi_{\sigma}:B(\sigma^\beef,s_0^\beef)\to B(\sigma^\beef,s_0^\beef)$ induced by the $p$-power map on $\mb{S}$. The completed tensor product of the two lifts now determines a Frobenius lift $\varphi=\varphi\widehat{\otimes}\varphi_{\sigma}$ on $R(\sigma^\beef,s_0^\beef)$.

The section $R^{\sab}\to\mb{\Xi}$ determining $\iota^{\sab}$ gives us a principally polarized $1$-motif $(\mathcal{Q}^{\cl},\lambda^{\cl})$ over $R^{\sab}$. On the other hand, as observed in~\eqref{boundary:subsubsec:tori_mult_1motifs}, $\iota^{\sab}$ also determines a map $\mb{\Xi}\to\mb{E}$, and thus a pairing
\[
 \tau_0:\mathsf{X}\times\mathsf{X}\to\Reg{\mb{\Xi}}^\times\to\Reg{V(\sigma^\beef,s_0^\beef)}^\times.
\]
Equivalently, it determines an object $\mathcal{Q}(0)$ in $\mb{DD}(S(\sigma^\beef,s_0^\beef),V(\sigma^\beef,s_0^\beef))$ with
\[
 \mathcal{Q}(0)^{\et}=\mathcal{Q}(0)^{\mult}=\mathsf{X}\;;\;\mathcal{Q}(0)^{\ab}=0.
\]

The pair $(\mathcal{Q}^{\cl},\tau_{0})$ is now an object in $\widetilde{\mb{DD}}(S(\sigma^\beef,s_0^\beef),V(\sigma^\beef,s_0^\beef))$; cf.~\eqref{semistable:subsubsec:ddpoltilde}. In the notation of \emph{loc. cit.}, we have a canonical isomorphism:
\[
 \on{B}_{(S({\sigma^\beef,s_0^\beef}),V({\sigma^\beef,s_0^\beef}))}\bigl((\mathcal{Q}^{\cl},\tau_{0})\bigr)\xrightarrow{\simeq}\mathcal{Q}({\sigma^\beef,s_0^\beef}).
\]

Set $\mathcal{M}({\sigma^\beef,s_0^\beef})=\Dieu(\mathcal{A}({\sigma^\beef,s_0^\beef}))(S({\sigma^\beef,s_0^\beef}))$ and $\mathcal{M}^{\cl}({\sigma^\beef,s_0^\beef})=\Dieu(\mathcal{Q}^{\cl})(S({\sigma^\beef,s_0^\beef}))$: these are both finite free $\Reg{S({\sigma^\beef,s_0^\beef})}$-modules equipped with a $\varphi$-module structure and compatible integrable logarithmic connections. They are also equipped with Hodge filtrations $F^\bullet\mathcal{M}({\sigma^\beef,s_0^\beef})$, $F^\bullet\mathcal{M}^{\cl}({\sigma^\beef,s_0^\beef})$; and weight filtrations $W_\bullet\mathcal{M}({\sigma^\beef,s_0^\beef})$, $W_\bullet\mathcal{M}^{\cl}({\sigma^\beef,s_0^\beef})$.

We have canonical isomorphisms:
\begin{align}\label{boundary:eqn:weightfilt}
 \gr^W_2\mathcal{M}({\sigma^\beef,s_0^\beef})\xrightarrow{\simeq}\gr^W_2\mathcal{M}^{\cl}({\sigma^\beef,s_0^\beef})&\xrightarrow{\simeq}\Reg{S(\sigma^\beef,s_0^\beef)}(-1)\otimes \mathsf{X};\nonumber\\ W_0\mathcal{M}({\sigma^\beef,s_0^\beef})\xrightarrow{\simeq}W_0\mathcal{M}^{\cl}({\sigma^\beef,s_0^\beef})&\xrightarrow{\simeq}\SHom(\mathsf{X},\Reg{S(\sigma^\beef,s_0^\beef)});\\
 \gr^W_1\mathcal{M}({\sigma^\beef,s_0^\beef})\xrightarrow{\simeq}\gr^W_1\mathcal{M}^{\cl}({\sigma^\beef,s_0^\beef})&\xrightarrow{\simeq}\Dieu(\mathcal{B})(S(\sigma^\beef,s_0^\beef)).\nonumber
\end{align}

Let $\Dieu(0)$ be the extension of $W_1\Dieu(\mathcal{Q}^{\cl})$ by $\mb{1}(-1)\otimes\mathsf{X}$ in $\mb{LDieu}(S(\sigma^\beef,s_0^\beef),V(\sigma^\beef,s_0^\beef))$ obtained by pushing $\Dieu(\mathcal{Q}(0))$ out along the inclusion:
\[
 \underline{\Hom}(\mb{1},\mathsf{X}) = W_0\Dieu(\mathcal{Q}^{\cl})\into W_1\Dieu(\mathcal{Q}^{\cl}).
\]

From the construction of $\Dieu(\mathcal{A}({\sigma^\beef,s_0^\beef}))$ in (\ref{semistable:subsubsec:logffull}) from $\Dieu(\mathcal{Q}^{\cl})$ and the pairing $\tau_{0}$, we see that $\Dieu(\mathcal{A}(\sigma^\beef,s_0^\beef))$ is canonically isomorphic to the Baer sum of $\Dieu(\mathcal{Q}^{\cl})$ and $\Dieu(0)$ as extensions of $W_1\Dieu(\mathcal{Q}^{\cl})$ by $\mb{1}(-1)\otimes\mathsf{X}$.

From the explicit description of $\Dieu(0)$ in~\eqref{semistable:subsubsec:noabcrys}, we deduce that, as a $\varphi$-module over $R(\sigma^\beef,s_0^\beef)$, $\Dieu(0)(S(\sigma^\beef,s_0^\beef))$ is naturally isomorphic to the \emph{split} extension of $\Reg{S(\sigma^\beef,s_0^\beef)}(-1)\otimes\mathsf{X}$ by $\SHom(\mathsf{X},\Reg{S})$.

Therefore, we have an isomorphism of $\varphi$-modules $\mathcal{M}^{\cl}({\sigma^\beef,s_0^\beef})\xrightarrow{\simeq}\mathcal{M}({\sigma^\beef,s_0^\beef})$ preserving Hodge and weight filtrations and inducing the canonical isomorphisms from (\ref{boundary:eqn:weightfilt}) on the associated graded pieces of the weight filtrations.

Furthermore, let $\nabla^{\cl}:\mathcal{M}({\sigma^\beef,s_0^\beef})\to\mathcal{M}({\sigma^\beef,s_0^\beef})\otimes\flogdiff{S({\sigma^\beef,s_0^\beef})/W}$ be the integrable connection induced via this isomorphism from that on $\mathcal{M}^{\cl}({\sigma^\beef,s_0^\beef})$. If $\nabla$ is the natural connection on $\mathcal{M}({\sigma^\beef,s_0^\beef})$, the difference
\begin{align}\label{boundary:eqn:theta_difference}
\Theta=\nabla-\nabla^{\cl}:\mathcal{M}({\sigma^\beef,s_0^\beef})\to\mathcal{M}({\sigma^\beef,s_0^\beef})\otimes\flogdiff{R({\sigma^\beef,s_0^\beef})/W}
\end{align}
is an $R({\sigma^\beef,s_0^\beef})$-linear map. It factors through:
\[
\gr^W_2\mathcal{M}({\sigma^\beef,s_0^\beef})=\Reg{S({\sigma^\beef,s_0^\beef})}(-1)\otimes \mathsf{X}\to\SHom(\mathsf{X},\flogdiff{S({\sigma^\beef,s_0^\beef})/W})=W_0\mathcal{M}({\sigma^\beef,s_0^\beef})\otimes\flogdiff{S({\sigma^\beef,s_0^\beef})/W},
\]
where the map in the middle $\Reg{S({\sigma^\beef,s_0^\beef})}(-1)\otimes \mathsf{X}\to\SHom(\mathsf{X},\flogdiff{S({\sigma^\beef,s_0^\beef})/W})$ is the one induced by the pairing:
\begin{align*}
  \mathsf{X}\times \mathsf{X}&\to\flogdiff{R({\sigma^\beef,s_0^\beef})/W};\\
  (x_1,x_2)&\mapsto-\dlog(\tau_{0}(x_1,x_2)).
\end{align*}

\subsubsection{}\label{boundary:subsubsec:manlogtriv}
We will now apply the theory of \eqref{semistable:subsec:unip}, especially \eqref{semistable:subsubsec:uniplogsmooth}. Let $a\in B({\sigma,s'_0})$ be an equation for the boundary divisor $S({\sigma^\beef,s_0^\beef})\backslash V({\sigma^\beef,s_0^\beef})$. Let $M({\sigma^\beef,s_0^\beef})=M_0(\mathcal{A}({\sigma^\beef,s_0^\beef}))$ be the module of unipotent nearby cycles.

Now,
\[
R({\sigma^\beef,s_0^\beef})[a^{-1}]^\times/R({\sigma^\beef,s_0^\beef})^\times\xrightarrow{\simeq}B({\sigma^\beef,s_0^\beef})[a^{-1}]^\times/B({\sigma^\beef,s_0^\beef})^\times\xrightarrow{\simeq}\mb{S}/\mb{S}(\sigma)^\times.
\]
Set $\Lambda(\sigma)=\mb{S}/\mb{S}(\sigma)^\times$. Then, in the notation of (\ref{semistable:subsubsec:unipequiv}), $M({\sigma^\beef,s_0^\beef})$ is an object in $\mb{LFI}(\overline{\Field}_p,\Lambda(\sigma))$.

Since the monodromy operator
 \[
  N({\sigma^\beef,s_0^\beef}):M({\sigma^\beef,s_0^\beef})\to M({\sigma^\beef,s_0^\beef})\otimes\Lambda(\sigma)
 \]
 is just the residue of the connection $\nabla$, it can be described explicitly using information from \eqref{boundary:subsubsec:fcrystal}: $M({\sigma^\beef,s_0^\beef})$ is equipped with a weight filtration $W_\bullet M({\sigma^\beef,s_0^\beef})$ with $W_0M({\sigma^\beef,s_0^\beef})=\Hom(\mathsf{X},K_0)$ and $\gr^W_2M({\sigma^\beef,s_0^\beef})=K_0(-1)\otimes \mathsf{X}$. Then $N({\sigma^\beef,s_0^\beef})$ is induced by a natural map $\mathsf{X}\to\Hom(\mathsf{X},\Lambda(\sigma))$, or, equivalently, a natural pairing
\[
 \mathsf{X}\times \mathsf{X}\to\Lambda(\sigma).
\]
This is exactly the symmetric pairing induced from the map $\mb{S}_0\to\mb{S}\to\mb{S}/\mb{S}(\sigma)^\times=\Lambda(\sigma)$.

\subsubsection{}\label{boundary:subsubsec:blowupslogfcrys}
Let $\tilde{\sigma}^\beef\subset\sigma^\beef\subset\mb{H}(\Phi^\beef)$ and $\tilde{s}^\beef_0\in\mathcal{Z}(\tilde{\sigma}^\beef)(\overline{\Field}_p)$ be as in (\ref{boundary:subsubsec:blowups}).

Then~\eqref{boundary:eqn:blowupab} produces a canonical identification in $\mb{LDieu}_{\wt}(S{(\tilde{\sigma}^\beef,\tilde{s}^\beef_0)},V{(\tilde{\sigma}^\beef,\tilde{s}^\beef_0)})$:
\begin{align}\label{boundary:eqn:mtildesigma}
 f_{\tilde{s}^\beef_0,s_0^\beef}^*\mathcal{M}({\sigma^\beef,s_0^\beef})=\mathcal{M}{(\tilde{\sigma}^\beef,\tilde{s}^\beef_0)}.
\end{align}

The map $f_{\tilde{s}^\beef_0,s_0^\beef}$ also produces a functor
\[
f_{\tilde{s}^\beef_0,s_0^\beef}^*:\mb{LFI}(\overline{\Field}_p,\Lambda(\sigma^\beef))\to\mb{LFI}(\overline{\Field}_p,\Lambda(\tilde{\sigma}^\beef)).
\]
Concretely, given a tuple $(D,\varphi_D,N_D)$ on the left hand side, we have $f_{\tilde{s}^\beef_0,s_0^\beef}^*(D,\varphi_D,N_D)=(D,\varphi_D,f_{\tilde{s}^\beef_0,s_0^\beef}^*N_D)$, where $f_{\tilde{s}^\beef_0,s_0^\beef}^*N_D$ is the composition:
\[
 D\xrightarrow{N_D}D\otimes\Lambda(\sigma)\to D\otimes\Lambda(\tilde{\sigma}^\beef).
\]

Now (\ref{boundary:eqn:mtildesigma}) gives us a canonical identification in $\mb{LFI}(\overline{\Field}_p,\Lambda(\tilde{\sigma}))$: $f_{\tilde{s}^\beef_0,s_0^\beef}^*M({\sigma^\beef,s_0^\beef})=M{(\tilde{\sigma}^\beef,\tilde{s}^\beef_0)}$.

Given $(\sigma^\beef,s_0^\beef)$ in the set \eqref{boundary:eqn:equivset}, we find from above that the $\varphi$-module over $K_0$ underlying $M({\sigma^\beef,s_0^\beef})$ depends only on the equivalence class of $(\sigma^\beef,s_0^\beef)$.

\emph{Fix such an equivalence class}. Write $M_0$ for the associated $\varphi$-module over $K_0$. Given a particular representative $(\sigma^\beef,s_0^\beef)$ in this equivalence class, there is a map $N({\sigma^\beef,s_0^\beef}):M_0\to M_0\otimes\Lambda(\sigma^\beef)$ giving $M_0$ the structure of an object $M({\sigma^\beef,s_0^\beef})$ in $\mb{LFI}(\overline{\Field}_p,\Lambda(\sigma^\beef))$.

\subsubsection{}\label{boundary:subsubsec:liewp_embedding}
Notice that the principal polarization on $\mathcal{A}(\sigma^\beef,s_0^\beef)$ induces a non-degenerate alternating pairing:
\[
 \psi(\sigma^\beef,s_0^\beef):\mathcal{M}(\sigma^\beef,s_0^\beef)\times\mathcal{M}(\sigma^\beef,s_0^\beef)\to\Reg{S(\sigma^\beef,s_0^\beef)}(-1),
\]
which in turn produces a symplectic pairing:
\[
\psi_0:M_0\times M_0\to K_0(-1).
\]
As above, it is easily checked that this pairing does not depend on the choice of $(\sigma^\beef,s_0^\beef)$ in its equivalence class.

Let $P_{\wt}\subset\GSp(M_0,\psi_0)$ be the parabolic subgroup preserving the weight filtration, and let $U_{\wt}\subset P_{\wt}$ be the center of its unipotent radical. Then $\Lie U_{\wt}\subset\End(M_0)$ consists of endomorphisms that factor as:
\[
M_0\to\gr^W_2M_0 = K_0(-1)\otimes\mathsf{X}\xrightarrow{f}\Hom(\mathsf{X},K_0)=W_0M_0\into M_0,
\]
where $f:K_0(-1)\otimes\mathsf{X}\to\Hom(\mathsf{X},K_0)$ is induced from a symmetric bilinear pairing $\mathsf{X}\times\mathsf{X}\to K_0(1)$.

In particular, $(\Lie U_{\wt})(-1)=K_0(-1)\otimes_{K_0}\Lie U_{\wt}$ admits a natural rational structure: The space of symmetric bilinear pairings on $\mathsf{X}$ with values in $\Rat$. But this in turn is canonically identified with the Lie algebra $\Lie W_{\Phi^{\beef}}$.

Putting this all together, we find that we have a canonical isomorphism:
\begin{equation}\label{boundary:eqn:liewp_embedding}
K_0\otimes_{\Rat}\Lie W_{\Phi^{\beef}}\xrightarrow{\simeq}(\Lie U_{\wt})(-1).
\end{equation}
This identifies $\Rat_p\otimes\Lie W_{\Phi^{\beef}}$ with the space of $\varphi$-invariants on the right-hand side.

\subsubsection{}\label{boundary:subsubsec:padichodge}
Let $\Reg{\widehat{\mathcal{U}}^{\an}({\sigma^\beef,s_0^\beef})}^{\log}$ be the $\Reg{\widehat{\mathcal{U}}^{\an}({\sigma^\beef,s_0^\beef})}$-algebra defined in \eqref{semistable:subsubsec:osanlog}: It has the structure of an ind-object over $\mb{LFI}(S({\sigma^\beef,s_0^\beef}),V({\sigma^\beef,s_0^\beef}))$. There is a canonical isomorphism of (ind-)log $F$-isocrystals~\eqref{semistable:eqn:paralleltrans}:
\begin{align}\label{boundary:eqn:parallel}
 \xi_{\sigma^\beef,s_0^\beef}:\Reg{\widehat{\mathcal{U}}({\sigma^\beef,s_0^\beef})^{\an}}^{\log}\otimes_{K_0}M_0\xrightarrow{\simeq}\mathcal{M}({\sigma^\beef,s_0^\beef})^{\an,\log}\coloneqq\Reg{\widehat{\mathcal{U}}({\sigma^\beef,s_0^\beef})^{\an}}^{\log}\otimes_{\Reg{\widehat{\mathcal{U}}^{\an}({\sigma^\beef,s_0^\beef})}}\mathcal{M}^{\an}({\sigma^\beef,s_0^\beef}).
\end{align}

Let $\widehat{\mathcal{U}}^{\an,\circ}({\sigma^\beef,s_0^\beef})\subset\widehat{\mathcal{U}}^{\an}({\sigma^\beef,s_0^\beef})$ be the complement of the boundary divisor. There is a canonical horizontal isomorphism of vector bundles over $\widehat{\mathcal{U}}^{\an,\circ}({\sigma^\beef,s_0^\beef})$ with integrable connections:
\begin{align}\label{boundary:eqn:manderham}
 \mathcal{M}^{\an}({\sigma^\beef,s_0^\beef})\vert_{\widehat{\mathcal{U}}^{\an,\circ}({\sigma^\beef,s_0^\beef})}\xrightarrow{\simeq}\underline{H}^1_{\dR}(\mathcal{A}({\sigma^\beef,s_0^\beef})/V({\sigma^\beef,s_0^\beef}))\vert_{\widehat{\mathcal{U}}^{\an,\circ}({\sigma^\beef,s_0^\beef})}
\end{align}
This respects the Hodge filtrations on both sides.

Fix an algebraic closure $\overline{K}_0$ for $K_0$. For any finite extension $L/K_0$ within $\overline{K}_0$ and any point $x\in\widehat{\mathcal{U}}^{\an,\circ}(L)$, write $\mathcal{A}_x$ for the fiber of $\mathcal{A}({\sigma^\beef,s_0^\beef})$ at $x$.

A choice of uniformizer $\pi\in L$ allows us to define a specialization map $x_{\pi}$ on $\Reg{\widehat{\mathcal{U}}^{\an}({\sigma^\beef,s_0^\beef})}^{\log}$; cf.~\eqref{semistable:subsubsec:hyodokatocomp}. Specializing (\ref{boundary:eqn:parallel}) along $x_{\pi}$ and using (\ref{boundary:eqn:manderham}) gives us the Hyodo-Kato isomorphism:
\begin{align}\label{boundary:eqn:hyodokatox}
 \beta_{\on{H-K},x,\pi}:L\otimes_{K_0}M_0\xrightarrow{\simeq}H^1_{\dR}(\mathcal{A}_{x}/L).
\end{align}

Furthermore, from (\ref{semistable:prp:comparison}) and (\ref{semistable:prp:phinmodule}), we obtain a natural comparison isomorphism
\begin{align}\label{boundary:eqn:comp}
\beta_{\st,x}:\Bst\otimes_{\Rat_p}H^1_{\et}\bigl(\mathcal{A}_{x,\overline{K}_0},\Rat_p\bigr)\xrightarrow{\simeq}\Bst\otimes_{K_0}M_0.
\end{align}
This isomorphism is compatible with the actions of $\Gamma_L=\Gal(\overline{K}_0/L)$ and $\varphi$, where $\Gamma_L$ (resp $\varphi$) acts via its diagonal action on the left hand side (resp. right hand side) and via its action on $\Bst$ on the right hand side (resp. left hand side). If we equip $M_0$ with the monodromy operator
\[
 N_{x}:M_0\xrightarrow{N({\sigma^\beef,s_0^\beef})}M_0\otimes\Lambda(\sigma^\beef)\xrightarrow{1\otimes x^\sharp} M_0\otimes(\overline{K}_0^\times/\Reg{\overline{K}_0}^\times),
\]
then (\ref{boundary:eqn:comp}) is also compatible with monodromy operators on both sides. Again, here, we are equipping the right hand side with the operator $N\otimes 1+1\otimes N_{x}$ and the left hand side with the operator $N\otimes 1$, where $N:\Bst\to\Bst\otimes(\overline{K}_0^\times/\Reg{\overline{K}_0}^\times)$ is the canonical monodromy operator from \cite[3.2.2]{fontaine:periodes}.

Using the embedding $L\otimes_{K_0}\Bst\into\Bdr$ induced by the choice of uniformizer $\pi$, we obtain the composition
\begin{align}\label{boundary:eqn:hkst}
\beta_{\dR,x}:\Bdr\otimes_{\Rat_p}H^1_{\et}\bigl(\mathcal{A}_{x,\overline{K}_0},\Rat_p\bigr)\xrightarrow{1\otimes\beta_{\st,x}}\Bdr\otimes_L(L\otimes_{K_0}M_0)\xrightarrow[\simeq]{1\otimes\beta_{\on{H-K},x,\pi}}\Bdr\otimes_LH^1_{\dR}(\mathcal{A}_{x}/L),
\end{align}
which is just the canonical de Rham comparison isomorphism.

\subsubsection{}\label{boundary:subsubsec:m0_classical}
Let $\mathcal{Q}^{\cl}$ be the $1$-motif over $R^{\sab}$ obtained from the splitting~\eqref{boundary:eqn:noncansplit}, and let $\mathcal{M}^{\cl,\an}$ be the corresponding $F$-isocrystal over $\widehat{\mathcal{U}}^{\sab,\an}$. Let $M_0^{\cl}$ be the corresponding module of unipotent nearby cycles over $K_0$. Then we have a canonical isomorphism of $F$-isocrystals~\eqref{semistable:eqn:paralleltrans}:
\[
 \xi^{\cl}:\Reg{\widehat{\mathcal{U}}^{\sab,\an}}\otimes_{K_0}M_0^{\cl}\xrightarrow{\simeq}\mathcal{M}^{\cl,\an}.
\]

Given any pair $(\sigma^\beef,s_0^\beef)$ in our fixed equivalence class, we now have a canonical identification of $\varphi$-modules:
\begin{equation}\label{boundary:eqn:m0cl_isomorphism}
 \Reg{\widehat{\mathcal{U}}^{\an}(\sigma^\beef,s_0^\beef)}\otimes_{K_0}M_0^{\cl}\xrightarrow{\simeq}\mathcal{M}^{\an}(\sigma^\beef,s_0^\beef)
\end{equation}
preserving Hodge and weight filtrations and inducing the canonical isomorphisms from~\eqref{boundary:eqn:weightfilt} on the associated graded pieces. The induced connection on the right hand side differs from the natural one by the linear map $\Theta$ defined in~\eqref{boundary:eqn:theta_difference}. Therefore,~\eqref{boundary:eqn:m0cl_isomorphism} gives us an identification:
\begin{equation}\label{boundary:eqn:m0cl_description}
M_0^{\cl} = \{x\in H^0\bigl(\widehat{\mathcal{U}}^{\an}(\sigma^\beef,s_0^\beef),\mathcal{M}^{\an}(\sigma^\beef,s_0^\beef)\bigr):\;\nabla(x) = \Theta(x)\}.
\end{equation}

By construction, we have a map of groups $R(\sigma^\beef,s_0^\beef)[a^{-1}]^\times\xrightarrow{u\mapsto\ell_u}\Reg{\widehat{\mathcal{U}}^{\an}(\sigma^\beef,s_0^\beef)}^{\log}$ carrying any element of $u\in R(\sigma^\beef,s_0^\beef)^\times$ to $\log(u)$. Consider the map
\[
\mb{S}_0\into\mb{S}\xrightarrow{m\mapsto\ell_m}\Reg{\widehat{\mathcal{U}}^{\an}(\sigma^\beef,s_0^\beef)}^{\log},
\]
where we view $\mb{S}$ as a subgroup of $B(\sigma^\beef,s_0^\beef)[a^{-1}]^\times\subset R(\sigma^\beef,s_0^\beef)[a^{-1}]^\times$. This gives us a $\varphi$-equivariant symmetric pairing
\[
 \mathsf{X}\times\mathsf{X}\to\Reg{\widehat{\mathcal{U}}^{\an}(\sigma^\beef,s_0^\beef)}^{\log}(1),
\]
and thus a $\varphi$-equivariant map
\[
 \Reg{\widehat{\mathcal{U}}^{\an}(\sigma^\beef,s_0^\beef)}^{\log}(1)\otimes\mathsf{X}\to \SHom(\mathsf{X},\Reg{\widehat{\mathcal{U}}^{\an}(\sigma^\beef,s_0^\beef)}^{\log}),
\]
which we can view as a $\varphi$-equivariant endomorphism $\mathsf{A}\in \Reg{\widehat{\mathcal{U}}^{\an}(\sigma^\beef,s_0^\beef)}^{\log}\otimes_{K_0}\Lie U_{\wt}$. Comparing with the construction of $\Theta$ in~\eqref{boundary:eqn:theta_difference}, we find
\[
d\mathsf{A} = -\Theta\in \Omega^{1,\log}_{\widehat{\mathcal{U}}^{\an}(\sigma^\beef,s_0^\beef)/K_0}\otimes\Lie U_{\wt}.
\]

From this, it is easy to see that the $\varphi$-equivariant automorphism $\exp(\mathsf{A})$ of $\mathcal{M}^{\an,\log}(\sigma^\beef,s_0^\beef)$ induces a $\varphi$-equivariant isomorphism $M_0^{\cl}\xrightarrow{\simeq}M_0$. Tensoring with $\Reg{\widehat{\mathcal{U}}^{\sab,\an}}$ produces an isomorphism:
\begin{align}\label{boundary:eqn:exp_A}
 \exp(\mathsf{A}):\mathcal{M}^{\cl,\an}\xrightarrow{\simeq}\Reg{\widehat{\mathcal{U}}^{\sab,\an}}\otimes_{K_0}M_0^{\cl}\xrightarrow{\simeq}\Reg{\widehat{\mathcal{U}}^{\sab,\an}}\otimes_{K_0}M_0.
\end{align}
The image of the Hodge filtration $F^\bullet\mathcal{M}^{\cl,\an}$ under this isomorphism produces a filtration $F^\bullet_{\cl}(\Reg{\widehat{\mathcal{U}}^{\sab,\an}}\otimes_{K_0}M_0)$ on the target.

\subsubsection{}\label{boundary:subsubsec:good_trivialization}
For future use, it will be useful to have a different realization of the isomorphism~\eqref{boundary:eqn:liewp_embedding}. Let $\mathcal{Q}_x$ be the $1$-motif over $L$ obtained by pulling $\mathcal{Q}(\sigma^\beef,s_0^\beef)$ along $x$: It is equipped with a principal polarization $\lambda_x$. Since $x$ is also a point of $\Ss_{K^\beef_{\Phi^\beef}}(Q_{\Phi^\beef},D_{\Phi^\beef})$, $(\mathcal{Q}_x,\lambda_x)$ is equipped with a $K^\beef_{\Phi^\beef}$-level structure. In particular, if $T_p(\mathcal{Q}_x)$ is the $p$-adic Tate module of $\mathcal{Q}_x$, this means that we can find an isomorphism:
\[
 \eta_p:H^g(\Int_p)\xrightarrow{\simeq}T_p(\mathcal{Q}_x),
\]
and a trivialization $u_p:\Int_p\xrightarrow{\simeq}\Int_p(1)$ over $\overline{K}_0$ (equivalently, a compatible family of $p$-power roots of unity in $\overline{K}_0^\times$) with the following properties:
\begin{itemize}
\item $\eta_p$ preserves $W_{\bullet}$ filtrations;
\item The map
\[
 \Int_p\otimes\mathsf{X}=\gr^W_0H^g(\Int_p)\xrightarrow[\simeq]{\gr^W_0\eta_p}\gr^W_0T_p(\mathcal{Q}_x)=\Int_p\otimes\mathsf{X}
\]
is the identity;
\item Via $\eta_p$, the Weil pairing $e_{\lambda_x}$ on $T_p(\mathcal{Q}_x)$ pulls back to $u_p\circ\psi^g$.
\end{itemize}

Using the canonical identifications
\[
H^1_{\et}\bigl(\mathcal{A}_{x,\overline{K}_0},\Int_p\bigr) = \dual{T_p(\mathcal{A}_x)} = \dual{T_p(\mathcal{Q}_x)},
\]
we now obtain an isomorphism:
\[
 \alpha=\beta_{\st,x}\circ(\dual{\eta}_p)^{-1}:\Bst\otimes\dual{H}(\Rat)\xrightarrow{\simeq}\Bst\otimes_{K_0}M_0,
\]
which preserves $W_{\bullet}$-filtrations and polarization pairings (up to $\Int_p^\times$-multiples). Moreover, let $u_0:K_0(-1)\xrightarrow{\simeq}K_0$ be the isomorphism obtained by applying the Dieudonn\'e module functor to the trivialization $u$. Then we have commuting diagrams:
\begin{align}\label{boundary:eqn:trivializations_commute}
\begin{diagram}
\Bst\otimes W_0\dual{H}(\Rat)&\rEquals&\Hom(\mathsf{X},\Bst)\\
\dTo_{\simeq}^{W_0\alpha}&&\dEquals\\
\Bst\otimes W_0M_0&\rEquals&\Hom(\mathsf{X},\Bst)
\end{diagram}\;&;\;
\begin{diagram}
\Bst\otimes\gr^W_2\dual{H}(\Rat)&\rEquals&\Bst\otimes\mathsf{X}\\
\dTo_{\simeq}^{\gr^W_2\alpha}&&\dTo^{\simeq}_{u^{-1}_0\otimes 1}\\
\Bst\otimes\gr^W_2M_0&\rEquals&\Bst(-1)\otimes\mathsf{X}
\end{diagram}.
\end{align}
Here, all equalities are canonical identifications.

Since $\alpha$ preserves $W_{\bullet}$-filtrations and polarization pairings, it induces an isomorphism:
\begin{equation}\label{boundary:eqn:liewp_embedding_bst}
 \Bst\otimes_{\Rat}\Lie W_{\Phi^{\beef}}\xrightarrow{\simeq}\Bst\otimes_{K_0}\Lie U_{\wt}.
\end{equation}

\begin{lem}\label{boundary:lem:liewp_embedding}
The embedding~\eqref{boundary:eqn:liewp_embedding_bst} restricts to an isomorphism
\[
 K_0\otimes_{\Rat}\Lie W_{\Phi^{\beef}}\xrightarrow{\simeq}(\Lie U_{\wt})(-1)
\]
which is precisely the one defined in~\eqref{boundary:eqn:liewp_embedding}.
\end{lem}
\begin{proof}
This is immediate from the commuting diagrams~\eqref{boundary:eqn:trivializations_commute}. Note that the inclusion $K_0(-1)\subset\Bst$ is the natural one generated by Fontaine's cyclotomic period $t$.
\end{proof}

\subsubsection{}\label{boundary:subsubsec:Nx_sigma}
Consider the monodromy operator
\[
 N_x:M_0\to M_0\otimes(\overline{K}_0^\times/\Reg{\overline{K}_0}^\times).
\]

Equip $M_0\otimes(\overline{K}_0^\times/\Reg{\overline{K}_0}^\times)$ with the $\varphi$-module structure given by the $K_0$-semi-linear endomorphism:
\[
 m\otimes a\mapsto \varphi_0(m)\otimes a^p.
\]

Then we have a $\varphi$-equivariant isomorphism:
\[
 \nu_p:M_0\otimes(\overline{K}_0^\times/\Reg{\overline{K}_0}^\times)\xrightarrow[\simeq]{1\otimes\nu_p}M_0\otimes_{K_0} K_0(-1)=M_0(-1).
\]

\begin{lem}\label{boundary:lem:Nx_sigma}
The map $\nu_p\circ N_x:M_0\to M_0(-1)$ is $\varphi$-equivariant, and in fact belongs to
\[
(\sigma^{\beef,\circ}\cap\Lie W_{\Phi^{\beef}})\subset K_0(-1)\otimes_{K_0}\Lie U_{\wt}\subset\End(M_0)(-1).
\]
\end{lem}
\begin{proof}
Indeed, $\nu_p\circ N_x$ is associated with the symmetric pairing
\[
 \mathsf{X}\times\mathsf{X}\to\Lambda(\sigma^\beef)\xrightarrow{x^\sharp}\overline{K}_0^\times/\Reg{\overline{K}_0}^\times\xrightarrow{\nu_p}\Rat.
\]
\end{proof}

\subsection{Canonical Tate tensors}\label{boundary:subsec:tatetensors}

Fix a place $v\vert p$ of $E$. Let $\Ss_K$ be the normalization of $\mathcal{S}_{{K}^{\beef}}$ in $\Sh_K$. In this sub-section, we will investigate the intersection of ${\Ss_K}$ with the formal boundary charts constructed above. In particular, we will show that every point in the special fiber of such an intersection carries a canonical family of crystalline `Tate tensors'. 

Away from the boundary, this follows quite readily from the argument used in the proof of~\cite[(2.3.5)]{kisin:abelian}, using the fact that the Hodge tensors on the abelian scheme $\mathcal{A}\to\Sh_K$ are parallel for the Gauss-Manin connection. The same idea works at the boundary as well, but its correct execution turns out to be a bit more intricate. 

The main additional difficulty is that, by the theory of~\eqref{semistable:subsubsec:unip}, elements of the space of unipotent nearby cycles propagate to parallel tensors over $\mathcal{M}^{\an,\log}(\sigma^\beef,\tilde{s}^\beef_0)$, and not over $\mathcal{M}^{\an}(\sigma^\beef,\tilde{s}^\beef_0)$.

The key thing is to show that, for the tensors arising from the $\{s_\alpha\}$ via the $p$-adic comparison isomorphism, their restrictions over $\widehat{U}^{\an,\circ}_G(\sigma^\beef,t_0)$ are in fact sections of $\mathcal{M}^{\an}(\sigma^\beef,s_0^\beef)$, and that these sections can therefore be identified with the de Rham realizations of the $\{s_{\alpha}\}$ by checking at any one point of the space. The main technical input is~\eqref{boundary:lem:parallellem}.

\subsubsection{}\label{boundary:subsubsec:tsigmat0}
Assume that we are given $(\sigma^\beef,s_0^\beef)$ in the set \eqref{boundary:eqn:equivset} such that
\[
 {\Ss_K}\times_{\mathcal{S}_{{K}^{\beef}},i{(\Phi,\sigma,s_0^\beef)}}V({\sigma^\beef,s_0^\beef})\neq\emptyset.
\]
This fiber product is finite over $V(\sigma^\beef,s^\beef_0)$, and the normalization in this finite cover of $S({\sigma^\beef,s_0^\beef})$ is of the form $\Spec R_G({\sigma^\beef,s_0^\beef})$, where $R_G({\sigma^\beef,s_0^\beef})$ is finite over the semi-local ring $\Reg{E,(v)}\otimes_{\Int_{(p)}}R({\sigma^\beef,s_0^\beef})$.

Choose a point $t_0\in(\Spec R_G({\sigma^\beef,s_0^\beef}))(\overline{\Field}_p)$, and let $R_G({\sigma^\beef,t_0})$ be the corresponding local quotient of $R_G({\sigma^\beef,s_0^\beef})$. Then $R_G({\sigma^\beef,t_0})$ is normal and finite over the local ring $\Reg{F}\otimes_WR({\sigma^\beef,s_0^\beef})$, where $F$ is a finite extension of $K_0$ within $\overline{K}_0$ appearing as a factor of the algebra $E\otimes_{\Int_{(p)}}W$.

\subsubsection{}
Let $\widehat{\mathcal{U}}^{\an}_G({\sigma^\beef,t_0})$ be the rigid analytic space over $F$ attached to the formal $\Reg{F}$-scheme $\widehat{\mathcal{U}}_G({\sigma^\beef,t_0})\coloneqq \Spf R_G({\sigma^\beef,t_0})$. Then we have a finite map of normal, irreducible analytic spaces over $F$,
\[
u_{\sigma^\beef,t_0}:\widehat{\mathcal{U}}^{\an}_G({\sigma^\beef,t_0})\to\widehat{\mathcal{U}}^{\an}(\sigma^\beef,s_0^\beef)_F.
\]
Let $\widehat{\mathcal{U}}^{\an,\circ}_G({\sigma^\beef,t_0})\subset\widehat{\mathcal{U}}^{\an}_G({\sigma^\beef,t_0})$ be the pre-image of $\widehat{\mathcal{U}}^{\an,\circ}(\sigma^\beef,s_0^\beef)_F$.

Let $\overline{\Sh}_K$ be the generic fiber of $\overline{\Ss}_K$, and let $\overline{\Sh}_{K,F}^{\an}$ be the associated rigid analytic space over $F$: It contains the rigid analytic space $\Sh_{K,F}^{\an}$ as a dense open. Then $\widehat{\mathcal{U}}^{\an}_G({\sigma^\beef,t_0})$ can be identified with an open sub-space in $\overline{\Sh}_{K,F}^{\an}$, and $\widehat{\mathcal{U}}^{\an,\circ}_G({\sigma^\beef,t_0})\subset\widehat{\mathcal{U}}^{\an}_G({\sigma^\beef,t_0})$ is the intersection of $\widehat{\mathcal{U}}^{\an}_G({\sigma^\beef,t_0})$ with $\Sh_{K,F}^{\an}$. 

% Since $\Sh_K$ is a normal, closed sub-scheme of $\Sh_{{K}^{\beef}}\otimes E$, this implies that the map $\widehat{\mathcal{U}}^{\an,\circ}_G({\sigma^\beef,t_0})\to\widehat{\mathcal{U}}^{\an,\circ}({\sigma^\beef,s^\beef_0})_F$ is a closed immersion.

\subsubsection{}\label{boundary:subsubsec:tequiv}
Just as in~\eqref{boundary:subsec:logfcrys}, for technical reasons, we must allow ourselves to vary the polyhedral cone $\sigma^\beef$.

Suppose now that we have another pair $(\tilde{\sigma}^\beef,\tilde{s}^\beef_0)$ as in~\eqref{boundary:eqn:equivset} that is equivalent to $(\sigma^\beef,s_0^\beef)$ in the sense of~\eqref{boundary:subsubsec:equiv}, and is such that the set
\begin{align}\label{boundary:eqn:teqintersection}
u_{\sigma^\beef,t_0}\bigl(\widehat{\mathcal{U}}^{\an,\circ}_G({\sigma^\beef,t_0})(\overline{K}_0)\bigr)\cap\bigl(\widehat{\mathcal{U}}^{\an,\circ}({\sigma^\beef,s_0^\beef})\cap\widehat{\mathcal{U}}^{\an,\circ}{(\tilde{\sigma}^\beef,\tilde{s}^\beef_0)}\bigr)(\overline{K}_0)\subset\widehat{\mathcal{U}}^{\an,\circ}({\sigma^\beef,s_0^\beef})(\overline{K}_0)
\end{align}
is non-empty. 

Then we have:
\begin{align}\label{boundary:eqn:teqnonempty}
 {\Ss_K}\times_{\mathcal{S}_{{K}^{\beef}},i{(\Phi,\tilde{\sigma},\tilde{s}^\beef_0)}}V{(\tilde{\sigma}^\beef,\tilde{s}^\beef_0)}\neq\emptyset.
\end{align}
Once again, we can consider the ring of functions $R_G(\tilde{\sigma}^\beef,\tilde{s}_0^\beef)$ of the normalization of $S(\tilde{\sigma}^\beef,\tilde{s}_0)$ in this finite $V(\tilde{\sigma}^\beef,\tilde{s}_0)$-scheme: It is a semi-local ring. For any closed point $\tilde{t}_0$ of $R_G(\tilde{\sigma}^\beef,\tilde{s}_0^\beef)$, we obtain a corresponding local normal quotient $R_G(\tilde{\sigma}^\beef,\tilde{t}_0)$, and the analytic spaces $\widehat{\mathcal{U}}^{\an}_G(\tilde{\sigma}^\beef,\tilde{t}_0)$ and $\widehat{\mathcal{U}}^{\an,\circ}_G(\tilde{\sigma}^\beef,\tilde{t}_0)$, along with the finite map
\[
u_{\tilde{\sigma}^\beef,\tilde{t}_0}:\widehat{\mathcal{U}}^{\an}_G(\tilde{\sigma}^\beef,\tilde{t}_0)\to \widehat{\mathcal{U}}^{\an}{(\tilde{\sigma}^\beef,\tilde{s}_0)}.
\]

We will now say that a triple $(\tilde{\sigma}^\beef,\tilde{s}^\beef_0,\tilde{t}_0)$ with $\tilde{t}_0$ a closed point of $\Spec R_G(\tilde{\sigma}^\beef,\tilde{s}_0^\beef)$ is \defnword{$\mathcal{S}_K$-equivalent} to $(\sigma^\beef,s^\beef_0,t_0)$ if the intersection from~\eqref{boundary:eqn:teqintersection} is non-empty and contained in the image of $u_{\tilde{\sigma}^\beef,\tilde{t}_0}$.

\subsubsection{}\label{boundary:subsubsec:salphax}
Fix a finite extension $L/F$ and a point $x\in\widehat{\mathcal{U}}^{\an,\circ}_G({\sigma^\beef,t_0})(L)\subset\Sh_K(L)$. We then have $\Gamma_L$-invariant tensors:
\[
 \{s_{\alpha,\et,x}\}\subset H^1_{\et}\bigl(\mathcal{A}_{x,\overline{K}_0},\Rat_p)^{\otimes}.
\]

Via the isomorphism \eqref{boundary:eqn:comp}, we now obtain tensors:
\[
 \{s_{\alpha,\st,x}\}\subset M_0^\otimes
\]
such that, for each $\alpha$, $1\otimes s_{\alpha,\st,x}=\beta_{\st,x}(1\otimes s_{\alpha,\et,x})\in\Bst\otimes M_0^\otimes$. By construction, for each $\alpha$, $s_{\alpha,\st,x}$ is $\varphi$-invariant, and satisfies $N_{x}(s_{\alpha,\st,x})=0$. In other words, if $M_{0,x}=x^*M_0$ is the induced $(\varphi,N)$-module over $K_0$, we can view $s_{\alpha,\st,x}$ as a morphism of $(\varphi,N)$-modules over $K_0$:
\[
 s_{\alpha,\st,x}:\mb{1}\to M_{0,x}^\otimes.
\]

\subsubsection{}
Note that, by~\eqref{boundary:lem:Nx_sigma}, for any $y\in\widehat{\mathcal{U}}^{\an,\circ}({\sigma^\beef,s_0^\beef})(\overline{K}_0)$, the monodromy element $\nu_p\circ N_y$ lies in $\sigma^\beef\cap\Lie W_{\Phi^{\beef}}$. Let
\begin{equation}\label{boundary:eqn:sigma_G_first_def}
 \sigma_G=\langle \nu_p\circ N_{y}: y\in \widehat{\mathcal{U}}^{\an,\circ}_G({\sigma^\beef,t_0})(\overline{K}_0)\rangle\subset\sigma^\beef
\end{equation}
be the rational polyhedral cone generated by the monodromy operators attached to points lifting to $\widehat{\mathcal{U}}^{\an,\circ}_G({\sigma^\beef,t_0})$.

Since $R_G({\sigma^\beef,t_0})$ is normal, there exists $s'_0\in\mathcal{Z}(\sigma_G)(\overline{\Field}_p)$ such that the map $\widehat{\mathcal{U}}_G({\sigma^\beef,t_0})\to\widehat{\mathcal{U}}({\sigma^\beef,s_0^\beef})$ lifts to a map $\widehat{\mathcal{U}}_G({\sigma^\beef,t_0})\to\widehat{\mathcal{U}}(\sigma_G,s'_0)$. Let $\xi_{\sigma_G,s'_0}$ be as in~\eqref{boundary:eqn:parallel}, and set
\[
\tilde{s}_{\alpha,\st,x}=\xi_{\sigma_G,s'_0}(1\otimes s_{\alpha,\st,x})\in H^0\bigl(\widehat{\mathcal{U}}^{\an}({\sigma_G,s'_0}),\mathcal{M}(\sigma_G,s'_0)^{\an,\log,\otimes}\bigr).
\]

\begin{prp}\label{boundary:prp:parequiv}
The following statements are equivalent:
\begin{enumerate}
  \item\label{rat:et}For any $y\in\widehat{\mathcal{U}}^{\an,\circ}_G({\sigma^\beef,t_0})(\overline{K}_0)$, and any index $\alpha$, we have:
  \[
   s_{\alpha,\st,y}=s_{\alpha,\st,x}\in M_0^\otimes.
  \]
  \item\label{rat:dr} For any $y\in\widehat{\mathcal{U}}^{\an,\circ}_G({\sigma^\beef,t_0})(\overline{K}_0)$, and any index $\alpha$, we have:
  \[
   \beta_{\on{H-K},y,\pi}^{-1}(s_{\alpha,\dR,y})=1\otimes s_{\alpha,\st,x}\in\overline{K}_0\otimes_{K_0}M_0^\otimes.
  \]
  \item\label{rat:mon} For any $y\in\widehat{\mathcal{U}}^{\an,\circ}_G({\sigma^\beef,t_0})(\overline{K}_0)$, and any index $\alpha$, we have:
  \[
   N_{y}(s_{\alpha,\st,x})=0\in M_0^\otimes\otimes\bigl(\overline{K}_0^\times/\Reg{\overline{K}_0}^\times\bigr).
  \]
  \item\label{rat:par} For any index $\alpha$, the parallel tensor $\tilde{s}_{\alpha,\st,x}$ belongs to $H^0\bigl(\widehat{\mathcal{U}}^{\an}(\sigma_G,s'_0),\mathcal{M}(\sigma_G,s'_0)^{\an,\otimes}\bigr)$.
  \item\label{rat:shimura}For any index $\alpha$, the restriction of $\tilde{s}_{\alpha,\st,x}$ to $\widehat{\mathcal{U}}^{\an,\circ}_G({\sigma^\beef,t_0})$ coincides with that of $s_{\alpha,\dR}$ under the canonical isomorphism of filtered vector bundles with integrable connection~\eqref{boundary:eqn:manderham}:
\[
 \mathcal{M}^{\an}({\sigma^\beef,s_0^\beef})\vert_{\widehat{\mathcal{U}}^{\an,\circ}_G({\sigma^\beef,t_0})}\xrightarrow{\simeq}\underline{H}^1_{\dR}(\mathcal{A}/\Sh_K)\vert_{\widehat{\mathcal{U}}^{\an,\circ}_G({\sigma^\beef,t_0})}.
\]
\end{enumerate}
\end{prp}
\begin{proof}
  We will show:
  \[
   \eqref{rat:et}\Leftrightarrow\eqref{rat:dr}\Rightarrow\eqref{rat:mon}\Leftrightarrow\eqref{rat:par}\Rightarrow\eqref{rat:shimura}\Rightarrow\eqref{rat:dr}.
  \]

  Let $y\in\widehat{\mathcal{U}}^{\an,\circ}(\overline{K}_0)$ be another point. After enlarging $L$ if necessary, we can assume that $y$ is an $L$-valued point. Statement~\eqref{rat:et} is equivalent to the assertion that, for any $\alpha$, we have:
  \[
   \beta_{\st,y}(1\otimes s_{\alpha,\et,y})=1\otimes s_{\alpha,\st,y}\in \Bst\otimes_{K_0}M_0^\otimes.
  \]

  This equality holds if and only if it is still true after a change of scalars along the embedding $L\otimes_{K_0}\Bst\into\Bdr$ attached to the choice of a uniformizer $\pi\in L$. Therefore, using (\ref{boundary:eqn:hkst}) and (\ref{boundary:prp:blasius}), \eqref{rat:et} is now equivalent to the assertion:
  \[
   \beta_{\on{H-K},y,\pi}^{-1}(s_{\alpha,\dR,y})=1\otimes s_{\alpha,\st,x}\in L\otimes_{K_0}M_0^\otimes.
  \]
  So we have shown \eqref{rat:et}$\Leftrightarrow$\eqref{rat:dr}.

  Since $s_{\alpha,\st,y}$ induces a map of $(\varphi,N)$-modules $\mb{1}\to M_{0,y}^\otimes$, \eqref{rat:mon} is implied by \eqref{rat:et}.

  Statement~\eqref{rat:par} holds if and only if, for any $\alpha$, the parallel tensor $\tilde{s}_{\alpha,\st,x}$ induces a morphism in $\mb{LFI}(S(\sigma_G,s'_0),V(\sigma_G,s'_0))$:
  \[
   \tilde{s}_{\alpha,\st,x}:\mb{1}\to \mathcal{M}(\sigma_G,s'_0)^{\an,\otimes}.
  \]
  By the equivalence of categories from~\eqref{semistable:subsubsec:unipequiv}, this holds if and only if, for any $\alpha$, $s_{\alpha,\st,x}\in M_0^\otimes$ induces a morphism in $\mb{LFI}(\overline{\Field}_p,\Lambda(\sigma_G))$:
  \[
   s_{\alpha,\st,x}:\mb{1}\to M(\sigma_G,s'_0)^\otimes.
  \]

  In turn, this is equivalent to requiring that, for all maps $\nu:\Lambda(\sigma_G)\to\Rat$, and for any index $\alpha$, the composition
  \[
   M_0\xrightarrow{N(\sigma_G,s'_0)}M_0\otimes\Lambda(\sigma_G)\xrightarrow{1\otimes\nu} M_0
  \]
  carries $s_{\alpha,\st,x}$ to $0$.

  Now, by the definition of $\sigma_G$, $\Lambda(\sigma_G)$ is the smallest quotient of $\Lambda(\sigma^\beef)$ with the following property: All maps $y^\sharp:\Lambda(\sigma^\beef)\to\overline{K}_0^\times/\Reg{\overline{K}_0}^\times$ for $y\in\widehat{\mathcal{U}}^{\an,\circ}_G({\sigma^\beef,t_0})(K_0)$ factor through $\Lambda(\sigma_G)$. In other words, the vector space $\Hom(\Lambda(\sigma_G),\Rat)$ is generated by maps of the form $\nu_p\circ N_y$, for $y\in\widehat{\mathcal{U}}^{\an,\circ}_G(\sigma^\beef,t_0)$.
  From this, the equivalence of \eqref{rat:mon} and \eqref{rat:par} is immediate.

  Assuming \eqref{rat:par}, we find that both $\tilde{s}_{\alpha,\st,x}$ and $s_{\alpha,\dR,\widehat{\mathcal{U}}^{\an,\circ}_G({\sigma^\beef,t_0})}$ are parallel tensors in
  \[
  H^0\bigl(\widehat{\mathcal{U}}^{\an,\circ}_G(\sigma^\beef,t_0),\mathcal{M}^{\an,\otimes}({\sigma^\beef,s_0^\beef})\bigr),
  \]
  whose fibers at $x$ coincide with $s_{\alpha,\dR,x}$.

  Since $\widehat{\mathcal{U}}^{\an,\circ}_G({\sigma^\beef,t_0})$ is a smooth, irreducible analytic space, we find that they must coincide everywhere on $\widehat{\mathcal{U}}^{\an,\circ}_G({\sigma^\beef,t_0})$. This shows \eqref{rat:par}$\Rightarrow$\eqref{rat:shimura}.

  Now, if we assume~\eqref{rat:shimura}, then, for any $y\in\widehat{\mathcal{U}}^{\an,\circ}_G({\sigma^\beef,t_0})(\overline{K}_0)$, the fiber of $\tilde{s}_{\alpha,\st,x}$ at $y$ coincides with $s_{\alpha,\dR,y}$. On the other hand, by the construction of $\tilde{s}_{\alpha,\st,x}$, this fiber is precisely $\beta_{\on{H-K},y,\pi}(s_{\alpha,\st,x})$. This shows that~\eqref{rat:shimura} implies~\eqref{rat:dr}.
\end{proof}

\begin{lem}\label{boundary:lem:parallellem}
Let $(\tilde{\sigma}^\beef,\tilde{s}_0,\tilde{t}_0)$ be a triple $\Ss_K$-equivalent to $(\sigma^\beef,s_0,t_0)$ with $\tilde{\sigma}^{\beef,\circ}\subset\sigma^{\beef,\circ}$. Assume that $\sigma^\beef$ (resp. $\tilde{\sigma}^\beef$) is generated by elements of the form $\nu_p\circ N_y$ with $y\in \widehat{\mathcal{U}}_G^{\an,\circ}(\sigma^\beef,t_0)(\overline{K}_0)$ (resp. $y\in \widehat{\mathcal{U}}_G^{\an,\circ}(\tilde{\sigma}^\beef,\tilde{t}_0)(\overline{K}_0)$).

Then the natural map:
\[
 H^0\bigl(\widehat{\mathcal{U}}_G^{\an,\circ}(\sigma^\beef,t_0),\Reg{\widehat{\mathcal{U}}^{\an}(\sigma^\beef,s_0^\beef)}^{\log}\bigr)^{\nabla=0} \to  H^0\bigl(\widehat{\mathcal{U}}_G^{\an,\circ}(\tilde{\sigma}^\beef,\tilde{t}_0),\Reg{\widehat{\mathcal{U}}^{\an}(\tilde{\sigma}^\beef,\tilde{s}^\beef_0)}^{\log}\bigr)^{\nabla=0}
\]
is injective.
\end{lem}
\begin{proof}
Without loss of generality, we can assume that $\tilde{\sigma}$ has codimension $1$ in $\sigma$, and that there exists a point $y_0\in \widehat{\mathcal{U}}^{\an,\circ}_G(\sigma^\beef,t_0)(\overline{K}_0)$ such that $B_0 = \nu_p\circ N_{y_0}$ is contained in $\sigma^{\beef,\circ}$, but not in $\tilde{\sigma}^{\beef,\circ}$. 

By~\eqref{boundary:subsubsec:noncansplit}, we have compatible splittings:
\[
 R(\sigma^\beef,s_0^\beef) = R^{\sab}\widehat{\otimes}B(\sigma^\beef,s'_0)\;;\;R(\tilde{\sigma}^\beef,\tilde{s}^\beef_0) = R^{\sab}\widehat{\otimes}B(\tilde{\sigma}^\beef,\tilde{s}'_0).
\]
Using these splittings, we can view every $h\in\mb{S}$ as a non-vanishing global function on both $\widehat{\mathcal{U}}^{\an,\circ}(\sigma^\beef,s_0^\beef)$ and $\widehat{\mathcal{U}}^{\an,\circ}(\tilde{\sigma}^\beef,\tilde{s}^\beef_0)$.

Suppose that $\sigma^\beef$ has rank $k$. We then have a surjection of free groups
\[
 \Lambda(\sigma^\beef) = \mb{S}/\mb{S}(\sigma^\beef)^\times \to\mb{S}/\mb{S}(\tilde{\sigma}^\beef)^\times = \Lambda(\tilde{\sigma}^\beef)
\]
of ranks $k$ and $k-1$, respectively. Let $v_k\in\Lambda(\sigma^\beef)$ be a generator for the kernel of this map, and fix a splitting:
\[
 \Lambda(\sigma^\beef) = \Lambda(\tilde{\sigma}^\beef)\oplus\langle v_k\rangle.
\]
Note that $v_k$ belongs to $\mb{S}(\tilde{\sigma}^\beef)^\times$ but not to $\mb{S}(\sigma^\beef)^\times$. In particular, the value of the pairing $\langle B_0,v_k\rangle$ must be non-zero.

Fix a basis $\{v_1,\ldots,v_{k-1}\}$ for $\Lambda(\tilde{\sigma}^\beef)$ and a splitting:
\[
 \mb{S} = \mb{S}(\sigma^\beef)^\times\oplus\langle v_1,\ldots,v_k\rangle.
\]
This allows us to view $v_k$ as a function on $\widehat{\mathcal{U}}^{\an,\circ}(\sigma^\beef,s_0^\beef)$. Note that, for any point $x\in\widehat{\mathcal{U}}^{\an,\circ}(\sigma^\beef,s_0^\beef)$, we have $\vert v_k(x)\vert_p = p^{-\langle \nu_p\circ N_x,v_k\rangle}$.

We now obtain isomorphisms of sheaves of algebras:
\begin{align*}
 \Reg{\widehat{\mathcal{U}}^{\an}(\sigma^\beef,s_0^\beef)}[X_1,\ldots,X_k]\xrightarrow[\simeq]{X_i\mapsto\ell_{v_i}}\Reg{\widehat{\mathcal{U}}^{\an}(\sigma^\beef,s_0^\beef)}^{\log}\;;\;
 \Reg{\widehat{\mathcal{U}}^{\an}(\tilde{\sigma}^\beef,\tilde{s}^\beef_0)}[X_1,\ldots,X_{k-1}]\xrightarrow[\simeq]{X_i\mapsto\ell_{v_i}}\Reg{\widehat{\mathcal{U}}^{\an}(\tilde{\sigma}^\beef,\tilde{s}^\beef_0)}^{\log}.
\end{align*}
Here, the induced connection on the left-hand side carries $X_i$ to $\dlog(v_i)$, $i=1,\ldots,k$.

Moreover, under these isomorphisms the restriction map $\Reg{\widehat{\mathcal{U}}^{\an}(\sigma^\beef,s_0^\beef)}^{\log}\to\Reg{\widehat{\mathcal{U}}^{\an}(\tilde{\sigma}^\beef,\tilde{s}^\beef_0)}^{\log}$ is identified with:
\begin{align}\label{boundary:eqn:zanlogexpl}
\Reg{\widehat{\mathcal{U}}^{\an}(\sigma^\beef,s_0^\beef)}[X_1,\ldots,X_k]&\to\Reg{\widehat{\mathcal{U}}^{\an}(\tilde{\sigma}^\beef,\tilde{s}^\beef_0)}[X_1,\ldots,X_{k-1}]\\
X_j&\mapsto\begin{cases}
             X_j\text{, if $j\leq k-1$}\\
             \log(v_k)\text{, if $j=k$}.
             \end{cases}\nonumber
\end{align}

We will abbreviate a $k$-tuple $(j_1,\ldots,j_k)\in\Nat^k$ by $\underline{j}$. For $1\leq j\leq k$, let $\underline{e}_k$ be the $i$-tuple with $1$ in the $j^{\text{th}}$ position and $0$s elsewhere. Suppose that we have an element
 \[
 a=\sum_{\underline{j}\in\Nat^k}a_{\underline{j}}X_1^{j_1}\cdots X_i^{j_i}\in H^0\bigl(\widehat{\mathcal{U}}^{\an,\circ}_G(\sigma^\beef,t_0), \Reg{\widehat{\mathcal{U}}^{\an}(\sigma^\beef,s_0^\beef)}^{\log}\bigr)^{\nabla=0}.
\]
Here, the coefficients $a_{\underline{j}}$ belong to $H^0(\widehat{\mathcal{U}}^{\an,\circ}_G(\sigma^\beef,t_0), \Reg{\widehat{\mathcal{U}}^{\an}(\sigma^\beef,s^\beef_0)})$. One easily checks that the condition $\nabla(a)=0$ translates to:
\[
 da_{\underline{j}}=-\sum_{m=1}^k(j_m+1)a_{\underline{j}+\underline{e}_m}\dlog(v_m),\text{ for all $\underline{j}\in\Nat^i$}.
\]
In particular, if $\underline{j}\in\Nat^i$ is such that $a_{\underline{j}}\neq 0$, but $a_{\underline{j}+\underline{e}_s}=0$, for $1\leq s\leq k$, then we must have $da_{\underline{j}}=0$, and so $a_{\underline{j}}$ must be a non-zero constant.

Suppose that $a\neq 0$, and choose an index $\underline{j}$ such that $a_{\underline{j}}$ is a non-zero constant and such that $a_{\underline{j}+\underline{e}_k}=0$. The lemma will follow if we can show that the image of $a$ under~\eqref{boundary:eqn:zanlogexpl} is non-zero when restricted over $\widehat{\mathcal{U}}^{\an,\circ}_G(\tilde{\sigma}^\beef,\tilde{t}_0)$.

Assume that this is untrue. Then we must have:
\[
\sum_{s=0}^{j_k}a_{(j_1,\ldots,j_{k-1},s)}\log(v_k)^{s}=0\in H^0(\widehat{\mathcal{U}}^{\an,\circ}_G(\tilde{\sigma}^\beef,\tilde{t}_0), \Reg{\widehat{\mathcal{U}}^{\an}(\sigma^\beef,s^\beef_0)})
\]
Indeed, this is simply the coefficient for the monomial $X_1^{j_1}\cdots X_{i-1}^{j_{k-1}}$ in the image of $a$ under~\eqref{boundary:eqn:zanlogexpl}. Therefore, we find that $\log(v_k)$ satisfies a non-zero monic polynomial over $H^0(\widehat{\mathcal{U}}^{\an,\circ}_G(\sigma^\beef,t_0), \Reg{\widehat{\mathcal{U}}^{\an}(\sigma^\beef,s^\beef_0)})$.

By~\eqref{rigid:lem:logarithm}, this implies that $\vert v_k(x)\vert_p=1$, for all $x\in\widehat{\mathcal{U}}^{\an}_G(\sigma,s_0)(\overline{K}_0)$. However, we have $\vert v_k(x_0)\vert_p = p^{-\langle B_0,v_k\rangle}$, and we have already seen that $\langle B_0,v_k\rangle\neq 0$. This gives us the desired contradiction.
\end{proof}

\begin{prp}\label{boundary:prp:parallel}
The equivalent statements of~\eqref{boundary:prp:parequiv} are true.
\end{prp}
\begin{proof}
It is enough to prove assertion~\eqref{rat:shimura} of~\eqref{boundary:prp:parequiv}. For this, replacing $\sigma$ by $\sigma_G$, we can assume that $\sigma$ is generated by elements of the form $\nu_p\circ N_x$, for $x\in\widehat{\mathcal{U}}^{\an,\circ}_G({\sigma^\beef,t_0})(\overline{K}_0)$.

  Via~\eqref{boundary:eqn:parallel}, we now obtain a canonical identifications:
  \[
  H^0\bigl(\widehat{\mathcal{U}}^{\an,\circ}_G({\sigma^\beef,t_0}),\mathcal{M}^{\an,\log,\otimes}({\sigma^\beef,s_0^\beef})\bigr)^{\nabla=0}=H^0\bigl(\widehat{\mathcal{U}}^{\an,\circ}_G({\sigma^\beef,t_0}),\Reg{\widehat{\mathcal{U}}^{\an}({\sigma^\beef,s_0^\beef})}^{\log}\bigr)^{\nabla=0}\otimes_{K_0}M_0^\otimes.
  \]
  In particular, for every $\alpha$, the restriction of $s_{\alpha,\dR}$ over $\widehat{\mathcal{U}}^{\an,\circ}_G({\sigma^\beef,t_0})$ can naturally be viewed as an element
  \[
   s_{\alpha,\dR,(\sigma,s_0)}\in H^0\bigl(\widehat{\mathcal{U}}^{\an,\circ}_G({\sigma^\beef,t_0}),\Reg{\widehat{\mathcal{U}}^{\an}({\sigma^\beef,s_0^\beef})}^{\log}\bigr)^{\nabla=0}\otimes_{K_0}M_0^\otimes.
  \]
  Assertion~\eqref{rat:shimura} holds if and only if, for all $\alpha$:
  \begin{align}\label{boundary:eqn:salphaident}
  s_{\alpha,\dR,(\sigma,s_0)}=1\otimes s_{\alpha,\st,x}\in H^0\bigl(\widehat{\mathcal{U}}^{\an,\circ}_G({\sigma^\beef,t_0}),\Reg{\widehat{\mathcal{U}}^{\an}({\sigma^\beef,s_0^\beef})}^{\log}\bigr)^{\nabla=0}\otimes_{K_0}M_0^\otimes.
  \end{align}

  Let $\tilde{\sigma}\subset\sigma$ be the polyhedral cone generated by $\nu_p\circ N_x$, and let $\tilde{s}^\beef_0\in\mathcal{Z}(\tilde{\sigma})(\overline{\Field}_p)$ be the specialization of $x$. From~\eqref{boundary:lem:parallellem} we find that the map:
  \begin{align}\label{boundary:eqn:restrictinj}
   H^0\bigl(\widehat{\mathcal{U}}^{\an,\circ}_G({\sigma^\beef,t_0}),\Reg{\widehat{\mathcal{U}}^{\an}({\sigma^\beef,s_0^\beef})}^{\log}\bigr)^{\nabla=0}\to H^0\bigl(\widehat{\mathcal{U}}^{\an,\circ}{(\tilde{\sigma}^\beef,\tilde{t}_0)},\Reg{\widehat{\mathcal{U}}^{\an}{(\tilde{\sigma}^\beef,\tilde{s}^\beef_0)}}^{\log}\bigr)^{\nabla=0}
  \end{align}
  is injective.

  Since $\tilde{\sigma}$ is $1$-dimensional, the identity in assertion~\eqref{rat:mon} is trivially valid for $y\in\widehat{\mathcal{U}}^{\an,\circ}_G{(\tilde{\sigma}^\beef,\tilde{t}_0)}(\overline{K}_0)$. Therefore, we find that, for every $\alpha$, the restriction of $s_{\alpha,\dR}$ over $\widehat{\mathcal{U}}^{\an,\circ}{(\tilde{\sigma},\tilde{t}_0)}$, viewed as an element
  \[
   s_{\alpha,\dR,(\tilde{\sigma},\tilde{t}_0)}\in H^0\bigl(\widehat{\mathcal{U}}^{\an,\circ}{(\tilde{\sigma},\tilde{t}_0)},\Reg{\widehat{\mathcal{U}}^{\an}{(\tilde{\sigma}^\beef,\tilde{s}^\beef_0)}}^{\log}\bigr)^{\nabla=0}\otimes_{K_0}M_0^\otimes,
  \]
  coincides with $1\otimes s_{\alpha,\st,x}$. By the injectivity of~\eqref{boundary:eqn:restrictinj}, the identity~\eqref{boundary:eqn:salphaident} has to hold for every $\alpha$. This completes the proof.
\end{proof}

\begin{corollary}\label{boundary:cor:tatetensorprps}
For any index $\alpha$, there is a canonical tensor $s_{\alpha,\st,0}\in M_0^\otimes$ with the following properties:
\begin{itemize}
  \item It is $\varphi$-invariant: $\varphi_0(s_{\alpha,\st,0})=s_{\alpha,\st,0}$.
  \item For any $(\tilde{\sigma}^\beef,\tilde{s}^\beef_0,\tilde{t}_0)$ in the ${\Ss_K}$-equivalence class of $(\sigma^\beef,s_0^\beef,t_0)$, and any $y\in\widehat{\mathcal{U}}^{\an,\circ}_G{(\tilde{\sigma},\tilde{t}_0)}(\overline{K}_0)$, we have: $N_{y}(s_{\alpha,\st,0})=0$.
  \item If $y$ is an $L$-valued point for some finite extension $L/K_0$ within $\overline{K}_0$, for any uniformizer $\pi\in L$, we have:
  \begin{align*}
   \beta_{\st,y}^{-1}(1\otimes s_{\alpha,\st,0})&=s_{\alpha,\et,y}\in\biggl(H^1_{\et}\bigl(\mathcal{A}_{y,\overline{K}_0},\Rat_p\bigr)^\otimes\biggr)^{\Gamma_L};\\
   \beta_{\on{H-K},y,\pi}(1\otimes s_{\alpha,\st,0})&=s_{\alpha,\dR,y}\in F^0\bigl(H^1_{\dR}(\mathcal{A}_y/L)^\otimes\bigr).
  \end{align*}
\end{itemize}
\end{corollary}
\begin{proof}
By assertion~\eqref{rat:et} of~\eqref{boundary:prp:parequiv}, we can take $s_{\alpha,\st,0}=s_{\alpha,\st,x}$, for any $x\in\widehat{\mathcal{U}}^{\an,\circ}_G({\sigma^\beef,t_0})(\overline{K}_0)$.
\end{proof}

\begin{rem}
\label{boundary:rem:pel_case_tensors}
Note that in the PEL case, where the tensors can be taken to be realizations of polarizations or endomorphisms, the canonical tensor above is simply the realization of the reduction of such a polarization or endomorphism. The compatibility with the comparison isomorphisms simply amounts to the functoriality of these isomorphisms.
\end{rem}

\subsection{The structure of the boundary}\label{boundary:subsec:structure}
In this sub-section, we will set up the notation required to state the main technical theorem~\eqref{boundary:thm:main} on the structure of $\Spec R_G({\sigma^\beef,t_0})$. This theorem shows that it has the expected shape: It is a completion of a twisted toric embedding over a finite normal $R^{\sab}$-scheme.

\subsubsection{}\label{boundary:subsubsec:tautred}
Consider the scheme over $K_0$ that parameterizes, for any $K_0$-algebra $B$, the set of isomorphisms
\[
 B\otimes_{\Rat_p}\dual{H}(\Rat_p)\xrightarrow{\simeq}B\otimes_{K_0}M_0
\]
carrying $\{1\otimes s_{\alpha}\}$ onto $\{1\otimes s_{\alpha,\st,s_0}\}$. By~\eqref{boundary:cor:tatetensorprps} and~\eqref{boundary:eqn:comp}, this scheme has a $\Bst$-valued point, and is thus a $G$-torsor over $K_0$. By Steinberg's theorem, it must have a $K_0$-valued point. In particular, the point-wise stabilizer of $\{s_{\alpha,\st,s_0}\}$ within $\GL(M_0)$ is isomorphic to $G_{K_0}$. In a slight abuse of notation, we will identify this stabilizer with $G_{K_0}$.

Set
\[
U_{\wt,G}=U_{\wt}\cap G_{K_0}\;;\;\mb{B}_G=\mb{B}\cap(\Lie U_{\wt,G})(-1)\subset(\Lie U_{\wt})(-1).
\]
Here, we are using~\eqref{boundary:lem:liewp_embedding} to view $\mb{B}\subset W_{\Phi^{\beef}}(\Rat)$ as a subgroup of $(\Lie U_{\wt})(-1)$.

Let $\mb{E}_G\subset\mb{E}$ be the sub-torus with co-character group $\mb{B}_G$; then $\mb{E}^G\coloneqq\mb{E}/\mb{E}_G$ is again a torus, with co-character group $\mb{B}/\mb{B}_G$. The quotient
\[
\mb{\Xi}^G = \mb{\Xi}/\mb{E}_G
\]
is an $\mb{E}^G$-torsor over $R^{\sab}$.

\subsubsection{}
Given any $R^{\sab}$-scheme $Z$ and a section $\beta:X\to\mb{\Xi}^G$, the pre-image of this section in $\mb{\Xi}\vert_Z$ is an $\mb{E}_G$-torsor over $Z$, which we will denote by $\mb{\Xi}_G(\beta)$.

Let $R_G^{\sab}$ be the normalization in $R_G({\sigma^\beef,t_0})$ of the image of $R^{\sab}$. It is not hard to see that, if we chose a different pair $(\tilde{\sigma}^\beef,\tilde{s}^\beef_0)$ in the ${\Ss_K}$-equivalence class of $(\sigma^\beef,s_0^\beef,t_0)$, then the normalization of $\Reg{F}\otimes_WR^{\sab}$ in $R_G{(\tilde{\sigma}^\beef,\tilde{t}_0)}$ will coincide with $R_G^{\sab}$.

If $Z=\mb{\Xi}$, then we have the tautological section $\iota\in\mb{\Xi}(\mb{\Xi})$, inducing a section:
\[
 \iota^G:\mb{\Xi}\to\mb{\Xi}^G.
\]
Suppose that $a_{\sigma}\in R({\sigma^\beef,s_0^\beef})$ is an equation for the boundary divisor; then $\iota^G$ induces a section over $\Spec R_G({\sigma^\beef,t_0})[a_{\sigma}^{-1}]$, which we denote by $\iota^G(\sigma,s_0)$.

\begin{thm}\label{boundary:thm:main}
\mbox{}
\begin{enumerate}[itemsep=0.12in]
\item\label{main:rationality}The inclusion $\mb{B}_G\subset(\Lie U_{\wt,G})(-1)$ induces an isomorphism:
\[
 K_0\otimes\mb{B}_G\xrightarrow{\simeq}(\Lie U_{\wt,G})(-1).
\]
\item\label{main:reduction}The section
\[
\iota^G_{(m)}(\sigma,s_0):\Spec R_G({\sigma^\beef,t_0})[a_{\sigma}^{-1}]\to\mb{\Xi}^G
\]
is already defined over $\Spec R_G^{\sab}$, and is independent of the choice of $(\sigma^\beef,s_0^\beef,t_0)$ within its ${\Ss_K}$-equivalence class. In particular, $\mb{\Xi}\vert_{\Spec R_G^{\sab}}$ admits a canonical reduction of structure group to an $\mb{E}_G$-torsor:
\[
 \mb{\Xi}_G\coloneqq\mb{\Xi}_G(\iota^G(\sigma,s_0)).
\]
\item\label{main:local}There exist: 
\begin{itemize}
  \item A torus $\mb{E}^\diamond_G$ equipped with an isogeny $\mb{E}^\diamond_G\to \mb{E}_G$;
  \item An $\mb{E}^\diamond_G$-torsor $\mb{\Xi}^\diamond_G$ over $R_G^{\sab}$ inducing $\mb{\Xi}_G$ along the above isogeny;
\end{itemize}
such that, for any $(\tilde{\sigma}^\beef,\tilde{s}_0^\beef,\tilde{t}_0)$ in the ${\Ss_K}$-equivalence class of $(\sigma^\beef,s_0^\beef,t_0)$, the map
\[
\Spec R_G{(\tilde{\sigma}^\beef,\tilde{t}_0)}\to\Spec R{(\tilde{\sigma}^\beef,\tilde{s}^\beef_0)}\to\mb{\Xi}(\tilde{\sigma})
\]
lifts canonically to a map $\Spec R_G{(\tilde{\sigma}^\beef,\tilde{t}_0)}\to\mb{\Xi}^\diamond_G(\tilde{\sigma}_G)$, and identifies $\widehat{\mathcal{U}}_G(\tilde{\sigma}^\beef,\tilde{t}_0)$ with the completion at a point of $\mb{\Xi}^\diamond_G(\tilde{\sigma}_G)$.
\end{enumerate}
\end{thm}

\subsubsection{}\label{boundary:subsubsec:remarks}
Here are some remarks to clarify the assertions of the theorem:
\begin{enumerate}[itemsep=0.11in]
\item Claim~\eqref{main:rationality} should be viewed as a rationality property for Hodge cycles on abelian varieties with respect to $p$-adic uniformizations, and is closely related to ideas from \cite{andre:padic_betti}. It is the crystalline avatar of Theorem~\ref{rationality}, and, as observed in the introduction, is a significant result for our purposes. However, in the PEL case it admits a simple proof, which is indicated in~\eqref{boundary:rem:pel_case}.

\item\label{main:concrete}As explained in the introduction, claim (\ref{main:reduction}) should be viewed as the key result of this paper. It has a very concrete meaning: Choose a section $\iota^{\sab}$ of $\mb{\Xi}$ over $R^{\sab}_{s_0^\beef}$ as in (\ref{boundary:subsubsec:noncansplit}) with associated element $\alpha\in\mb{E}(\mb{\Xi})$ carrying $\iota^{\sab}$ to $\iota$. We can think of $\alpha$ as a homomorphism $\mb{S}\to H^0(\mb{\Xi},\Gm)$.

Use $\iota^{\sab}$ to also denote the section of the $\mb{E}^G$-torsor $\mb{\Xi}^G$ induced from $\iota^{\sab}$. Let $\overline{\alpha}\in\mb{E}^G(\mb{\Xi})$ be the image of $\alpha$. Then $\overline{\alpha}$ carries $\iota^{\sab}$ to $\iota^G$. Therefore we find that, if $\mb{S}^G\subset\mb{S}$ is the character group of $\mb{E}^G$, then (\ref{main:reduction}) is equivalent to:

For all $h\in\mb{S}^G$, the image of $\overline{\alpha}(h)=\alpha(h)$ in $R_G({\sigma^\beef,t_0})[a_{\sigma}^{-1}]^\times$ lies in the sub-group $R_G^{\sab,\times}$.

In fact, since $R_G^{\sab}$ is integrally closed in $R_G(\sigma^\beef,t_0)$ by construction, it is enough to show:

There exists an integer $m\geq 1$ such that, for all $h\in\mb{S}^G$, the image of $\alpha(h)^m$ in $R_G({\sigma^\beef,t_0})[a_{\sigma}^{-1}]^\times$ lies in the sub-group $R_G^{\sab,\times}$

 The independence from the choice of $(\sigma^\beef,s_0^\beef,t_0)$ is easily checked.

\item Claim~\eqref{main:local} gives the desired description of the local structure at the boundary of $\Ss_K$. As the reader will note, its proof involves the theory of Chai-Faltings and $p$-adic Hodge theory. In particular, it makes no use of the general characteristic $0$ theory of compactifications. 

We still have to explain the notation in claim (\ref{main:local}). We have:
\[
\tilde{\sigma}_G=\tilde{\sigma}^\beef\cap\bigl(\mb{B}_G\otimes\Real\bigr).
\]
Though it is not \emph{a priori} clear, it is a consequence of claim~\eqref{main:local} that this definition agrees with the one given in~\eqref{boundary:eqn:sigma_G_first_def}. 

Given this cone and the $\mb{E}_G^\diamond$-torsor $\mb{\Xi}^\diamond_G$, we can construct the twisted toric embedding $\mb{\Xi}_G\into\mb{\Xi}_G(\tilde{\sigma}_G)$ of schemes over $R^{\sab}_G$. This is the scheme that enters the statement of~\eqref{main:local}.
\end{enumerate}

The rest of the section is dedicated to the proof of (\ref{boundary:thm:main}). It will be completed in~\eqref{boundary:subsec:final}.

\subsection{A rationality property of Hodge cycles}\label{boundary:subsec:rationality}
In this subsection, we will prove assertion \eqref{main:rationality} of \eqref{boundary:thm:main}. We already know by~\eqref{boundary:cor:tatetensorprps} that the monodromy operators $N_x$ corresponding to the the points $x\in\widehat{\mathcal{U}}^{\an,\circ}_G({\sigma^\beef,t_0})(\overline{K}_0)$ all lie in $\mb{B}\cap(\Lie U_{\wt,G})(-1)=\mb{B}_G$. We have to show that they generate a large enough space.

This is of course a numerical question. Namely, set $d_{\wt}\coloneqq\dim_{K_0}U_{\wt,G}$; then we have to show $\rank\mb{B}_G = d_{\wt}$. This will follow from two bounds:~\eqref{boundary:lem:tsabdim} and~\eqref{boundary:lem:nudim}. The first of these is a consequence of the versality of the deformation ring $R^{\sab}$ and the second is a simple fact about sub-schemes of completed torus embeddings.

\subsubsection{}\label{boundary:subsubsec:shimuradim}
We will need a little preparation before we can prove the first bound. Recall that the $G(\Real)$-homogeneous space $X$ determines a canonical conjugacy class of co-characters $[\mu]$ of $G$ defined over the reflex field $E$, defined in the following way: If we fix $x\in X$, then a representative $\mu_x$ for $[\mu]$ over $\Comp$ is given by:
\[
\mu_x:\Gmh{\Comp}\xrightarrow{z\mapsto (z,1)}\Gmh{\Comp}\times\Gmh{\Comp}\xrightarrow{\simeq}\mathbb{S}_{\Comp}\xrightarrow{h_x}G_{\Comp}.
\]
Let $P_{\mu_x}\subset G_{\Comp}$ be the parabolic sub-group whose Lie algebra consists of the non-negative weight spaces of $\mu_x$. By construction, the map carrying a point $x\in X$ to the Hodge filtration on $V_{\Comp}$ induced by $h_x$ exhibits $X$ as a finite analytic space over an open sub-space of the Grassmannian $G_{\Comp}/P_{\mu_x}$. We therefore have:
\[
 d\coloneqq\dim\Sh_K(G,X)=\dim X=\dim G-\dim P_{\mu_x}.
\]

\subsubsection{}\label{boundary:subsubsec:cocharacters}
Suppose now that we have a finite extension $L/K_0$ within $\overline{K}_0$ and a point $y\in\widehat{\mathcal{U}}^{\an,\circ}_G({\sigma^\beef,t_0})(L)$. Attached to this (and a uniformizer $\pi\in L$) is the Hyodo-Kato isomorphism:
\[
 \beta_{\on{H-K},y,\pi}:L\otimes_{K_0}M_0\xrightarrow{\simeq}H^1_{\dR}(\mathcal{A}_y/L).
\]
Via this isomorphism, we obtain a Hodge filtration on $L\otimes_{K_0}M_0$:
\[
 F^\bullet_y\bigl(L\otimes_{K_0}M_0\bigr)=\beta_{\on{H-K},y,\pi}^{-1}\bigl(F^\bullet H^1_{\dR}(\mathcal{A}_y/L)\bigr).
\]

\begin{lem}\label{boundary:lem:cocharacters}
\mbox{}
\begin{enumerate}
\item\label{cochar:weight}The weight filtration $W_{\bullet}M_0$ is split by a co-character
\[
 w:\Gmh{K_0}\to G_{K_0}.
\]
In particular, the stabilizer $P_{\wt,G}\subset G_{K_0}$ of $W_{\bullet}M_0$ is a parabolic sub-group.
\item\label{cochar:hodge}The Hodge filtration $F^\bullet_y\bigl(L\otimes_{K_0}M_0\bigr)$ is split by a co-character
\[
 \tilde{\mu}_y:\Gmh{L}\to G_L
\]
such that $\tilde{\mu}^{-1}_y$ belongs to the conjugacy class $[\mu]$.
\item\label{cochar:qwtg}Let $Q_{\wt,G}\subset P_{\wt,G}$ be the largest subgroup acting trivially on $W_0M_0$. Then we can choose $\tilde{\mu}_y$ so that it factors through $L\otimes_{K_0}Q_{\wt,G}$.
\end{enumerate}
\end{lem}
\begin{proof}
Attached to $y$ is the monodromy element
\[
B=\nu_p\circ N_y\in\Rat\otimes\mb{B}_G\subset\Lie U_{\wt,G}(-1)\subset\End(M_0)(-1).
\]

Fix an isomorphism $K_0(-1)\xrightarrow{\simeq}K_0$. Then we can view $B$ as a nilpotent endomorphism of $M_0$:
\[
 M_0\to \gr^W_2M_0=K_0(-1)\otimes \mathsf{X}\to\Hom(\mathsf{X},K_0(-1))=W_0M_0(-1)\into M_0(-1)\xrightarrow{\simeq}M_0,
\]
where the middle map is induced by the pairing $B:\mathsf{X}\times \mathsf{X}\to\Rat\into K_0$. Since $B$ belongs to $\sigma^{\circ}$, it induces a positive definite pairing on $\mathsf{X}$. In particular, we have $\ker B=W_1M_0$ and $\im B=W_0M_0$.

For any integer $i$, set $U^iM_0=W_{-i+1}M_0$. Then $U^\bullet M_0$ is the filtration on $M_0$ attached to the nilpotent endomorphism $B$ via the recipe in (2.5.1) of \cite[IV]{saavedra_rivano}. In particular, since $B\in\Lie G_{K_0}$, it follows from Prop. 2.5.3 of \emph{loc. cit.} that $U^\bullet M_0$ is split by a co-character $w':\Gmh{K_0}\to G_{K_0}$. Since $G_{K_0}$ contains the central sub-group $\Gmh{K_0}\subset\GL(M_0)$,~\eqref{cochar:weight} is now immediate.

Choose an embedding $L\into\Comp$, and view $y$ as a point $[(x,g)]\in\Sh_K(\Comp)$. Then the existence of the co-character $\tilde{\mu}_y$ in~\eqref{cochar:hodge} is immediate from Prop. 2.2.2 of \emph{loc. cit.} and the following: There exists an isomorphism
\[
 \beta:\Comp\otimes \dual{H}(\Rat)\xrightarrow{\simeq}H^1_{\dR}(\mathcal{A}_y/\Comp)
\]
carrying $\{1\otimes s_{\alpha}\}$ to $\{s_{\alpha,\dR,y}\}$, and such that the induced filtration
\[
F^\bullet_y(\Comp\otimes \dual{H}(\Rat))=\beta^{-1}(F^\bullet H^1_{\dR}(\mathcal{A}_x/\Comp))
\]
is split by the co-character $\mu_x^{-1}$.

It follows from~\cite[4.2.17]{dat_orlik_rapoport} that we can choose $\tilde{\mu}_y$ to factor through $L\otimes_{K_0}P_{\wt,G}$. But we have
\[
F^1_y(L\otimes_{K_0}M_0)\cap(L\otimes_{K_0}W_0M_0) = 0.
\]
Therefore, $\tilde{\mu}_y(\Gm)$ must act trivially on $L\otimes_{K_0}W_0M_0$, and so $\tilde{\mu}_y$ must in fact factor through $Q_{\wt,G,L}$. This shows~\eqref{cochar:qwtg}.
\end{proof}

\subsubsection{}\label{boundary:subsubsec:grassmann}
Write $\on{Gr}_{G,[\mu]}$ for the Grassmannian over $K_0$ defined as follows: For any $K_0$-algebra $A$, $\on{Gr}_{G,[\mu]}(A)$ is the set of decreasing filtrations $F^\bullet(A\otimes_{K_0}M_0)$, which, \'etale locally on $\Spec A$, can be split by a co-character $\Gmh{A}\to G_A$ in the conjugacy class $[\mu]$.

Let $\on{U}_{G,[\mu]}\subset\on{Gr}_{G,[\mu]}$ be the open sub-variety such that, for any $K_0$-algebra $A$, we have:
\[
 \on{U}_{G,[\mu]} (A)=\{F^\bullet(A\otimes_{K_0}M_0)\in\on{Gr}_{G,[\mu]}(A):\;F^1(A\otimes_{K_0}M_0)\cap(A\otimes_{K_0} W_0M_0)=0\}.
\]

Choose $x\in X$; then we have
\[
\dim\on{U}_{G,[\mu]}=\dim\on{Gr}_{G,[\mu]}=\dim G_{\Comp}/P_{\mu_x}=d.
\]

Set $M_0^{\sab}=M_0/W_0M_0$. Let $\on{Gr}^{\sab}$ be the Grassmannian over $K_0$ such that, for all $K_0$-algebras $A$, $\on{Gr}^{\sab}(A)$ is the set of decreasing filtrations $F^\bullet(A\otimes_{K_0}M_0^{\sab})$ satisfying:
\begin{itemize}
\item $F^i(A\otimes_{K_0}M_0^{\sab})=0$, if $i<1$, and $F^0(A\otimes_{K_0}M_0^{\sab})=A\otimes_{K_0}M_0^{\sab}$;
\item $\rank F^1(A\otimes_{K_0}M_0^{\sab})=g$;
\item $F^1(A\otimes_{K_0}M_0^{\sab})+A\otimes_{K_0}W_1M_0^{\sab}=A\otimes_{K_0}M_0^{\sab}$.
\end{itemize}

Then we obtain a morphism:
\begin{align}\label{boundary:eqn:sabproj}
\on{U}_{G,[\mu]}&\to\on{Gr}^{\sab}\\
F^\bullet(A\otimes_{K_0}M_0)&\mapsto\bigl(F^\bullet(A\otimes_{K_0}M_0)+(A\otimes_{K_0}W_0M_0)\bigr)/(A\otimes_{K_0}W_0M_0).\nonumber
\end{align}

Note that $Q_{\wt,G}$ acts naturally on $M_0^{\sab}$. Let $\on{U}^{\sab}_{G,[\mu]}\subset\on{Gr}^{\sab}$ be the closed sub-scheme such that, for any $A$ as above, $\on{U}^{\sab}_{G,[\mu]}(A)$ consists of the filtrations $F^\bullet(A\otimes_{K_0}M_0^{\sab})$ that, \'etale locally on $\Spec A$, can be split by a co-character $\tilde{\mu}:\Gmh{A}\to Q_{\wt,G,A}\subset G_A$ lying in the conjugacy class $[\mu]$. Then~\eqref{boundary:eqn:sabproj} clearly factors through a $Q_{\wt,G}$-equivariant map $\bm{\pr}:\on{U}_{G,[\mu]}\to\on{U}^{\sab}_{G,[\mu]}$.

\begin{lem}\label{boundary:lem:uwttorsor}
$\bm{\pr}$ exhibits $\on{U}_{G,[\mu]}$ as a $U_{\wt,G}$-torsor over $\on{U}^{\sab}_{G,[\mu]}$. In particular, $\on{U}^{\sab}_{G,[\mu]}$ is smooth and connected of dimension $d-d_{\wt}$.
\end{lem}
\begin{proof}
Since any filtration in $\on{U}^{\sab}_{G,[\mu]}$ is \'etale locally split by a cocharacter of $Q_{\wt,G}$ in the conjugacy class $[\mu]$, it is clear that the fibers of $\bm{\pr}$ are all non-empty.

Now, $U_{\wt,G}$ acts trivially on $M_0^{\sab}$ and thus on $\on{U}^{\sab}_{G,[\mu]}$. So it suffices to show: Given a $K_0$-algebra $A$, and two filtrations
\[
F_i^\bullet(A\otimes_{K_0}M_0)\in\on{U}_{G,[\mu]}(A),\;\text{$i=1,2$}
\]
inducing the same filtration of $A\otimes_{K_0}M_0^{\sab}$, there is a unique element of $U_{\wt,G}(A)$ carrying one to the other.

By \'etale descent, we can assume that there is a cocharacter $\tilde{\mu}:\Gmh{A}\to Q_{\wt,G,A}$ splitting $F_1^\bullet(A\otimes_{K_0}M_0)$:
\[
 A\otimes_{K_0}M_0 = F_1^1(A\otimes_{K_0}M_0)\oplus M_1^0,
\]
where $M_1^0$ is the space on which $\tilde{\mu}(\Gmh{A})$ acts trivially. In particular, $A\otimes_{K_0}W_0M_0\subset M_1^0$.

Now, since
\[
F_1^\bullet(A\otimes_{K_0}M_0)\oplus (A\otimes_{K_0}W_0M_0) = F_2^\bullet(A\otimes_{K_0}M_0)\oplus (A\otimes_{K_0}W_0M_0),
\]
there exists a homomorphism $\Phi:F^1_1(A\otimes_{K_0}M_0)\to(A\otimes_{K_0} W_0M_0)$ such that
\[
 F_2^1(A\otimes_{K_0}M_0) = (1+\Phi)(F_1^1(A\otimes_{K_0}M_0)).
\]
We can extend $\Phi$ to a map on $M_0$ by setting it equal to $0$ on $M_1^0$, and so view it as an endomorphism of $M_0$. Then $\Phi$ is an element of $A\otimes_{K_0}\Lie U_{\wt}$.

Consider the map:
\begin{align*}
 A\otimes_{K_0}\Lie U_{\wt}& \to \bigoplus_{\alpha}(A\otimes_{K_0}M^{\otimes})\\
 f&\mapsto (f(1\otimes s_{\alpha,\st,0}))_{\alpha}.
\end{align*}
This map is equivariant of the action of $\tilde{\mu}(\Gmh{A})$, and since $\tilde{\mu}(z)$ acts on $A\otimes_{K_0}\Lie U_{\wt}$ via $z\mapsto z^{-1}$, its image must land within the eigenspace of the target where $\tilde{\mu}(\Gmh{A})$ acts by the same character.

On the other hand, since $F^\bullet_2=(1+\Phi)(F^\bullet_1)$ is also a $G_A$-split filtration of $A\otimes_{K_0}M_0$, we find that, for all $\alpha$, $\Phi(1\otimes s_{\alpha,\st,0})$ must lie in $F^0_1(A\otimes_{K_0}M^\otimes)$, which is the sum of the non-negative eigenspaces for $\tilde{\mu}(\Gmh{A})$. This shows that we must have $\Phi(1\otimes s_{\alpha,\st,0}) = 0$, for all indices $\alpha$. In other words, $\Phi\in A\otimes_{K_0}\Lie U_{\wt,G}$, and $1+\Phi\in U_{\wt,G}(A)$ is the unique element carrying $F^\bullet_1$ to $F^\bullet_2$.
\end{proof}

\subsubsection{}\label{boundary:subsubsec:periodgrass}
Let $\mathcal{G}$ be the tautological semi-abelian extension of $\mathcal{B}$ over $R^{\sab}$ (cf.~\ref{boundary:eqn:semi_abelian_ext}). Set $\mathcal{M}^{\sab}=\Dieu(\mathcal{G})(\Spec R^{\sab})$: this is a Dieudonn\'e $F$-crystal over $\Spec R^{\sab}$.

Restricting to $\widehat{\mathcal{U}}^{\sab,\an} = (\Spf R^{\sab})^{\an}$, we obtain an $F$-isocrystal $\mathcal{M}^{\sab,\an}$ over $\widehat{\mathcal{U}}^{\sab,\an}$. There is a canonical identification $M_0^{\sab}=\bigl(\mathcal{M}^{\sab,\an}\bigr)^{\nabla=0}$ giving us an isomorphism of $F$-isocrystals:
\[
 \xi^{\sab}:\Reg{\widehat{\mathcal{U}}^{\sab,\an}}\otimes_{K_0}M_0^{\sab}\xrightarrow{\simeq}\mathcal{M}^{\sab,\an}.
\]

By construction, given any pair $(\sigma^\beef,s_0^\beef)$, there is a canonical isomorphism:
\begin{align}\label{boundary:eqn:msabmsigma}
 \mathcal{M}({\sigma^\beef,s_0^\beef})/W_0\mathcal{M}({\sigma^\beef,s_0^\beef})\xrightarrow{\simeq}\mathcal{M}^{\sab}\vert_{S({\sigma^\beef,s_0^\beef})}.
\end{align}
Under this isomorphism, the trivialization of $\mathcal{M}({\sigma^\beef,s_0^\beef})/W_0\mathcal{M}({\sigma^\beef,s_0^\beef})$ induced from~\eqref{boundary:eqn:parallel} agrees with $\xi^{\sab}$.

Let $F^\bullet\mathcal{M}^{\sab,\an}$ be the Hodge filtration; then we obtain the filtration $(\xi^{\sab})^{-1}(F^\bullet\mathcal{M}^{\sab,\an})$ on $\Reg{\widehat{\mathcal{U}}^{\sab,\an}}\otimes_{K_0}M_0^{\sab}$. If $\on{Gr}^{\on{sab},\an}$ is the analytic space over $K_0$ attached to $\on{Gr}^{\on{sab}}$, then this filtration gives rise to a map of $K_0$-analytic spaces
\[
 \eta:\widehat{\mathcal{U}}^{\sab,\an}\to\on{Gr}^{\on{sab},\an}.
\]

\begin{lem}\label{boundary:lem:tsabdim}
The map $\eta$ is unramified, and carries $\widehat{\mathcal{U}}^{\sab,\an}_G = (\Spf R^{\sab}_G)^{\an}$ into $\on{U}^{\sab,\an}_{G,[\mu]}$. In particular, we have:
\[
 \dim R_G^{\sab}\leq d-d_{\wt}+1.
\]
\end{lem}
\begin{proof}
Fix a finite extension $L/K_0$ within $\overline{K}_0$ and a point $x\in\widehat{\mathcal{U}}^{\sab,\an}(L)$. Let $L[\epsilon]=L[\epsilon]/(\epsilon^2)$ be the ring of dual numbers over $L$, and let $\tilde{x}\in\widehat{\mathcal{U}}^{\sab,\an}(L[\epsilon])$ be a point lifting $x$.

By the crystalline property of the de Rham cohomology of $\mathcal{G}$, we obtain a canonical isomorphism:
\[
 j_{\tilde{x}}:L[\epsilon]\otimes_LH^1_{\dR}(\mathcal{G}_x/L)\xrightarrow{\simeq}H^1_{\dR}\bigl(\mathcal{G}_{\tilde{x}}/L[\epsilon]\bigr).
\]

This gives us a filtration:
\[
F^\bullet_{\tilde{x}}\bigl(L[\epsilon]\otimes_LH^1_{\dR}(\mathcal{G}_x/L))=j_{\tilde{x}}^{-1}\biggl(F^\bullet H^1_{\dR}\bigl(\mathcal{G}_{\tilde{x}}/L[\epsilon]\bigr)\biggr).
\]

By basic deformation theory, the induced map:
\begin{align}\label{boundary:eqn:tangent_space_1}
\biggl(\text{Lifts $\tilde{x}\in\widehat{\mathcal{U}}^{\sab,\an}(L[\epsilon])$ of $x$}\biggr)&\to\biggl(\text{Lifts $F^\bullet\bigl(L[\epsilon]\otimes_LH^1_{\dR}(\mathcal{G}_x/L)\bigr)$ of $F^\bullet H^1_{\dR}(\mathcal{G}_x/L)$}\biggr)\\
\tilde{x}&\mapsto F^\bullet_{\tilde{x}}\bigl(L[\epsilon]\otimes_LH^1_{\dR}(\mathcal{G}_x/L)\bigr)\nonumber
\end{align}
is an injection.

On the other hand, by the very definition of $\on{Gr}^{\sab}$, we have:
\begin{align}\label{boundary:eqn:tangent_space_2}
\biggl(\text{Lifts $\widetilde{\eta(x)}\in\on{Gr}^{\sab,\an}(L[\epsilon])$ of $\eta(x)$}\biggr)&=\biggl(\text{Lifts $F^\bullet\bigl(L[\epsilon]\otimes_{K_0}M^{\sab}_0\bigr)$ of $F^\bullet_x(L\otimes_{K_0}M^{\sab}_0)$}\biggr).
\end{align}

Specializing $\xi^{\sab}$ at $x$ produces an isomorphism $\xi^{\sab}_x:L\otimes_{K_0}M_0^{\sab}\xrightarrow{\simeq}H^1_{\dR}(\mathcal{G}_x/L)$. Specializing at $\tilde{x}$, we obtain the isomorphism:
\[
 \xi^{\sab}_{\tilde{x}}:L[\epsilon]\otimes_{K_0}M^{\sab}_0\xrightarrow{\simeq}H^1_{\dR}\bigl(\mathcal{G}_{\tilde{x}}/L[\epsilon]\bigr).
\]

Composing $j_{\tilde{x}}$ with $1\otimes\xi^{\sab}_x$ gives us another isomorphism:
\[
 j_{\tilde{x}}\circ(1\otimes\xi^{\sab}_x):L[\epsilon]\otimes_{K_0}M^{\sab}_0\xrightarrow{\simeq}H^1_{\dR}\bigl(\mathcal{G}_{\tilde{x}}/L[\epsilon]\bigr).
\]

Using the arguments along the lines of those in~\eqref{semistable:prp:phinmodule}, it is not hard to see that we have $j_{\tilde{x}}\circ(1\otimes\xi^{\sab}_x)=\xi^{\sab}_{\tilde{x}}$. This implies that the right hand side of~\eqref{boundary:eqn:tangent_space_1} is canonically identified with that of~\eqref{boundary:eqn:tangent_space_2}. Therefore, $\eta$ induces injections of tangent spaces at all points of $\widehat{\mathcal{U}}^{\sab,\an}$, and is thus unramified.

To see that the restriction of $\eta$ to $\widehat{\mathcal{U}}_G^{\sab,\an}$ factors through $\on{U}^{\sab,\an}_{G,[\mu]}$, choose a triple $(\sigma^\beef,s_0^\beef,t_0)$ in our fixed ${\Ss_K}$-equivalence class, and consider the composition:
\begin{align}\label{boundary:eqn:sigmasabcomp}
 \widehat{\mathcal{U}}^{\an}_G({\sigma^\beef,t_0})\to\widehat{\mathcal{U}}^{\sab,\an}_G\xrightarrow{\eta}\on{Gr}^{\sab,\an}.
\end{align}
Since the map $\widehat{\mathcal{U}}^{\an}_G({\sigma^\beef,t_0})\to\widehat{\mathcal{U}}^{\sab,\an}_G$ is dominant (cf.~\ref{rigid:lem:dominant}), it is enough to show that~\eqref{boundary:eqn:sigmasabcomp} factors through $\on{U}^{\sab,\an}_{G,[\mu]}$. But this is clear from~\eqref{boundary:lem:cocharacters}.

The bound on $\dim R_G^{\sab}$ now follows from~\eqref{boundary:lem:uwttorsor}.
\end{proof}

\subsubsection{}\label{boundary:subsubsec:nudim}
We now turn our attention to the second bound. \emph{The notation from here to the end of the proof of (\ref{boundary:lem:nudim}) will be strictly local, and has no connection to that used anywhere else in this section.}

We will put ourselves in the following situation: Let $\mb{Y}$ be a free abelian group of finite rank and let $\tau\subset\Real\otimes\dual{\mb{Y}}$ be a non-degenerate rational polyhedral cone of maximal dimension. Then the monoid $\mb{Y}_{\tau}=\mb{Y}\cap\dual{\tau}$ has no non-trivial invertible elements.

Fix a finite extension $L/K_0$ within $\overline{K}_0$. Let $R=\Reg{L}\pow{\mb{Y}_{\tau}}$ be the completion of the monoid ring $\Reg{L}[\mb{Y}_{\tau}]$ along the ideal generated by the non-invertible elements of $\mb{Y}_{\tau}$. Let $T$ be a quotient domain of $R$, flat over $W$. Write $\widehat{\mathcal{U}}$ (resp. $\widehat{\mathcal{V}}$) for the formal schemes $\Spf R$ (resp. $\Spf T$). Let $\widehat{\mathcal{U}}^{\an,\circ}\subset\widehat{\mathcal{U}}^{\an}$ be the complement of the boundary divisor, that is, the vanishing locus of elements in $\mb{Y}_{\tau}\backslash\{1\}$. Set $\widehat{\mathcal{V}}^{\an,\circ}=\widehat{\mathcal{U}}^{\an,\circ}\cap\widehat{\mathcal{V}}^{\an}$.

For each point $x\in\widehat{\mathcal{U}}^{\an}(\overline{K}_0)$, we obtain a homomorphism of groups:
\[
 \nu_x:\mb{Y}\xrightarrow{x^\sharp}\overline{K}_0^\times\xrightarrow{\nu_p}\Rat.
\]
Set:
\[
 d(T)=\dim_{\Rat}\langle \nu_x:\;x\in\widehat{\mathcal{V}}^{\an,\circ}(\overline{K}_0)\rangle\subset\Hom(\mb{Y},\Rat).
\]

\begin{lem}\label{boundary:lem:nudim}
Suppose that $\widehat{\mathcal{V}}^{\an,\circ}$ is non-empty. Then we have:
\[
 d(T)\geq\dim T-1.
\]
\end{lem}
\begin{proof}
  It is clear from the definition that $d(R)=\rank \mb{Y}=\dim R-1$.

  We will prove the lemma by induction on the rank $r$ of $\mb{Y}$. If $r=0$, we must have $R=T$, and so we are done. If $r=1$ and $R\neq T$, then $\dim T=1$, and the assertion is vacuous. Therefore, we can assume that $r\geq 2$.

  Set
  \[
   \Sigma=\{D\subset\Hom(\mb{Y},\Rat):\;D\text{ a hyperplane}; D\cap\tau^{\circ}\neq\emptyset\}\subset\bb{P}(\dual{\mb{Y}})(\Rat).
  \]
  For any $D\in\Sigma$, set $\mb{Y}_D=\dual{D}$ and $\tau_D=\tau\cap(\Real\otimes D)$. Observe that $\tau_D$ is again maximal dimensional in $\Real\otimes D$, and so we can form the completed monoid ring $R_D=\Reg{L}\pow{\mb{Y}_D\cap\dual{\tau}_D}$. The corresponding map
  \[
   \Spec R_D\to\Spec R
  \]
  is the normalization of a closed immersion, and its image is an irreducible Weil divisor $S_D\subset\Spec R$.

  Set $\widehat{\mathcal{U}}_D=\Spf R_D$ and
  \[
  \widehat{\mathcal{U}}_D^{\an,\circ}=\widehat{\mathcal{U}}^{\an}_D\times_{\widehat{\mathcal{U}}^{\an}}\widehat{\mathcal{U}}^{\an,\circ}\;;\;\widehat{\mathcal{V}}_D^{\an,\circ}=\widehat{\mathcal{V}}^{\an}\times_{\widehat{\mathcal{U}}^{\an}}\widehat{\mathcal{U}}^{\an,\circ}.
  \]

  Fix $x\in\widehat{\mathcal{V}}^{\an,\circ}(\overline{K}_0)$ and set:
  \[
   \Sigma_x=\{D\in\Sigma:\;\nu_x\notin D\}.
  \]

  We claim that there exists $D\in\Sigma_x$ such that $\widehat{\mathcal{V}}_D^{\an,\circ}\neq\emptyset$. Assume that this is not true. Let $\mathcal{D}\subset\Spec R$ be the boundary divisor. Then our hypothesis says that, for every $D\in\Sigma_x$, the intersection $E_D\coloneqq\Spec T\cap S_D$ is contained in $\mathcal{D}'\coloneqq \Spec(\Field_p\otimes R)\cup\mathcal{D}$ (as a topological sub-space).

  However, the irreducible components of both $E_D$ and $(\Spec T)\cap \mathcal{D}'$ are divisors in $\Spec T$. Therefore, the union $\bigcup_{D\in\Sigma_x}E_D$ contains only finitely many irreducible components (since this is certainly true for $(\Spec T)\cap\mathcal{D}'$).

  Let $\{E_i:1\leq i\leq n\}$ be the irreducible components of this union. Then we can partition $\Sigma_x$ into finitely many sub-sets: $\Sigma_x=\bigsqcup_{i=1}^k\Sigma_i$, such that, for each index $i$, we have:
  \[
   \bigcap_{D\in\Sigma_i}S_D\supset E_i\neq\emptyset.
  \]
  But this can only happen if, for each $i$, the intersection $K(i)\coloneqq\cap_{D\in\Sigma_i}D$ is a non-zero sub-space of $\Hom(\mb{Y},\Rat)$. Moreover, it would also imply that, for every hyperplane $D\in\Sigma_x$, there exists an index $i$ such that $K(i)\subset D$. Therefore, $\Sigma_x$ is a finite union of Zariski closed sub-spaces of $\bb{P}(\dual{\mb{Y}})(\Rat)$; but, by definition, it is also a Zariski open sub-space. This gives us the contradiction.

 Fix $D\in\Sigma_x$ such that $\widehat{\mathcal{V}}_D^{\an,\circ}\neq\emptyset$. Let $\Spec T_D$ be an irreducible component of $\Spec(R_D\otimes_RT)$ such that $(\Spf T_D)^{\an}\cap\widehat{\mathcal{U}}_D^{\an,\circ}$ is non-empty. Then we have:
 \[
 \dim T-1=\dim T_D\leq d(T_D)+1\leq d(T).
 \]
 Here, the inequality in the middle holds by our inductive hypothesis, and the last inequality holds because $\nu_x\notin D$.
\end{proof}

\begin{proof}[Proof of claim~\eqref{main:rationality} of~\eqref{boundary:thm:main}]
It is enough to show: $\dim_{\Rat}(\Rat\otimes\mb{B}_G)= d_{\wt}$.

By choosing $(\sigma^\beef,s_0^\beef,t_0)$ appropriately in its ${\Ss_K}$-equivalence class, we can assume that $\sigma^\beef\subset\Real\otimes\mb{B}$ has maximal dimension. We have a splitting $R({\sigma^\beef,s_0^\beef})=R^{\sab}\widehat{\otimes}B({\sigma^\beef,s_0^\beef})$, where $B({\sigma^\beef,s_0^\beef})=W\pow{\mb{S}(\sigma^\beef)}$ is simply the completed monoid ring attached to the monoid $\mb{S}(\sigma)$ (cf.~\ref{boundary:subsubsec:noncansplit}).

Set $\widehat{\mathcal{U}}^{\sab}=\Spf R_G^{\sab}$ and $\widehat{\mathcal{U}}^{\sab}=\Spf R^{\sab}$. Suppose that we are given a finite extension $L/K_0$, and $y\in\widehat{\mathcal{U}}^{\sab}(\Reg{L})$ corresponding to a map $y^\sharp:R^{\sab}\to\Reg{L}$. Then we can identify $R({\sigma^\beef,s_0^\beef})\otimes_{R^{\sab},y^\sharp}\Reg{L}$ with the completed monoid ring $\Reg{L}\pow{\mb{S}(\sigma^\beef)}$.

Choose any point $x\in\widehat{\mathcal{U}}^{\an,\circ}_G({\sigma^\beef,t_0})(L)$ with associated map $x^\sharp:R_G({\sigma^\beef,t_0})\to\Reg{L}$. Let $y\in\widehat{\mathcal{U}}^{\sab}(\Reg{L})\subset\widehat{\mathcal{U}}^{\sab}(\Reg{L})$ be the restriction of $x$. Let
\[
\Spec\tilde{A}\subset\Spec(R_G({\sigma^\beef,t_0})\otimes_{R_G^{\sab},y^\sharp}\Reg{L})
\]
be the irreducible component such that $x$ belongs to $(\Spf\tilde{A})^{\an}(L)$. Then we have a finite map:
\[
 \Reg{L}\pow{\mb{S}(\sigma^\beef)}=R({\sigma^\beef,s_0^\beef})\otimes_{R^{\sab},y^\sharp}\Reg{L}\to\tilde{A}.
\]

Let $A\subset\tilde{A}$ be the image of this map. We can now apply~\eqref{boundary:lem:nudim} with $\mb{Y}=\mb{S}$, $\tau=\sigma^\beef$ and $T=A$. This gives us the inequality:
\begin{align}\label{boundary:eqn:nudimA}
 d(A)\geq\dim(A)-1.
\end{align}
Here, $d(A)$ is the dimension of the vector space:
\[
 \langle \nu_p\circ N_{x'}:\; x'\in(\Spf A)^{\an}(\overline{K}_0)\cap\widehat{\mathcal{U}}^{\an,\circ}(\overline{K}_0)\rangle\subset\Rat\otimes\mb{B}.
\]

By~\eqref{boundary:cor:tatetensorprps}, for all $x'\in\widehat{\mathcal{U}}^{\an,\circ}(\overline{K}_0)$, we have:
\[
 \nu_p\circ N_{x'}\in(\Lie G_{K_0})(-1)\cap(\Rat\otimes\mb{B})=\Rat\otimes\mb{B}_G.
\]
Therefore,~\eqref{boundary:eqn:nudimA} shows:
\begin{align}\label{boundary:eqn:dimBG}
 \dim_{\Rat}(\Rat\otimes\mb{B}_G)\geq\dim(A)-1.
\end{align}

It follows from~\eqref{boundary:lem:tsabdim} that
\begin{align}\label{boundary:eqn:tsabdim}
\dim R_G^{\sab}\leq d-d_{\wt}+1.
\end{align}

Let $\mf{P}\subset R_G({\sigma^\beef,t_0})$ be the kernel of the map $R_G({\sigma^\beef,t_0})\to A$. Then $\mf{P}$ is minimal over $\pf_yR_G({\sigma^\beef,t_0})$. Therefore, using \cite[Theorem 15.1]{matsumura:ring_theory}, we find:
\begin{align}\label{boundary:eqn:dimbounds}
\dim\bigl(R_G({\sigma^\beef,t_0})_{\mf{P}}\bigr)\leq \dim\bigl((R_G^{\sab})_{\pf_y}\bigr)\leq\dim(R_G^{\sab}[p^{-1}])=d-d_{\wt}.
\end{align}
Using the fact that $R_G({\sigma^\beef,t_0})$ is catenary~\cite[Theorem 29.4]{matsumura:ring_theory}, we get:
\begin{align}\label{boundary:eqn:dimbounds2}
\dim(A)&=\dim(R_G({\sigma^\beef,t_0})/\mf{P})=\dim(R_G({\sigma^\beef,t_0}))-\dim\bigl((R_G({\sigma^\beef,t_0}))_{\mf{P}}\bigr)\nonumber\\
&= d+1-\dim\bigl((R_G({\sigma^\beef,t_0}))_{\mf{P}}\bigr)\geq d+1-(d-d_{\wt})\geq d_{\wt}+1.
\end{align}

Combining this with~\eqref{boundary:eqn:dimBG} shows:
\[
 d_{\wt}\geq\dim_{\Rat}(\Rat\otimes\mb{B}_G)\geq d_{\wt}.
\]

This finishes the proof.
\end{proof}

\begin{rem}
\label{boundary:rem:pel_case}
In the PEL case, the above proof can be simplified considerably, and in fact will not require any $p$-adic Hodge theory beyond the functoriality of logarithmic Dieudonn\'e theory.

Note that $\mb{B}$ can be identified with a subspace of the space $B(\mathsf{X})$ of symmetric bilinear forms on $\mathsf{X}$, which in turn can be identified with the space of symmetric maps $\mathsf{X}\to \mathsf{X}^\vee$. The identification of this space with a subspace of $\Lie U_{\wt}$ proceeds by identifying the latter with the space of maps $M_0\to M_0$ that factor through a symmetric homomorphism
\[
\gr^W_2M_0 = \mathsf{X}\otimes K_0\to \Hom(\mathsf{X},K_0) = W_0M_0.
\]

Suppose that $\{s_{\alpha}\}\subset \End(V)\subset V^\otimes$ can be taken to be a collection of endomorphisms. Then the corresponding Tate tensors
\[
\{s_{\alpha,\st,0}\}\subset \End(M_0)\subset M_0^\otimes
\]
are simply the log crystalline realizations of the endomorphisms of the universal log $1$-motif that are induced from the $\{s_{\alpha}\}$. 

The subspace $\Lie U_{\wt,G}\subset \Lie U_{\wt}$ consists of those elements that anti-commute with the collection $\{s_{\alpha,\st,0}\}$. Therefore, the desired rationality statement amounts to showing that each endomorphism $s_{\alpha,\st,0}$ of $M_0$ respects the weight filtration $W_\bullet M_0$, and that the induced endomorphism of 
\[
W_0M_0 = \Hom(\mathsf{X},K_0)\;;\; \gr^W_2M_0 = \mathsf{X}\otimes K_0
\]
respects the integral structures
\[
\mathsf{X}^\vee \subset \Hom(\mathsf{X},K_0)\;;\; \mathsf{X}\subset \mathsf{X}\otimes K_0.
\]
But this can easily be deduced from the fact that $s_{\alpha,\st,0}$ is the realization of an endomorphism of a log $1$-motif whose log Dieudonn\'e realization is $M_0$, and whose \'etale and multiplicative parts are identified with $\mathsf{X}$ in a way compatible with the realization.
\end{rem}

\subsection{Reduction of structure group}\label{boundary:subsec:reduction}

\subsubsection{}\label{boundary:subsubsec:unitst}
Fix a trivialization $\iota^{\sab}$ as in~\eqref{boundary:subsubsec:noncansplit}, with the corresponding splitting:
\begin{align}\label{boundary:eqn:noncansplitunits}
R({\sigma^\beef,s_0^\beef})=R^{\sab}\widehat{\otimes}B({\sigma^\beef,s_0^\beef}).
\end{align}

Let $\alpha:\mb{S}\to\Gmh{\mb{\Xi}}$ be the homomorphism attached to the trivialization $\iota^{\sab}$. Let $Q(R_G({\sigma^\beef,t_0}))$ be the fraction field of $R_G({\sigma^\beef,t_0})$. Via the restriction $H^0(\mb{\Xi},\Gm)\to Q(R_G({\sigma^\beef,t_0}))^\times$, we can view $\alpha$ as a map $\mb{S}\to Q(R_G({\sigma^\beef,t_0}))^\times$. As mentioned in remark~\eqref{main:concrete} of~\eqref{boundary:subsubsec:remarks}, to prove assertion~\eqref{main:reduction} of the theorem, we have to show that there exists $m\in\Int_{>0}$ such that, for all $h\in\mb{S}^G$, we have $\alpha(h)^m\in R_G^{\sab,\times}\subset Q(R_G({\sigma^\beef,t_0}))^\times$.

We will do this in stages.

\begin{lem}\label{boundary:lem:tsigmaunits}
For any $h\in\mb{S}^G$, $\alpha(h)$ lies in $R_G({\sigma^\beef,t_0})^\times\subset Q(R_G({\sigma^\beef,t_0}))^\times$.
\end{lem}
\begin{proof}
Given $x\in\widehat{\mathcal{U}}^{\an,\circ}_G({\sigma^\beef,t_0})(\overline{K}_0)$, the homomorphism $\alpha$ specializes to a map $\alpha_x:\mb{S}\to\overline{K}_0^\times$. The induced map:
\begin{align}\label{boundary:eqn:xsharp}
\mb{S}\xrightarrow{\alpha_x}\overline{K}_0^\times\rightarrow\overline{K}_0^\times/\Reg{\overline{K}_0}^\times
\end{align}
is precisely the one corresponding to the monodromy operator $N_x:M_0\to M_0\otimes(\overline{K}^\times_0/\Reg{\overline{K}_0}^\times)$; cf.~\eqref{boundary:subsubsec:manlogtriv}. By~\eqref{boundary:cor:tatetensorprps}, it follows that \eqref{boundary:eqn:xsharp} factors through $\mb{S}_G\coloneqq\mb{S}/\mb{S}^G$.

In particular, we find that, for any point $x\in\widehat{\mathcal{U}}^{\an,\circ}_G({\sigma^\beef,t_0})(\overline{K}_0)$, $\alpha(h)(x)=\alpha_x(h)\in\overline{K}_0^\times$ belongs to $\Reg{\overline{K}_0}^\times$. The lemma now follows from~\eqref{rigid:lem:units}.
\end{proof}

\subsubsection{}
Consider the logarithm homomorphism:
\begin{align*}
\log:\widehat{\bb{G}}_m^{\an}&\to\mathbb{G}_a^{\an}\\
x&\mapsto \sum_{i=1}^\infty(-1)^{i+1}\frac{(x-1)^i}{i}.
\end{align*}

It induces a map of analytic groups:
\begin{align}\label{boundary:eqn:ellmap}
\ell:\widehat{\mb{E}}^{G,\an}=\SHom\bigl(\mb{S}^G,\widehat{\bb{G}}_m^{\an}\bigr)\xrightarrow{\log}\SHom\bigl(\mb{S}^G,\bb{G}_a^{\an}\bigr)=(\Lie U_{\wt}/\Lie U_{\wt,G})(-1)\otimes\bb{G}_a^{\an}.
\end{align}
Here, the last identity follows from assertion~\eqref{main:rationality} of~\eqref{boundary:thm:main}.

If we think of $\alpha$ as a map $\alpha:\mb{\Xi}\to\mb{E}$, it follows from~\eqref{boundary:lem:tsigmaunits} that the composition:
\[
 \Spec Q(R_G({\sigma^\beef,t_0}))\to\mb{\Xi}\to\mb{E}\to\mb{E}^G
\]
arises from a map $\overline{\alpha}:\Spec R_G({\sigma^\beef,t_0})\to\mb{E}^G$. Moreover, since $\overline{\Field}_p^\times$ is torsion, for a sufficiently divisible $m\in\Int_{\geq 1}$, $\overline{\alpha}^m$ will give rise to a map of formal schemes $\widehat{\mathcal{U}}_G({\sigma^\beef,t_0})=\Spf R_G({\sigma^\beef,t_0})\to\widehat{\mb{E}}^G$, which we denote again by $\overline{\alpha}^m$.

\begin{prp}\label{boundary:prp:tsabpinvunits}
There exists a (unique) map
\[
\ell_{\overline{\alpha},m}:\widehat{\mathcal{U}}^{\sab,\an}_G\to(\Lie U_{\wt}/\Lie U_{\wt,G})(-1)\otimes\bb{G}_a^{\an}
\]
such that the following diagram commutes:
\begin{diagram}
\widehat{\mathcal{U}}^{\an}_G({\sigma^\beef,t_0})&\rTo^{\overline{\alpha}^m}&\widehat{\mb{E}}^{G,\an}\\
\dTo&&\dTo_{\ell}\\
\widehat{\mathcal{U}}^{\sab,\an}_G&\rTo_{\ell_{\overline{\alpha},m}}&(\Lie U_{\wt}/\Lie U_{\wt,G})(-1)\otimes\bb{G}_a^{\an}.
\end{diagram}
\end{prp}
\begin{proof}
Let $F^\bullet_{\cl}(\Reg{\widehat{\mathcal{U}}^{\sab,\an}_G}\otimes_{K_0}M_0)$ be the filtration on $\Reg{\widehat{\mathcal{U}}^{\sab,\an}}\otimes_{K_0}M_0$ constructed in~\eqref{boundary:subsubsec:m0_classical}. The induced filtration on
\[
 \mathcal{M}^{\sab,\an} = \Reg{\widehat{\mathcal{U}}^{\sab,\an}_G}\otimes_{K_0}M_0^{\sab}
\]
is the canonical Hodge filtration considered in~\eqref{boundary:subsubsec:periodgrass}. Therefore, by~\eqref{boundary:lem:tsabdim}, the associated map $\widehat{\mathcal{U}}^{\sab,\an}_G\to\on{Gr}^{\sab,\an}$ maps into $\on{U}^{\sab,\an}_{G,[\mu]}$.

Moreover, by (the proof of)~\eqref{boundary:lem:uwttorsor}, there is a unique $U_{\wt,G}$-orbit of endomorphisms $B\in\Reg{\widehat{\mathcal{U}}^{\sab,\an}_G}\otimes\Lie U_{\wt}$ such that $\exp(B)(F^\bullet_{\cl})$ is split by a cocharacter $\mu:\Gm\to Q_{\wt,G}$ in the conjugacy class $[\mu]$. Let
\[
 \ell_{\overline{\alpha}}:\widehat{\mathcal{U}}^{\sab,\an}_G\to(\Lie U_{\wt}/\Lie U_{\wt,G})\otimes\bb{G}_a^{\an}
\]
be the map associated with this orbit. It can now be checked that $\ell_{\overline{\alpha},m} \coloneqq -m\ell_{\overline{\alpha}}$ is the map whose existence is being asserted by the proposition.
\end{proof}

The following corollary completes the proof of~\eqref{main:reduction} and shows that $\mb{\Xi}\vert_{\Spec R^{\sab}_G}$ admits a canonical reduction of structure group to an $\mb{E}_G$-torsor $\mb{\Xi}_G$.

\begin{corollary}\label{boundary:cor:tsabunits}
Assertion~\eqref{main:reduction} of~\eqref{boundary:thm:main} holds. More precisely: There exists $m\in\Int_{\geq 1}$ such that, for all $h\in\mb{S}^G$, we have:
\[
\alpha(h)^m\in R_G^{\sab,\times}.
\]
\end{corollary}
\begin{proof}
This is now immediate from~\eqref{rigid:lem:logarithm}.
\end{proof}

\subsection{The end of the proof}\label{boundary:subsec:final}

We are almost there. The only remaining point is to find the torus $\mb{E}^\diamond_G$ and the $\mb{E}^\diamond_G$-torsor $\mb{\Xi}^\diamond_G$ over $R^{\sab}_G$ promised by claim~\eqref{main:local} of~\eqref{boundary:thm:main}. The purpose of this subsection is to resolve this, and thereby complete the proof. The main additional observation required is~\eqref{boundary:lem:finite_cover_torsor}. Essentially, what we have shown above covers the case where $K_p = K^\beef_p \cap G(\Rat_p)$. The main observation here shows that the desired result is true when $G = G^\beef$ and $K_p = K^{\beef,\diamond}_p$ is an arbitrary compact open subgroup of $K^\beef_p$. Combining the two situations gives the result for general $G$ and $K_p$.

\subsubsection{}\label{boundary:subsubsec:finite_cover}
Choose a compact open subgroup $K^{\beef,\diamond}\subset K^\beef$. Then, $\Sh_{K^{\beef,\diamond}}\to \Sh_{K^\beef}$ is a finite \'etale cover. Let $\Ss_{K^{\beef,\diamond}}^{\Sigma^\beef}\to \Ss^{\Sigma}_{K^\beef}$ be the normalization of $\Ss_{K^\beef}$ in $\Sh_{K^{\beef,\diamond}}$.

We also have the tower of mixed Shimura varieties
\begin{align}\label{boundary:eqn:siegel_tower_char0}
\Sh_{K^{\beef,\diamond}_{\Phi^\beef}}(Q_{\Phi^\beef},D_{\Phi^\beef})\to \Sh_{\overline{K}^{\beef,\diamond}_{\Phi^\beef}}(\overline{Q}_{\Phi^\beef},\overline{D}_{\Phi^\beef}),
\end{align}
which is a torsor under the torus $\mb{E}_{K^{\beef,\diamond}}(\Phi^\beef)$.

This tower is finite \'etale over the $\mb{E}=\mb{E}_{K^\beef}(\Phi^\beef)$-torsor
\begin{align*}
\Sh_{K^\beef_{\Phi^\beef}}(Q_{\Phi^\beef},D_{\Phi^\beef})\to \Sh_{\overline{K}^\beef_{\Phi^\beef}}(\overline{Q}_{\Phi^\beef},\overline{D}_{\Phi^\beef}),
\end{align*}
which admits the canonical integral model
\begin{align}\label{boundary:eqn:siegel_tower1_integral}
\Ss_{K^\beef_{\Phi^\beef}}(Q_{\Phi^\beef},D_{\Phi^\beef})\to \Ss_{\overline{K}^\beef_{\Phi^\beef}}(\overline{Q}_{\Phi^\beef},\overline{D}_{\Phi^\beef}),
\end{align}

We can now consider the normalization
\begin{align}\label{boundary:eqn:siegel_tower2_integral}
\Ss_{K^{\beef,\diamond}_{\Phi^\beef}}(Q_{\Phi^\beef},D_{\Phi})\to \Ss_{\overline{K}^{\beef,\diamond}_{\Phi^\beef}}(\overline{Q}_{\Phi^\beef},\overline{D}_{\Phi^\beef})
\end{align}
of~\eqref{boundary:eqn:siegel_tower1_integral} in~\eqref{boundary:eqn:siegel_tower_char0}.

\begin{lem}
\label{boundary:lem:finite_cover_torsor}
The map in~\eqref{boundary:eqn:siegel_tower2_integral} is an $\mb{E}_{K^{\beef,\diamond}}(\Phi^\beef)$-torsor.
\end{lem}
\begin{proof}
For any integer $n\geq 1$, set
\[
K^\beef(n) = \ker(\GSp(H(\widehat{\Int}))\to \GSp(H(\Int/n\Int))).
\] 
Choose an integer $n\geq 3$ such that
\[
K^\beef(n)\subset K^{\beef,\diamond},
\]
so that we have
\begin{align}\label{boundary:eqn:cocharacter_groups}
n\cdot B(\mathsf{X}) = \mb{B}_{K^\beef(n)}(\Phi^\beef)\subset \mb{B}_{K^{\beef,\diamond}}(\Phi^\beef)\subset B(\mathsf{X}) = \mb{B}_{K^\beef(1)}(\Phi^\beef).
\end{align}
Here, as in~\eqref{boundary:rem:pel_case}, $B(\mathsf{X})$ is the space of symmetric bilinear forms on $\mathsf{X}$, which can be identified with $\mb{B}_{K^\beef(0)}(\Phi^\beef)$.

Let $\mathcal{Q}^{\ab}$ be the universal principally polarized abelian scheme over $\Ss_{K^\beef_h}(G_{\Phi^\beef,h},D_{\Phi^\beef,h})$. Then the scheme
\[
\SHom(\mathsf{X},\mathcal{Q}^{\ab})\to \Ss_{K^\beef_h}(G_{\Phi^\beef,h},D_{\Phi^\beef,h}) 
\]
parameterizes semi-abelian schemes $\mathcal{Q}^{\sab}$ with principally polarized abelian part $\mathcal{Q}^{\ab}$ and multiplicative part $\mathsf{X}$.

There now exists a sequence of finite flat morphisms of smooth $\Ss_{K^\beef_h}(G_{\Phi^\beef,h},D_{\Phi^\beef,h})$-schemes
\[
\SHom(\frac{1}{n}\mathsf{X},\mathcal{Q}^{\ab})\to \Ss_{\overline{K}^\beef_{\Phi^\beef}}(\overline{Q}_{\Phi^\beef},\overline{D}_{\Phi^\beef}) \to\SHom(\mathsf{X},\mathcal{Q}^{\ab}).
\]
Indeed, the second morphism is tautological from the moduli description, since it parameterizes certain prime-to-$p $level structures on $\mathcal{Q}^{\sab}$ lifting those on $\mathcal{Q}^{\ab}$. The first morphism arises from the following fact: Giving a map
$\frac{1}{n}\mathsf{X}\to \mathcal{Q}^{\ab}$
lifting the classifying map of the semi-abelian scheme $\mathcal{Q}^{\sab}$ is equivalent to giving a splitting of the surjection $\mathcal{Q}^{\sab}[n]\to \mathcal{Q}^{\ab}[n]$ of finite flat group schemes. In particular, one can use this splitting to give the desired level structures on $\mathcal{Q}^{\sab}$ that the scheme in the middle is supposed to parameterize.

Now, consider the scheme
\begin{align}\label{boundary:eqn:symmetric_triv}
\SHom^{\mathrm{symm}}(\frac{1}{n}\mathsf{X},\mathcal{Q}^{\sab})\to \SHom(\frac{1}{n}\mathsf{X},\mathcal{Q}^{\ab})
\end{align}
parameterizing lifts $\tilde{c}_n:\frac{1}{n}\mathsf{X}\to\mathcal{Q}^{\sab}$ of the universal morphism $c_n:\frac{1}{n}\mathsf{X}\to\mathcal{Q}^{\ab}$ that are \emph{symmetric} in the following sense: In the notation of~\eqref{semistable:subsec:1motifs}, the restriction of $\tilde{c}_n$ to $\mathsf{X}$ corresponds to a trivialization $\tau$ of the universal biextension of $\mathsf{X}\times \mathsf{X}$ obtained by pulling the Poincar\'e bundle on $\mathcal{Q}^{\ab}\times \mathcal{Q}^{\ab,\vee}$ along the morphism
\[
\mathsf{X}\times \mathsf{X}\xrightarrow{c\times \lambda^{\ab}c}\mathcal{Q}^{\ab}\times \mathcal{Q}^{\ab,\vee}.
\]
Here, $\lambda^{\ab}:\mathcal{Q}^{\ab}\xrightarrow{\simeq} \mathcal{Q}^{\ab,\vee}$ is the tautological principal polarization. Then the scheme in~\eqref{boundary:eqn:symmetric_triv} parameterizes homomorphisms such that the corresponding trivialization $\tau$ is symmetric.

It is easy to check that~\eqref{boundary:eqn:symmetric_triv} is a torsor under the torus
\[
\mb{E}_{K^\beef(n)}(\Phi^\beef) = \SHom(\frac{1}{n}B(\mathsf{X}),\Gm) = \SHom^{\mathrm{symm}}(\frac{1}{n}\mathsf{X},\SHom(\mathsf{X},\Gm))
\]
that parameterizes pairings $\frac{1}{n}\mathsf{X}\times \mathsf{X}\to \Gm$ that are symmetric on $\mathsf{X}\times \mathsf{X}$.

Therefore, one can push it forward along the isogeny $\mb{E}_{K^\beef(n)}(\Phi^\beef)\to \mb{E}_{K^{\beef,\diamond}}(\Phi^\beef)$ obtained via the inclusion of cocharacter groups in~\eqref{boundary:eqn:cocharacter_groups}, and obtain a canonical $\mb{E}_{K^{\beef,\diamond}}(\Phi^\beef)$-torsor over $\SHom(\frac{1}{n}\mathsf{X},\mathcal{Q}^{\ab})$.

We now claim that the obtained $\mb{E}_{K^{\beef,\diamond}}(\Phi^\beef)$-torsor over the finite flat cover
\[
\SHom(\frac{1}{n}\mathsf{X},\mathcal{Q}^{\ab})\times_{\Ss_{\overline{K}^\beef_{\Phi^\beef}}(\overline{Q}_{\Phi^\beef},\overline{D}_{\Phi^\beef}) }\Ss_{\overline{K}^{\beef,\diamond}_{\Phi^\beef}}(\overline{Q}_{\Phi^\beef},\overline{D}_{\Phi^\beef}) \to \Ss_{\overline{K}^{\beef,\diamond}_{\Phi^\beef}}(\overline{Q}_{\Phi^\beef},\overline{D}_{\Phi^\beef}) 
\]
descends to an $\mb{E}_{K^{\beef,\diamond}}(\Phi^\beef)$-torsor over the base that is identified with the morphism~\eqref{boundary:eqn:siegel_tower2_integral}.

This is a statement that can be checked over the generic fiber, where it can be deduced from the moduli description of the spaces involved.
\end{proof}

\subsubsection{}\label{boundary:subsubsec:closed_normalization}
Assume now that $K^{\beef,\diamond}$ has been chosen so that $K = K^{\beef,\diamond}\cap G(\Adele_f)$, and such that the map
\[
\Sh_K \to \Sh_{K^{\beef,\diamond}}
\]
is a closed immersion. This is always possible by~\cite[Prop. 1.15]{deligne:travaux}.

From~\eqref{boundary:lem:finite_cover_torsor}, we obtain a finite map
\begin{align}\label{boundary:eqn:finite_cover_phi}
\Ss_{K^{\beef,\diamond}_{\Phi^\beef}}(Q_{\Phi^\beef},D_{\Phi^\beef},\tilde{\sigma}^\beef)\to \Ss_{K^\beef_{\Phi^\beef}}(Q_{\Phi^\beef},D_{\Phi^\beef},\tilde{\sigma}^\beef)
\end{align}
of normal twisted toric varieties.

Since $\mathcal{S}_K\to \mathcal{S}_{K^\beef}$ lifts to a map $\mathcal{S}_K\to \mathcal{S}_{K^{\beef,\diamond}}$, one finds that the composition
\[
\Spec R_G(\tilde{\sigma}^\beef,\tilde{t}_0)\to \Spec R(\tilde{\sigma}^\beef,\tilde{s}^\beef_0)\to \Ss_{K^\beef_{\Phi^\beef}}(Q_{\Phi^\beef},D_{\Phi^\beef},\tilde{\sigma}^\beef)
\]
lifts to a map
\[
\Spec R_G(\tilde{\sigma}^\beef,\tilde{t}_0)\to \Ss_{K^{\beef,\diamond}_{\Phi^\beef}}(Q_{\Phi^\beef},D_{\Phi^\beef},\tilde{\sigma}^\beef).
\]

Let
\[
\mathcal{Z}_{K^{\beef,\diamond}_{\Phi^\beef}}(Q_{\Phi^\beef},D_{\Phi^\beef},\tilde{\sigma}^\beef)\subset \Ss_{K^{\beef,\diamond}_{\Phi^\beef}}(Q_{\Phi^\beef},D_{\Phi^\beef},\tilde{\sigma}^\beef)
\] 
be the closed stratum, and let $\tilde{s}_0^{\beef,\diamond}\in\mathcal{Z}_{K^{\beef,\diamond}_{\Phi^\beef}}(Q_{\Phi^\beef},D_{\Phi^\beef},\tilde{\sigma}^\beef)(\overline{\Field}_p)$ be the image of $\tilde{t}_0$. 

Let $R(\tilde{\sigma}^\beef,\tilde{s}_0^{\beef,\diamond})$ be the complete local ring of $ \Ss_{K^{\beef,\diamond}_{\Phi^\beef}}(Q_{\Phi^\beef},D_{\Phi^\beef},\tilde{\sigma}^\beef)$ at $\tilde{s}_0^{\beef,\diamond}$, so that we have a finite map of complete local rings
\[
R(\tilde{\sigma}^\beef,\tilde{s}_0^{\beef,\diamond})\to R_G(\tilde{\sigma}^\beef,\tilde{t}_0).
\]
The fact that $\Sh_K\to \Sh_{K^{\beef,\diamond}}$ is a closed immersion now implies that $R_G(\tilde{\sigma}^\beef,\tilde{t}_0)$ is the normalization of the image of the above map. Moreover, $R_G^{\sab}$ is also identified with the normalization of a quotient of the complete local ring of $\Ss_{\overline{K}^{\beef,\diamond}_{\Phi}}(\overline{Q}_{\Phi},\overline{D}_{\Phi})$ at the image of $\tilde{s}^{\beef,\diamond}_0$.

\subsubsection{}\label{boundary:subsubsec:pre_proof}
The restriction of $\Ss_{K^{\beef,\diamond}_{\Phi^\beef}}(Q_{\Phi^\beef},D_{\Phi^\beef})$ over $\Spec R_G^{\sab}$ is an $\mb{E}_{K^{\beef,\diamond}}(\Phi^\beef)$-torsor, which we denote by $\mb{\Xi}^\diamond$. 

Let $\mb{E}_G^\diamond$ be the torus with co-character group
\[
\mb{B}_G\cap \mb{B}_{K^{\beef,\diamond}}(\Phi^\beef)\subset \mb{B}_G.
\]
It is a sub-torus of $\mb{E}_{K^\beef,\diamond}(\Phi^\diamond)$ equipped with an isogeny $\mb{E}^\diamond_G\to \mb{E}_G$.

% It will be helpful to unpack the definition of $\mb{\Xi}_G$ a bit. Write $\alpha^{G}$ for the composition
% \[
%  \mb{\Xi}\xrightarrow{\alpha}\mb{E}\to\mb{E}^G.
% \]
% By~\eqref{boundary:cor:tsabunits}, we can view $\overline{\alpha}^m$ as a map:
% \[
%  \Spec R_G^{\sab}\to\mb{E}^G.
% \]
% Then, for any $R_G^{\sab}$-scheme $Y$, we have:
% \begin{align}\label{boundary:eqn:xigmpoints}
% \mb{\Xi}_G(Y)&=\{f\in\mb{\Xi}(Y):\;\alpha^{G}\circ f=\overline{\alpha}\vert_Y\}.
% \end{align}

% Let $\mb{\Xi}_G\into\mb{\Xi}_G(\sigma_G)$ be the twisted torus embedding over $R^{\sab}_G$ attached to the cone $\sigma_G$.

We can now complete the proof of our main result:
\begin{proof}[Proof of~\eqref{boundary:thm:main}]
Assertions~\eqref{main:rationality} and~\eqref{main:reduction} have already been shown, so we only have to prove~\eqref{main:local}.

Since $\Spec R_G(\tilde{\sigma}^\beef,t_0)$ is a scheme over the twisted torus embedding $\mb{\Xi}^\diamond(\tilde{\sigma}^\beef)$, we obtain a tautological section
\begin{align}\label{boundary:eqn:taut_section_diamond}
\Spec Q(R_G(\tilde{\sigma}^\beef,\tilde{t}_0))\to \mb{\Xi}^\diamond/\mb{E}_G^\diamond.
\end{align}
By the results of~\eqref{boundary:subsec:reduction}, the composition of this section with the finite map
\[
\mb{\Xi}^\diamond/\mb{E}_G^\diamond\to \mb{\Xi}/\mb{E}_G
\]
is defined over $\Spec R^{\sab}_G$. Since $R^{\sab}_G$ is integrally closed in $Q(R_G(\tilde{\sigma}^\beef,\tilde{t}_0))$, this implies that~\eqref{boundary:eqn:taut_section_diamond} is also defined over $R^{\sab}_G$. In other words, $\mb{\Xi}^\diamond$ has a canonical reduction of structure group $\mb{\Xi}^\diamond_G$ to an $\mb{E}^\diamond_G$-torsor over $R^{\sab}_G$, given by the pre-image of this section.

Suppose that $\tilde{\sigma}_G = \tilde{\sigma}^\beef\cap(\mb{B}_G\otimes \Real)$, and $\mb{\Xi}_G^\diamond(\tilde{\sigma}_G)$ is the associated twisted torus embedding over $R^{\sab}_G$. 

From the definition of $\mb{\Xi}^\diamond_G$, we find that the natural map $\Spec Q(R_G({\tilde{\sigma}^\beef,\tilde{t}_0}))\to\mb{\Xi}$ factors through $\mb{\Xi}^\diamond_G$. Therefore, the composition of finite maps
\begin{align}\label{boundary:eqn:tsigrsigxisig}
 \Spec R_G({\tilde{\sigma}^\beef,\tilde{t}_0})\to\Spec R({\tilde{\sigma}^\beef,\tilde{s}_0^\beef})\to\mb{\Xi}(\sigma^\beef)
\end{align}
must lift to a map $\Spec R_G({\tilde{\sigma}^\beef,\tilde{t}_0})\to\mb{\Xi}^\diamond_G(\tilde{\sigma}_G)$.

The discussion in~\eqref{boundary:subsubsec:closed_normalization} shows that $R_G(\tilde{\sigma}_0,\tilde{t}_0)$ is the normalization of a quotient of a complete local ring of $\mb{\Xi}^\diamond_G(\tilde{\sigma}_G)$. But, on the other hand, we have:
\[
 \dim\mb{\Xi}^\diamond_G=\dim(R_G^{\sab})+\rank\mb{B}_G\leq d-d_{\wt}+1+d_{\wt}=d+1=\dim(R_G({\tilde\sigma^\beef,\tilde{t}_0})).
\]
Here, the inequality in the middle follows from~\eqref{boundary:lem:tsabdim}. Therefore, under~\eqref{boundary:eqn:tsigrsigxisig}, $\Spec R_G({\tilde{\sigma}^\beef,\tilde{t}_0})$ must map finitely and dominantly onto the completion of the normal scheme $\mb{\Xi}^\diamond_G(\tilde{\sigma}_G)$ at an $\overline{\Field}_p$-valued point. 

This completes the proof of the theorem.
\end{proof}

\section{Compactifications of Hodge type and their stratifications}\label{sec:strata}

In this section, we will deduce the main theorems of the paper from the results of \S~\ref{sec:boundary}. We preserve the notation specified in~\eqref{boundary:subsec:abshodge}.

\subsection{Toroidal compactifications of Hodge type}\label{strata:subsec:toroidal}

\subsubsection{}
Fix a clr $\Phi$ for $(G,X)$, and let ${\Phi}^\beef$ be the induced clr for $(G^\beef,X^\beef)$.

Fix a neat level $K=K_pK^p\subset G(\Adele_f)$, as well as a neat sub-group ${K}^{\beef}={K}^{\beef}_p{K}^{\beef,p}\subset\mathcal{G}(\Adele_f)$ containing $K$ such that the map $\Sh_K\to\Sh_{{K}^{\beef}}$ is a closed immersion.

We then obtain a finite map of mixed Shimura varieties over $E$~\eqref{background:eqn:phitildephi}:
\[
 \Sh_{K_\Phi}(Q_\Phi,D_\Phi)\to E\otimes\Sh_{K^\beef_{\Phi^\beef}}(Q_{\Phi^\beef},D_{\Phi^\beef}).
\]
This respects the natural tower structures on each side.

Let $\Ss_{K^\beef_{\Phi^\beef}}(Q_{\Phi^\beef},D_{\Phi^\beef})$ be the natural integral model for $\Sh_{K^\beef_{\Phi^\beef}}(Q_{\Phi^\beef},D_{\Phi^\beef})$ defined in~\eqref{background:subsubsec:1motifsmodulip}. It has a tower structure:
\begin{align}\label{strata:eqn:towertildephi}
\Ss_{K^\beef_{\Phi^\beef}}(Q_{\Phi^\beef},D_{\Phi^\beef})\to\Ss_{\overline{K}^\beef_{\Phi^\beef}}(\overline{Q}_{\Phi^\beef},\overline{D}_{\Phi^\beef})\to\Ss_{K^{\beef,h}_{\Phi^\beef}}(G_{\Phi^\beef,h},D_{\Phi^\beef,h})
\end{align}

Let $\Ss_{K_\Phi}(Q_\Phi,D_\Phi)\to\Ss_{\overline{K}_\Phi}(\overline{Q}_\Phi,\overline{D}_\Phi)\to\Ss_{K_{\Phi,h}}(G_{\Phi,h},D_{\Phi,h})$ be the tower obtained by taking the normalization of~\eqref{strata:eqn:towertildephi} in the the corresponding tower for $\Sh_{K_\Phi}(Q_\Phi,D_\Phi)$; cf.~\eqref{boundary:subsec:tatetensors} for the definition. We will soon see that the first map in the tower is an $\mb{E}_K(\Phi)$-torsor, and that the second, under certain conditions, is a torsor under an abelian scheme $\mathcal{A}_K(\Phi)$ over $\Ss_{K_{\Phi,h}}(G_{\Phi,h},D_{\Phi,h})$. In particular, the singularities of the tower are all concentrated in $\Ss_{K_{\Phi,h}}(G_{\Phi,h},D_{\Phi,h})$.

\subsubsection{}
For any rational polyhedral cone $\sigma^\beef\subset\Real\otimes\mb{B}_{{K}^{\beef}}({\Phi}^\beef)$, we obtain the twisted torus embedding $\Ss_{K^\beef_{\Phi^\beef}}(Q_{\Phi^\beef},D_{\Phi^\beef})\into\Ss_{K^\beef_{\Phi^\beef}}(Q_{\Phi^\beef},D_{\Phi^\beef},\sigma^\beef)$. Within the target of this embedding, we have the closed stratum $Z_{K^\beef_{\Phi^\beef}}(Q_{\Phi^\beef},D_{\Phi^\beef},\sigma^\beef)$.

Set $\sigma=\sigma^\beef\cap W_\Phi(\Real)(-1)$, and let $\Ss_{K_\Phi}(Q_{\Phi},D_{\Phi},\sigma)$ be the normalization of $\Ss_{K^\beef_{\Phi^\beef}}(Q_{\Phi^\beef},D_{\Phi^\beef},\sigma^\beef)$ in $\Sh_{K_\Phi}(Q_\Phi,D_\Phi)$: It does not depend on the choice of $\sigma^\beef$ intersecting $W_\Phi(\Real)(-1)$ in $\sigma$. By a lemma of Harris~\cite[Lemma 3.1]{harris:functorial}, the generic fiber of $\Ss_{K_\Phi}(Q_{\Phi},D_{\Phi},\sigma)$ is precisely the twisted toric variety $\Sh_{K_\Phi}(Q_{\Phi},D_{\Phi},\sigma)$. Let $Z_{K_\Phi}(Q_\Phi,D_\Phi,\sigma)$ be the closed stratum in $\Sh_{K_\Phi}(Q_{\Phi},D_{\Phi},\sigma)$, and let $\mathcal{Z}_{K_\Phi}(Q_\Phi,D_\Phi,\sigma)$ be the normalization of $Z_{K^\beef_{\Phi^\beef}}(Q_{\Phi^\beef},D_{\Phi^\beef},\sigma^\beef)$ in $Z_{K_\Phi}(Q_\Phi,D_\Phi,\sigma)$. Then we obtain a finite map $\mathcal{Z}_{K_\Phi}(Q_\Phi,D_\Phi,\sigma)\to\Ss_{K_\Phi}(Q_{\Phi},D_{\Phi},\sigma)$ extending the closed immersion 
\[
Z_{K_\Phi}(Q_\Phi,D_\Phi,\sigma)\into\Sh_{K_\Phi}(Q_{\Phi},D_{\Phi},\sigma).
\]

\subsubsection{}\label{strata:subsubsec:phifunctorialp}
Suppose that we are given two clrs $\Phi_1$, $\Phi_2$ for $(G,X)$ with $\Phi_1\xrightarrow{(\gamma,q)_K}\Phi_2$, for $\gamma\in G(\Rat)$, $q\in Q_{\Phi_2}(\Adele_f)$, and with $\gamma\cdot P_{\Phi_1}=P_{\Phi_2}$. This gives us an isomorphism~\eqref{background:subsubsec:cuspsisom}: $\rho(\gamma,q):\Sh_{K_{\Phi_1}}(Q_{\Phi_1},D_{\Phi_1})\xrightarrow{\simeq}\Sh_{K_{\Phi_2}}(Q_{\Phi_2},D_{\Phi_2})$.

Now, we also have ${\Phi}^\beef_1\xrightarrow{(\iota(\gamma),\iota(q))_{{K}^{\beef}}}{\Phi}^\beef_2$, and so an isomorphism of mixed Shimura varieties $\tilde{\rho}(\gamma,q):\Sh_{K^\beef_{\Phi^\beef_1}}(Q_{\Phi^\beef_1},D_{\Phi^\beef_1})\xrightarrow{\simeq}\Sh_{K^\beef_{\Phi^\beef_2}}(Q_{\Phi^\beef_2},D_{\Phi^\beef_2})$. By~\eqref{background:subsubsec:intfunct}, this extends to an isomorphism $\Ss_{K^\beef_{\Phi^\beef_1}}(Q_{\Phi^\beef_1},D_{\Phi^\beef_1})\xrightarrow{\simeq}\Ss_{K^\beef_{\Phi^\beef_2}}(Q_{\Phi^\beef_2},D_{\Phi^\beef_2})$.

We therefore find that $\rho(\gamma,q)$ also extends to an isomorphism $\rho(\gamma,q):\Ss_{K_{\Phi_1}}(Q_{\Phi_1},D_{\Phi_1})\xrightarrow{\simeq}\Ss_{K_{\Phi_2}}(Q_{\Phi_2},D_{\Phi_2})$. By construction, it preserves the tower structures on either side.

If, further, we choose a rational polyhedral cone $\sigma_1\subset W_{\Phi_1}(\Real)(-1)$ and set $\sigma_2=\on{int}(\gamma)(\sigma_2)\in W_{\Phi_2}(\Real)(-1)$, then $\rho(\gamma,q)$ induces isomorphisms:
\[
\Ss_{K_{\Phi_1}}(Q_{\Phi_1},D_{\Phi_1},\sigma_1)\xrightarrow{\simeq}\Ss_{K_{\Phi_2}}(Q_{\Phi_2},D_{\Phi_2},\sigma_2)\;;\;\mathcal{Z}_{K_{\Phi_1}}(Q_{\Phi_1},D_{\Phi_1},\sigma_1)\xrightarrow{\simeq}\mathcal{Z}_{K_{\Phi_2}}(Q_{\Phi_2},D_{\Phi_2},\sigma_2).
\]

\subsubsection{}
Fix an admissible rpcd $\Sigma^\beef$ for $(G^\beef,X^{\beef},{K}^{\beef})$ and let $\Sigma$ be the induced admissible rpcd for $(G,X,K)$. Associated with this is a map of toroidal compactifications (cf.~\ref{background:prp:funct}) $\Sh^{\Sigma}_K\to\Sh^{\Sigma^\beef}_{{K}^{\beef}}$. By~\cite[3.4]{harris:functorial}, this map identifies $\Sh^{\Sigma}_K$ with the normalization of $\Sh^{\Sigma^\beef}_{{K}^{\beef}}$ in $\Sh_K$.

Now, assume that $\Sigma^\beef$ is smooth. Let $\mathcal{S}^{\Sigma^\beef}_{{K}^{\beef}}$ be the Chai-Faltings compactification of $\mathcal{S}_{{K}^{\beef}}$ (cf.~\ref{background:thm:chaifal}). Let $\Ss^{\Sigma}_K$ be the normalization of $\mathcal{S}^{\Sigma^\beef}_{{K}^{\beef}}$ in $\Sh_K$

Choose $\Upsilon=[(\Phi,\sigma)]$ in $\on{Cusp}_K^{\Sigma}(G,X)$; and set $\Upsilon^\beef=\iota_*\Upsilon=[({\Phi}^\beef,\sigma^\beef)]$. We obtain a diagram:
\begin{diagram}
  Z_{K_\Phi}(Q_\Phi,D_\Phi,\sigma)&\rTo&Z_{{K}^{\beef}_{\Phi^\beef}}(Q_{\Phi^\beef},D_{\Phi^\beef},\sigma^\beef)&\rInto&\mathcal{Z}_{K^\beef_{\Phi^\beef}}(Q_{\Phi^\beef},D_{\Phi^\beef},\sigma^\beef)\\
  \dTo^{\simeq}&&\dTo^{\simeq}&&\dTo^{\simeq}\\
  Z_K(\Upsilon)&\rTo&Z_{{K}^{\beef}}(\Upsilon^\beef)&\rInto&\mathcal{Z}_{{K}^{\beef}}(\Upsilon^\beef),
\end{diagram}
where the vertical maps are canonical isomorphisms. Here, we are using the following: For any $h\in X$, twisting by $h(i)$ is a Cartan involution on $G/\Gm$. Therefore, $Z_G^{\circ}$ is isogenous to a product of $\Gm$ and a compact torus. By~\eqref{background:lem:deltacirc_trivial}, we then conclude that the group $\Delta_K^{\circ}(\Phi)$ is trivial, and so $Z_{K_\Phi}(Q_\Phi,D_\Phi,\sigma)$ does indeed map isomorphically onto $Z_K(\Upsilon)$.

Let $\mathcal{Z}_K(\Upsilon)$ be the normalization of $\mathcal{Z}_{{K}^{\beef}}(\Upsilon^\beef)$ in $Z_K(\Upsilon)$. Then we see that there is a canonical isomorphism $\mathcal{Z}_{K_\Phi}(Q_\Phi,D_\Phi,\sigma)\xrightarrow{\simeq}\mathcal{Z}_K(\Upsilon)$. Moreover, the map
\[
\mathcal{Z}_K(\Upsilon)\to\mathcal{Z}_{{K}^{\beef}}(\Upsilon^\beef)\into\mathcal{S}^{\Sigma^\beef}_{{K}^{\beef}}
\]
lifts to a map $\mathcal{Z}_K(\Upsilon)\to\Ss^{\Sigma}_K$, extending the locally closed immersion $Z_K(\Upsilon)\into\Sh^{\Sigma}_K$.

Also, for an integer $n\geq 1$, let $K^\beef(n)\subset G^\beef(\Adele_f)$ be the full level-$n$ compact open subgroup defined as in~\eqref{boundary:lem:finite_cover_torsor}.

We can now state our main result on the structure of $\Ss^{\Sigma}_K$:
\begin{thm}\label{strata:thm:strata}
\mbox{}
\begin{enumerate}
 \item~\label{strata:abelian_torsor}Suppose that $K_p = K^\beef(n)_p\cap G(\Rat_p)$, for some $n\geq 1$. Then, for any clr $\Phi$ for $(G,X)$, the abelian scheme $A_K(\Phi)\to \Sh_{K_{\Phi,h}}(G_{\Phi,h},D_{\Phi,h})$ extends to an abelian scheme $\mathcal{A}_K(\Phi)\to\Ss_{K_{\Phi,h}}(G_{\Phi,h},D_{\Phi,h})$, and the $A_K(\Phi)$-torsor structure on 
\[
 \Sh_{\overline{K}_\Phi}(\overline{Q}_\Phi,\overline{D}_\Phi)\to \Sh_{K_{\Phi,h}}(G_{\Phi,h},D_{\Phi,h})
 \] 
 extends to an $\mathcal{A}_K(\Phi)$-torsor structure on 
\[
 \Ss_{\overline{K}_\Phi}(\overline{Q}_\Phi,\overline{D}_\Phi)\to\Ss_{K_{\Phi,h}}(G_{\Phi,h},D_{\Phi,h}).
\]
  \item~\label{strata:strataemb}For any $\Upsilon\in\on{Cusp}^{\Sigma}_K(G,X)$, the map $\mathcal{Z}_K(\Upsilon)\to\Ss^{\Sigma}_K$ is a locally closed immersion. We have a stratification:
      \[
       \Ss^{\Sigma}_K=\bigsqcup_{\Upsilon}\mathcal{Z}_K(\Upsilon),
      \]
      where $\Upsilon$ ranges over $\on{Cusp}_K^{\Sigma}(G,X)$.
\item~\label{strata:strataclosure}
 For any fixed $\Upsilon$, the closure of $\mathcal{Z}_K(\Upsilon)$ in $\Ss_K^\Sigma$ is precisely the closed sub-space:
\[
 \overline{\mathcal{Z}}_K(\Upsilon)=\bigsqcup_{\Upsilon'\preccurlyeq\Upsilon}\mathcal{Z}_K(\Upsilon').
\]
  \item~\label{strata:torsor}For any clr $\Phi$, $\Ss_{K_\Phi}(Q_\Phi,D_\Phi)$ has the structure of an $\mb{E}_K(\Phi)$-torsor over $\Ss_{\overline{K}_\Phi}(\overline{Q}_\Phi,\overline{D}_\Phi)$, extending that of $\Sh_{K_\Phi}(Q_\Phi,D_\Phi)$ over $\Sh_{\overline{K}_\Phi}(\overline{Q}_\Phi,\overline{D}_\Phi)$. In particular, for any rational polyhedral cone $\sigma\subset W_\Phi(\Real)(-1)$, $\Ss_{K_\Phi}(Q_{\Phi},D_{\Phi},\sigma)$ is a twisted torus embedding for the torus $\mb{E}_K(\Phi)$, and $\mathcal{Z}_{K_\Phi}(Q_\Phi,D_\Phi,\sigma)$ is its closed stratum.
  \item~\label{strata:strataisom}Suppose that $\Upsilon=[(\Phi,\sigma)]$. Let $\widehat{\Ss}_{K_{\Phi}}(Q_{\Phi},D_{\Phi},\sigma)$ be the formal completion of $\Ss_{K_\Phi}(Q_{\Phi},D_{\Phi},\sigma)$ along the closed stratum $\mathcal{Z}_{K_\Phi}(Q_\Phi,D_\Phi,\sigma)$. Then the canonical isomorphism 
  \[
  \mathcal{Z}_{K_\Phi}(Q_\Phi,D_\Phi,\sigma)\xrightarrow{\simeq}\mathcal{Z}_K({\Upsilon})
   \]
  lifts to an isomorphism of formal schemes:
      \[
       \widehat{\Ss}_{K_{\Phi}}(Q_{\Phi},D_{\Phi},\sigma)\xrightarrow{\simeq}\bigl(\Ss^{\Sigma}_K\bigr)^{\wedge}_{\mathcal{Z}_K(\Upsilon)}.
      \]
      This isomorphism restricts to the one from~\eqref{background:thm:pink}\eqref{pink:completion} over the generic fiber.
\end{enumerate}
\end{thm}

\begin{rem}\label{strata:rem:conedecomp}
\mbox{}
\begin{itemize}
\item It follows from~\eqref{strata:strataisom} and~\eqref{appendix:lem:abstract_hecke} that the integral model $\Ss^{\Sigma}_K$ for $\Sh^{\Sigma}_K$ does not depend on the choice of $\Sigma^\beef$ such that $\Sigma=\iota^*\Sigma^\beef$.
\item Suppose that $\Sigma'$ is an admissible rpcd for $(G,X,K)$ that refines $\Sigma$; then the methods of~\cite[Ch. II]{kkms} allow us to construct an open embedding $\Ss_K\into\Ss^{\Sigma'}_K$ over $\Reg{E,(v)}$, and a birational map $\Ss^{\Sigma'}_K\to\Ss^\Sigma_K$, which is an integral model over $\Reg{E,(v)}$ for $\Sh^{\Sigma'}_K\to\Sh^{\Sigma}_K$, and which has a stratification satisfying, \emph{mutatis mutandum}, the conclusions of~\eqref{strata:thm:strata}. Alternatively, one can also argue as in~\cite[THeorem (2.4.12)]{hoermann:thesis}, and show that we can arrange the choice of $\Sigma^\beef$ so that $\Sigma$ is smooth.

In any case, we can and will assume that there exist good integral models as above for compactifications associated with \emph{smooth} complete rpcds.  
\end{itemize}
\end{rem}

We postpone the proof of the theorem to \eqref{strata:subsec:integral}; in brief, it amounts to putting together~\eqref{background:thm:pink} and~\eqref{boundary:thm:main}. For now, we can quickly deduce Theorems~\ref{main} and~\ref{morita:good_reduction} of the introduction.

\begin{proof}[Proof of Theorem~\ref{main}]
Fix $\Upsilon=[(\Phi,\sigma)]$ in $\on{Cusp}_K^{\Sigma}(G,X)$. Combining assertion~\eqref{strata:strataisom} of~\eqref{strata:thm:strata} with Artin approximation, we find there is an \'etale neighborhood $V\to\Ss^{\Sigma}_K$ of $\mathcal{Z}_K(\Upsilon)$ such that the open immersion $V\vert_{\Ss_K}\into V$ is again \'etale over the open immersion $\Ss_{K_\Phi}(Q_\Phi,D_\Phi)\into\Ss_{K_\Phi}(Q_{\Phi},D_{\Phi},\sigma)$. By assertion~\eqref{strata:torsor}, this is a twisted toric embedding, and so the complement of $\Ss_{K_\Phi}(Q_\Phi,D_\Phi)$ in $\Ss_{K_\Phi}(Q_{\Phi},D_{\Phi},\sigma)$ is a relative effective Cartier divisor over $\Reg{E,(v)}$. Hence, the complement of $V\vert_{\Ss_K}$ in $V$ is a relative Cartier divisor over $\Reg{E,(v)}$. Now, using assertion~\eqref{strata:strataemb}, we see that the complement of $\Ss$ in $\Ss^{\Sigma}_K$ must be a relative Cartier divisor.

By~\eqref{strata:rem:conedecomp}, we can replace $\Sigma$ by a smooth refinement.

Then both $\Ss_{K_\Phi}(Q_\Phi,D_\Phi)$ and $\Ss_{K_\Phi}(Q_{\Phi},D_{\Phi},\sigma)$ are smooth over $\Ss_{\overline{K}_\Phi}(\overline{Q}_\Phi,\overline{D}_\Phi)$. Since $\Ss_K\into\Ss^{\Sigma}_K$ is \'etale locally isomorphic to $\Ss_{K_\Phi}(Q_\Phi,D_\Phi)\into\Ss_{K_\Phi}(Q_{\Phi},D_{\Phi},\sigma)$, we find that the singularities of $\Ss^{\Sigma}_K$ can be no worse than those of $\Ss_K$.
\end{proof}

\begin{corollary}\label{strata:cor:proper}
$\Ss_K$ is projective over $\Reg{E,(v)}$ if and only if $G^{\ad}$ is anisotropic; equivalently, if and only if $G^{\ad}(\Rat)$ contains no non-trivial unipotent elements.
\end{corollary}
\begin{proof}
Indeed, it is clear from the theorem and the description of the stratification that $\Ss_K$ is projective if and only if $G$ does not admit any proper parabolic sub-groups defined over $\Rat$.
\end{proof}

\subsubsection{}\label{strata:subsubsec:morita}
Before we prove Theorem~\ref{morita:good_reduction}, let us first recall the \defnword{Mumford-Tate group} $\MT_A$ associated with an abelian variety $A$ over $\Comp$: One way to define it is as the fundamental group of the Tannakian category of polarizable rational Hodge structures generated by the rational Hodge structure $H^1(A(\Comp),\Rat)$ (cf.~\cite[Ch. II]{dmos}). In particular, it is a connected reductive group and there is a canonical map $h_A:\bb{S}\rightarrow\MT_{A,\Real}$ that gives rise to the Hodge decomposition of $H^1(A(\Comp),\Comp)$. The pair $(\MT_A,X_A)$, where $X_A$ is the $\MT_A(\Real)$-conjugacy class of $h_A$, is a Shimura datum of Hodge type.

Suppose now that $A$ is defined over a number field $F$. The \defnword{Mumford-Tate group} $\MT_A$ of $A$ is $\MT_{\sigma^*A}$, for any embedding $\sigma:F\into\Comp$. The main result of \cite[Ch. I]{dmos} shows that $\MT_A$ does not depend on the choice of embedding. We can now restate Theorem~\ref{morita:good_reduction} as follows:
\begin{thm}\label{strata:thm:morita}
Suppose that $\MT_A$ is anisotropic modulo center. Then $A$ has potentially good reduction at all finite places of $F$.
\end{thm}
\begin{proof}
 Using Zarhin's trick~\cite[\S 6]{zarhin:finiteness}, we can replace $A$ by an abelian variety isogenous to $A^8$, and assume that $A$ is principally polarized over $F$. Extending $F$ if necessary, we can assume that it contains the reflex field $E=E(\MT_A,X_A)$. Fix $\sigma:F\into\Comp$, and set $H=H^1\bigl((\sigma^*A)(\Comp),\Rat\bigr)$ equipped with a pairing attached to the principal polarization on $\sigma^*A$. We then have a natural embedding of Shimura data:
 \[
  (G,X)\coloneqq(\MT_A,X_A)\into (G^\beef,X^{\beef})\coloneqq (\GSp(H),\on{S}^{\pm}(H)).
 \]

 Fix a prime $p$ and a place $v\vert p$ for $E$, and a neat level sub-group ${K}^{\beef}\subset\mathcal{G}(\Adele_f)$, so that ${K}^{\beef}_p$ is the stabilizer of $H^1(A,\Int_p)$. Set $K={K}^{\beef}\cap G(\Adele_f)$. Let $\Ss_K\to \mathcal{S}_{{K}^{\beef}}$ be the corresponding finite map of integral models of Shimura varieties over $\Reg{E,(v)}$. By~\eqref{strata:cor:proper}, $\Ss_K$ is proper over $\Reg{E,(v)}$.

 Let $\mathcal{A}\to\mathcal{S}_{{K}^{\beef}}$ be the universal abelian scheme. By construction, there is a finite extension $F'/F$ and a point $x\in\Ss_K(F')$ such that $\mathcal{A}_x$ is isomorphic to $A$. Since $\Ss_K$ is proper, for any place $v'\vert v$ of $F'$, $x$ extends to an $\Reg{F',(v')}$-valued point of $\Ss_K$, implying that $A$ has potentially good reduction over $v$.

 Since $v$ was arbitrary, this proves the theorem.
\end{proof}

\subsubsection{}
For any scheme $S$, write $\pi_0(S)$ for its set of connected components. The following easy corollary to~\eqref{strata:thm:strata} is often useful. For instance, it implies the geometric irreducibility of the moduli space of polarized K3 surfaces of degree $2d$ over $\Field_p$ when $p\nmid d^2$ (cf.~\cite{mp:tatek3}).
\begin{corollary}\label{strata:cor:conncomp}
Suppose that the special fiber $k(v)\otimes \Ss_K$ is geometrically reduced. Then, for any finite extension $F/E$ and any place $w\vert v$ of $F$, the natural maps:
\[
 \pi_0\bigl(F\otimes_E\Sh_K\bigr)\leftarrow\pi_0\bigl(\Reg{F,(w)}\otimes_{\Reg{E,(v)}}\Ss_K\bigr)\rightarrow\pi_0\bigl(k(w)\otimes_{\Reg{E,(v)}}\Ss_K\bigr)
\]
are both isomorphisms.
\end{corollary}
\begin{proof}
 By Theorem~\ref{main}, the hypothesis implies that $k(v)\otimes\Ss^{\Sigma}_K$ is also geometrically reduced. Since $\Ss_K$ is fiber-wise dense in $\Ss^{\Sigma}_K$, we reduce to showing the following general statement:

 Suppose that $S$ is a flat, proper algebraic space over $\Reg{E,(v)}$ with geometrically reduced special fiber. Then the natural maps:
 \[
  \pi_0\bigl(F\otimes_ES)\leftarrow\pi_0\bigl(\Reg{F,(w)}\otimes_{\Reg{E,(v)}}S\bigr)\rightarrow\pi_0\bigl(k(w)\otimes_{\Reg{E,(v)}}S\bigr)
 \]
 are isomorphisms. Indeed, by replacing $S$ with $\Spec H^0(S,\Reg{S})$, we are reduced to the case where $S$ is finite and \'etale over $\Reg{E,(v)}$, where the statement is obvious.
\end{proof}

\subsubsection{}\label{strata:subsubsec:hecke}
We can also extend Hecke actions to the compactifications. More generally, suppose that we have an embedding of Shimura data:
\[
 \eta:(G',X')\into (G,X).
\]
Fix a compact open sub-group $K'\subset G'(\Adele_f)$, and $g\in G(\Adele_f)$ such that $g\eta(K')g^{-1}\subset K$. Let $\Sigma'$ be an admissible rpcd for $(G',X')$ refining $(\eta,g)^*\Sigma$.

Let $E'=E(G',X')$ be the reflex field, and let $v'\vert v$ be a place of $E'$ above $v'$. Consider the open immersion of algebraic spaces over $E'$:
\[
 \Sh_{K'}=\Sh_{K'}(G',X')\into\Sh_{K'}^{\Sigma'}(G',X')=\Sh_{K'}^{\Sigma'}.
\]
Using the symplectic embedding $\iota\circ\eta$ of $(G',X')$,~\eqref{strata:thm:strata} and~\eqref{strata:rem:conedecomp}, we obtain a normal integral model $\Ss_{K'}\into\Ss_{K'}^{\Sigma'}$ over $\Reg{E,(v')}$ for this immersion.

By construction, the map $(\eta,g):\Sh_{K'}\to E'\otimes_E\Sh_K$ extends to a finite map of normal schemes $\Ss_{K'}\to\Reg{E',(v')}\otimes_{\Reg{E,(v)}}\Ss_K$. More generally, for any clr $\Phi'$ for $(G',X')$ with $(\eta,g)_*\Phi'=\Phi$, the map $\Sh_{K'_{\Phi'}}(Q_{\Phi'},D_{\Phi'})\to E'\otimes_E\Sh_{K_\Phi}(Q_\Phi,D_\Phi)$ (cf.~\ref{background:eqn:phitildephi}) extends to a finite map:
\begin{align}\label{strata:eqn:phi'phi}
 \Ss_{K'_{\Phi'}}(Q_{\Phi'},D_{\Phi'})\to\Reg{E',(v')}\otimes_{\Reg{E,(v)}}\Ss_{K_\Phi}(Q_\Phi,D_\Phi).
\end{align}

Given $\sigma'\in\Sigma'(\Phi')$ and $\sigma\in\Sigma(\Phi)$ such that $\eta(\sigma')$ is contained in $\sigma$ and intersects $\sigma^{\circ}$ non-trivially,~\eqref{strata:eqn:phi'phi} extends to a map:
\[
 \Ss_{K'_{\Phi'}}(Q_{\Phi'},D_{\Phi'},\sigma')\to\Reg{E',(v')}\otimes_{\Reg{E,(v)}}\Ss_{K_\Phi}(Q_{\Phi},D_{\Phi},\sigma).
\]
This carries $\mathcal{Z}_{K'_{\Phi'}}(Q_{\Phi'},D_{\Phi'},\sigma')$ into $\Reg{E',(v')}\otimes_{\Reg{E}}\mathcal{Z}_{K_\Phi}(Q_\Phi,D_\Phi,\sigma)$.

Therefore, for any $\Upsilon'=[(\Phi',\sigma')]$ in $\on{Cusp}_{K'}^{\Sigma'}(G',X')$ with $\Upsilon=(\eta,g)_*\Upsilon'=[(\Phi,\sigma)]$, we obtain a map:
\begin{align}\label{strata:eqn:functstratacomp}
(\eta,g):\bigl(\Ss^{\Sigma'}_{K'}\bigr)^{\wedge}_{\mathcal{Z}_{K'}(\Upsilon')}\xrightarrow{\simeq}\widehat{\Ss}_{K'_{\Phi'}}(Q_{\Phi'},D_{\Phi'},\sigma')\to\widehat{\Ss}_{K_{\Phi}}(Q_{\Phi},D_{\Phi},\sigma)\xrightarrow{\simeq}\bigl(\Ss^{\Sigma}_{K}\bigr)^{\wedge}_{\mathcal{Z}_{K}(\Upsilon)}.
\end{align}

From~\eqref{background:prp:funct} and~\eqref{appendix:lem:abstract_hecke}, we now obtain:
\begin{prp}\label{strata:prp:hecke}
Suppose that $\Sigma$ is complete. Then the map $\Ss_{K'}\to \Reg{E',(v')}\otimes_{\Reg{E,(v)}}\Ss_K$ extends uniquely to a map
\[
 (\eta,g):\Ss^{\Sigma'}_{K'}\to\Reg{E',(v')}\otimes_{\Reg{E,(v)}}\Ss^{\Sigma}_K
\]
satisfying the following property: For every cusp label $\Upsilon'$ with $\Upsilon=(\eta,g)_*\Upsilon'$, $(\eta,g)$ carries $\mathcal{Z}_{K'}(\Upsilon')$ into $\Reg{E',(v')}\otimes_{\Reg{E,(v)}}\mathcal{Z}_K(\Upsilon)$, and the corresponding map between the formal completions along these locally closed sub-schemes is identified with that obtained from~\eqref{strata:eqn:functstratacomp}.
\end{prp}
\qed

\subsection{Stratifications of the integral model}\label{strata:subsec:integral}

The notation will be as above.

\subsubsection{}\label{strata:subsubsec:abscheme_shimura_data}
Let $\Phi$ be a clr for $(G,X,K)$ and let ${\Phi}^\beef$ be the induced clr for $(G^\beef,X^{\beef},{K}^{\beef})$. By the analytic description of $A_K(\Phi)$ and $A_{{K}^{\beef}}({\Phi}^\beef)\vert_{\Sh_{K_{\Phi,h}}(G_{\Phi,h},D_{\Phi,h})}$ in~\eqref{background:subsubsec:abscheme}, we find that the former is associated with the $G_{\Phi,h}$-representation $V_\Phi$ with the lattice $K_{\Phi,V}\subset V_\Phi(\Adele_f)$ and the latter is associated with the representation $V_{\Phi^\beef}$ and the lattice ${K}^{\beef}_{{\Phi}^\beef,V}\subset V_{\Phi^{\beef}}(\Adele_f)$. The homomorphism of abelian schemes
\begin{align}\label{strata:eqn:Akphi_generic_embedding}
 A_K(\Phi)\to A_{{K}^{\beef}}({\Phi}^\beef)\vert_{\Sh_{K_{\Phi,h}}(G_{\Phi,h},D_{\Phi,h})}
\end{align}
is associated with the natural inclusion $V_\Phi\into V_{\Phi^{\beef}}$. 

\begin{prp}\label{strata:prp:abelian_scheme}
Suppose that $K_{\Phi,V,p}={K}^{\beef}(n)_{{\Phi}^\beef,V,p}\cap V_\Phi(\Rat_p)$, for some $n\geq 1$. Then the normalization of $\mathcal{A}_{{K}^{\beef}}({\Phi}^\beef)\vert_{\Ss_{K_{\Phi,h}}(G_{\Phi,h},D_{\Phi,h})}$ in $A_K(\Phi)$ is an abelian scheme $\mathcal{A}_K(\Phi)\to\Ss_{K_{\Phi,h}}(G_{\Phi,h},D_{\Phi,h})$.
\end{prp}
\begin{proof}
First, assume that $n=1$. In this case, the homomorphism~\eqref{strata:eqn:Akphi_generic_embedding} is a prime-to-$p$ isogeny onto its image

Let $W_\bullet H\subset H$ be the filtration determined by the parabolic subgroup $P^{\beef}$. Then
\[
 V_{\Phi^{\beef}} \xrightarrow{\simeq} \Hom\bigl(\gr^W_0H,\gr^W_{-1}H\bigr)
\]
is equipped with a symplectic form arising from that on $\gr^W_{-1}H$. This induces a polarization on $A_{{K}^{\beef}}({\Phi}^\beef)$ of prime-to-$p$ degree.

Let $V'\subset V_{\Phi^{\beef}}$ be the orthogonal complement to $V_\Phi$. Then, just as in~\eqref{background:subsubsec:abscheme}, to the $G_{\Phi,h}$-representation $V'$ and the lattice $V'(\Adele_f)\cap{K}^{\beef}_{{\Phi}^\beef,V}$, we can attach an abelian subscheme $A'\subset A_{{K}^{\beef}}({\Phi}^\beef)\vert_{\Sh_{K_{\Phi,h}}(G_{\Phi,h},D_{\Phi,h})}$ equipped with a canonical polarization. Moreover, the product homomorphism:
\begin{align}\label{strata:eqn:product_hom}
 \beta:A_K(\Phi)\times_{\Sh_{K_{\Phi,h}}(G_{\Phi,h},D_{\Phi,h})} A'\to A_{{K}^{\beef}}({\Phi}^\beef)\vert_{\Sh_{K_{\Phi,h}}(G_{\Phi,h},D_{\Phi,h})}
\end{align}
is an isogeny of abelian schemes.

The proposition in the case $n=1$ is now immediate from~\eqref{appendix:lem:abelian_scheme_extension}. In particular, if $K^\flat = K^p(K^\sharp_p\cap G(\Rat_p))$, then
\[
\mathcal{A}_{K^\flat}(\Phi)\to \Sh_{K^\flat_{\Phi,h}}(G_{\Phi,h},D_{\Phi,h})
\]
is an abelian scheme.

If $n\geq 2$, then there exists a compact open subgroup $K^{\beef,\diamond}\subset G^\beef(\Adele_f)$ such that $K^{\beef,\diamond}_p = K^\beef(n)_p$ and such that $K = K^{\beef,\diamond}\cap G(\Adele_f)$.

One can show as in the proof of~\eqref{boundary:lem:finite_cover_torsor} that 
\[
A_{K^\beef(n)}(\Phi^\beef)\to \Sh_{K^\beef(n)_{\Phi^\beef}}(G_{\Phi^\beef,h},D_{\Phi^\beef,h})
\]
can be identified with the abelian scheme $\SHom(\frac{1}{n} \mathsf{X},\mathcal{Q}^{\ab})$, where $\mathcal{Q}^{\ab}$ is the universal principally polarized abelian scheme over the base Shimura variety. From this one can deduce that
\[
\mathcal{A}_{K^{\beef,\diamond}}(\Phi^\beef)\to \Ss_{K^{\beef,\diamond}_{\Phi^\beef}}(G_{\Phi^\beef,h},D_{\Phi^\beef,h})
\]
is a abelian scheme. 

By abuse of notation, write $\mathcal{A}_{K^\flat}(\Phi)$, $\mathcal{A}_{K^\beef}(\Phi^\beef)$, $\mathcal{A}_{K^{\beef,\diamond}}(\Phi^\beef)$ for the base change of these abelian schemes over $\Ss_{K_{\Phi,h}}(G_{\Phi,h},D_{\Phi,h})$ along the obvious morphisms. 

We can then consider the fiber product
\[
\mathcal{A}_{K^\flat}(\Phi)\times_{\mathcal{A}_{K^\beef}(\Phi^\beef)}\mathcal{A}_{K^{\beef,\diamond}}(\Phi^\beef),
\]
which is a disjoint union of abelian schemes finite over $\mathcal{A}_{K^\flat}(\Phi)$. Now, as can be checked over the generic fiber, $\mathcal{A}_K(\Phi)$ is identified with one of these abelian schemes.
\end{proof}

\begin{corollary}\label{strata:cor:abelian_torsor}
Assertion~\eqref{strata:abelian_torsor} of~\eqref{strata:thm:strata} holds.
\end{corollary}
\begin{proof}
It only remains to show that $\Ss_{\overline{K}_\Phi}(\overline{Q}_\Phi,\overline{D}_\Phi)\to\Ss_{K_{\Phi,h}}(G_{\Phi,h},D_{\Phi,h})$ is an $\mathcal{A}_K(\Phi)$-torsor over $\Ss_{K_{\Phi,h}}(G_{\Phi,h},D_{\Phi,h})$. It is easy to see that this morphism is projective and $\mathcal{A}_K(\Phi)$-equivariant. 

To show that it is an $\mathcal{A}_K(\Phi)$-torsor, we can work \'etale locally on the base, and assume that we have a surjective \'etale morphism $\mathcal{U}\to \Ss_{K_{\Phi,h}}(G_{\Phi,h},D_{\Phi,h})$ such that $\Sh_{\overline{K}_\Phi}(\overline{Q}_\Phi,\overline{D}_\Phi)\vert_{\mathcal{U}[p^{-1}]}$ is a trivializable $A_K(\Phi)$-torsor over $\mathcal{U}[p^{-1}]$, and that $\Ss_{\overline{K}^\beef_{\Phi^\beef}}(\overline{Q}_{\Phi^\beef},\overline{D}_{\Phi^\beef})\vert_{\mathcal{U}}$ is a trivializable $\mathcal{A}_{{K}^{\beef}}({\Phi}^\beef)$-torsor over $\mathcal{U}$.

Fix a section $\mathcal{U}[p^{-1}]\to \Sh_{\overline{K}_\Phi}(\overline{Q}_\Phi,\overline{D}_\Phi)$; it induces a section of $\Sh_{\overline{K}^\beef_{\Phi^\beef}}(\overline{Q}_{\Phi^\beef},\overline{D}_{\Phi^\beef})$, which, by the valuative criterion for properness and Weil's extension theorem~\cite[\S~4.4, Theorem 1]{blr}, extends to a section of $\Ss_{\overline{K}^\beef_{\Phi^\beef}}(\overline{Q}_{\Phi^\beef},\overline{D}_{\Phi^\beef})$ over $\mathcal{U}$.\footnote{Note that this is precisely what fails if we worked instead with $\mb{E}_K(\Phi)$-torsors: There can be plenty of sections that are not integral, and so one has to work to show integrality.} Here, we are using the fact that $\Ss_{\overline{K}^\beef_{\Phi^\beef}}(\overline{Q}_{\Phi^\beef},\overline{D}_{\Phi^\beef})$ is isomorphic to an abelian scheme over $\mathcal{U}$.

This section now allows us to identify the map of $\mathcal{U}$-schemes $\Sh_{\overline{K}_\Phi}(\overline{Q}_\Phi,\overline{D}_\Phi)\vert_{\mathcal{U}}\to\Ss_{\overline{K}^\beef_{\Phi^\beef}}(\overline{Q}_{\Phi^\beef},\overline{D}_{\Phi^\beef})\vert_{\mathcal{U}}$
with $A_K(\Phi)\vert_{\mathcal{U}}\to\mathcal{A}_{{K}^{\beef}}({\Phi}^\beef)\vert_{\mathcal{U}}$.

This implies that it must also identify $\Ss_{\overline{K}_\Phi}(\overline{Q}_\Phi,\overline{D}_\Phi)\vert_{\mathcal{U}}$ with $\mathcal{A}_K(\Phi)\vert_{\mathcal{U}}$ as schemes over $\mathcal{U}$, and so finishes the proof.
\end{proof}

\begin{prp}\label{strata:prp:ztildephiintersect}
For any $\Upsilon^\beef$ in $\on{Cusp}^{\Sigma^\beef}_{{K}^{\beef}}(G^\beef,X^{\beef})$, the reduced locally closed sub-scheme
\[
 \bigl(\Ss_K^{\Sigma}\times_{\mathcal{S}_{{K}^{\beef}}^{\Sigma^\beef}}\mathcal{Z}_{{K}^{\beef}}(\Upsilon^\beef)\bigr)_{\on{red}}\into\Ss_K^{\Sigma}
\]
is normal and flat over $\Int_{(p)}$.
\end{prp}
\begin{proof}
 This is a local statement. Choose a closed point $s_0\in\Ss_K^{\Sigma}(\overline{\Field}_p)$ mapping to $s_0^\beef\in\mathcal{Z}_{{K}^{\beef}}(\Upsilon^\beef)(\overline{\Field}_p)$. Suppose that $\Upsilon^\beef=[({\Phi}^\beef,\sigma^\beef)]$, so that we can identify the complete local ring of $\mathcal{S}_{{K}^{\beef}}^{\Sigma^\beef}$ at $s_0^\beef$ with $R\coloneqq R({\Phi}^\beef,\sigma^\beef,s_0^\beef)$, the complete local ring of $\Ss_{K^\beef_{\Phi^\beef}}(Q_{\Phi^\beef},D_{\Phi^\beef},\sigma^\beef)$ at a closed point $s^\beef_0$ of $Z_{K^\beef_{\Phi^\beef}}(Q_{\Phi^\beef},D_{\Phi^\beef},\sigma^\beef)$. 

 Let $R_G$ be the complete local ring of $\Ss^{\Sigma}_K$ at $s_0$, let $R^{\sab}$ be the complete local ring of $\Ss_{\overline{K}^\beef_{\Phi^\beef}}(\overline{Q}_{\Phi^\beef},\overline{D}_{\Phi^\beef})$ at the image of $s^\beef_0$, and let $R_G^{\sab}$ be the normalization in $R_G$ of the image of $R^{\sab}$.

 According to~\eqref{boundary:thm:main}\eqref{main:reduction}, the $\mb{E}_{{K}^{\beef}}({\Phi}^\beef)$-torsor $\Ss_{K^\beef_{\Phi^\beef}}(Q_{\Phi^\beef},D_{\Phi^\beef})$ admits a canonical reduction of structure group to an $\mb{E}_G$-torsor $\mb{\Xi}_G$ over $\Spec R_G^{\sab}$. Here, $\mb{E}_G$ is a sub-torus of $\mb{E}_{{K}^{\beef}}({\Phi}^\beef)$ with co-character group $\mb{B}_G$. 

 Moreover, by assertion~\eqref{main:local} of the same theorem, there exists an isogeny $\mb{E}_G^\diamond\to \mb{E}_G$ of tori, and an $\mb{E}_G^\diamond$-torsor $\mb{\Xi}_G^\diamond$ over $R_G^{\sab}$ that induces $\mb{\Xi}_G$ with the following property: If $\sigma_G=\sigma^\beef\cap(\Real\otimes\mb{B}_G)$ with corresponding torus embedding $\mb{E}_G^\diamond\hookrightarrow \mb{E}_G^\diamond(\sigma_G)$ and twisted torus embedding
 \[
\mb{\Xi}_G^\diamond\hookrightarrow \mb{\Xi}_G^\diamond(\sigma_G),
 \]
 then $R_G$ is identified with a complete local ring at a closed point in the closed stratum of $\mb{\Xi}_G^\diamond(\sigma_G)$.

 Fix a trivialization
 \[
\Spec R_G^{\sab}\times\mb{E}_G^\diamond \xrightarrow{\simeq}\mb{\Xi}^\diamond_G.
 \]
 Using this, we find compatible isomorphisms of $R_G^{\sab}$-algebras:
 \begin{align*}
 R_G^{\sab}\widehat{\otimes}_WB(\sigma^\beef,s^\beef_0)&\xrightarrow{\simeq}R_G^{\sab}\otimes_{R^{\sab}}R;\\
 R_G^{\sab}\widehat{\otimes}_WB^\diamond(\sigma_G,t_0)&\hookrightarrow R_G.
 \end{align*}
 Here, $B(\sigma^\beef,s^\beef_0)$ is the complete local ring of $\mb{E}_{{K}^\beef}({\Phi}^\beef,\sigma^\beef)$ at a point $s^\beef_0$ in its closed stratum, and $B^\diamond(\sigma_G,t_0)$ is the complete local ring of $\mb{E}_G^{\diamond}(\sigma_G)$ at a point $t_0$ mapping to $s^\beef_0$.

 Let $I\subset R$ be the ideal so that $R/I$ is the complete local ring of $\mathcal{Z}_{{K}^{\beef}}(\Upsilon^\beef)$ at $s_0^\beef$. Let $I(\tilde{\sigma}^\beef,s^\beef_0)\subset B(\sigma^\beef,s^\beef_0)$ be the ideal defining the closed stratum in $\Spec B(\sigma^\beef,s^\beef_0)$. Then, using the explicit description of $R_G^{\sab}$-algebras above, we find that:
 \begin{align*}
 R_G\otimes_R(R/I)&\xrightarrow{\simeq}R_G\otimes_{R_G^{\sab}\otimes_{R^{\sab}}R}(R_G^{\sab}\otimes_{R^{\sab}}(R/I)\\
 &\xrightarrow{\simeq}R_G^{\sab}\widehat{\otimes}_W(B^\diamond(\sigma_G,t_0)/I(\sigma^\beef,s^\beef_0)B^\diamond(\sigma_G,t_0)).
 \end{align*}

 So, to finish the proof, it is enough to show that the maximal reduced quotient of the ring
 \[
 B^\diamond(\sigma_G,t_0)/I(\sigma^\beef,s^\beef_0)B^\diamond(\sigma_G,t_0)
 \]
 is formally smooth over $W$. In fact, let $I^\diamond(\sigma_G,t_0)\subset B^\diamond(\sigma_G,t_0)$ be the ideal defining the closed stratum in $\Spec B^\diamond(\sigma_G,t_0)$. Then some finite power of this ideal is contained in the image of $I(\sigma^\beef,s^\beef_0)$. So it is enough to show that the closed stratum in $\mb{E}_G^\diamond(\sigma_G)$ is smooth over $\Int$. But this stratum is isomorphic to the torus with co-character group $\mb{B}^\diamond_G/\langle\sigma_G\rangle$, where $\langle\sigma_G\rangle$ is the sub-group of $\mb{B}^\diamond_G$ generated by $\sigma_G\cap\mb{B}^\diamond_G\subset\mb{B}^\diamond_G$; cf.~\cite[p. 16]{kkms}.
\end{proof}

\begin{corollary}\label{strata:cor:stratadecomp}
Assertions~\eqref{strata:strataemb} and~\eqref{strata:strataclosure} of~\eqref{strata:thm:strata} hold.
\end{corollary}
\begin{proof}
Using~\eqref{background:thm:chaifal}\eqref{chaifal:strata}, we see that we have a decomposition into locally closed sub-spaces:
\[
 \Ss^{\Sigma}_K=\bigsqcup_{\Upsilon^\beef}\bigl(\Ss^{\Sigma}_K\times_{\mathcal{S}^{\Sigma^\beef}_{{K}^{\beef}}}\mathcal{Z}_{{K}^{\beef}}(\Upsilon^\beef)\bigr)_{\on{red}},
\]
where $\Upsilon^\beef$ runs over $\on{Cusp}^{\Sigma^\beef}_{{K}^{\beef}}(G^\beef,X^{\beef})$. So to prove assertion~\eqref{strata:strataemb}, we have to show that, for each such $\Upsilon^\beef$, the natural map
\[
 \bigsqcup_{\Upsilon:\;\iota_*\Upsilon=\Upsilon^\beef}\mathcal{Z}_K(\Upsilon)\to\bigl(\Ss^{\Sigma}_K\times_{\mathcal{S}^{\Sigma^\beef}_{{K}^{\beef}}}\mathcal{Z}_{{K}^{\beef}}(\Upsilon^\beef)\bigr)_{\on{red}}
\]
is an isomorphism. But, by construction,~\eqref{background:thm:pink}\eqref{pink:stratification},~\eqref{background:prp:funct}, and~\eqref{strata:prp:ztildephiintersect}, this is a finite, birational map of normal schemes, and is thus an isomorphism.

Assertion~\eqref{strata:strataclosure} is proven similarly, using the analogue in characteristic $0$~\eqref{background:thm:pink}\eqref{pink:stratification}, and that for the Chai-Faltings compactification~\eqref{background:thm:chaifal}\eqref{chaifal:strata}.
\end{proof}

\subsubsection{}\label{strata:subsubsec:tensors}
Suppose that we have $\Upsilon=[(\Phi,\sigma)]$ in $\on{Cusp}_K^\Sigma(G,X)$. Fix a point $s_0\in\mathcal{Z}_K(\Upsilon)(\overline{\Field}_p)$, and let $R(\Upsilon,s_0)$ be the complete local ring of $\Ss^\Sigma_K$ at $s_0$. Set
\[
V\coloneqq=V(\Upsilon,s_0)=\Spec R(\Upsilon,s_0)\times_{\Ss^\Sigma_K}\Ss_K.
\]
We have a canonical abelian scheme $\mathcal{A}$ over $V$, equipped with canonical tensors:
\[
 \{s_{\alpha,\et}\}\subset H^0\bigl(V[p^{-1}],\Adele_f\otimes\widehat{T}(\mathcal{A})^\otimes\bigr).
\]

Suppose that ${\Phi}^\beef=\iota_*\Phi$. By construction, we have a map $V\to\Ss_{K^\beef_{\Phi^\beef}}(Q_{\Phi^\beef},D_{\Phi^\beef})$ corresponding to a $1$-motif $\mathcal{Q}$ over $V$ such that, for any integer $n$, $\mathcal{A}[n]$ is canonically isomorphic over $V$ to $\mathcal{Q}[n]$. In particular, over $V[p^{-1}]$ we can identify $\widehat{T}(\mathcal{Q})$ with $\widehat{T}(\mathcal{A})$. In particular, we can view $s_{\alpha,\et}$ as a section of $\widehat{T}(\mathcal{Q})^\otimes$.

Now, let $I(\Upsilon,s_0)\subset\underline{\Isom}(H^g(\widehat{\Int}),\widehat{T}(\mathcal{Q}))$ (cf.~\ref{background:subsubsec:1motifsmoduli}) be the \'etale sub-sheaf over $V[p^{-1}]$ consisting of sections $(\eta,u)$ such that, for every $\alpha$, $\eta$ carries $s_{\alpha}\in H(\Adele_f)^\otimes$ to $s_{\alpha,\et}$. The stabilizer $K_{0,\Phi}$ of $H^g(\widehat{\Int})$ in $Q_\Phi(\Adele_f)$ acts on $I(\Upsilon,s_0)$ via pre-composition.
\begin{lem}\label{strata:lem:etaletorsor}
This action makes $I(\Upsilon,s_0)$ a torsor under $K_{0,\Phi}$.
\end{lem}
\begin{proof}
It is easy to see that any two sections of $I(\Upsilon,s_0)$ differ by a unique element of $K_{0,\Phi}$. So we only have to show that $I(\Upsilon,s_0)$ is non-empty. For this, choose any $t\in Z_K(\Upsilon)$ specializing to $s_0$. Let $R(\Upsilon,t)$ be the complete local ring of $\Sh^\Sigma_K$ at $t$, and set
\[
 V(\Upsilon,t)=\Spec R(\Upsilon,t)\times_{\Sh^\Sigma_K}\Sh_K.
\]
Then we have formally \'etale map $V(\Upsilon,t)\to V[p^{-1}]$.

Now, we can identify $R(\Upsilon,t)$ with the complete local ring $R(\Phi,\sigma,t)$ of $\Sh_{K_\Phi}(Q_{\Phi},D_{\Phi},\sigma)$ at $t$. By~\eqref{boundary:prp:twotensors}, for each $\alpha$, the restriction of $s_{\alpha,\et}$ to $V(\Upsilon,t)$ is identified with $s_{\alpha,{\Phi}^\beef,\et}$. So we can use the complex analytic uniformization of $\Sh_{K_\Phi}(Q_\Phi,D_\Phi)(\Comp)$ to deduce that $I(\Upsilon,s_0)$ is non-empty over $V(\Upsilon,t)$, and thus over $V[p^{-1}]$.
\end{proof}

\begin{prp}\label{strata:prp:xikphitorsor}
Assertion~\eqref{strata:torsor} of~\eqref{strata:thm:strata} holds.
\end{prp}

Before presenting the proof, we make a couple of remarks.

\begin{rem}
\label{strata:rem:naive_torsor}
Note that there is genuine content to~\eqref{strata:prp:xikphitorsor}. In general,  if $X\to S$ is a torsor under a torus $T$ over a flat, normal $\mathbb{Z}_{(p)}$-scheme $S$, and if $Y\hookrightarrow X[p^{-1}]$ is a torsor under a sub-torus $T'\subset T$, then the normalization of the Zariski closure of $Y$ in $X$ need not be a $T'$-torsor.

For an example, take $S = \mathrm{Spec}~\mathbb{Z}_{(p)}$, and $X = \mathbb{G}_m^2$, viewed as a trivial torsor over $\mathbb{G}_m^2$, and $Y = \mathrm{Spec}~\mathbb{Q}[T^{\pm 1},U^{\pm 1}]/(TU^{-1} - p^{-1})$, which is a torsor under the diagonal subtorus $\mathbb{G}_m\subset T$. The Zariski closure of $Y$ in $X$ is
\[
\mathrm{Spec}~\mathbb{Z}_{(p)}[T^{\pm 1},U^{\pm 1}]/(pTU^{-1} - 1),
\]
which has empty special fiber! In particular, it is not faithfully flat (though it is normal), and so is no longer a torsor under the diagonal sub-torus.
\end{rem}

\begin{rem}
\label{strata:rem:local_models}
Suppose that one were able to prove~\eqref{strata:prp:xikphitorsor} directly somehow. This would give us good integral models for the mixed Shimura varieties appearing at the boundary. However, this would not be sufficient to prove the main theorem~\ref{strata:thm:strata}. One has to still show that every point of the boundary is accessible through one of these rational boundary components; equivalently, one would have to know that the boundary divisor is flat over $\Int_{(p)}$. The only way I know to see this is to use~\eqref{boundary:thm:main}.
\end{rem}

\begin{proof}[Proof of~\eqref{strata:prp:xikphitorsor}]
  Let ${\Phi}^\beef$ be the clr for $(G^\beef,X^{\beef})$ induced from $\Phi$. Set $\mb{\Xi}\coloneqq\Ss_{K_\Phi}(Q_\Phi,D_\Phi)$ and
  \[
  \mb{\Xi}^\beef\coloneqq \Ss_{\overline{K}_\Phi}(\overline{Q}_\Phi,\overline{D}_\Phi)\times_{\Ss_{\overline{K}^\beef_{\Phi^\beef}}(\overline{Q}_{\Phi^\beef},\overline{D}_{\Phi^\beef})}\Ss_{K^\beef_{\Phi^\beef}}(Q_{\Phi^\beef},D_{\Phi^\beef}).
  \]

  Also, set $\mb{E}=\mb{E}_K(\Phi)$ and $\mb{E}^\beef=\mb{E}_{{K}^{\beef}}({\Phi}^\beef)$, so that we have a homomorphism $\mb{E}\to \mb{E}^\beef$ of tori that is an isogeny onto its image.

  Then $\mb{\Xi}^\beef$ is an $\mb{E}^\beef$-torsor over $\Ss_{\overline{K}_{\Phi}}(\overline{Q}_\Phi,\overline{D}_\Phi)$, while $\mb{\Xi}[p^{-1}]$ is an $\mb{E}$-torsor over its generic fiber. There is an $\mb{E}$-equivariant map $\mb{\Xi}[p^{-1}]\to \mb{\Xi}^\beef$, and so the $\mb{E}$-action on $\mb{\Xi}[p^{-1}]$ lifts to an action on $\mb{\Xi}$.

  We need to show that $\mb{\Xi}$ is an $\mb{E}$-torsor over $\Ss_{\overline{K}_\Phi}(\overline{Q}_\Phi,\overline{D}_\Phi)$. For this, it is sufficient to work over a complete local ring of the base. Fix a point $t_0\in\Ss_{\overline{K}_\Phi}(\overline{Q}_\Phi,\overline{D}_\Phi)(\overline{\Field}_p)$ and let $R(\Phi,t_0)$ be the complete local ring of $\Ss_{\overline{K}_\Phi}(\overline{Q}_\Phi,\overline{D}_\Phi)$ at $t_0$. 

  Choose a rational polyhedral cone $\sigma\in \Sigma(\Phi)$, and a point $s_0\in\mathcal{Z}_{K_\Phi}(Q_\Phi,D_\Phi,\sigma)(\overline{\Field}_p)$ mapping to $t'_0$. If $\Upsilon=[(\Phi,\sigma)]$, we can view $s_0$ as a point of $\mathcal{Z}_K(\Upsilon)$, and thus of $\Ss^{\Sigma}_K$. Suppose that $\sigma^\beef\in\Sigma^\beef({\Phi}^\beef)$ contains $\sigma^\circ$ in its interior and $s_0^\beef$ is the image of $s_0$ in $Z_{K^\beef_{\Phi^\beef}}(Q_{\Phi^\beef},D_{\Phi^\beef},\sigma^\beef)$. Then the complete local ring $R(\sigma,s_0)$ of $\Ss^{\Sigma}_K$ at $s_0$ is a finite algebra over the complete local ring $R(\sigma^\beef,s_0^\beef)$ of $\mathcal{S}^{\Sigma^\beef}_{{K}^{\beef}}$. This algebra is described explicitly in~\eqref{boundary:thm:main}.

  By construction, the sub-ring $R(\Phi,t_0)\subset R(\sigma,s_0)$ is identified with the ring denoted $R_G^{\sab}$ in~\eqref{boundary:thm:main}. Let $\mb{E}^\diamond_G$ be the torus defined there, equipped with a finite homomorphism $\mb{E}^\diamond_G\to \mb{E}^\beef$\footnote{Note that what we have denoted by $\mb{E}^\beef$ here is denoted $\mb{E}$ in ~\eqref{boundary:thm:main}.}

  Let $\mb{B}^\diamond_G\subset \mb{B}^\beef$ be the co-character group of $\mb{E}^\diamond_G$. We claim that $\mb{B}^\diamond_G=\mb{B}$. We will prove this claim in~\eqref{strata:lem:bgb} below. It implies that the homomorphism $\mb{E}^\diamond_G\to\mb{E}^\beef$ is canonically identified with $\mb{E}\to \mb{E}^\beef$. 

  Now it follows from assertion~\eqref{main:local} of~\eqref{boundary:thm:main} that $\mb{\Xi}\vert_{\Spec R(\Phi,t_0)}$ is an $\mb{E}^\diamond_G = \mb{E}$-torsor. It can be verified from the construction that this $\mb{E}$-torsor structure is compatible with that on its generic fiber.
\end{proof}

\begin{lem}\label{strata:lem:bgb}
In the notation of the proof above, we have $\mb{B}^\diamond_G=\mb{B}$.
\end{lem}
\begin{proof}
We have to recall the construction of the sub-group $\mb{B}^\diamond_G\subset\mb{B}^\beef$. First, we have the canonical $\varphi$-module $M_0$ over $K_0=W[p^{-1}]$ associated with the point $s_0^\beef$ of $\mathcal{Z}_{{K}^{\beef}}({\Phi}^\beef,\tilde{\sigma})$; cf.~\eqref{boundary:subsubsec:blowupslogfcrys}. It is equipped with a weight filtration $W_\bullet M_0$ and a symplectic pairing $\psi_0$ with values in $K_0(-1)$. Let $U_{\wt}\subset\GSp(M_0,\psi_0)$ be the center of the unipotent radical of the parabolic subgroup stabilizing $W_\bullet M_0$.

We have a canonical family of $\varphi$-invariant tensors $\{s_{\alpha,\st,0}\}\subset M_0^{\otimes}$ whose stabilizer in $\GL(M_0)$ is isomorphic to $G_{K_0}$.

We have a canonical isomorphism~\eqref{boundary:eqn:liewp_embedding}:
\[
 K_0\otimes\Lie W_{\Phi^{\beef}}\xrightarrow{\simeq}K_0(-1)\otimes\Lie U_{\wt}.
\]

Choose any point $x\in(\Spec R(\sigma,s_0))(\overline{K}_0)$ that does not lie on the boundary divisor. In~\eqref{boundary:lem:liewp_embedding}, we showed that~\eqref{boundary:eqn:liewp_embedding} is obtained as the descent of an isomorphism arising from a composition
\[
\Bst\otimes_{\Rat}H(\Rat)\xrightarrow[\simeq]{1\otimes(\dual{\eta}_p)^{-1}}\Bst\otimes H^1_{\et}\bigl(\mathcal{A}_{x,\overline{K}_0},\Rat_p)\xrightarrow[\simeq]{\beta_{\st,x}}\Bst\otimes_{K_0}M_0.
\]
Here, $\beta_{\st,x}$ is the semistable comparison isomorphism, and $\eta_p$ is the $p$-primary part of any section of $\underline{\Isom}(H^g(\widehat{\Int}),\widehat{T}(\mathcal{Q}))$ over $\overline{K}_0$.

By~\eqref{boundary:cor:tatetensorprps}, for every index $\alpha$, $\beta_{\st,x}$ carries $1\otimes s_{\alpha,\et,x}$ to $1\otimes s_{\alpha,\st,0}$. Therefore, if we choose $\eta$ to be a section of $I(\Upsilon,s_0)$ (which is always possible by~\eqref{strata:lem:etaletorsor}), then, for every $\alpha$,  the above composition will carry $1\otimes s_{\alpha}$ to $1\otimes s_{\alpha,\st,0}$. In particular, if $U_{\wt,G}=U_{\wt}\cap G_{K_0}$, this composition will induce an isomorphism
\[
 \Bst\otimes_{\Rat}\Lie W_\Phi\xrightarrow{\simeq}\Bst\otimes_{K_0}\Lie U_{\wt,G},
\]
which will descend to an isomorphism
\[
 K_0\otimes_{\Rat}\Lie W_\Phi\xrightarrow{\simeq}K_0(-1)\otimes_{K_0}\Lie U_{\wt,G}
\]
that is compatible with~\eqref{boundary:eqn:liewp_embedding}.

Now, by definition, $\mb{B}^\diamond_G\otimes\Rat \subset\mb{B}^\beef\otimes\Rat$ is the intersection of the pre-image of $\Lie U_{\wt,G}(-1)$ under~\eqref{boundary:eqn:liewp_embedding} with a compact open subgroup $K^{\beef,\diamond}_{\Phi^\beef}\subset Q_{\Phi^\beef}(\Adele_f)$ such that
\[
K^{\beef,\diamond}_{\Phi^\beef}\cap Q_{\Phi}(\Adele_f) = K_{\Phi}.
\]
By what we have seen before this is also the intersection of this compact open subgroup with $\Lie W_\Phi$, and this is precisely $\mb{B}$.
\end{proof}

\begin{rem}\label{strata:rem:etale_analog}
The above result has the following \'etale analog, which implies Theorem~\ref{rationality}: Set $\widehat{V}(\mathcal{A}_x)=\Adele_f\otimes\widehat{T}(\mathcal{A}_x)$, and observe that this is equipped with the canonical Weil pairing $\psi_x$ induced from the polarization on $\mathcal{A}_x$. Let $U_{\Adele_f}\subset\GSp(\widehat{V}(\mathcal{A}_x),\psi_x)$ be the center of the unipotent radical preserving the weight filtration.

Then as above, we have a canonical isomorphism of Galois modules:
\[
\Adele_f(1)\otimes\Lie W_{\Phi^{\beef}}\xrightarrow{\simeq}\Lie U_{\Adele_f}.
\]
The stabilizer of the tensors $\{s_{\alpha,\et,x}\}\subset\widehat{V}(\mathcal{A}_x)$ gives us a canonical copy $G_{\Adele_f}\subset\GSp(\widehat{V}(\mathcal{A}_x),\psi_x)$. The non-emptiness of the torsor $I(\Upsilon,s_0)$ now shows that the canonical isomorphism above must map $\Adele_f(1)\otimes\Lie W_\Phi$ isomorphically onto $\Lie U_{G,\Adele_f} = \Lie G_{\Adele_f}\cap\Lie U_{\Adele_f}.$
\end{rem}

The next result completes the proof of~\eqref{strata:thm:strata}.

\begin{prp}\label{strata:prp:completion}
Assertion~\eqref{strata:strataisom} of~\eqref{strata:thm:strata} holds.
\end{prp}
\begin{proof}
Set $({\Phi}^\beef,\sigma^\beef)=\iota_*(\Phi,\sigma)$, and let $\Upsilon^\beef$ be its class in $\on{Cusp}_{{K}^{\beef}}^{\Sigma^\beef}(G^\beef,X^{\beef})$. We obtain a commutative diagram:
\begin{diagram}
\widehat{\Sh}_{K_\Phi}(Q_\Phi,D_\Phi,\sigma)&\rInto&\widehat{\Ss}_{K_{\Phi}}(Q_{\Phi},D_{\Phi},\sigma)&\rTo&\widehat{\Ss}_{K^\beef_{\Phi^\beef}}(Q_{\Phi^\beef},D_{\Phi^\beef},\sigma^\beef)\\
\dTo^{\simeq}&&&&\dTo_{\simeq}\\
\bigl(\Sh^\Sigma_K\bigr)^{\wedge}_{Z_K(\Upsilon)}&\rInto&\bigl(\Ss^\Sigma_K\bigr)^{\wedge}_{\mathcal{Z}_K(\Upsilon)}&\rTo&\bigl(\mathcal{S}^{\Sigma^\beef}_{{K}^{\beef}}\bigr)^{\wedge}_{\mathcal{Z}_{{K}^{\beef}}(\Upsilon^\beef)}.
\end{diagram}

The horizontal maps on the right-hand side are finite. So, applying~\eqref{appendix:lem:formalext}, we only have to show the following: Suppose that we are given a point $s_0\in\mathcal{Z}_{K_\Phi}(Q_\Phi,D_\Phi,\sigma)(\overline{\Field}_p)$ with image $s_0^\beef$ in $Z_{K^\beef_{\Phi^\beef}}(Q_{\Phi^\beef},D_{\Phi^\beef},\sigma^\beef)$. Let $R(\Phi,\sigma,s_0)$ (resp. $R(\Upsilon,s_0)$, $R({\Phi}^\beef,\sigma^\beef,s_0^\beef)$) be the complete local ring of $\Ss_{K_\Phi}(Q_{\Phi},D_{\Phi},\sigma)$ (resp. $\Ss^{\Sigma}_K$, resp. $\Ss_{K^\beef_{\Phi^\beef}}(Q_{\Phi^\beef},D_{\Phi^\beef},\sigma^\beef)$) at $s_0^\beef$ (resp. $s_0$, $s_0^\beef$). Then $R(\Phi,\sigma,s_0)$ and $R(\Upsilon,s_0)$ are isomorphic as $R({\Phi}^\beef,\sigma^\beef,s_0^\beef)$-algebras.
This can be readily deduced from the proof of~\eqref{strata:prp:xikphitorsor} and~\eqref{boundary:thm:main}\eqref{main:local}.
\end{proof}

\subsection{The case of hyperspecial level}\label{strata:subsec:smooth}

\subsubsection{}\label{main:subsubsec:derhamext}
The relative first de Rham homology of the universal abelian scheme $\mathcal{A}\to \mathcal{S}_{{K}^{\beef}}$ provides us with an extension of the filtered vector bundle $(\bm{H}_{\dR},F^0\bm{H}_{\dR})$ over $\mathcal{S}_{{K}^{\beef}}$. We will continue to use the same symbol to refer to this extension as well.

More generally, given a clr ${\Phi}^\beef$ for $(G^\beef,X^{\beef})$, the (covariant) de Rham realization of the universal $1$-motif over $\Ss_{K^\beef_{\Phi^\beef}}(Q_{\Phi^\beef},D_{\Phi^\beef})$ gives us an extension $(\bm{H}_{\dR}({\Phi}^\beef),W_\bullet\bm{H}_{\dR}({\Phi}^\beef),F^0\bm{H}_{\dR}({\Phi}^\beef))$ of the doubly filtered vector bundle over $\Sh_{K^\beef_{\Phi^\beef}}(Q_{\Phi^\beef},D_{\Phi^\beef})$, which we had denoted by the same symbol.

Given any $\sigma^\beef\in\Sigma^\beef({\Phi}^\beef)$, and any point $s_0^\beef\in \mathcal{Z}_{K^\beef_{\Phi^\beef}}(Q_{\Phi^\beef},D_{\Phi^\beef},\sigma^\beef)$, write as usual $R=R({\Phi}^\beef,\sigma^\beef,s_0^\beef)$ for the complete local ring of $\Ss_{K^\beef_{\Phi^\beef}}(Q_{\Phi^\beef},D_{\Phi^\beef},\sigma^\beef)$ at $s_0^\beef$, and let $V\subset\Spec R$ be the intersection with $\Ss_{K^\beef_{\Phi^\beef}}(Q_{\Phi^\beef},D_{\Phi^\beef})$. Then the evaluation over $\Spec R$ of the log Dieudonn\'e crystal associated with the tautological $1$-motif $\mathcal{Q}\vert_V$ (or rather the degenerating abelian scheme associated with this $1$-motif) gives us a canonical extension over $\Spec R$ of the restriction of $\bm{H}_{\dR}({\Phi}^\beef)$ to $V$, along with its two filtrations.

From this, faithfully flat descent, and~\cite[Prop. 4.2]{ferrand_raynaud}, we see that there exists a canonical extension $\bm{H}_{\dR}({\Phi}^\beef,\sigma^\beef)$ over $\Ss_{K^\beef_{\Phi^\beef}}(Q_{\Phi^\beef},D_{\Phi^\beef},\sigma^\beef)$ of $\bm{H}_{\dR}({\Phi}^\beef)$ as a doubly filtered vector bundle. It is equipped with an integrable connection with logarithmic poles along the boundary divisor.

It can be checked from the explicit description of the log Dieudonn\'e crystal in~\eqref{boundary:subsubsec:fcrystal} that this construction agrees with the analytic one given in~\eqref{background:subsubsec:derham_extension} over $\Sh_{K_\Phi}(Q_{\Phi},D_{\Phi},\sigma)(\Comp)$.

There is now a canonical extension $(\bm{H}_{\dR}(\Sigma^\beef),F^0\bm{H}_{\dR}(\Sigma^\beef))$ of $(\bm{H}_{\dR},F^0\bm{H}_{\dR})$ to a filtered vector bundle over $\mathcal{S}^{\Sigma^\beef}_{{K}^{\beef}}$ equipped with an integrable connection with logarithmic poles along the boundary, and characterized by the following property: For any cusp label $\Upsilon^\beef=[({\Phi}^\beef,\sigma^\beef)]$, the restriction of this extension to
\[
(\mathcal{S}^{\Sigma^\beef}_{{K}^{\beef}})^{\wedge}_{\mathcal{Z}_{{K}^{\beef}}(\Upsilon^\beef)}=\widehat{\Ss}_{K^\beef_{\Phi^\beef}}(Q_{\Phi^\beef},D_{\Phi^\beef},\sigma^\beef)
\]
is isomorphic, as a filtered vector bundle with integrable connection, to the restriction of $\bm{H}_{\dR}({\Phi}^\beef,\sigma^\beef)$.

\subsubsection{}
We will now make some further assumptions: First, the prime $p$ will be such that the group $G$ admits a reductive model $G_{\Int_{(p)}}$ over $\Int_{(p)}$. Second, we will assume that the principally polarized lattice $H(\Int)\subset H(\Rat)$ has been chosen such that the symplectic embedding $\iota:G\into\GSp(H)$ arises from a closed immersion of $\Int_{(p)}$-groups $G_{\Int_{(p)}}\into\GL\bigl(H_{\Int_{(p)}}\bigr)$. It is always possible to choose such a lattice; cf.~\cite[2.3.1]{kisin:abelian}.\footnote{There is an additional hypothesis in the statement of the cited lemma that the group $G$ not have factors of type $B$ when $p=2$. However, this is not necessary; cf.~\cite[Lemma 4.7]{Kim2016-fb}.} Finally, we will assume that $K_p = G_{\Int_{(p)}}(\Int_p)$ is hyperspecial. We will call such an embedding \defnword{$p$-integral}.

The notation and hypotheses from the previous sections will continue to hold. The main result now is the following~\cite[2.3.5]{kisin:abelian}, which is augmented for the case $p=2$ by~\cite{Kim2015-bp}:

\begin{thm}[Kisin]\label{strata:thm:smooth}
$\Ss_K$ is smooth over $\Reg{E,(v)}$.
\end{thm}
\qed

\subsubsection{}\label{strata:subsubsec:hyperspecial_boundary}
Fix a clr $\Phi$ for $(G,X,K)$. The parabolic sub-group $P_\Phi$ extends to a parabolic sub-group $P_{\Phi,\Int_{(p)}}\subset G_{\Int_{(p)}}$. Also, it follows from~\cite[1.6.9]{hoermann:thesis} that the normal sub-group $Q_\Phi\subset P_\Phi$ extends to a normal sub-group $Q_{\Phi,\Int_{(p)}}\subset P_{\Phi,\Int_{(p)}}$ with reductive image in the Levi quotient $P_{\Phi,\Int_{(p)}}/U_{\Phi,\Int_{(p)}}$; this image gives us a reductive model $G_{\Phi,h,\Int_{(p)}}$ for the group $G_{\Phi,h}$.

Now, the level at $p$ $K_{\Phi,h,p}\subset G_{\Phi,h}(\Rat_p)$ is a conjugate of $G_{\Phi,h,\Int_{(p)}}(\Int_p)$, and is thus hyperspecial. Therefore, it follows from~\eqref{strata:thm:smooth} that $\Ss_{K_{\Phi,h}}(G_{\Phi,h},D_{\Phi,h})$ is smooth over $\Reg{E,(v)}$. Since $\Ss_{K_\Phi}(Q_\Phi,D_\Phi)$ is smooth over $\Ss_{K_{\Phi,h}}(G_{\Phi,h},D_{\Phi,h})$ (it is a torus torsor over a torsor under an abelian scheme over $\Ss_{K_{\Phi,h}}(G_{\Phi,h},D_{\Phi,h})$), we see that it must also be smooth over $\Reg{E,(v)}$. Of course, we can also deduce this from the fact that it is \'etale locally isomorphic to $\Ss_K$.

We will consider the projective limit:
\[
 \Ss_p(Q_\Phi,D_\Phi)=\varprojlim_{K^p\subset G(\Adele_f^p)}\Ss_{K_{\Phi,h,p}K^p_{\Phi}}(Q_\Phi,D_\Phi).
\]
Since the transition maps in this system are finite \'etale, this is a scheme over $\Reg{E,(v)}$ that is locally of finite type. From the previous paragraph, we also find that it is regular and formally smooth over $\Reg{E,(v)}$. We will denote the generic fiber of this scheme by $\Sh_p(Q_\Phi,D_\Phi)$.

\begin{prp}\label{strata:prp:canonical}
The scheme $\Ss_p(Q_\Phi,D_\Phi)$ is a \defnword{canonical model} for $\Sh_p(Q_\Phi,D_\Phi)$ over $\Reg{E,(v)}$. That is, given any regular, formally smooth scheme $S$ over $\Reg{E,(v)}$, any map $E\otimes S\to\Sh_p(Q_\Phi,D_\Phi)$ extends to a map $S\to\Ss_p(Q_\Phi,D_\Phi)$.

In particular, $\Ss_p(Q_\Phi,D_\Phi)$, as well as its $Q_{\Phi}(\Adele_f^p)$-equivariant structure, is canonically determined by the Shimura datum $(G,X)$ and the reductive model $G_{\Int_{(p)}}$. It does not depend on the choice of symplectic embedding.
\end{prp}
\begin{proof}
This is shown as in~\cite[2.3.8]{kisin:abelian}, except that, instead of extension theorems for abelian schemes, we need to use ones for $1$-motifs; cf.~\eqref{appendix:lem:neronogg1motifs} and~\eqref{appendix:lem:1motifspurity}.
\end{proof}

\begin{proof}[Proof of Theorem~\ref{smooth}]
The immersion $\Ss_K\into\Ss^{\Sigma}_K$ is \'etale locally isomorphic to $\Ss_{K_\Phi}(Q_\Phi,D_\Phi)\into\Ss_{K_\Phi}(Q_{\Phi},D_{\Phi},\sigma)$, for some clr $\Phi$ and $\sigma\in\Sigma(\Phi)$. If we choose $\Sigma$ to be smooth and complete, then it follows that $\Ss^{\Sigma}_K$ is again smooth, and that the boundary is a normal crossing divisor.

The claim that $\Ss^{\Sigma}_K$ depends only on $\Sigma$ and not on the choice of symplectic embedding is a consequence of~\eqref{strata:thm:strata},~\eqref{strata:prp:canonical} and~\eqref{appendix:lem:abstract_hecke}
\end{proof}

\subsubsection{}\label{min:subsubsec:smooth_tensors}
Suppose that the tensors $\{s_{\alpha}\}\subset H(\Rat)^\otimes$ actually lie in $H(\Int_{(p)})^\otimes$, and that $G_{\Int_{(p)}}$ is their point-wise stabilizer in $\GL\bigl(H_{\Int_{(p)}}\bigr)$. By~\cite[1.3.2]{kisin:abelian}, this can always be arranged. It is shown in~\cite[2.3.9]{kisin:abelian} that the associated sections $\{s_{\alpha,\dR}\}$ of $\bm{H}_{\dR}^\otimes\vert_{\Sh_K}$ extend over $\Ss_K$.

Define the functor $\mathcal{P}_{\dR}(\iota)$ on $\Ss_K$-schemes, so that, for any $\Ss_K$-scheme $T$, $\mathcal{P}_{\dR}(\iota)(T)$ consists of isomorphisms of vector bundles
\[
\xi:\Reg{T}\otimes_{\Int_{(p)}}H(\Int_{(p)})\xrightarrow{\simeq}\bm{H}_{\dR}\vert_T
\]
such that, for all $\alpha$, $\xi(1\otimes s_{\alpha})=s_{\alpha,\dR}\in H^0(T,\bm{H}_{\dR}^\otimes)$. Then $\mathcal{P}_{\dR}(\iota)$ is a $G_{\Int_{(p)}}$-torsor over $\Ss_K$. This follows from the proof of~\cite[2.3.9]{kisin:abelian}, and is explained in detail for a special case in~\cite[\S 4]{mp:reg}.

\begin{prp}\label{min:prp:canonical_extension}
\mbox{}
\begin{enumerate}
\item~\label{can_ext:tensors}For all $\alpha$, $s_{\alpha,\dR}$ extends (necessarily uniquely) to a section
\[
 s_{\alpha,\dR,\Sigma}\in H^0\bigl(\Ss^\Sigma_K,\bm{H}^\otimes_{\dR}(\Sigma)\bigr).
\]
\item~\label{can_ext:torsor}Consider the functor $\mathcal{P}_{\dR}^{\Sigma}(\iota)$ on $\Ss^\Sigma_K$-schemes assigning to every $T$ the set of isomorphisms:
\[
 \xi:\Reg{T}\otimes_{\Int_{(p)}}H(\Int_{(p)})\xrightarrow{\simeq}\bm{H}_{\dR}(\Sigma)\vert_T
\]
such that, for all $\alpha$, $\xi(1\otimes s_{\alpha})=s_{\alpha,\dR,\Sigma}\in H^0(T,\bm{H}_{\dR}^\otimes(\Sigma))$: this is a $G_{\Int_{(p)}}$-torsor over $\Ss^\Sigma_K$, extending $\mathcal{P}_{\dR}(\iota)$.
\end{enumerate}
\end{prp}
\begin{proof}
Let $U\subset\Ss^{\Sigma}_K$ be the complement of the special fiber of the boundary divisor $\mb{D}^{\Sigma}_K$: The complement of $U$ has codimension $2$ in the normal scheme $\Ss^{\Sigma}_K$. Therefore, the restriction functor from vector bundles on $\Ss^{\Sigma}_K$ to those on $U$ is fully faithful.

So it suffices to show that, given an index $\alpha$, $s_{\alpha,\dR}$ extends over $U$, and that $\mathcal{P}^{\Sigma}_{\dR}(\iota)$ admits a section over an \'etale cover of $U$. For this, it is sufficient to know that its restriction to $\Sh_K$ extends over $\Sh^{\Sigma}_K$, and that $\mathcal{P}^{\Sigma}_{\dR}(\iota)$ has a section over an \'etale cover of $\Sh^{\Sigma}_K$.

From~\eqref{boundary:prp:twotensors} and~\eqref{boundary:prp:abshodge},~\eqref{can_ext:tensors} will follow if we can show that, for each $\alpha$, $s_{\alpha,\Phi,\dR}$ extends to a section of $\bm{H}_{\dR}(\Phi,\sigma)^\otimes$ over $\Sh_{K_\Phi}(Q_{\Phi},D_{\Phi},\sigma)(\Comp)$. This is a consequence of the discussion in~\eqref{background:subsubsec:derham_extension}. In particular, the functor $\mathcal{P}_{\dR}^{\Sigma}(\iota)$ is now well-defined over all of $\Ss^{\Sigma}_K$, and the same discussion implies that it has \'etale local sections over a neighborhood of the boundary as well.
\end{proof}

\subsubsection{}\label{min:subsubsec:automorphic}
As in~\eqref{boundary:subsubsec:shimuradim}, let $[\mu]$ be the conjugacy class of co-characters of $G$ associated with $X$. We can find a finite extension $F/E$ and a place $w\vert v$ of $F$ such that $[\mu]$ has a representative $\mu:\Gmh{\Reg{F,(w)}}\to G_{\Reg{F,(w)}}$. Let $P_{\mu}\subset G_{\Reg{F,(w)}}$ be the parabolic sub-group whose Lie algebra consists of the non-negative weight spaces for $\mu$, and let $U_{\mu}\subset P_{\mu}$ be its unipotent radical.

The co-character $\mu^{-1}$ splits a decreasing filtration on $H_{\Reg{F,(w)}}$:
\[
 0=F^1H_{\Reg{F,(w)}}\subset F^0H_{\Reg{F,(w)}}\subset F^{-1}H_{\Reg{F,(w)}}=H_{\Reg{F,(w)}}.
\]
The sub-group $P_{\mu}$ is precisely the stabilizer of this filtration.

We will consider the sub-functor $\mathcal{P}^{\Sigma}_{\dR,\mu}(\iota)$ of $\Reg{F,(w)}\otimes_{\Reg{E,(v)}}\mathcal{P}^{\Sigma}_{\dR}(\iota)$ parameterizing isomorphisms $\xi$ that also satisfy:
\[
 \xi(\Reg{T}\otimes_{\Reg{F,(w)}}F^0H_{\Reg{F,(w)}})=F^0\bm{H}_{\dR}(\Sigma)\vert_T.
\]

\begin{prp}\label{min:prp:automorphic_torsor}
$\mathcal{P}^{\Sigma}_{\dR,\mu}(\iota)$ is a canonical reduction of structure group of $\Reg{F,(w)}\otimes_{\Reg{E,(v)}}\mathcal{P}^{\Sigma}_{\dR}(\iota)$ to a $P_{\mu}$-torsor over $\Reg{F,(w)}\otimes_{\Reg{E,(v)}}\Ss^\Sigma_K$.
\end{prp}
\begin{proof}
  It suffices to show that $\mathcal{P}^{\Sigma}_{\dR,\mu}(\iota)$ admits \'etale local sections over $\Ss_K$ and over $\Sh^{\Sigma}_K(\Comp)$. Over $\Ss_K$, this again follows from the proof of~\cite[(2.3.9)]{kisin:abelian}; cf.~\cite[(4.17)]{mp:reg} for an explication in a special case. Over $\Sh^{\Sigma}_K(\Comp)$, it can be deduced from the explicit local description in~\eqref{background:subsubsec:derham_extension}
\end{proof}

\begin{prp}\label{min:prp:automorphic_torsor_independence}
Let $\iota':G_{\Int_{(p)}}\into\GL\bigl(H'_{\Int_{(p)}}\bigr)$ be another $p$-integral embedding underlying an embedding of Shimura data:
\[
 \iota^\diamond:(G,X)\into (G^\diamond,X^{\diamond})\coloneqq(\GSp(H^\diamond),\on{S}^{\pm}(H^\diamond)).
\]
Then there is a canonical isomorphism $\mathcal{P}^{\Sigma}_{\dR}(\iota)\xrightarrow{\simeq}\mathcal{P}^{\Sigma}_{\dR}(\iota^\diamond)$ of $G_{\Int_{(p)}}$-torsors over $\Ss^{\Sigma}_K$, as well as a canonical isomorphism $\mathcal{P}^{\Sigma}_{\dR,\mu}(\iota)\xrightarrow{\simeq}\mathcal{P}^{\Sigma}_{\dR,\mu}(\iota^\diamond)$ of $P_{\mu}$-torsors over $\Reg{F,(w)}\otimes\Ss^{\Sigma}_K$.
\end{prp}
\begin{proof}
  Set $H^\sharp_{\Int_{(p)}}=H_{\Int_{(p)}}\oplus H^\diamond_{\Int_{(p)}}$, and equip it with the direct sum symplectic structure. We then obtain a third $p$-integral embedding:
  \[
   \iota^\sharp:G_{\Int_{(p)}}\into\GL\bigl(H_{\Int_{(p)}}\bigr)\times\GL\bigl(H^\diamond_{\Int_{(p)}}\bigr)\into\GL\bigl(H^\sharp_{\Int_{(p)}}\bigr),
  \]
  where the first map is the diagonal embedding. This gives us a third $G_{\Int_{(p)}}$-torsor $\mathcal{P}_{\dR}^{\Sigma}(\iota^\sharp)$ parameterizing certain trivializations
  \[
   \xi^\sharp:\Reg{T}\otimes_{\Int_{(p)}} H^\sharp_{\Int_{(p)}}\xrightarrow{\simeq}\bm{H}^\sharp_{\dR}(\Sigma)\vert_T.
  \]

  We claim that the restriction of $\xi^\sharp$ to $\Reg{T}\otimes_{\Int_{(p)}}H_{\Int_{(p)}}$ maps isomorphically onto $\bm{H}_{\dR}(\Sigma)\vert_T$ and gives a section of $\mathcal{P}_{\dR}^{\Sigma}(\iota)$. It suffices to check this over the generic fiber, and hence over $\Comp$, where it is clear.

  Therefore, we obtain a canonical isomorphism of $G_{\Int_{(p)}}$-torsors:
  \[
  \mathcal{P}_{\dR}^{\Sigma}(\iota^\sharp)\xrightarrow{\simeq}\mathcal{P}_{\dR}^{\Sigma}(\iota).
  \]
  A similar assertion holds with $\iota$ replaced by $\iota^\diamond$. This proves the assertion about $G_{\Int_{(p)}}$-torsors.

  To prove the assertion about $P_{\mu}$-torsors, it suffices to check that the canonical isomorphism of $G_{\Int_{(p)}}$-torsors respects the reductions of structure group to $P_{\mu}$-torsors. This can be checked over $\Sh_K(\Comp)$, where it is clear from the analytic uniformization.
\end{proof}

\section{The minimal compactification}\label{sec:min}

We will now construct integral models for the minimal or Baily-Borel-Satake compactification of the integral model $\Ss_K$ above. We will follow the strategy in \cite[\S V.2]{faltings_chai}, which is extended to the PEL (good reduction) case in \cite[\S 7.2]{lan:thesis}. Since the method here is not very different from that used in \emph{loc. cit.}, our treatment will be somewhat compressed.

The notation will be as in the previous section. We will only consider complete rpcds, so that the spaces $\mathcal{S}_{{K}^{\beef}}^{\Sigma^\beef}$ and $\Ss_K^\Sigma$ are \emph{proper} over $\Reg{E,(v)}$.

\subsection{The Hodge bundle and Fourier-Jacobi expansions}\label{min:subsec:hodge}

\begin{defn}\label{min:defn:hodge}
The \defnword{Hodge bundle} $\omega_{{K}^{\beef}}(\Sigma^\beef)$ over $\mathcal{S}^{\Sigma^\beef}_{{K}^{\beef}}$ is:
\[
\omega_{{K}^{\beef}}(\Sigma^\beef)\coloneqq\underline{\det}\bigl(\bm{H}_{\dR}(\Sigma^\beef)/F^0\bm{H}_{\dR}(\Sigma^\beef)\bigr)^{\otimes -1}.
\]
\end{defn}

\subsubsection{}\label{min:subsubsec:omega}
Write $\omega_K(\Sigma)$ for the restriction of $\omega_{{K}^{\beef}}(\Sigma^\beef)$ to $\Ss_K^{\Sigma}$: It is easy to see that this invertible sheaf does not depend on the choice of ${K}^{\beef}$ containing $K$ or on the choice of rpcd $\Sigma^\beef$ compatible with $\Sigma$.

If $\Sigma'$ is a refinement of $\Sigma$, induced from an admissible rpcd $\widetilde{\Sigma'}$ for $(G^{\beef},X^{\beef},{K}^{\beef})$, then $\omega_K(\Sigma')$ is canonically isomorphic to the pull-back of $\omega_K(\Sigma)$ along the map $\Ss^{\Sigma'}_K\to\Ss^{\Sigma}_K$.

Therefore, given any refinement $\Sigma'$ of $\Sigma$, we can without qualms denote the pull-back of $\omega_K(\Sigma)$ to $\Ss^{\Sigma'}_K$ by $\omega_K(\Sigma')$.

\begin{prp}\label{min:prp:omega}
\mbox{}
\begin{enumerate}
\item\label{omega:globalgen}A suitable power of $\omega_K(\Sigma)$ is generated by its global sections over $\Ss^{\Sigma}_K$.
\item\label{omega:projisom}If $\Sigma'$ is a refinement of $\Sigma$, for any $n\in\Int_{\geq 0}$, the natural map:
\[
 H^0(\Ss^{\Sigma}_K,\omega^{\otimes n}_K(\Sigma))\to H^0(\Ss^{\Sigma'}_K,\omega^{\otimes n}_K(\Sigma')).
\]
is an isomorphism.
\item\label{omega:gradedfin}The graded $\Reg{E,(v)}$-algebra
\[
 \bigoplus_nH^0\bigl(\Ss^\Sigma_K,\omega^{\otimes n}_K(\Sigma)\bigr)
\]
is finitely generated.
\end{enumerate}
\end{prp}
\begin{proof}
Assertion~\eqref{omega:globalgen} is shown as in~\cite[V.2.1]{faltings_chai}, using a result of Moret-Bailly~\cite[IX.2.1]{moret-bailly:pinceaux}.

Assertion~\eqref{omega:projisom} is immediate from the projection formula, and the following facts about the map $f:\Ss^{\Sigma'}_K\to\Ss^{\Sigma}_K$:
\begin{itemize}
\item $f_*\Reg{\Ss^{\Sigma'}_K}=\Reg{\Ss^{\Sigma}_K}$;
\item $R^if_*\Reg{\Ss^{\Sigma'}_K}=0$, for $i>0$.
\end{itemize}
This can be deduced as in~\cite[V.1.2(b)]{faltings_chai} from the \'etale local structure of the map $f$ (cf.~\ref{strata:prp:hecke}).

Finally,~\eqref{omega:gradedfin} follows from~\cite[7.2.2.6]{lan:thesis}.
\end{proof}

\subsubsection{}\label{min:subsubsec:omegakphi}
Fix a pair $(\Phi,\sigma)$ for $(G,X)$, and let $({\Phi}^\beef,\sigma^\beef)$ be the pair induced for $(G^\beef,X^{\beef})$. Write $\bm{H}_{\dR}(\Phi,\sigma)$ for the restriction of $\bm{H}_{\dR}({\Phi}^\beef,\sigma^\beef)$ over $\Ss_{K_\Phi}(Q_{\Phi},D_{\Phi},\sigma)$.

We then have:
\begin{align}\label{min:eqn:omegakphi1}
\underline{\det}\bigl(\bm{H}_{\dR}(\Phi,\sigma)/F^0\bm{H}_{\dR}(\Phi,\sigma)\bigr)^{\otimes -1}&\xrightarrow{\simeq}\otimes_{i=-2}^0\underline{\det}\bigl(\gr^W_i\bm{H}_{\dR}(\Phi,\sigma)/F^0\gr^W_i\bm{H}_{\dR}(\Phi,\sigma)\bigr)^{\otimes -1}\\
&\xrightarrow{\simeq}\omega^{\et}(\Phi)\otimes_{\Int}\omega^{\ab}_K(\Phi).\nonumber
\end{align}
We need to explain the symbols in the last row. Here, $\omega^{\ab}_K(\Phi)$ is the Hodge bundle associated with the universal abelian scheme $\mathcal{B}\to\Ss_{K_{\Phi,h}}(G_{\Phi,h},D_{\Phi,h})$; we denote its pull-back to $\Ss_{K_\Phi}(Q_{\Phi},D_{\Phi},\sigma)$ by the same symbol. Also, $\omega^{\et}(\Phi)=\det(\gr^WH^g(\Int))^{\otimes -1}$ is a free $\Int$-module of rank $1$.

In particular, the line bundle defined in~\eqref{min:eqn:omegakphi1} does not depend on $\sigma$, and so we denote it by $\omega_K(\Phi)$. Note that $\omega_K(\Phi)$ is actually already defined over $\Ss_{K_{\Phi,h}}(G_{\Phi,h},D_{\Phi,h})$.

\subsubsection{}\label{min:subsubsec:fourier_jacobi}
As a scheme over $\Ss_{\overline{K}_\Phi}(\overline{Q}_\Phi,\overline{D}_\Phi)$, we have:
\[
 \Ss_{K_\Phi}(Q_\Phi,D_\Phi)=\underline{\Spec}\bigl(\bigoplus_{\ell\in\mb{S}_K(\Phi)}\Psi^{(\ell)}_K(\Phi)\bigr),
\]
where, for each $\ell\in\mb{S}_K(\Phi)$, $\Psi^{(\ell)}_K(\Phi)$ is a line bundle over $\Ss_{\overline{K}_\Phi}(\overline{Q}_\Phi,\overline{D}_\Phi)$. Therefore, we have:
\begin{align}\label{min:eqn:hatmbxi}
 \widehat{\Ss}_{K_{\Phi}}(Q_{\Phi},D_{\Phi},\sigma)&=\underline{\Spf}\bigl(\widehat{\bigoplus}_{\ell\in\mb{S}_K(\Phi,\sigma)}\Psi^{(\ell)}_K(\Phi)\bigr).
\end{align}
Here, $\widehat{\bigoplus}_{\ell\in\mb{S}_K(\Phi,\sigma)}\Psi^{(\ell)}_K(\Phi)$ is the completion of $\bigoplus_{\ell\in\mb{S}_K(\Phi,\sigma)}\Psi^{(\ell)}_K(\Phi)$ along the ideal generated by $\Psi^{(\ell)}_K(\Phi)$, with $\ell\in\mb{S}_K(\Phi,\sigma)\backslash\mb{S}_K(\Phi,\sigma)^\times$.

Consider the map $\pr:\Ss_{\overline{K}_\Phi}(\overline{Q}_\Phi,\overline{D}_\Phi)\to\Ss_{K_{\Phi,h}}(G_{\Phi,h},D_{\Phi,h})$: It is a projective morphism, and so, for each $\ell\in\mb{S}_K(\Phi)$, we obtain a coherent sheaf
\[
 \on{FJ}^{(\ell)}_K(\Phi)=\pr_*\Psi^{(\ell)}_K(\Phi)
\]
over $\Ss_{K_{\Phi,h}}(G_{\Phi,h},D_{\Phi,h})$.

If $\Upsilon=[(\Phi,\sigma)]$, for any $n\in\Int_{\geq 0}$, we now obtain an evaluation map:
\begin{align}\label{min:eqn:fjphisig}
\on{FJ}_{(\Phi,\sigma)}:H^0\bigl(\Ss^{\Sigma}_K,\omega^{\otimes n}_K(\Sigma)\bigr)&\to H^0\biggl(\bigl(\Ss^{\Sigma}_K\bigr)^{\wedge}_{\mathcal{Z}_K(\Upsilon)},\omega^{\otimes n}_K(\Sigma)\biggr)\\\nonumber
&\xrightarrow[\simeq]{\eqref{min:eqn:omegakphi1}}H^0\bigl(\widehat{\Ss}_{K_{\Phi}}(Q_{\Phi},D_{\Phi},\sigma),\omega^{\otimes n}_K(\Phi)\bigr)\\\nonumber
&\into\prod_{\ell\in\mb{S}_K(\Phi)}H^0\bigl(\Ss_{K_{\Phi,h}}(G_{\Phi,h},D_{\Phi,h}),\on{FJ}^{(\ell)}_K(\Phi)\otimes\omega^{\otimes n}_K(\Phi)\bigr).\nonumber
\end{align}
Here, the last map is obtained from~\eqref{min:eqn:hatmbxi} and the projection formula.

\subsubsection{}\label{min:subsubsec:deltakphi}
Suppose that we have another clr $\Phi'$ for $(G,K)$, and that $\gamma\in G(\Rat)$ and $q\in Q_{\Phi'}(\Adele_f)$ are such that $\Phi\xrightarrow{(\gamma,q)_K}\Phi'$ and $\gamma\cdot P_\Phi=P_{\Phi'}$. Then, as explained in~\eqref{strata:subsubsec:phifunctorialp}, we have an associated tower-structure preserving isomorphism:
\[
 \rho(\gamma,q):\Ss_{K_\Phi}(Q_\Phi,D_\Phi)\xrightarrow{\simeq}\Ss_{K_{\Phi'}}(Q_{\Phi'},D_{\Phi'}).
\]

Let $\mathcal{Q}(\Phi)$ (resp. $\mathcal{Q}(\Phi')$) be the canonical $1$-motif over $\Ss_{K_\Phi}(Q_\Phi,D_\Phi)$ (resp. $\Ss_{K_{\Phi'}}(Q_{\Phi'},D_{\Phi'})$. By the construction of $\rho(\gamma,q)$ (cf.~\ref{background:subsubsec:intfunct}), there is a canonical isomorphism $\rho(\gamma,q)^*\mathcal{Q}(\Phi')\xrightarrow{\simeq}\mathcal{Q}(\Phi)$. This induces an isomorphism of line bundles
\begin{align}\label{min:eqn:rhogammaq}
 \rho(\gamma,q)^*\omega_K(\Phi')\xrightarrow{\simeq}\omega_K(\Phi).
\end{align}

Let $\Sigma'$ be another admissible rpcd for $(G,X,K)$, and suppose that we have $\sigma'\in\Sigma'(\Phi')$. Then we get a diagram:
\begin{equation}\label{min:eqn:fj_independence}
\begin{diagram}
H^0\bigl(\Ss^{\Sigma}_K,\omega^{\otimes n}_K(\Sigma)\bigr)&\rTo_{\simeq}^{\eqref{min:prp:omega}~\eqref{omega:projisom}}&H^0\bigl(\Ss^{\Sigma'}_K,\omega^{\otimes n}_K(\Sigma)\bigr)&\rTo^{\on{FJ}_{(\Phi',\sigma')}}&\prod_{\ell'\in\mb{S}_K(\Phi')}H^0\bigl(\Ss_{K_{\Phi',h}}(G_{\Phi',h},D_{\Phi',h}),\on{FJ}^{(\ell')}_K(\Phi')\otimes\omega^{\otimes n}_K(\Phi')\bigr)\\
&\rdTo(4,2)_{\on{FJ}_{(\Phi,\sigma)}}&&&\dTo^{\simeq}_{\rho(\gamma,q)^*}\\
&&&&\prod_{\ell\in\mb{S}_K(\Phi)}H^0\bigl(\Ss_{K_{\Phi,h}}(G_{\Phi,h},D_{\Phi,h}),\on{FJ}^{(\ell)}_K(\Phi)\otimes\omega^{\otimes n}_K(\Phi)\bigr).
\end{diagram}
\end{equation}
Here, the vertical isomorphism is induced via~\eqref{min:eqn:rhogammaq}.

\begin{lem}\label{min:lem:indphisig}
The above diagram commutes.
\end{lem}
\begin{proof}
 This is immediately reduced to the case where $\Sigma$ is a refinement of $\Sigma'$, and $\on{int}(\gamma)(\sigma^{\circ})$ is contained in $(\sigma')^{\circ}$. Let $\Upsilon'=[(\Phi',\sigma')]$ be the image of $(\Phi',\sigma')$ in $\mb{Cusp}_K^{\Sigma'}(G,X)$. Then we have a map $\Ss^{\Sigma'}_K\to\Ss^{\Sigma}_K$ carrying $\mathcal{Z}_K(\Upsilon')$ to $\mathcal{Z}_K(\Upsilon)$. We also have a map of formal schemes:
  \[
   \rho(\gamma,q):\widehat{\Ss}_{K_{\Phi}}(Q_{\Phi},D_{\Phi},\sigma)\to\widehat{\Ss}_{K_{\Phi'}}(Q_{\Phi'},D_{\Phi'},\sigma')
  \]
  compatible with the isomorphisms with $\bigl(\Ss^{\Sigma}_K\bigr)^{\wedge}_{\mathcal{Z}_K(\Upsilon)}$, and $\bigl(\Ss^{\Sigma'}_K\bigr)^{\wedge}_{\mathcal{Z}_K(\Upsilon')}$, respectively, and pulling $\omega_K(\Phi')$ back to $\omega_K(\Phi)$. Therefore, it is clear that the diagram commutes.
\end{proof}

\begin{corollary}\label{min:cor:fourier_jacobi}
Let $\Delta_K(\Phi)$ be as in~\eqref{background:eqn:deltaphi}. Let $\mb{P}_K(\Phi)\subset\mb{S}_K(\Phi)$ be the sub-monoid of elements that pair non-negatively with $\mb{H}(\Phi)$. Then $\on{FJ}_\Phi=\on{FJ}_{(\Phi,\sigma)}$ does not depend on the choice of $\sigma\in\Sigma(\Phi)$, and its image lands in the sub-algebra:
\[
 \biggl[\prod_{\ell\in\mb{P}_K(\Phi)}H^0\bigl(\Ss_{K_{\Phi,h}}(G_{\Phi,h},D_{\Phi,h}),\on{FJ}^{(\ell)}_K(\Phi)\otimes\omega^{\otimes n}_K(\Phi)\bigr)\biggr]^{\Delta_K(\Phi)}
\]
\end{corollary}
\qed

\subsection{Construction of the minimal compactification}\label{min:subsec:const}

\subsubsection{}
Set
\[
\Ss^{\min}_K=\Proj\bigl(\oplus_nH^0(\Ss^\Sigma_K,\omega^{\otimes n}_K(\Sigma))\bigr).
\]
This is a projective $\Reg{E,(v)}$-scheme, and using~\eqref{min:prp:omega}, we find that, for any sufficiently fine complete rpcd $\Sigma'$ for $(G,X,K)$, we have a canonical map (cf.~\cite[7.2.3]{lan:thesis}):
\[
 \oint^{\Sigma'}:\Ss^{\Sigma'}_K\to\Proj\bigl(\oplus_nH^0(\Ss^{\Sigma'}_K,\omega^{\otimes n}_K(\Sigma'))\bigr)=\Ss^{\min}_K.
\]

This map can be described as follows: Choose $N\in\Int_{\geq 1}$ sufficiently large so that $\omega^{\otimes N}_K(\Sigma')$ is generated by global sections. Then we have a proper map
\[
\Ss^{\Sigma'}_K\to\bb{P}\bigl(H^0(\Ss^{\Sigma'}_K,\omega^{\otimes N}_K(\Sigma'))\bigr),
\]
and $\oint^{\Sigma'}$ appears in the Stein factorization of this map:
\begin{align}\label{min:eqn:steinfact}
 \Ss^{\Sigma'}_K\xrightarrow{\oint^{\Sigma'}}\Ss^{\min}_K\to \bb{P}\bigl(H^0(\Ss^{\Sigma'}_K,\omega^{\otimes N}_K(\Sigma'))\bigr).
\end{align}

From this, we obtain:
\begin{prp}\label{min:prp:intsigma}
$\oint^{\Sigma'}$ is a surjective map of normal algebraic spaces over $\Reg{E,(v)}$ with geometrically connected fibers.

It is characterized by the following property: Given any $\Reg{E,(v)}$-scheme $T$, a very ample line bundle $\mathscr{L}$ over $T$, an integer $N\in\Int_{\geq 1}$, and a map $f:\Ss^{\Sigma'}_K\to T$ equipped with an isomorphism $f^*\mathscr{L}\xrightarrow{\simeq}\omega^{\otimes N}_K(\Sigma')$, there exists a (unique) map $f^{\min}:\Ss^{\min}_K\to T$ such that $f=f^{\min}\circ\oint^{\Sigma'}$.
\end{prp}
\qed

\subsubsection{}
Fix a pair $(\Phi,\sigma)$ for $(G,X)$ with $\sigma\in\Sigma'(\Phi)$. Given any section $f\in H^0\bigl(\Ss^{\Sigma}_K,\omega^{\otimes n}_K(\Sigma)\bigr)$, it follows from~\eqref{min:cor:fourier_jacobi} that its restriction to $\mathcal{Z}_{K_\Phi}(Q_\Phi,D_\Phi,\sigma)$ is determined by the section:
\[
 \on{FJ}^{(0)}_{\Phi}(f)\in H^0\bigl(\Ss_{K_{\Phi,h}}(G_{\Phi,h},D_{\Phi,h}),\omega_K^{\otimes n}(\Phi)\bigr)^{\Delta_K(\Phi)}.
\]

This implies that the composition:
\[
 \mathcal{Z}_{K_\Phi}(Q_\Phi,D_\Phi,\sigma)\to\Ss^{\Sigma'}_K\xrightarrow{\oint^{\Sigma'}}\Ss^{\min}_K.
\]
factors through a map\footnote{Recall from~\eqref{background:subsubsec:cuspsisom} that $\Delta_K(\Phi)$ acts on $\Ss_{K_{\Phi,h}}(G_{\Phi,h},D_{\Phi,h})$ via a finite quotient $\Delta_K^{\mathrm{fin}}(\Phi)$.}
\begin{align}\label{min:eqn:zkphimap}
\Delta^{\mathrm{fin}}_K(\Phi)\setminus\Ss_{K_{\Phi,h}}(G_{\Phi,h},D_{\Phi,h})\to\Ss^{\min}_K.
\end{align}

\subsubsection{}\label{min:subsubsec:zkphimap}
Let $\on{Cusp}_K(G,X)$ be the set of equivalence classes of clrs for $(G,X)$ for the relation $\Phi_1\sim\Phi_2$ whenever there exists $\Phi_1\xrightarrow{(\gamma,q)_K}\Phi_2$ with $\gamma\cdot P_{\Phi_1}=P_{\Phi_2}$. We can give it the structure of a poset with $[\Phi_1]\preceq[\Phi_2]$ whenever there exists an arrow $\Phi_1\xrightarrow{(\gamma,q)_K}\Phi_2$.

Then, we deduce from~\eqref{min:lem:indphisig} that:
\begin{itemize}
\item The scheme $\Delta^{\mathrm{fin}}_K(\Phi)\setminus\Ss_{K_{\Phi,h}}(G_{\Phi,h},D_{\Phi,h})$ depends only on the class $[\Phi]$ in $\on{Cusp}_K(G,X)$. We will therefore denote it by $\mathcal{Z}_K([\Phi])$.

\item The map $\mathcal{Z}_K([\Phi])\to\Ss^{\min}_K$ induced from~\eqref{min:eqn:zkphimap} depends only on the class $[\Phi]$, and not on a choice of representative $\Phi$, or of the cone $\sigma\in\Sigma'(\Phi)$.
\end{itemize}

\begin{lem}\label{min:lem:rank1_zkphi}
\begin{enumerate}
\item\label{rank1:rank0}Suppose that $\Phi=(G,X^+,g)$; then the map $\Ss_{K_\Phi}(Q_\Phi,D_\Phi)\to\mathcal{Z}_K([\Phi])$ is an isomorphism.
\item\label{rank1:rank1}Suppose that $\sigma\in\Sigma(\Phi)$ is a $1$-dimensional cone, so that $\mathcal{Z}_{K_\Phi}(Q_\Phi,D_\Phi,\sigma)$ has co-dimension $1$ in $\Ss_{K_\Phi}(Q_{\Phi},D_{\Phi},\sigma)$. Then the map $\mathcal{Z}_{K_\Phi}(Q_\Phi,D_\Phi,\sigma)\to\mathcal{Z}_K([\Phi])$ is an isomorphism if and only if the unipotent radical $U_\Phi\subset P_\Phi$ has dimension at most $1$.
\end{enumerate}
\end{lem}
\begin{proof}
For~\eqref{rank1:rank0}, since $P=G$, we have $\Ss_{K_\Phi}(Q_\Phi,D_\Phi)=\Ss_{K_{\Phi,h}}(G_{\Phi,h},D_{\Phi,h})$. We claim that the group $\Delta_K(\Phi)$ is actually trivial in this case. Indeed, since $W_\Phi=1$, it coincides with the group $\Delta^{\circ}_K(\Phi)$, which we know to be trivial by~\eqref{background:lem:deltacirc_trivial}. This implies that $\Ss_{K_\Phi}(Q_\Phi,D_\Phi)$ maps isomorphically onto $\mathcal{Z}_K([\Phi])$.

Now, suppose that $U_\Phi$ has dimension $1$. Then we have $W_\Phi=U_\Phi$, so that $\Ss_{\overline{K}_\Phi}(\overline{Q}_\Phi,\overline{D}_\Phi)=\Ss_{K_{\Phi,h}}(G_{\Phi,h},D_{\Phi,h})$, and $\Ss_{K_\Phi}(Q_\Phi,D_\Phi)$ is a $\Gm$-torsor over $\Ss_{K_{\Phi,h}}(G_{\Phi,h},D_{\Phi,h})$, with $\mathcal{Z}_{K_\Phi}(Q_\Phi,D_\Phi,\sigma)$ the zero section of the corresponding line bundle over $\Ss_{K_{\Phi,h}}(G_{\Phi,h},D_{\Phi,h})$. Therefore, the map $\mathcal{Z}_{K_\Phi}(Q_\Phi,D_\Phi,\sigma)\to\Ss_{K_{\Phi,h}}(G_{\Phi,h},D_{\Phi,h})$ is an isomorphism. Furthermore, since $W_\Phi$ is $1$-dimensional, $\PGL(W_\Phi)$ is trivial, which implies again that $\Delta_K(\Phi)=\Delta^{\circ}_K(\Phi)$ is trivial. Therefore, $\Ss_{K_{\Phi,h}}(G_{\Phi,h},D_{\Phi,h})$ maps isomorphically onto $\mathcal{Z}_K([\Phi])$.

Conversely, if this map is an isomorphism, then we must have $\Ss_{\overline{K}_\Phi}(\overline{Q}_\Phi,\overline{D}_\Phi)=\Ss_{K_{\Phi,h}}(G_{\Phi,h},D_{\Phi,h})$, and $\Ss_{K_\Phi}(Q_\Phi,D_\Phi)$ must be a $\Gm$-torsor over $\Ss_{\overline{K}_\Phi}(\overline{Q}_\Phi,\overline{D}_\Phi)$, with $\mathcal{Z}_{K_\Phi}(Q_\Phi,D_\Phi,\sigma)$ the zero section of the associated line bundle. This can happen exactly when $W_\Phi=U_\Phi$ has rank $1$.
\end{proof}

\begin{rem}\label{min:rem:codim_1}
$G$ admits an admissible parabolic $P$ with $W_P=U_P$ of dimension $1$, exactly when it admits $\PGL_2$ as a quotient, and in this case $P$ is the pre-image of a Borel subgroup of $\PGL_2$.
\end{rem}

\begin{prp}\label{min:prp:nointersection}
Given clrs $\Phi_1$ and $\Phi_2$ for $(G,X)$, the images of $\mathcal{Z}_K([\Phi_1])$ and $\mathcal{Z}_K([\Phi_2])$ in $\Ss^{\min}_K$ are disjoint unless $[\Phi_1]=[\Phi_2]$.
\end{prp}
\begin{proof}
Suppose that we have a point $x$ in the intersection of the two images. Consider the fiber $F_x$ of $\Ss^{\Sigma'}_K$ over $x$. This fiber is geometrically irreducible, and we can assume that it is positive dimensional.

We can find, for $i=1,2$, $\sigma_i\in\Sigma'(\Phi_i)$, and points $y_i\in\mathcal{Z}_K([(\Phi_i,\sigma_i)]\cap F_x$ connected by a geometrically integral curve $C\subset F_x$. Let $\eta$ be the generic point of $C$, and suppose $\Upsilon=[(\Phi,\sigma)]$ is such that $\eta\in\mathcal{Z}_K(\Upsilon)$. Then by~\eqref{strata:thm:strata}\eqref{strata:strataclosure}, we must have $[(\Phi_i,\sigma_i)]\preceq [(\Phi,\sigma)]$. In particular, this implies:
\begin{align}\label{min:eqn:phisiginequal}
[\Phi_i]\preceq[\Phi],
\end{align}
for $i=1,2$.

On the other hand, it follows from the argument in~\cite[7.2.3.6]{lan:thesis} that we have:
\begin{align}\label{min:eqn:phisigequal}
 \iota_*[\Phi_1]=\iota_*[\Phi]=\iota_*[\Phi_2]\in\on{Cusp}_{{K}^{\beef}}(G^\beef,X^{\beef}).
\end{align}

The proposition now follows from the conjunction of~\eqref{min:eqn:phisiginequal} and~\eqref{min:eqn:phisigequal}.
\end{proof}

The proof of the next result proceeds exactly as that of~\cite[7.2.3.13]{lan:thesis}.
\begin{corollary}\label{min:cor:completion}
Let $x$ be a geometric point of $\Ss^{\min}_K$ lifting to a point $y$ of $\mathcal{Z}_K([\Phi])$. Let $\widehat{\on{FJ}}^{(\ell)}_K(\Phi)_y$ be the completion of the stalk of $\on{FJ}^{(\ell)}_K(\Phi)$ at $y$. Then we have a canonical isomorphism of complete local rings:
\begin{align*}
\widehat{\Rg}_{\Ss^{\min}_K,x}&\xrightarrow{\simeq}\biggl[\prod_{\ell\in\mb{P}_K(\Phi)}\widehat{\on{FJ}}^{(\ell)}_K(\Phi)_y\biggr]^{\Delta_K(\Phi)}.
\end{align*}
\end{corollary}
\qed

\begin{corollary}\label{min:cor:zkiphilocimmersion}
The map $\mathcal{Z}_K([\Phi])\to\Ss^{\min}_K$ is a locally closed immersion.
\end{corollary}
\begin{proof}
Call this map $f$; then \eqref{min:cor:completion} shows that $f$ is unramified. Moreover, by construction, it is proper onto its image, and must therefore in fact be finite over its image. Since $f$ is unramified, its image is also normal.\footnote{The notion of `image' here has to be understood appropriately: For a closed stratum, this is just the schematic image. For a general stratum $\mathcal{Z}_K(\Upsilon)$, this is defined inductively by removing from $\Ss^{\min}_K$ the images of the strata in the closure of $\mathcal{Z}_K(\Upsilon)$, and then taking the schematic image of $\mathcal{Z}_K(\Upsilon)$ in this complement.} By Zariski's main theorem, it is therefore enough to know that $f$ is birational onto its image. We can check this over the generic fiber, or even over $\Comp$. Here, it follows from~\cite[6.3]{pink:thesis}. Note that we are using the fact that Stein factorization is compatible with flat base change to identify the complex fiber of $\Ss_K^{\min}$ with the construction in \emph{loc. cit.}
\end{proof}

\begin{rem}\label{min:rem:finite_quotient}
There seems to be some inconsistency between~\cite[6.3]{pink:thesis} (and hence between the above statement) and ~\cite[7.2.3.15]{lan:thesis}, where the action of the finite quotient of $\Delta_K(\Phi)$ does not intervene. It appears to us that the gap lies in~\cite[7.2.3.5]{lan:thesis}, which, in our notation, claims that the map $\Ss_{K_{\Phi,h}}(G_{\Phi,h},D_{\Phi,h})\to\Ss^{\min}_K$ is injective on geometric points. Although, as shown on pp. 8 of~\cite{morel:siegel_1}, this is certainly true in the Siegel case, it is not so clear to us that it should hold in general.
\end{rem}

\begin{thm}\label{min:thm:min}
The scheme $\Ss^{\min}_K$ is a normal, projective $\Reg{E,(v)}$ scheme, which enjoys the following properties:
\begin{enumerate}
\item\label{min:open}The map $\Ss_K\into\Ss^{\Sigma}_K\xrightarrow{\oint^{\Sigma}}\Ss^{\min}_K$ is an open immersion, and its complement in $\Ss^{\min}_K$ is flat over $\Reg{E,(v)}$.
\item\label{min:hodge}The Hodge bundle $\omega_K$ on $\Ss_K$ extends to an ample line bundle $\omega_K^{\min}$ on $\Ss^{\min}_K$. There is a canonical isomorphism $(\oint^\Sigma)^*\omega^{\min}_K\xrightarrow{\simeq}\omega_K(\Sigma)$ extending the identity over $\Ss_K$.
\item\label{min:strata}We have a stratification:
\[
 \Ss_K^{\min}=\bigsqcup_{[\Phi]}\mathcal{Z}_K([\Phi]),
\]
where $[\Phi]$ ranges over $\on{Cusp}_K(G,X)$. In this stratification, $\mathcal{Z}_K([\Phi'])$ is in the closure of $\mathcal{Z}_K([\Phi])$ if and only if $[\Phi']\preceq[\Phi]$.
\item\label{min:oint}The map $\oint^\Sigma$ is compatible with stratifications: For any $[\Phi]$, we have
\[
 (\oint^\Sigma)^{-1}\bigl(\mathcal{Z}_K([\Phi])\bigr)=\bigsqcup_{\Upsilon}\mathcal{Z}_K(\Upsilon),
\]
where $\Upsilon$ ranges over classes of the form $[(\Phi,\sigma)]$ with $\sigma\in\Sigma(\Phi)$. Moreover, the map $\mathcal{Z}_K([(\Phi,\sigma)])\to\mathcal{Z}_K([\Phi])$ is canonically isomorphic to the natural map $\mathcal{Z}_{K_\Phi}(Q_\Phi,D_\Phi,\sigma)\to\Delta^{\mathrm{fin}}_K(\Phi)\setminus\Ss_{K_{\Phi,h}}(G_{\Phi,h},D_{\Phi,h})$.
\item\label{min:koecher}Let $\Ss^1_K\subset\Ss^{\Sigma}_K$ be the pre-image of the complement in $\Ss^{\min}_K$ of the union of the strata of co-dimension at least $2$. Then $\Ss^1_K$ maps isomorphically into $\Ss^{\min}_K$. Moreover, the natural map
\[
 H^0\bigl(\Ss^{\Sigma}_K,\omega^{\otimes n}_K(\Sigma)\bigr)\to H^0\bigl(\Ss^1_K,\omega^{\otimes n}_K(\Sigma)\bigr)
\]
is an isomorphism. In particular, when $G^{\ad}$ does not admit $\PGL_2$ as a factor, the natural map
\[
 H^0\bigl(\Ss^{\Sigma}_K,\omega^{\otimes n}_K(\Sigma)\bigr)\to H^0\bigl(\Ss_K,\omega^{\otimes n}_K\bigr)
\]
is an isomorphism.
\end{enumerate}
\end{thm}
\begin{proof}
That $\Ss_K^{\min}$ is normal and projective over $\Reg{E,(v)}$ is immediate from its construction and~\cite[\S~6.7, Lemma 2]{blr}.

Assertions~\eqref{min:strata} and~\eqref{min:oint} are immediate from~\eqref{min:prp:nointersection}, the properness of $\oint^{\Sigma}$, and the corresponding assertions about the stratification of $\Ss^{\Sigma}_K$.

Since this shows that the strata are all flat over $\Reg{E,(v)}$, it also implies assertion~\eqref{min:open} (one also needs~\eqref{min:lem:rank1_zkphi}~\eqref{rank1:rank0}).

\eqref{min:koecher} is standard, except for the last statement, which follows from~\eqref{min:lem:rank1_zkphi}~\eqref{rank1:rank1} and~\eqref{min:rem:codim_1}.

It only remains to prove~\eqref{min:hodge}. In fact, we can take
\[
\omega^{\min}_K = \oint^{\Sigma}_*\omega^{\Sigma}_K.
\]
It is enough to show that this is an ample line bundle over $\Ss^{\min}_K$. The proof of this proceeds exactly as that of~\cite[7.2.4.1]{lan:thesis}.
\end{proof}

\subsubsection{}
We can now extend the Hecke correspondences. Let the notation be as in~\eqref{strata:subsubsec:hecke}. The assignment $[\Phi']\mapsto [\Phi]\coloneqq[\iota_*\Phi']$ is a well-defined map from $\on{Cusp}_{K'}(G',X')$ to $\on{Cusp}_K(G,X)$. The map~\eqref{strata:eqn:phi'phi} induces a finite map:
\begin{align}\label{min:eqn:zk'phi'}
 \mathcal{Z}_{K'}([\Phi'])=\Delta^{\mathrm{fin}}_{K'}(\Phi')\backslash\Ss_{K'_{\Phi',h}}(G_{\Phi',h},D_{\Phi',h})\to\Delta^{\mathrm{fin}}_K(\Phi)\backslash\Ss_{K_{\Phi,h}}(G_{\Phi,h},D_{\Phi,h})\bigr)=\mathcal{Z}_K([\Phi]).
\end{align}

\begin{prp}\label{min:prp:hecke}
The map
\[
 (\eta,g):\Ss_{K'}\to\Reg{E',(v')}\otimes_{\Reg{E,(v)}}\Ss_K
\]
extends uniquely to a map $\Ss^{\min}_{K'}\to\Reg{E',(v')}\otimes_{\Reg{E,(v)}}\Ss_K^{\min}$ characterized by the following property: It carries the stratum $\mathcal{Z}_{K'}(\Phi')$ to the stratum $\Reg{E',(v')}\otimes\mathcal{Z}_K([\Phi])$ via~\eqref{min:eqn:zk'phi'}.
\end{prp}
\begin{proof}
Immediate from~\eqref{min:prp:intsigma} and~\eqref{strata:prp:hecke}.
\end{proof}

\subsection{The case of hyperspecial level}\label{min:subsec:smooth}
Here, we will assume that we are in the situation of~\eqref{strata:subsec:smooth}. The assumptions and notation will as in that subsection. The following result shows that the integral model of the minimal compactification is also canonical.

\begin{thm}\label{min:thm:smooth}
The minimal compactification $\Ss^{\min}_K$ and its stratification are, up to unique isomorphism, independent of the choice of $p$-integral symplectic embedding $\iota$. Moreover, for any admissible rpcd $\Sigma$ for $(G,X,K)$, any clr $\Phi$, and any $\sigma\in\Sigma(\Phi)$, the map $\Ss^{\Sigma}_K\to\Ss^{\min}_K$ induces a smooth morphism:
\[
 \mathcal{Z}_K([(\Phi,\sigma)])\to \mathcal{Z}_K([\Phi]).
\]
% In particular, $\mathcal{Z}_K([\Phi])$ is smooth over $\Reg{E,(v)}$.
% \item\label{minsmooth:omegamin}Given another choice of $p$-integral embedding $\iota'$ satisfying the conditions of~\eqref{strata:subsec:smooth}, let $\omega_K^{\min}(\iota')$ be the associated Hodge bundle over $\Ss^{\min}_K$ obtained from \eqref{min:thm:min}\eqref{min:hodge}. Then there exist $n,m\in\Int_{>0}$, and an isomorphism of line bundles:
% \[
% \omega_K^{\min,\otimes n}\xrightarrow{\simeq}\omega_K^{\min,\otimes m}(\iota').
% \]
% \end{enumerate}
\end{thm}
\begin{proof}
Let $\iota'$ be another $p$-integral embedding satisfying the conditions of~\eqref{strata:subsec:smooth}, and let $\Ss^{\mathrm{min}}_K(\iota')$ be the minimal compactification of $\Ss_K$ constructed using $\iota'$. Fix an admissible rpcd $\Sigma$ for $(G,X,K)$, so that the identity morphism on $\Ss_K$ extends to morphisms
\[
\oint_1:\Ss^\Sigma_K \to \Ss_K^{\mathrm{min}}\;;\; \oint_2:\Ss^\Sigma_K\to \Ss_K^{\mathrm{min}}(\iota')
\]
of algebraic spaces, where $\Ss^\Sigma_K$ is the canonical compactification of $\Ss_K$ given by Theorem~\ref{smooth}.

Let $\Ss'$ be the normalization of the image of the product morphism
\[
\oint_1\times \oint_2: \Ss^\Sigma_K\to \Ss^{\mathrm{min}}_K\times \Ss^{\mathrm{min}}_K(\iota').
\]
It contains the diagonal embedding of $\Ss_K$ as an open subscheme.

It is now enough to show that the natural projections of $\Ss'$ onto $\Ss^{\mathrm{min}}_K$ and $\Ss^{\min}_K(\iota')$ are both isomorphisms. Since the generic fiber $\Sh^{\mathrm{min}}_K$ of the minimal compactification is a canonical object in characteristic $0$, the morphisms $\oint_1$ and $\oint_2$ are canonically isomorphic over $E$, we see that the restrictions of the projections over the open subscheme $\Ss_K\cup \Sh^{\mathrm{min}}_K$ are isomorphisms. We now conclude by observing that this subscheme has complement of codimension at least $2$ in both $\Ss^{\mathrm{min}}_K$ and $\Ss^{\mathrm{min}}_K(\iota')$.

To finish, it is enough to show that the quotient map
\begin{align}\label{min:eqn:mkphi_quotient_unramified}
\Ss_{K_{\Phi,h}}(G_{\Phi,h},D_{\Phi,h})\to \Delta^{\mathrm{fin}}_K(\Phi)\backslash\Ss_{K_{\Phi,h}}(G_{\Phi,h},D_{\Phi,h})
\end{align}
is \'etale. The best way to show this would be to use Kisin's theory of twisting abelian varieties~\cite[\S~3]{kisin:abelian} to explicitly describe the action of $\Delta^{\mathrm{fin}}_K(\Phi)$. This method has the advantage of also working without any hypotheses on the level.

However, we will give a quicker proof here, in the case of hyperspecial level. Let ${\Phi}^\beef=\iota_*\Phi$ be the induced clr for $(G^\beef,X^{\beef},{K}^{\beef})$. Then, as explained in~\eqref{min:rem:finite_quotient}, $\Delta^{\mathrm{fin}}_{{K}^{\beef}}({\Phi}^\beef)$ is trivial. Therefore, the natural map $\eta_{K^p}:\Ss_{K_{\Phi,h}}(G_{\Phi,h},D_{\Phi,h})\to\Ss_{K^{\beef}_{\Phi^\beef,h}}(G_{\Phi^\beef,h},D_{\Phi^\beef,h})$ factors through~\eqref{min:eqn:mkphi_quotient_unramified}. So it is enough to show that $\eta_{K^p}$ is unramified.

This follows from the proof of~\cite[2.3.9]{kisin:abelian}, which shows that all complete local rings of $\Ss_{K_{\Phi,h}}(G_{\Phi,h},D_{\Phi,h})$ are quotients of those of $\Ss_{K^{\beef}_{\Phi^\beef,h}}(G_{\Phi^\beef,h},D_{\Phi^\beef,h})$.
\end{proof}

\appendix

\section{}\label{sec:appendix}

\subsection{Logarithmic Dieudonn\'e theory for degenerating abelian varieties}\label{appendix:subsec:logfcrys}
The notation will be as in (\ref{semistable:subsec:logfcrys}). The goal is to prove (\ref{semistable:prp:logdieu}). The construction we give here is essentially given in \cite{kato_trihan}, but that reference only covers the case of complete discrete valuation rings, and so we have chosen to give some more details for our more general situation.

\subsubsection{}
We will begin by considering objects in the full sub-category $\mb{DD}^{\on{tot}}(S,U)$ of $\mb{DD}(S,U)$ consisting of objects with $Q^{\ab}\vert_U=0$.

Any such object is determined by a bilinear pairing $\tau:Y_U\times X_U\to\Gmh{U}$. We will write $(Q_{\tau},\lambda^{\et})$ for the corresponding object in $\mb{DD}^{\on{tot}}(S,U)$. It corresponds to a complex $[Y_U\xrightarrow{f_{\tau}}T_U]$, where $f_{\tau}(y)(x)=-\tau(y,x)$.

The dual object $\dual{Q}_U$ is attached to the pairing $\dual{\tau}:X_U\times Y_U\to\Gmh{U}$ obtained by flipping the order of $Y$ and $X$ in the product. It corresponds to the dual map:
\begin{align*}
  f_{\dual{\tau}}: X_U &\to \dual{T}_U=\SHom(Y_U,\Gmh{U})\\
 x&\mapsto (y\mapsto-\dual{\tau}(x,y)).
\end{align*}

\subsubsection{}\label{semistable:subsubsec:ntorsion}
Let $\mathcal{G}_{\tau}\to Y\vert_U$ be the pull-back of the multiplication-by-$n$ map $T\xrightarrow{n}T$ along $f_{\tau}$. It is an extension:
\[
 0\to T_U[n]\to\mathcal{G}({\tau},n)\to Y_U\to 0.
\]
Then, as an extension of $Y_U\otimes\Int/n\Int$ by $T_U[n]\vert_U$, $Q_U[n]$ is canonically identified with $\mathcal{G}(\tau,n)\otimes\Int/n\Int$.

The pairing $Q_U[n]\times\dual{Q}_U[n]\to\mu_{n,U}$ is (up to sign) the one obtained from the pairing:
\begin{align*}
 \mathcal{G}({\tau},n)\times \mathcal{G}({\dual{\tau}},n)&\to\mu_{n,U}\\
 (y,t)\times (x,\dual{t})&\mapsto x(t)y(\dual{t})^{-1}.
\end{align*}
This makes sense, since $x(t)^ny(\dual{t})^{-n}=x(f(y))y(\dual{f}(x))^{-1}=\tau(y,x)\tau(y,x)^{-1}=1$.

\subsubsection{}\label{semistable:subsubsec:noabderham}
Any vector extension of $T_U$ can be trivialized. Therefore, the universal vector extension $[Y_U\to E_{Q_U}]$ of $[Y_U\xrightarrow{f_{\tau}}T_U]$ can be identified with
\[
 Y_U\xrightarrow{1\oplus f}(Y_U\otimes\bb{G}_a)\oplus T_U.
\]

We then have canonical isomorphisms of filtered vector bundles:
\begin{align}\label{semistable:eqn:qabtrivial}
   \underline{H}^1_{\dR}(Q_U)&=\SHom\bigl(Y,\Reg{U}\bigr)\oplus\bigl(X\otimes\Reg{U}(-1)\bigr)\\
   \underline{H}^1_{\dR}(\dual{Q}_U)&=\bigl(Y\otimes\Reg{U}(-1)\bigr)\oplus\SHom\bigl(X,\Reg{U}\bigr).
\end{align}
Moreover, the $\Reg{U}(-1)$-valued pairing on $\underline{H}^1_{\dR}(Q_U)\times\underline{H}^1_{\dR}(\dual{Q}_U)$ translates to the obvious one under these identifications.

\subsubsection{}\label{semistable:subsubsec:noabconnection}
Now, assume that $S$ is a scheme over $S_0$. We can also describe the connection
\[
\nabla:\underline{H}^1_{\dR}(Q_U)\to\underline{H}^1_{\dR}(Q_U)\otimes\Omega^1_{U/S_0}
\]
in terms of the isomorphism (\ref{semistable:eqn:qabtrivial}).

Let $U^{(1)}\subset U\times_SU$ be the first-order neighborhood of the diagonal embedding of $U$. Let $p_1,p_2:U^{(1)}\to U$ be the two projections. Then there is a canonical isomorphism of complexes:
\begin{align}\label{semistable:eqn:p1p2}
 p_1^*[Y_U\to E_{Q_U}]&\xrightarrow{\simeq}p_2^*[Y_U\to E_{Q_U}].
\end{align}

Since the categories of \'etale locally constant sheaves, and hence the categories of \'etale locally trivial tori, over $U^{(1)}$ and $U$ are canonically isomorphic, we can identify both $p_1^*E_{Q_U}$ and $p_2^*E_{Q_U}$ with $(Y\otimes\bb{G}_{a,U^{(1)}})\oplus T_{U^{(1)}}$, where $T_{U^{(1)}}$ is the canonical lift over $U^{(1)}$ of $U$.

The map $y\mapsto p_2^*f(y)p_1^*f(y)^{-1}$ defines a homomorphism from $Y_U$ to $\ker(T_{U^{(1)}}\to T_U)$. Since
\[
\Omega^1_{U/S_0}=\ker(\Reg{U^{(1)}}\to \Reg{U}),
\]
this last group can be canonically identified with $\SHom(X_U,1+\Omega^1_{U/S_0})=\SHom(X_U,\Omega^1_{U/S_0})$, and the resulting map
\[
 Y_U\to\SHom(X_U,\Omega^1_{U/S_0})
\]
is simply $y\mapsto -\dlog(\tau(y,\cdot))$.

We therefore see that (\ref{semistable:eqn:p1p2}) is identified with:
\begin{align*}
 [Y_U\xrightarrow{\on{1}\oplus p_1^*f}(Y_U\otimes\bb{G}_{a,U^{(1)}})\oplus T_{U^{(1)}}]&\xrightarrow{\simeq}[Y_U\xrightarrow{\on{1}\oplus p_2^*f}(Y_U\otimes\bb{G}_{a,U^{(1)}})\oplus T_{U^{(1)}}],
\end{align*}
where the underlying automorphism of $Y_U$ is the identity, and that of $(Y_U\otimes\bb{G}_{a,U^{(1)}})\oplus T_{U^{(1)}}$ is given by:
\[
 (y\otimes a)\oplus t\mapsto (y\otimes a)\oplus (1+a\dlog(\tau(y,\cdot)))\cdot t.
\]

From this, we find that, via the identification~\eqref{semistable:eqn:qabtrivial}, the connection on $\underline{H}^1_{\dR}(Q_U)$ has the following description: For any section $\varphi$ of $\SHom\bigl(Y_U,\Int\bigr)$, we have: $\nabla((\varphi,0))=(\varphi\otimes 1,0)$. For any section $x$ of $X_U$, we have:
\[
 \nabla(x\otimes 1)=(-\dlog(\tau(\cdot,x)),x\otimes 1).
\]

\subsubsection{}\label{semistable:subsubsec:noabcrys}
We now come to the crystalline realization $\Dieu(Q)$. We will first construct the restriction of $\Dieu$ to the full sub-category $\mb{DD}^{\on{tot}}(S,U)$ of $\mb{DD}(S,U)$.

Let $S^{\log}$ be the log scheme attached to $(S,U)$. Given a log scheme $Z$ over $S^{\log}$ in which $p$ is nilpotent and an exact nilpotent thickening of log schemes $Z\into\widetilde{Z}$ equipped with divided powers, we need to construct a canonical $\Reg{\widetilde{Z}}$-module $\Dieu(Q_{\tau})(Z\into\widetilde{Z})$. We can view $f_{\tau}$ as a homomorphism
\[
f_{\tau}:Y\to\SHom(X,\Reg{U}^\times)=\SHom(X,\Mon_{S^{\log}}^{\gp}).
\]
In particular, for any log scheme $Z$ over $S^{\log}$, we obtain a homomorphism $f_{\tau}\vert_Z:Y\vert_Z\to\SHom(X,\Mon_Z^{\gp})$.

Since $Z\into\widetilde{Z}$ is an exact nilpotent thickening, if $\mathcal{I}=\ker(\Reg{\widetilde{Z}}\to\Reg{Z})$, we also have:
\[
 1+\mathcal{I}=\ker\bigl(\Mon_{\widetilde{Z}}^{\gp}\to\Mon_Z^{\gp}\bigr).
\]
We therefore obtain a short exact sequence of \'etale sheaves over $Z$:
\begin{align}\label{semistable:eqn:shom_x}
 1\to \SHom(X\vert_Z,1+\mathcal{I})\to\SHom(X\vert_Z,\Mon_{\widetilde{Z}}^{\gp})\to\SHom(X\vert_Z,\Mon_Z^{\gp})\to 1.
\end{align}

Set $T_{\mathcal{I}}=\ker(T_{\widetilde{Z}}\to T_Z)$; we then have the short exact sequence of $\widetilde{Z}$-group schemes:
\begin{align}\label{semistable:eqn:ttildez}
 1\to T_{\mathcal{I}}\to T_{\widetilde{Z}}\to T_Z\to 1.
\end{align}

Pulling~\eqref{semistable:eqn:shom_x} back along $f_{\tau}\vert_Z$, we obtain another short exact sequence:
\begin{align}\label{semistable:eqn:etaulog}
1\to T_{\mathcal{I}}\to\mathcal{E}(\tau,Z\into\widetilde{Z})\to Y\vert_Z\to 1.
\end{align}

Write $\Gmh{\mathcal{I}}$ for $\ker(\Gmh{\widetilde{Z}}\to\Gmh{Z})$. Then the divided powers on $\mathcal{I}$ give us a map of groups:
\begin{align*}
 \ell:\Gmh{\mathcal{I}}&\to \Reg{\widetilde{Z}}\\
 x&\mapsto\sum_{n\in\Int_{\geq 1}}(-1)^{n-1}(x-1)^{[n]}.
\end{align*}

We therefore obtain a map:
\begin{align}\label{semistable:eqn:log}
\ell\otimes 1: T_{\mathcal{I}}\otimes_{\Int}\Reg{\widetilde{Z}}=\SHom(X\vert_Z,\Gmh{\mathcal{I}})\otimes_{\Int}\widetilde{Z}&\to\SHom(X\vert_Z,\Reg{\widetilde{Z}}).
\end{align}

Let $\mathcal{F}(\tau,Z\into\widetilde{Z})$ be the push-forward along the map $\ell\otimes 1:T_{\mathcal{I}}\otimes_{\Int}\Reg{\widetilde{Z}}\to\Reg{\widetilde{Z}}$ of $\mathcal{E}(\tau,Z\into\widetilde{Z})\otimes_{\Int}\Reg{\widetilde{Z}}$. We set:
\[
 \Dieu(Q_{\tau})(Z\into\widetilde{Z})=\SHom\bigl(\mathcal{F}(\tau,Z\into\widetilde{Z}),\Reg{\widetilde{Z}}\bigr).
\]

If $\widetilde{Z}$ is also a log scheme over $S^{\log}$, then (\ref{semistable:eqn:etaulog}) has a canonical splitting, and so we obtain a canonical splitting:
\begin{align}\label{semistable:eqn:cansplit}
 \Dieu(Q_{\tau})(Z\into\widetilde{Z})=\SHom(Y\vert_Z,\Reg{\widetilde{Z}})\oplus(X\otimes\Reg{\widetilde{Z}}).
\end{align}
In particular, we obtain a canonical direct summand $F^1\Dieu(Q_{\tau})(S)=X\otimes\Reg{S}\subset\Dieu(Q_{\tau})(S)$: This induces the Hodge filtration $F^\bullet\Dieu(Q_{\tau})(S)$.

There is a canonical $2$-step weight filtration $W_{\bullet}\Dieu(Q_{\tau})$ on $\Dieu(Q_{\tau})$ with $W_2\Dieu(Q_{\tau})=\Dieu(Q_{\tau})$, $W_{-1}\Dieu(Q_{\tau})=0$ and
\[
 W_0\Dieu(Q_{\tau})=W_1\Dieu(Q_{\tau})=\SHom(Y,\mb{1})\;;\;\gr^W_2\Dieu(Q_{\tau})=X\otimes\mb{1}(-1).
\]

\subsubsection{}\label{semistable:subsubsec:noabfcrys}
We still need to describe the $F$-crystal structure on $\Dieu(Q)$; that is, the maps
\begin{align*}
\varphi_{\Dieu(Q)}:\on{Fr}^*\Dieu(Q)&\to\Dieu(Q);\\
 V_{\Dieu(Q)}:\Dieu(Q)&\to\on{Fr}^*\Dieu(Q).
\end{align*}
For this, we note that we can canonically identify $\on{Fr}^*\Dieu(Q)$ with $\Dieu(Q^{(p)})$, where $Q^{(p)}$ is attached to the pairing $\tau^p:Y\vert_U\times X\vert_U\to\Gmh{U}$. So, in the notation of the construction above, we only have to construct canonical maps
\begin{align*}
 \varphi(Z\into\widetilde{Z}):\mathcal{E}(\tau,Z\into\widetilde{Z})&\to \mathcal{E}(\tau^p,Z\into\widetilde{Z});\\
 V(Z\into\widetilde{Z}):\mathcal{E}(\tau^p,Z\into\widetilde{Z})&\to \mathcal{E}(\tau,Z\into\widetilde{Z}),
\end{align*}
whose composition in either direction is just multiplication by $p$.

This is easy: They are induced, respectively, by the endomorphisms $(y,t)\mapsto (y,t^p)$ and $(y,t)\mapsto (y^p,t)$ of $Y\vert_Z\times T_{\widetilde{Z}}$.

\subsubsection{}\label{semistable:subsubsec:pairing}
We also have to describe the canonical pairing $\Dieu(Q)\times\Dieu(\dual{Q})\to\mb{1}(-1)$. This is induced from the pairing:
\begin{align*}
  \bigl(\SHom(X,\Reg{\widetilde{Z}})\oplus\mathcal{E}(\tau,Z\into\widetilde{Z})\bigr)\times \bigl(\SHom(Y,\Reg{\widetilde{Z}})\oplus\mathcal{E}(\dual{\tau},Z\into\widetilde{Z})\bigr)&\to\Reg{\widetilde{Z}}\\
  (\varphi,y,t)\times(\psi,x,\dual{t})\mapsto \varphi(x)-\psi(y)-\ell(x(\dual{t})y(t)^{-1}).
\end{align*}
Observe that $x(\dual{t})y(t)^{-1}$ lies in $1+\mathcal{I}$, since its image in $\Mon^{\gp}_{Z}$ is $\tau(x,y)\tau(x,y)^{-1}=1$.

Finally, it can be checked from the definitions and the splitting (\ref{semistable:eqn:cansplit}) that there is a canonical isomorphism of $\Reg{U}$-modules with integrable connections:
\[
 H^1_{\dR}(Q_{\tau,U})\xrightarrow{\simeq}\Dieu(Q_{\tau})(S)\vert_U
\]
respecting Hodge and weight filtrations and polarization pairings.

\subsubsection{}\label{semistable:subsubsec:ddpoltilde}
Let $\widetilde{\mb{DD}}(S,U)$ be the category of pairs $(Q,\tau_0)$, where:
\begin{itemize}
\item $Q$ is a $1$-motif over $S$;
\item $\tau_0:Q^{\et}_U\times Q^{\mult}_U\to\Gmh{U}$ is a pairing associated with an object of $\mb{DD}^{\on{tot}}(S,U)$.
\end{itemize}

A map of two such pairs $\varphi:(Q_1,\tau_{1,0})\to (Q_2,\tau_{2,0})$ is a map of $1$-motifs $\varphi:Q_1\to Q_2$ such that the pair $(\varphi^{\et},\varphi^{\mult,C})$ determines a map
\[
Q_{\tau_{1,0}}\to Q_{\tau_{2,0}}
\]
in $\mb{DD}^{\on{tot}}(S,U)$.

There is a natural notion of duality in this category: Given a pair $(Q,\tau_0)$ its dual is the object $(\dual{Q},\dual{\tau}_0)$.

There is a canonical fully faithful functor
\[
\mb{DD}^{\on{tot}}(S,U)\to\widetilde{\mb{DD}}(S,U)
\]
defined as follows: Given $Q_{\tau}$ on the left-hand side, we attach to it the pair $(Q,\tau)$, where $Q$ corresponds to the complex $[Q^{\et}_{\tau}\xrightarrow{0}Q^{\mult,C}_{\tau}]$.

We can now extend the functor $\Dieu:\mb{DD}^{\on{tot}}(S,U)\to\mb{LDieu}_{\wt}(S,U)$ to a functor
\[
 \widetilde{\Dieu}:\widetilde{\mb{DD}}(S,U)\to\mb{LDieu}_{\on{wt}}(S,U).
\]
Indeed, given $(Q,\tau_0)$ on the left hand side, we obtain:
\begin{itemize}
\item The Dieudonn\'e crystal $\Dieu(Q)$ over $S$ attached to $Q$ equipped with its weight filtration; cf.~\eqref{semistable:subsubsec:crystalline}. This can be viewed as an extension:
    \begin{align}\label{semistable:eqn:1motifextn}
     0\to\SHom(Q^{\et},\mb{1})\to\Dieu(Q)\to\Dieu(Q^{\sab})\to 0.
    \end{align}
\item The object $\Dieu(Q_{\tau_0})$ in $\mb{LDieu}_{\on{wt}}(S,U)$. This is an extension:
\[
 0\to\SHom(Q^{\et},\mb{1})\to\Dieu(Q_{\tau_0})\to\Dieu(Q^{\mult})\to 0.
\]
\end{itemize}
If we pull the latter extension back along the natural map $\Dieu(Q^{\sab})\to\Dieu(Q^{\mult})$, we obtain another extension:
\begin{align}\label{semistable:eqn:pullbackextn}
 0\to\SHom(Q^{\et},\mb{1})\to\Dieu(Q_{\tau_0})\times_{\Dieu(Q^{\mult})}\Dieu(Q^{\sab})\to\Dieu(Q^{\sab})\to 0.
\end{align}
The underlying log Dieudonn\'e crystal of $\widetilde{\Dieu}(Q)$ will now be the Baer sum of (\ref{semistable:eqn:1motifextn}) and (\ref{semistable:eqn:pullbackextn}). We can naturally equip it with a weight filtration, and it can be checked that $\widetilde{\Dieu}$ respects duality. This completes the construction of $\widetilde{\Dieu}$.

\subsubsection{}\label{semistable:subsubsec:bsufunct}
Given $(Q,\tau_0)$ in $\widetilde{\mb{DD}}(S,U)$, we can view $\tau_0$ as a trivialization of the trivial $\Gm$-bi-extension of $Q^{\et}_U\times Q^{\mult,C}_U$. Therefore, we obtain a natural functor
\[
 \on{B}_{(S,U)}:\widetilde{\mb{DD}}_{\on{pol}}(S,U)\to \mb{DD}_{\on{pol}}(S,U)
\]
defined as follows on objects: Given a pair $(Q,\tau_0)$ on the left-hand side, we assign to it the object of $\mb{DD}_{\on{pol}}(S,U)$ attached to the tuple
\[
(Q^{\ab},Q^{\et},Q^{\mult,C},c,\dual{c},\tau\vert_U\cdot\tau_0).
\]
A morphism $\varphi:(Q_1,\tau_{1,0})\to (Q_2,\tau_{2,0})$ is given by a tuple $(\varphi^{\ab},\varphi^{\et},\varphi^{\mult,C})$, which also determines a map in $\mb{DD}(S,U)$:
\[
 \on{B}_{(S,U)}(\varphi):\on{B}_{(S,U)}\bigl((Q_1,\tau_{1,0})\bigr)\to\on{B}_{(S,U)}\bigl((Q_2,\tau_{2,0})\bigr).
\]
It follows from the definitions that $\on{B}_{(S,U)}$ is a faithful functor.

\subsubsection{}\label{semistable:subsubsec:logffull}
Given $(Q^{\ab},Q^{\et},Q^{\mult,C},c,\dual{c})$, after perhaps a finite \'etale base change $S\to S'$, we can find a trivialization $\tau$ of the bi-extension $(c\times\dual{c})^*\mathcal{P}_{Q^{\ab}}$ that completes the given tuple to a $1$-motif over $S'$. Therefore, for any object $(Q_U,\lambda_U)$ in $\mb{DD}(S,U)$, there is a finite \'etale cover $S'\to S$ such that, with $U'=S'\times_SU$, the base-change of $(Q_U,\lambda_U)$ to $\mb{DD}((S',U'))$ belongs to the essential image of $\on{B}_{(S',U')}$. It is now not hard to see that (\ref{semistable:prp:logdieu}) is a consequence of the following assertion:

Given $(Q_1,\tau_{1,0})$ and $(Q_2,\tau_{2,0})$ in $\widetilde{\mb{DD}}(S,U)$ and an isomorphism in $\mb{DD}(S,U)$:
\[
 \varphi:\on{B}_{(S,U)}\bigl((Q_1,\tau_{1,0})\bigr)\xrightarrow{\simeq} \on{B}_{(S,U)}\bigl((Q_2,\tau_{2,0})\bigr),
\]
there is a canonically associated isomorphism in $\mb{LDieu}_{\wt}(S,U)$:
\[
 \Dieu(\varphi):\widetilde{\Dieu}\bigl((Q_1,\tau_{1,0})\bigr)\xrightarrow[\simeq]{\Dieu(\varphi)}\widetilde{\Dieu}\bigl((Q_2,\tau_{2,0})\bigr).
\]

To see this, first note that the existence of $\varphi$ implies that we can assume $Q_1^{\ab}=Q_2^{\ab}$, and $Q_1^{\et}=Q_2^{\et}$, $Q_1^{\mult}=Q_2^{\mult}$. Furthermore, under these identifications, we can assume that $c_{Q_1}=c_{Q_2}$ and $\dual{c}_{Q_1}=\dual{c}_{Q_2}$.

For convenience, we now drop the superfluous numerals in the subscript of these objects. We can view both $\tau_{Q_1}$ and $\tau_{Q_2}$ as trivializations of the bi-extension $(c_{Q}\times \dual{c}_{Q})^*\mathcal{P}_{Q^{\ab}}$  of $Q^{\et}\times Q^{\mult,C}$; and both $\tau_{1,0}$ and $\tau_{2,0}$ as pairings on $Q^{\et}_U\times Q^{\mult,C}_U$. They satisfy:
\[
 \tau_{Q_1}\vert_U\cdot\tau_{1,0}=\tau_{Q_2}\vert_U\cdot\tau_{2,0}.
\]

Set $\overline{\tau}_0=\tau_{2,0}\tau_{1,0}^{-1}$; the formula above shows that $\overline{\tau}_0$ extends to a pairing $Q^{\et}\times Q^{\mult,C}\to\Gmh{S}$, which we again denote by $\overline{\tau}_0$, and that it satisfies:
\begin{align}\label{semistable:eqn:tausum}
 \tau_{Q_2}\cdot\overline{\tau}_0=\tau_{Q_1}\;;\;\overline{\tau}_0\vert_U\tau_{1,0}=\tau_{2,0}.
\end{align}

Let $Q_{\overline{\tau}_0}$ be the $1$-motif over $S$ attached to $\overline{\tau}_0$. Then, as extensions of $\Dieu(Q^{\mult})$ by $\SHom(Q^{\et},\mb{1})$, the Baer sum of $\Dieu(Q_{\overline{\tau}_0})$ and $\Dieu(Q_{\tau_{1,0}})$ is canonically isomorphic to $\Dieu(Q_{\tau_{2,0}})$. To see this, one must show that the construction of~\eqref{semistable:subsubsec:noabcrys}, when applied to the pairing $\overline{\tau}_0$ in the natural way, gives us the Dieudonn\'e crystal $\Dieu(Q_{\overline{\tau}_0})$. This can be deduced using finite \'etale descent and~\cite[(4.3.5)]{dejong:formal_rigid}.

Also, as extensions of $\Dieu(Q^{\sab})$ by $\SHom(Q^{\et},\mb{1})$, the Baer sum of $\Dieu(Q_{\overline{\tau}_0})\times_{\Dieu(Q^{\mult})}\Dieu(Q^{\sab})$ and $\Dieu(Q_2)$ is canonically isomorphic to $\Dieu(Q_1)$.

The assertion we required is immediate from these facts and the construction of the functor $\widetilde{\Dieu}$ in~\eqref{semistable:subsubsec:ddpoltilde}.

\subsection{Comparison isomorphisms for semi-stable abelian varieties}\label{appendix:subsec:comparison}

Fix an algebraic closure $k$ for $\Field_p$, and let $W=W(k)$, $K_0=W[p^{-1}]$. Set $S=\Spec\Reg{K}$ for a finite extension $K/K_0$ and let $U=\Spec K\subset S$ be the generic point. Let $\overline{K}$ be an algebraic closure of $K$, and let $\Reg{\overline{K}}$ be its ring of integers. Let $\Gamma_K$ be the absolute Galois group $\Gal(\overline{K}/K)$. As above, we will fix a uniformizer $\pi\in K$ with corresponding formal divided power thickening $\mathcal{S}_{\pi}\twoheadrightarrow\Reg{K}$.

Set $\overline{S}=\Spec\Reg{\overline{K}}$ and $\overline{U}=\Spec\overline{K}$; then $\overline{S}$ has a natural log structure associated with the divisor $\overline{S}\backslash\overline{U}$. Write $\overline{S}^{\log}$ for the associated log scheme.

\subsubsection{}\label{semistable:subsubsec:torsionpoints}
Fix $(A,\psi)$ in $\mb{DEG}_{\pol}(S,U)$. Write $\Dieu\bigl(A_{\overline{K}}\bigr)$ for the induced log Dieudonn\'e crystal in $\mb{LDieu}(\overline{S},\overline{U})$. For any $1$-motif $H$ over $K$, write $T_p(H)$ for the the $p$-adic Tate module
\[
 T_p(H)=\varprojlim_{n} H[p^n](\overline{K}).
\]
This is naturally a continuous $\Int_p$-representation of $\Gamma_K$.

\begin{prp}\label{semistable:prp:torsiondieu}
There is a canonical map of $\Gamma_K$-modules:
\[
 T_p(A)\to\Hom_{\mb{LDieu}(\overline{S},\overline{U})}\bigl(\Dieu\bigl(A_{\overline{K}}\bigr),\mb{1}\bigr).
\]
\end{prp}
\begin{proof}
 The $\Gamma_K$-action on the right-hand side is via its action on $\Dieu\bigl(A_{\overline{K}}\bigr)$ through `transport-of-structure'.

 Suppose that $(A,\psi)$ corresponds to $(Q_U,\lambda_U)$ in $\mb{DD}_{\on{pol}}(S,U)$. Write $\Dieu(Q)$ (resp. $\Dieu(Q_{\overline{K}})$) for the log Dieudonn\'e crystal $\Dieu(A)$ (resp. $\Dieu(A_{\overline{K}})$). Then, by (\ref{semistable:eqn:torsion}), the proposition amounts to showing that there is a canonical map
 \[
  T_p(Q_U)\to\Hom_{\mb{LDieu}}(\overline{S},\overline{U})\bigl(\Dieu(Q_{\overline{K}}),\mb{1}\bigr).
 \]

 Let $(Q',\lambda',\tau_0)$ be a tuple in $\widetilde{\mb{DD}}_{\on{pol}}(S,U)$ mapping to $(Q_U,\lambda_U)$ under the functor in (\ref{semistable:subsubsec:bsufunct}). Then, in the category of extensions of $Q^{\et}\otimes\Int_p$ by $T_p(Q^{\sab})$, $T_p(Q_U)$ is the Baer sum of $T_p(Q')=T_p(Q'[p^{\infty}])$ and the push-forward of $T_p(Q_{\tau_0})$ along $T_p(Q^{\mult})\to T_p(Q^{\sab})$.

 Now, there is a natural map:
 \[
  T_p(Q')=\Hom_{\Reg{\overline{K}}}(\Rat_p/\Int_p,Q'[p^{\infty}])\xrightarrow{\Dieu}\Hom_{\mb{LDieu}(\overline{S},\overline{U})}(\Dieu(Q'_{\overline{K}}),\mb{1}\bigr).
 \]

 So, given the construction of $\Dieu(Q)$ in (\ref{semistable:subsubsec:logffull}), it now suffices to construct a natural map:
 \begin{align}\label{semistable:eqn:qtau0torsion}
  T_p(Q_{\tau_0})\to\Hom_{\mb{LDieu}(\overline{S},\overline{U})}(\Dieu(Q_{\tau_0,\overline{K}}),\mb{1}).
 \end{align}

 By the construction of (\ref{semistable:subsubsec:ntorsion}), we have
 \[
 Q_{\tau_0}[p^n]=\mathcal{G}(\tau_0,n)\otimes\Int/p^n\Int.
 \]
 Therefore, each section of $Q_{\tau_0}[p^n]$ is the image of a section $(y_n,t_n)$ of $\mathcal{G}(\tau_0,n)$, which satisfies $f_{\tau_0}(y_n)=t_n^{p^n}$.

 Suppose that we are now given a log scheme $Z$ over $\overline{S}^{\log}$ in which $p$ is nilpotent, and an exact nilpotent thickening of log schemes $Z\into\widetilde{Z}$ equipped with divided powers. To construct the map (\ref{semistable:eqn:qtau0torsion}), it suffices, in the notation of (\ref{semistable:subsubsec:noabfcrys}), to construct a canonical section of $\mathcal{E}(\tau_0,Z\into\widetilde{Z})\otimes_{\Int}\Reg{\widetilde{Z}}$ associated with an element $(\alpha_n)\in T_p(Q_{\tau_0})$.

 For this, let $\mathcal{I}=\ker(\Reg{\widetilde{Z}}\to\Reg{Z})$, and let $r\in\Int_{\geq 1}$ be such that $\mathcal{I}^{r+1}=0$. Then, given any section $u$ of $Q^{\mult}_{\mathcal{I}}$, we will have $u^r=1$; cf.~\cite[1.1.1]{katz:serre-tate}. This in turn implies that, given any section $t$ of $\SHom(Q^{\mult,C}\vert_Z,\Mon_Z^{\gp})$ and any lift $\tilde{t}$ of it to a section of $\SHom(Q^{\mult,C}\vert_Z,\Mon_{\widetilde{Z}}^{\gp})$, the section $\tilde{t}^r$ is a canonically determined lift of $t^r$.

 Choose $m$ such that $p^m\geq r$ and such that $p^m\Reg{\widetilde{Z}}=0$. Choose a lift $(y_m,t_m)\in\mathcal{G}(\tau_0,m)(\overline{K})$ of $\alpha_m$. Then the image of $f_{\tau}(y_m)=t_m^{p^m}$ in $\Hom(Q^{\mult,C}\vert_Z,\Mon_Z^{\gp})$ has a canonical lift $\tilde{t}_m\in \Hom(Q^{\mult,C}\vert_Z,\Mon_{\widetilde{Z}}^{\gp})$. The pair $(y_m,\tilde{t}_m)$ is a section of $\mathcal{E}(\tau_0,Z\into\widetilde{Z})$. It is easily checked that its image in $\mathcal{E}(\tau_0,Z\into\widetilde{Z})\otimes\Reg{\widetilde{Z}}$ does not depend on the choice of $m$ or the lift $(y_m,t_m)$, and so is our desired canonical section.
\end{proof}

\subsubsection{}\label{semistable:subsubsec:periodrings}
We will now prove~\eqref{semistable:prp:comparison}. For this, we will need a certain larger ring $\widehat{B}_{\st}$ containing $\Bst$, whose construction we now recall. For details, cf.~\cite[\S 2]{breuil:griffiths}.

Let $A_{\cris}$ be the $p$-adically complete $W$-algebra defined in \cite[\S 2.3]{fontaine:periodes}: it is equipped with a $\Gamma_K$-action and a compatible surjection $\theta:A_{\cris}\to\Reg{\overline{K}}$, whose kernel is equipped with divided powers. There is a canonical Frobenius lift $\varphi:A_{\cris}\to A_{\cris}$. Moreover, if $\mathcal{R}=\lim_{x\mapsto x^p}\Reg{\overline{K}}/p$ is the perfect envelope of $\Reg{\overline{K}}/p$, there is a canonical $\Gamma_K$-equivariant Teichm\"uller lift $[\cdot]:\mathcal{R}\backslash\{0\}\to A_{\cris}\backslash\{0\}$.

Any coherent sequence $(\zeta_n)_{n\in\Int_{\geq 0}}$ of $p$-power roots of unity determines an element $\underline{\zeta}\in\mathcal{R}^\times$ and thus an element $[\underline{\zeta}]\in A_{\cris}^\times$. Similarly, a coherent sequence $(\pi^{1/p^n})=(\pi_n)_{n\in\Int_{\geq 0}}$ of $p$-power roots of $\pi$ determines an element $[\underline{\pi}]\in A_{\cris}$.

Let $\widehat{A}_{\st}$ be the $p$-adic completion of the divided power algebra $A_{\cris}\langle X\rangle$, so that we have:
\[
 \widehat{A}_{\st}=\bigl\{\sum_ia_i\frac{X^i}{i!}\in A_{\cris}[p^{-1}]\pow{X}:\;a_i\to 0\text{ $p$-adically}\bigr\}.
\]
There is a natural log structure $\on{M}_{\widehat{A}_{\st}}\to \widehat{A}_{\st}$, where $\on{M}_{\widehat{A}_{\st}}\subset \widehat{A}_{\st}$ is the monoid generated by $\widehat{A}_{\st}^\times$ and the image of $\mathcal{R}\setminus\{0\}$ under the Teichm\"uller lift. The action of $\Gamma_K$ extends to $\widehat{A}_{\st}$ via the formula: $g(1+X)=[\underline{\epsilon}(g)](1+X)$. Here, $\underline{\epsilon}(g)\in\mathcal{R}^\times$ arises from the coherent sequence of $p$-power roots of unity $(g(\pi_n)\pi_n^{-1})_n$. The Frobenius lift $\varphi$ on $A_{\cris}$ extends to $\varphi:\widehat{A}_{\st}\to \widehat{A}_{\st}$ satisfying $\varphi(1+X)=(1+X)^p$. Both these additional structures respect log structures.

The surjection $\theta:A_{\cris}\to\Reg{\overline{K}}$ extends to a map $\theta:\widehat{A}_{\st}\to\Reg{\overline{K}}$ satisfying $\theta(X)=0$; clearly, $\ker\theta\subset \widehat{A}_{\st}$ admits divided powers. We view $\widehat{A}_{\st}$ as an $\mathcal{S}_{\pi}$-algebra via the map $u\mapsto [\underline{\pi}](1+X)^{-1}$. All maps involved are compatible with log structures. Moreover, $\theta:\widehat{A}_{\st}\to\Reg{\overline{K}}$ is an exact thickening, in the sense that the induced map:
\[
 \Mon_{\widehat{A}_{\st}}^{\gp}/\widehat{A}_{\st}^{\times}\to\overline{K}^\times/\Reg{\overline{K}}^\times
\]
is an isomorphism.

$\widehat{A}_{\st}$ is equipped with a natural $A_{\cris}$-linear derivation $N$ satisfying $N(1+X)=(1+X)$. It is easily checked that we have $p\varphi N = N\varphi$ as endomorphisms of $\widehat{A}_{\st}$.

Now, $\widehat{B}^+_{\st}=\widehat{A}_{\st}[p^{-1}]$, and $\widehat{B}_{\st}=\widehat{B}^+_{\st}[t^{-1}]$, where $t=\log[\underline{\varepsilon}]\in A_{\cris}$ is Fontaine's cyclotomic period, associated with a coherent sequence $(\varepsilon_n)_n$ of primitive $n^{\text{th}}$-roots of unity.  There is a natural map $\Bst\to\widehat{B}_{\st}$ compatible with additional structures.

\subsubsection{}\label{semistable:subsubsec:m0Atriv}
Let $\mathcal{S}_{\pi}^{\log}=\mathcal{S}_{\pi}[p^{-1}][\ell_u]$ be the polynomial algebra over $\mathcal{S}_{\pi}[p^{-1}]$ in the variable $\ell_u$. It is equipped with a natural integrable connection $\nabla:\mathcal{S}_{\pi}^{\log}\to\mathcal{S}_{\pi}^{\log}\dlog(u)$ satisfying $\nabla(\ell_u)=\dlog(u)$, as well as a Frobenius lift $\varphi:\mathcal{S}_{\pi}^{\log}\to\mathcal{S}_{\pi}$ extending that on $\mathcal{S}_{\pi}$ and satisfying $\varphi(\ell_u)=p\ell_u$. The natural map $\mathcal{S}_{\pi}\to\Reg{Y(e)}^{\an}$ extends to a map $\mathcal{S}_{\pi}^{\log}\to\Reg{Y(e)}^{\an,\log}$ in the obvious way: $\ell_u$ is carried to $\ell_u$. This map is compatible with the $\varphi$-module structures, as well as the logarithmic connections.

From \cite[6.2.1.1]{breuil:griffiths}, we now obtain:
\begin{prp}\label{semistable:prp:breuil}
The isomorphism (\ref{semistable:eqn:separallel}) restricts to an isomorphism of $\mathcal{S}_{\pi}^{\log}$-modules:
\[
 \mathcal{S}_{\pi}^{\log}\otimes_{K_0}M_0(A)\xrightarrow{\simeq}\mathcal{S}_{\pi}^{\log}\otimes_{\mathcal{S}_{\pi}}\mathcal{M}(A).
\]
\end{prp}
\qed

\subsubsection{}\label{semistable:subsubsec:compproof}
The map $\mathcal{S}_{\pi}\to\widehat{A}_{\st}$ extends to a $\varphi$-equivariant map $\mathcal{S}_{\pi}^{\log}\to\widehat{B}^+_{\st}$ carrying $\ell_u$ to the element $\log\bigl(\frac{[\underline{\pi}]}{\pi}\bigr)-\log(1+X)$, where both logarithms are developed using the usual power series\footnote{Note: $\log\bigl(\frac{[\underline{\pi}]}{\pi}\bigr)=\log\bigl(1+(\frac{[\underline{\pi}]}{\pi}-1)\bigr)$.}. If we equip $\mathcal{S}_{\pi}^{\log}$ with the derivation $N=-\nabla(u\frac{d}{du})$, then this map is also compatible with derivations.

We again fix $(A,\psi)$ in $\mb{DEG}_{\pol}(S,U)$. We can evaluate $\Dieu(A_{\overline{K}})$ along the formal divided power thickening $\Spf\Reg{\overline{K}}\into\Spf\widehat{A}_{\st}$ to obtain an $\widehat{A}_{\st}$-module that is canonically isomorphic to $\widehat{A}_{\st}\otimes_{\mathcal{S}_{\pi}}\mathcal{M}(A)$. Therefore, from \eqref{semistable:prp:torsiondieu}, we obtain canonical maps of $\Rat_p$-vector spaces
\begin{align*}
T_p(A)[p^{-1}]\to\Hom_{\mathcal{S}_{\pi}}\bigl(\mathcal{M}(A),\widehat{B}_{\st}\bigr)\to\Hom_{\mathcal{S}_{\pi}^{\log}}\bigl(\mathcal{S}_{\pi}^{\log}\otimes_{\mathcal{S}_{\pi}}\mathcal{M}(A),\widehat{B}_{\st}\bigr)\to\Hom_{K_0}\bigl(M_0(A),\widehat{B}_{\st}\bigr).
\end{align*}
Here, the final map is obtained via (\ref{semistable:prp:breuil}).

Tensoring the maps with $\widehat{B}^+_{\st}$ and dualizing, we now obtain a map of $\widehat{B}^+_{\st}$-modules:
\[
\widehat{B}_{\st}\otimes_{K_0}M_0(A)\to \widehat{B}_{\st}\otimes_{\Rat_p}H^1_{\et}(A_{\overline{K}},\Rat_p).
\]
By construction, it is compatible with all additional structures. We claim that it is an isomorphism. To see this, we note that the map is compatible with weight filtrations on both sides, and we only have to check that it induces isomorphisms on the associated gradeds for the weight filtration. But this follows from \cite[Theorem 7]{faltings:very_ramified}.

The proof of (\ref{semistable:prp:comparison}) will be completed by:
\begin{prp}\label{semistable:prp:comp_finish}
\mbox{}
\begin{enumerate}
  \item\label{comp_finish:descent}The isomorphism above restricts to an isomorphism
  \[
\beta_{\st,A}:B_{\st}\otimes_{K_0}M_0(A)\to B_{\st}\otimes_{\Rat_p}H^1_{\et}(A_{\overline{K}},\Rat_p).
\]
  \item\label{comp_finish:commute}The diagram (\ref{semistable:diagram}) commutes.
\end{enumerate}
\end{prp}
\begin{proof}
Assertion~\eqref{comp_finish:descent} follows from \cite[Th\'eor\`eme 3.3]{breuil:griffiths}.

For~\eqref{comp_finish:commute}, we observe that the obstruction to is an automorphism $u$ of $\Bdr\otimes_{\Rat_p}H$, where $H\coloneqq H^1_{\et}(A_{\overline{K}},\Rat_p)$, that respects the Hodge filtration (arising from that on $\Bdr$), the weight filtration (arising from that on $H$), as well as the diagonal $\Gamma_K$-action. Moreover, the induced automorphism on the associated graded pieces of the weight filtration is the identity. In weights $0,2$, this is clear, and in weight $1$, this follows from the discussion in~\cite[\S~6]{faltings:very_ramified}; see pp.132--133 of~\emph{loc. cit.} In fact, this discussion shows more: Since our comparison isomorphism agrees with the classical crystalline comparison isomorphism for $p$-divisible groups, the automorphism induced by $u$ on $\Bdr\otimes_{\Rat_p}W_1H$ and $\Bdr\otimes_{\Rat_p}(H/W_0H)$ is the identity.

Therefore, $u-1$ factors through a $\Gamma_K$-equivariant map $\Bdr\otimes_{\Rat_p}\gr^W_2H\to\Bdr\otimes_{\Rat_p}W_0H$,
which respects Hodge filtrations. Giving such a map is equivalent to giving a map
\[
\gr^W_2H^1_{\dR}(A/K)\to W_0H^1_{\dR}(A/K)
\]
that respects Hodge filtrations. But any such map must be identically zero.
\end{proof}

\subsection{Extension lemmas}\label{appendix:subsec:extension}

\emph{Unless otherwise specified, all schemes and algebraic spaces in this sub-section will be normal, noetherian and flat over $\Int_{(p)}$.}

\subsubsection{}
Let $Z$ be a scheme admitting locally closed immersions $Z\into X_i$, for $i=1,2$, where $X_1$ and $X_2$ are algebraic spaces.

Suppose that there exists another scheme $Z'$, which also admits locally closed immersions $Z'\into X'_i$ satisfying the following property: For $i=1,2$, there exists a finite map $X_i\to X'_i$ such that
\[
 Z=(X_i\times_{X'_i}Z')_{\on{red}}.
\]

Assume that we have a commutative diagram of formal algebraic spaces:
\begin{diagram}
\bigl(X_1[p^{-1}]\bigr)^{\wedge}_{Z[p^{-1}]}&\rTo&\bigl(X_1\bigr)^{\wedge}_{Z}&\rTo&\bigl(X'_1\bigr)^{\wedge}_{Z'}\\
\dTo_{\simeq}^{f[p^{-1}]}&&&&\dTo_{\simeq}^{f'}\\
\bigl(X_2[p^{-1}]\bigr)^{\wedge}_{Z[p^{-1}]}&\rTo&\bigl(X_2\bigr)^{\wedge}_{Z}&\rTo&\bigl(X'_2\bigr)^{\wedge}_{Z'},
\end{diagram}
where the two vertical isomorphisms lift the identity on $Z[p^{-1}]$ (resp. $Z'$).

For every point $z\in Z(\overline{\Field}_p)$, with image $z'$ in $Z'$, let $R_i(z')$ (resp. $R_i(z)$) be the complete local ring of $X'_i$ (resp. $X_i$) at $z$ (resp. $z'$).

\begin{lem}\label{appendix:lem:formalext}
Suppose that, for every $z$, the finite map $R_1(z')\xrightarrow{\simeq}R_2(z')\to R_2(z)$ lifts to an isomorphism $R_1(z)\xrightarrow{\simeq} R_2(z)$. Then $f[p^{-1}]$ extends to an isomorphism of formal algebraic spaces
\[
 f:\bigl(X_1\bigr)^{\wedge}_Z\xrightarrow{\simeq}\bigl(X_2\bigr)^{\wedge}_Z
\]
\end{lem}
\begin{proof}
We can work formal affine locally on $\bigl(X'_1\bigr)^{\wedge}_{Z'}$, and so we can assume that we have:
\[
 \bigl(X'_1\bigr)^{\wedge}_{Z'}=\Spf(A,I_A)\xrightarrow[\simeq]{f'}\bigr(X'_2\bigr)^{\wedge}_{Z'};
\]
and
\[
 \bigl(X_i\bigr)^{\wedge}_Z=\Spf(B_i,I_{B_i}),
\]
for $i=1,2$. Furthermore, if $\widehat{B_i[p^{-1}]}$ is the $I_{B_i}$-adic completion of $B_i[p^{-1}]$, $f[p^{-1}]$ induces an isomorphism of adic rings:
\begin{align}\label{appendix:eqn:b1b2pinv}
 (\widehat{B_1[p^{-1}]},I_{B_1}\widehat{B_1[p^{-1}]})&\xrightarrow{\simeq}(\widehat{B_2[p^{-1}]},I_{B_2}\widehat{B_2[p^{-1}]}).
\end{align}

We have to show that the finite map $A\to B_2$ lifts to a map $B_1\to B_2$. By the $\Int_{(p)}$-flatness and noetherianity of $B_i$, and Krull's intersection theorem, the map $B_i\to\widehat{B_i[p^{-1}]}$ is injective.

Choose an element $b\in B_2$; via the isomorphism of~\eqref{appendix:eqn:b1b2pinv}, we can view $b$ as an element of $\widehat{B_1[p^{-1}]}$. We have to show that it actually lies in $B_1$. Since both $B_1$ and $B_2$ are finite over $A$, the sub-algebra $B_1[b]\subset\widehat{B_1[p^{-1}]}$ generated by $b$ is also finite over $B_1$. We have to show that the map $B_1\to B_1[b]$ is an isomorphism. It suffices to check this after completing $B_1$ at any maximal ideal, and here it follows from the hypothesis of the lemma.
\end{proof}

\subsubsection{}
We will consider the category $\mathcal{C}_1$ of pairs $\bigl(S,\{Z_c\}_{c\in\on{C}}\bigr)$, where:
\begin{itemize}
\item $S$ is a normal, noetherian algebraic space, flat and locally of finite type over $\Int_{(p)}$;
\item $(\on{C},\preceq)$ is a partially ordered finite set;
\item $\{Z_c\}$ is a collection of locally closed sub-spaces of $S$ that are also normal and flat over $\Int_{(p)}$, and that form a stratification for $S$. That is, we have a decomposition:
\[
 S=\bigsqcup_{c\in\on{C}}Z_c;
\]
and, for every $c\in\on{C}$, the sub-space $\bigsqcup_{c'\preceq c}Z_{c'}$ is the Zariski closure of $Z_c$ in $S$.
\end{itemize}

A map $(S,\{Z_c\}_{c\in\on{C}})\to (T,\{W_d\}_{d\in\on{D}})$ of such tuples is a pair $(f,\varphi)$, where $f:S\to T$ is a map of algebraic spaces, and $\varphi:\on{C}\to\on{D}$ is a map of partially ordered sets such that, for every $d\in\on{D}$, we have:
\[
 (f^{-1}(W_d))_{\on{red}}=\bigsqcup_{\varphi(c)=d}Z_c.
\]

Let $\mathcal{C}_2$ be the category whose objects are the same as those of $\mathcal{C}_1$, but in which a morphism between $(S,\{Z_c\}_{c\in\on{C}})$ and $(T,\{W_d\}_{d\in\on{D}})$ is instead a pair $(f[p^{-1}],\varphi)$, where:
\begin{itemize}
\item
\[
(f[p^{-1}],\varphi):(S[p^{-1}],\{Z_c[p^{-1}]\}_{c\in\on{C}})\to (T[p^{-1}],\{W_d[p^{-1}]\}_{d\in\on{D}})
\]
is a map in $\mathcal{C}_1$;
\item For every $c\in\on{C}$ with $d=\varphi(c)$, there exists a map of formal algebraic spaces $\widehat{S}_{Z_c}\to \widehat{T}_{W_d}$ extending the map
\[
(S[p^{-1}])^{\wedge}_{Z_c[p^{-1}]}\to(T[p^{-1}])^{\wedge}_{W_d[p^{-1}]}
\]
induced by $f[p^{-1}]$.
\end{itemize}

\begin{lem}\label{appendix:lem:abstract_hecke}
The natural functor $\mathcal{C}_1\to\mathcal{C}_2$ is an equivalence of categories.
\end{lem}
\begin{proof}
Let
\[
(f[p^{-1}],\varphi):(S,\{Z_c\}_{c\in\on{C}})\to (T,\{W_d\}_{d\in\on{D}})
\]
be a map in $\mathcal{C}_2$. We have to show that $f[p^{-1}]$ extends to a map $f:S\to T$. Using a simple induction, this reduces to the situation where $\on{C}$ is a two-element set $\{c_0,c_1\}$ with $c_0\preceq c_1$. Set $d_0=\varphi(c_0)$ and $d_1=\varphi(c_1)$.

Let $S'$ be the normalization of the Zariski closure in $S\times T$ of the graph of $f[p^{-1}]$. Our hypothesis implies that the open immersion $Z_{c_0}[p^{-1}]\into S[p^{-1}]$ extends to an open immersion $Z_{c_0}\into S'$, and, similarly, the closed immersion $Z_{c_1}[p^{-1}]\into S[p^{-1}]$ extends to a closed immersion $Z_{c_1}\into S'$. Therefore, $S'=Z_{c_0}\sqcup Z_{c_1}$, and the natural projection $S'\to S$ is a bijection on points. Moreover, the completion of $S'$ along $Z_{c_1}$ maps isomorphically onto that of $S$. This implies that $S'$ maps isomorphically to $S$, and so $f[p^{-1}]$ extends to all of $S$.
\end{proof}

\begin{lem}\label{appendix:lem:neronogg1motifs}
Let $\Reg{L}$ be a complete discrete valuation ring of mixed characteristic $(0,p)$. Suppose that $Q$ is a $1$-motif over $L$ such that, for some $\ell\neq p$, the $\ell$-adic Tate module $T_{\ell}(Q)$ is unramified as a representation of $\Gal(\overline{L}/L)$. Then $Q$ extends (uniquely) to a $1$-motif over $\Reg{L}$.
\end{lem}
\begin{proof}
We can assume that $\Reg{L}$ has algebraically closed residue field. Suppose that $Q=(B,Y,X,c,\dual{c},\tau)$. By the usual good reduction criterion for abelian varieties, we see that $B$ extends to an abelian scheme over $\Reg{L}$. By the properness of $B$ and $\dual{B}$, the maps $c$ and $\dual{c}$ also extend over $\Reg{L}$. Let $J$ be the extension of $B$ by $T=\SHom(X,\Gm)$ determined by $\dual{c}$. For every $x\in X$, viewed as a character of $T$, we obtain via push-forward an extension $J_x$ of $B$ by $\Gm$. We have to show that, for any $y\in Y$, the section $\tau(y,x)\in J_x(L)$ extends over $\Reg{L}$.

For each $n\in\Int_{\geq 1}$, consider the connecting map $\partial_n:J_x(L)\to H^1(L,J_x[\ell^n])$ in the Kummer long exact sequence. Our hypothesis implies that $\partial_n(\tau(y,x))$ is trivial, for all $n$. Therefore, $\tau(y,x)$ is in the image of $[\ell^n]:J_x(L)\to J_x(L)$. To finish, it is enough to show that the inclusion
\[
 J_x(\Reg{L})\subset\bigcap_{n}\on{im}([\ell^n]:J_x(L)\to J_x(L))
\]
is bijective.

Choose an element $j$ of the right-hand side with image $b\in B(L)=B(\Reg{L})$; choose also $j'\in J(\Reg{L})$ lifting $b$. Then $j=tj'$, for a unique $t\in L^\times$. For any $n\geq 1$, choose $j_n\in J(L)$ such that $[\ell^n]j_n=j$. Let $b_n\in B(\Reg{L})$ be its image. Then we can find $j'_n\in J(\Reg{L})$ with image $b_n\in B(\Reg{L})$ and satisfying $[\ell^n]j'_n=j'$. In particular, for all $n\geq 1$, there exists $t_n\in L^\times$ such that $j_n=t_nj'_n$ and $t_n^{\ell^n}=t$. This implies that $t\in \Reg{L}^\times$, and hence $j\in J(\Reg{L})$.
\end{proof}

\begin{lem}\label{appendix:lem:1motifspurity}
Suppose that $S$ is a regular, formally smooth scheme over $\Int_{(p)}$. Let $U\subset S$ be an open sub-scheme containing $S[p^{-1}]$, whose complement has co-dimension at least $2$ in $S$. Then any $1$-motif over $U$ extends uniquely to one over $S$.
\end{lem}
\begin{proof}
Let $(B,Y,X,c,\dual{c},\tau)$ be a $1$-motif over $U$. By~\cite[Corollary 5]{vasiu_zink}, $B$ extends to an abelian scheme over $S$. The rest of the data also extends by Weil's extension theorem~\cite[\S~4.4, Theorem 1]{blr}: Given a smooth group scheme $H$ over $S$, any section of $H$ over $U$ extends to one over $S$.
\end{proof}

\subsubsection{}
Suppose that $S$ is a flat normal $\Int_{(p)}$-scheme, and that $(A_i,\lambda_i)$ for $i=1,2,3$ are three polarized abelian schemes over $S[p^{-1}]$. Suppose that there exist finite homomorphisms $f_i:A_i\to A_1$, for $i=1,2$ that are compatible with polarizations, and whose kernels are finite flat group schemes over $S[p^{-1}]$ of prime-to-$p$ order. 

Suppose also that the product homomorphism
\[
\beta:A_2\times A_3 \xrightarrow{f_2\times f_3}A_1
\]
is an isogeny, and that $A_1$ extends to an abelian scheme $\mathcal{A}_1$ over $S$.

\begin{lem}
\label{appendix:lem:abelian_scheme_extension}
In the above situation, suppose that $\lambda_1$ has prime-to-$p$ degree. Then both $A_2$ and $A_3$ also extend to abelian schemes $\mathcal{A}_2$ and $\mathcal{A}_3$ over $S$, and the homomorphisms $f_2,f_3$ extend to finite homomorphisms $\mathcal{A}_i\to \mathcal{A}_1$ for $i=2,3$.
\end{lem}
\begin{proof}
We can assume that $S$ is connected. This implies that the degrees of the polarizations $\lambda_i$ are constant. Let $d_i$ be the degree of the polarization $\lambda_i$ on $A_i$, for $i=1,2,3$. 

Our hypothesis on the degree of $\lambda_1$ and its compatibility with $\lambda_2$ and $\lambda_3$ along the finite homomorphisms $f_2$ and $f_3$ imply that the integers $d_i$ are prime-to-$p$ for $i=1,2,3$.

Moreover, since $\beta$ respects polarizations, the $p$-primary part $\ker\beta[p^{\infty}]$ of $\ker\beta$ must map isomorphically onto a subgroup scheme of both $\ker\lambda_1$ and $\ker\lambda_2$ via the natural projections. Write $\widetilde{A}$ for the quotient $(A_2\times A_3)/\ker\beta[p^\infty]$: this is a finite cover of $A_1$ of prime-to-$p$ degree. Let $\tilde{d}$ be the degree of the induced polarization on $\widetilde{A}$. 

Let $g_1$ (resp. $g_2$, $g_3$) be the dimension of the fibers of $A_1$ (resp. $A_2$, $A_3$). Let $\mathsf{A}_{g,\tilde{d}}$ be the moduli stack over $\Int_{(p)}$ of polarized abelian varieties of dimension $g$ and degree $\tilde{d}$.

Let $\mathsf{B}$ be the moduli stack over $\Int_{(p)}$ that, for each $\Int_{(p)}$-scheme $T$, parameterizes tuples $((B_2,\psi_2),(B_3,\psi_3),\varepsilon)$,
where, for $i=2,3$, $(B_i,\psi_i)$ is a polarized abelian scheme over $T$ of dimension $g_i$ and degree $d_i$, and
\[
\varepsilon:\ker\psi_2[p^\infty]\xrightarrow{\simeq}\ker\psi_3[p^\infty]
\]
is an isomorphism of $T$-group schemes such that the product polarization $\psi_2\times\psi_3$ on $B_2\times_TB_3$ descends to a polarization $\psi$ on $B=(B_2\times_TB_3)/L$, where $L=(1\times\varepsilon)(\ker\psi_2[p^\infty])$, which will necessarily have degree $\tilde{d}$. This is again an algebraic stack over $\Int_{(p)}$.

By definition, we now have a map of $\Int_{(p)}$-stacks
\begin{align}\label{appendix:eqn:stacks_finite}
\mathsf{B}&\to \mathsf{A}_{g,\tilde{d}}\\
\bigl((B_2,\psi_2),(B_3,\psi_3),\varepsilon\bigr)&\mapsto(B,\psi).\nonumber
\end{align}
Note that both $B_2$ and $B_3$ are abelian subschemes of $B$.

We claim that this map is finite. Since it is of finite type, it suffices to show that the map is quasi-finite and satisfies the valuative criterion for properness. The first property follows from the argument in~\cite{lenstra_oort_zarhin}, and the second follows from the Ner\'on-Ogg-Shafarevich criterion for good reduction.

The polarized abelian scheme $\widetilde{A}\to S[p^{-1}]$ is associated with a map
\[
f:S[p^{-1}]\to\mathsf{A}_{g,\tilde{d}}. 
\]
The discussion above shows that $f$ lifts to a map $S[p^{-1}]\to\mathsf{B}$. 

The hypothesis that $A_1$ extends to an abelian scheme over $S$, combined with the fact that the isogeny $\tilde{\beta}:\widetilde{A}\to A_1$ has prime-to-$p$ degree implies that $\widetilde{A}$ extends to a polarized abelian scheme over $S$, and hence that $f$ extends to a map $S\to\mathsf{A}_{g,\tilde{d}}$. 

So, to prove the proposition, it is enough to show that this extension also admits a lift $S\to\mathsf{B}$. But this follows from the finiteness of~\eqref{appendix:eqn:stacks_finite} and the normality of $S$.
\end{proof}

\subsection{Rigid analytic lemmas}\label{appendix:subsec:rigid}
The purpose of this part of the appendix is to collect some results about rigid analytic spaces obtained via Berthelot's analytification functor.

\subsubsection{}\label{rigid:subsubsec:berthelot}
Let $R$ be a complete local normal ring, topologically of finite type over $W$. Set $\widehat{\mathcal{U}}=\Spf R$. We have the attached analytic space $\widehat{\mathcal{U}}^{\an}$ over $K_0$ as defined in \cite[\S 7.1]{dejong:formal_rigid}\footnote{This space is denoted by $\widehat{\mathcal{U}}^{\on{rig}}$ in \emph{loc. cit.}}: Let $\mx\subset R$ be the maximal ideal, and let $R[\mx^n/p]$ be the $R$-sub-algebra of $R\bigl[p^{-1}\bigr]$ generated by the elements $a/p$, for $a\in\mx^n$. Let $B_n$ be the $\mx$-adic completion of $R[\mx^n/p]$ and set $C_n=B_n\bigl[p^{-1}\bigr]$. Then $C_n$ is an affinoid algebra over $K_0$. If $X_n=\on{Sp}(C_n)$ is the affinoid space over $K_0$ associated with $C_n$, the natural map $X_n\to X_{n+1}$ induced by the inclusion $R[\mx^{n+1}/p]\to R[\mx^n/p]$ exhibits $X_n$ as an affinoid open in $X_{n+1}$. We now have:
\[
 \widehat{\mathcal{U}}^{\an}=\bigcup_nX_n.
\]

There is a natural map of $W$-algebras: $R\to H^0\bigl(\widehat{\mathcal{U}}^{\an},\Reg{\widehat{\mathcal{U}}^{\an}}\bigr)$. By~\cite[7.3.6]{dejong:formal_rigid}, this map produces an identification:
\begin{align}\label{rigid:eqn:rnormal}
R=\{f\in H^0\bigl(\widehat{\mathcal{U}}^{\an},\Reg{\widehat{\mathcal{U}}^{\an}}\bigr): \vert f(x)\vert_p\leq 1\text{, for all $x\in\widehat{\mathcal{U}}^{\an}(\overline{K}_0)$}\}.
\end{align}

\begin{lem}\label{rigid:lem:units}
Let $Q(R)$ be the fraction field of $R$, and suppose that $f\in Q(R)$ is such that there exists a non-empty Zariski open sub-space $V\subset\widehat{\mathcal{U}}^{\an}$ with $\vert f(x)\vert_p=1$, for all $x\in V(\overline{K}_0)$. Then $f$ belongs to $R^\times$.
\end{lem}
\begin{proof}
Write $f=b_1/b_2$, for non-zero $b_1,b_2\in R$. Let $\mf{p}\subset R$ be a height $1$ prime with $p\notin\mf{p}$, and suppose that $\ord_{\mf{p}}(b_1)\geq\ord_{\mf{p}}(b_2)$. Let $V(\mf{p})\subset\widehat{\mathcal{U}}^{\an}$ be the Zariski closed subspace associated with $\mf{p}$. Then $f=b_1/b_2$ is defined on a non-empty open of $V(\mf{p})$. For any point in this open, $\vert f(x)\vert_p=1$, which implies that $\ord_{\mf{p}}(b_1)=\ord_{\mf{p}}(b_2)$. Arguing similarly with $b_2/b_1$, we see that $\ord_{\mf{p}}(f)=1$, for all such height $1$ primes.

Since $R$ is normal, this implies that $f$ belongs to $R[p^{-1}]^\times$, with $\vert f(x)\vert_p=1$ everywhere. Combining with~\eqref{rigid:eqn:rnormal} shows that $f$ must in fact belong to $R^\times$.
\end{proof}

From the definitions, we obtain:
\begin{lem}\label{rigid:lem:dominant}
Suppose that we have an injective map $R\to R'$ of complete local normal $W$-algebras, topologically of finite type over $W$. Set $\widehat{\mathcal{U}}'=\Spf R'$. Then the induced map
\[
 H^0\bigl(\widehat{\mathcal{U}}^{\an},\Reg{\widehat{\mathcal{U}}^{\an}}\bigr)\to H^0\bigl(\widehat{\mathcal{U}}^{',\an},\Reg{\widehat{\mathcal{U}}^{',\an}}\bigr)
\]
is again injective.
\end{lem}
\qed

\begin{lem}\label{rigid:lem:logarithm}
Let $R\to R'$ be as above, and let $\mx'\subset R'$ be the maximal ideal. 
\begin{enumerate}
\item\label{log:log}Suppose that $f\in 1+\mx'$ is such that the logarithm:
\[
 \log(f)=\sum_{i=1}^\infty(-1)^{i+1}\frac{(f-1)^i}{i}\in H^0\bigl(\widehat{\mathcal{U}}^{',\an},\Reg{\widehat{\mathcal{U}}^{',\an}}\bigr)
\]
is integral over $H^0\bigl(\widehat{\mathcal{U}}^{\an},\Reg{\widehat{\mathcal{U}}^{\an}}\bigr)$. Suppose that there exist a finitely generated $R$-algebra $S$, a maximal ideal $\mx_S\subset S$, an element $g\in S$, and an isomorphism of local $R$-algebras $\widehat{S}_{\mx_S}\xrightarrow{\simeq}R'$ carrying $g$ to $f$. Then $f$ is integral over $R$.
\item\label{log:power}If $\log(f)$ belongs to $H^0\bigl(\widehat{\mathcal{U}}^{\an},\Reg{\widehat{\mathcal{U}}^{\an}}\bigr)$, then there exists $r\in\Int_{>0}$ such that $f^r$ belongs to $1+\mx\subset R$.
\item\label{log:quotient}If $f$ belongs to the quotient field $Q(R)$ of $R$, then it lies in $1+\mx$.
\end{enumerate}
\end{lem}
\begin{proof}
For all these statements, we can, without loss of generality, replace $S$ by the normalization of the image in $R'$ of the map of $R$-algebras:
\begin{align*}
R[T]&\to S\\
T&\mapsto g-1;
\end{align*}
and $R'$ by the complete local ring of this normalization at the pre-image of $\mx_S$.

Now,~\eqref{log:log} amounts to the assertion that $R'$ is finite over $R$. This can be rephrased as follows: Let $R_1$ be the integral closure of $R$ in $S$. Since it is a domain, finite over the complete local ring $R$, it has to be local as well. We have to show that the map of local $R$-algebras $R_1\to S_{\mx_S}$ is an isomorphism; or, equivalently, that $g$ belongs to $R_1$.

By hypothesis, $\widehat{\mathcal{U}}^{',\an}\to\widehat{\mathcal{U}}^{\an}$ factors through a finite map $V\to\widehat{\mathcal{U}}^{\an}$ of rigid spaces such that $\log(f)$ lies in $H^0(V,\Reg{V})$. For any $n\geq 1$, let $X_n\subset\widehat{\mathcal{U}}^{\an}$ be the affinoid open defined in~\eqref{rigid:subsubsec:berthelot}. The pre-image $Y_n\subset V$ of $X_n$ is again affinoid. Since $g=\log(f)$ is an analytic function on $Y_n$, it is necessarily bounded, and so there exists $r\in\Int_{>0}$ such that $\vert p^r g\vert_p< p^{-\frac{1}{p-1}}$ everywhere on $Y_n$. In particular, the power series
\[
 f_r\coloneqq \exp(p^r g) = \sum_{j=0}^\infty\frac{(p^rg)^j}{j!}
\]
converges to a function on $Y_n$ satisfying $\vert f_r -1\vert_p$ everywhere, and $\log(f_r) = p^r\log(f)$. For $s\in\Int_{\geq 1}$ sufficiently large, we will now have $f_r^{p^s} = f^{p^{r+s}}$ over the pre-image in $\widehat{\mathcal{U}}^{',\an}$ of $Y_n$. Therefore, the pre-image in $\widehat{\mathcal{U}}^{',\an}$ of $X_n$ is finite over $X_n$.

Using~\cite[7.1.9]{dejong:formal_rigid} and the noetherianity of $R$, we now find that the localization of $R'[p^{-1}]$ at any maximal ideal $\mf{p}\subset R[p^{-1}]$ is finite over $R[p^{-1}]_{\mf{p}}=R_{\mf{p}}$. Therefore, $(\widehat{S}_{\mx_S})_{\mf{p}}$, and hence its sub-algebra $(S_{\mx_S})_{\mf{p}}$, has to be finite over $R_{\mf{p}}$. This implies that the map of $R$-algebras $R_1[p^{-1}]\to S_{\mx_S}[p^{-1}]$ is an isomorphism, since it is one after localizing at every maximal ideal of $R[p^{-1}]$. So we find that the element $g\in S_{\mx_S}$ lies in $R_1[p^{-1}]\cap S_{\mx_S}$.

Let $\widehat{\mathcal{U}}_1^{\an}$ be the rigid analytic space associated with $R_1$. The map $\widehat{\mathcal{U}}^{',\an}\to\widehat{\mathcal{U}}^{\an}$ factors through $\widehat{\mathcal{U}}_1^{\an}$. Now, $g$ is a bounded analytic function on $\widehat{\mathcal{U}}_1^{\an}$ that restricts to $f$ over $\widehat{\mathcal{U}}^{',\an}$. But, by what we have just seen, the map $\widehat{\mathcal{U}}^{',\an}\to\widehat{\mathcal{U}}_1^{\an}$ is surjective on $\overline{K}_0$-points. Therefore, since $\vert f-1\vert_p<1$ holds everywhere on $\widehat{\mathcal{U}}^{\an}(\overline{K}_0)$, we must have $\vert g-1\vert_p<1$ everywhere on $\widehat{\mathcal{U}}_1^{\an}(\overline{K}_0)$. By~\eqref{rigid:eqn:rnormal}, it now follows that $g$ must belong to $R_1^\times$. 

Assertion~\eqref{log:quotient} is now immediate. It remains to prove~\eqref{log:power}. Let $R_2\subset R'$ be the $R$-sub-algebra generated by $f$. By~\eqref{log:log}, it is finite over $R$. Moreover, the argument above with the $p$-adic exponential shows that, for all maximal ideals $\mf{p}\subset R[p^{-1}]$, there exists $s\in\Int_{>0}$ such that the image of $f^{p^s}$ in $(R_2)_{\mf{p}}$ lies in $R_{\mf{p}}$. This shows that if $s$ is sufficiently divisible, then we will have $f^{p^s}\in R[p^{-1}]$, and hence in $R$.
\end{proof}

\printbibliography

\end{document}